\documentclass{amsart}
\pdfoutput=1

\title[On the geometry of the Humbert surface of square discriminant]{On the geometry of the Humbert \\surface of square discriminant}

\author{Sam Frengley} 
\address{School of Mathematics, University of Bristol, Bristol, BS8 1UG, UK}
\email{sam.frengley@bristol.ac.uk}
\urladdr{\url{https://samfrengley.github.io/}}

\date{19 August 2024}

\usepackage[a4paper,margin=3cm,marginparwidth=2.5cm]{geometry}

\usepackage{colonequals}    
\usepackage{amsmath}        
\usepackage{amssymb}        
\usepackage{amsthm}         
\usepackage{amsfonts}       
\usepackage{enumitem}       
\usepackage[unicode,pagebackref=true,hypertexnames=true,plainpages=false,bookmarksdepth=3,breaklinks=true,pdfusetitle=true]{hyperref}

\usepackage{mathrsfs}
\usepackage{graphicx}
\usepackage{caption}
\usepackage{mathtools}
\usepackage{booktabs}
\usepackage{listings}
\usepackage{color}
\usepackage{subcaption}
\usepackage{float}
\usepackage{etoolbox}
\usepackage[utf8]{inputenc}

\usepackage[noabbrev,nameinlink,capitalise]{cleveref}

\usepackage{tikz}
\usepackage{tikz-cd}
\usetikzlibrary{decorations.pathreplacing,calligraphy}

\usepackage{array}              
\usepackage{makecell}           
\usepackage{adjustbox}          
\usepackage{diagbox}
\usepackage{colortbl}
\usepackage{multirow}
\usepackage{arydshln}
\usepackage{placeins}

\newcommand{\PreserveBackslash}[1]{\let\temp=\\#1\let\\=\temp}
\newcolumntype{C}[1]{>{\PreserveBackslash\centering}p{#1}}

\newcommand{\enabc}{\textnormal{(\alph*)}}
\newcommand{\eniii}{\textnormal{(\roman*)}}
\newcommand{\enABC}{\textnormal{(\Alph*)}}

\newtheorem{theorem}{Theorem}
\numberwithin{theorem}{section}
\newtheorem{prop}[theorem]{Proposition}
\newtheorem{coro}[theorem]{Corollary}  
\newtheorem{lemma}[theorem]{Lemma}
\newtheorem{conj}[theorem]{Conjecture}

\newtheorem*{theorem*}{Theorem}
\newtheorem*{lemma*}{Lemma}
\newtheorem*{prop*}{Proposition}
\newtheorem*{coro*}{Corollary}  
\theoremstyle{definition}\newtheorem{defn}[theorem]{Definition}
\theoremstyle{definition}\newtheorem{remark}[theorem]{Remark}
\theoremstyle{definition}\newtheorem{notation}[theorem]{Notation}
\theoremstyle{definition}\newtheorem{construction}[theorem]{Construction}
\theoremstyle{definition}\newtheorem*{defn*}{Definition}  

\crefname{prop}{Proposition}{Propositions}
\crefname{coro}{Corollary}{Corollaries}
\crefname{lemma}{Lemma}{Lemmas}
\crefname{conj}{Conjecture}{Conjectures}
\crefname{theorem}{Theorem}{Theorems}
\crefname{question}{Question}{Questions}
\crefname{defn}{Definition}{Definitions}
\crefname{remark}{Remark}{Remarks}
\crefname{notation}{Notation}{Notation}
\newcommand{\case}[1]{\par\vspace{2mm}\noindent\underline{#1:}\,}

\newcommand{\LMFDBLabelMC}[1]{%
  \begingroup
  \hypersetup{
    hidelinks
  }
  \textnormal{\href{https://beta.lmfdb.org/ModularCurve/Q/#1/}{\texttt{#1}}}%
  \endgroup
}

\def\MR#1{\href{http://www.ams.org/mathscinet-getitem?mr=#1}{MR#1}} 

\newcommand{\bibtitleref}[2]
{\hypersetup{urlbordercolor=0.8 1 1}%
  \href{#1}{#2}%
  \hypersetup{urlbordercolor=cyan}%
}
\DeclareFontFamily{U}{wncy}{}
\DeclareFontShape{U}{wncy}{m}{n}{<->wncyr10}{}
\DeclareSymbolFont{mcy}{U}{wncy}{m}{n}
\DeclareMathSymbol{\Sha}{\mathord}{mcy}{"58} 

\newcommand{\bbC}{\mathbb{C}}
\newcommand{\bbF}{\mathbb{F}}
\newcommand{\bbH}{\mathbb{H}}

\newcommand{\bbP}{\mathbb{P}}
\newcommand{\bbQ}{\mathbb{Q}}

\newcommand{\bbZ}{\mathbb{Z}}
\newcommand{\bbA}{\mathbb{A}}

\newcommand{\ffb}{\mathfrak{b}}
\newcommand{\ffa}{\mathfrak{a}}

\newcommand{\overbar}[1]{\mkern 1.5mu\overline{\mkern-1.5mu#1\mkern-1.5mu}\mkern 1.5mu} 
\newcommand{\Kbar}{\overbar{K}}

\newcommand{\ee}{\varepsilon}
\newcommand{\ffz}{\mathfrak{z}}
\newcommand{\ffj}{\mathfrak{j}}
  
\DeclareMathOperator{\Gal}{Gal}

\DeclareMathOperator{\End}{End}

\DeclareMathOperator{\Isom}{Isom}

\DeclareMathOperator{\Aut}{Aut}

\DeclareMathOperator{\GL}{GL}
\DeclareMathOperator{\SL}{SL}

\DeclareMathOperator{\PGL}{PGL}

\DeclareMathOperator{\sm}{sm}
\DeclareMathOperator{\NS}{NS}

\DeclareMathOperator{\ns}{ns}
\newcommand{\nsalt}{\overbar{\ns}}
\DeclareMathOperator{\s}{s}
\DeclareMathOperator{\borel}{b}
\DeclareMathOperator{\antidiag}{c}
\DeclareMathOperator{\weyl}{w}

\newcommand{\Isharp}{I^\sharp}
\newcommand{\borelsharp}{\borel^\sharp}

\newcommand{\Xred}[1]{\mathbf{X}(#1)}
\newcommand{\Yred}[1]{\mathbf{Y}(#1)}

\newcommand{\shortminus}{\scalebox{0.9}[1.00]{-}}

\newcommand{\apair}[2]{\left\langle #1 , #2 \right\rangle}

\newcommand{\tblacksquare}{{\vcenter{\hbox{\scalebox{0.5}{$\blacksquare$}}}}}

\newcommand{\ZNr}[2]{Z_{#1,#2}}

\newcommand{\ZNrtil}[2]{\widetilde{Z}_{#1,#2}}
\newcommand{\ZNro}[2]{Z^\circ_{#1,#2}}

\newcommand{\ZNrSym}[2]{Z^{\textnormal{sym}}_{#1,#2}}
\newcommand{\ZNrSymtil}[2]{\widetilde{Z}^{\textnormal{sym}}_{#1,#2}}

\newcommand{\YNr}[2]{\ZNr{#1}{#2}}
\newcommand{\YNrSym}[2]{\ZNrSym{#1}{#2}}

\newcommand{\WNr}[2]{W\!_{#1,#2}}

\newcommand{\WNro}[2]{W_{#1,#2}^{\mathsf{small}}}
\newcommand{\WNrtiny}[2]{W_{#1,#2}^{\hspace{0.08em}\mathsf{tiny}}}
\newcommand{\WNrmin}[2]{W^\mathsf{min}_{#1,#2}}

\newcommand{\KWo}{K_{\mathsf{small}}}

\newcommand{\Zone}{Z_{1}}

\newcommand{\Ram}{\boldsymbol{F}}
\newcommand{\Ramtil}{\boldsymbol{F}\!_{\tilde{Z}}}
\newcommand{\Ramo}{\boldsymbol{F}\!_{Z^\circ}}
\newcommand{\Ramotil}{\widetilde{\boldsymbol{F}}\!_{Z^\circ}}

\newcommand{\oo}{{\mathsf{small}}}
\newcommand{\minn}{{\mathsf{min}}}
\newcommand{\Ngon}[1][N]{\mathcal{E}_{#1}}

\newcommand{\UpsN}[1]{\Upsilon\!_{#1}}

\newcommand{\err}[2]{\mathfrak{e}{(#1, #2)}}
\newcommand{\DiagDiv}{\mathfrak{D}}
\newcommand{\three}[2]{[ #1, #2 ]}

\usepackage{bbm}
\newcommand{\indicator}[1]{\mathbbm{1}_{#1}}
\newcommand{\emid}{\mathrel{\Vert}}

\newcommand{\mymat}[4]{\ensuremath{\left( \begin{smallmatrix} #1 & #2 \\ #3 & #4 \end{smallmatrix} \right)}}

\newcommand{\myspecialmat}[4]{%
  \ensuremath{%
    \left(
      \begin{smallmatrix}
        #1 \hspace{0.5mm} & #2 \\[1mm]
        #3 \hspace{0.5mm} \vphantom{{2^2}^2} & #4 \\[1mm]
      \end{smallmatrix}
    \right)
  }
}

\usepackage{listofitems}
\setsepchar{,}

\def\labelformodcrv#1{%
  \readlist*\mylist{#1}%
  \ifnum\mylistlen=1%
    \mylist[1]%
  \else%
    ( {\foreachitem\x\in\mylist[]{\ifnum\xcnt=1 \else , \fi \x}} )%
  \fi
}

\newrobustcmd{\modcrvlabel}[1]{\labelformodcrv{#1}}

\newcommand{\Hg}[1]{{H}_{\modcrvlabel{#1}}}
\newcommand{\Hgplus}[1]{{H}^+_{\modcrvlabel{#1}}}
\newcommand{\Xg}[1]{{X}_{\modcrvlabel{#1}}}
\newcommand{\Xgplus}[2][]{{X}^+_{{#1}\modcrvlabel{#2}}}
\newcommand{\calXgplus}[2][]{{\mathcal{X}}^+_{{#1}\modcrvlabel{#2}}}

\newcommand{\Fgplus}[2][]{{F}^+_{{#1}\modcrvlabel{#2}}}

\newcommand{\Fgplustil}[2][]{\widetilde{{F}}^+_{{#1}\modcrvlabel{#2}}}
\newcommand{\calFgplustil}[2][]{\widetilde{\mathcal{F}}^+_{{#1}\modcrvlabel{#2}}}

\newcommand{\Xliftnum}[1]{X_0^{\mathsf{lift}}(#1)}

\newcommand{\Fliftnum}[2]{F^{\mathsf{lift}}_{#1, #2}}

\newcommand{\Fliftnumtil}[2]{\widetilde{F}^{\mathsf{lift}}_{#1, #2}}

\newcommand{\decorN}[1]{\widehat{#1}}

\newcommand{\LamN}[1]{\Lambda(#1)}
\newcommand{\LamNr}[2]{\Lambda(#1,#2)}

\newcommand{\ratcrit}[1]{\Cref{lemma:rat-crit}}

\DeclareCaptionFormat{cont}{#1 (continued)#2#3\par}


\begin{document}

\begin{abstract}
  For every positive integer $N$ we determine the Enriques--Kodaira type of the Humbert surface of discriminant $N^2$ which parametrises principally polarised abelian surfaces that are $(N,N)$-isogenous to a product of elliptic curves. A key step in the proof is to analyse the fixed point locus of a Fricke-like involution on the Hilbert modular surface of discriminant $N^2$ which was studied by Hermann and by Kani and Schanz. To this end, we construct certain ``diagonal'' Hirzebruch--Zagier divisors which are fixed by this involution. In our analysis we obtain a genus formula for these divisors, which includes the case of modular curves associated to (any) extended Cartan subgroup of $\GL_2(\bbZ/N\bbZ)$ and which may be of independent interest.
\end{abstract}

\maketitle

\vspace{-2mm}
\begingroup
\hypersetup{hidelinks}
\tableofcontents
\endgroup
\vspace{-2mm}

\section{Introduction}
\label{sec:introduction}
Let $N$ be a positive integer and let $\mathcal{A}_2$ denote the (coarse) moduli space of principally polarised abelian surfaces. The \emph{Humbert surface of discriminant $N^2$} is the surface $\mathcal{H}_{N^2} \subset \mathcal{A}_2$ parametrising principally polarised abelian surfaces that are $(N,N)$-isogenous to a product of elliptic curves (equipped with the product polarisation). The surfaces $\mathcal{H}_{N^2}$ also have a natural interpretation as a moduli space of curves. More precisely, let $\mathcal{M}_2$ denote the (coarse) moduli space of genus $2$ curves and embed $\mathcal{M}_2 \hookrightarrow \mathcal{A}_2$ via the Torelli morphism. The locus $\mathcal{H}_{N^2} \cap \mathcal{M}_2$ parametrises genus $2$ curves $C/\bbC$ admitting a morphism $\psi \colon C \to E$ of degree $N$, where $E/\bbC$ is an elliptic curve and $\psi$ does not factor through a non-trivial isogeny $E' \to E$.

Extending work of Hermann~\cite{H_SMDDp2} (who considered the case where $N$ is a prime number) we classify the Enriques--Kodaira type of the Humbert surfaces $\mathcal{H}_{N^2}$.

\begin{theorem}
  \label{coro:humbert-geom}
  The Humbert surface $\mathcal{H}_{N^2}$ of discriminant $N^2$ is birational over $\bbC$ to a smooth projective algebraic surface which is:
  \begin{enumerate}[label=\eniii]
  \item
    rational if $N \leq 16$ or if $N = 18, 20, 24$,
  \item
    an elliptic K3 surface if $N = 17$,
  \item
    an elliptic surface of Kodaira dimension $1$ if $N = 19$, $21$, and
  \item
    of general type if $N \geq 22$ and $N \neq 24$.
  \end{enumerate}
\end{theorem}

In fact, we prove the more general \Cref{thm:WNr-KD}. For each $r \in (\bbZ/N\bbZ)^\times$, we define the \emph{Hilbert modular surface} $\YNr{N}{r}$ (of discriminant $N^2$) to be the (coarse) moduli space whose $K$-points parametrise triples $(E, E', \phi)$ where $E/K$ and $E'/K$ are elliptic curves and $\phi \colon E[N] \cong E'[N]$ is an \emph{$(N,r)$-congruence}. That is, $\phi$ is an isomorphism of $\Gal(\overbar{K}/K)$-modules such that
\begin{equation}
  \label{eqn:power-cong}
  e_{E,N}(P, Q)^r = e_{E', N}(\phi(P), \phi(Q))
\end{equation}
where $e_{E,N} \colon E[N] \times E[N] \to \mu_N$ denotes the $N$-Weil pairing. The element $r$ is said to be the \emph{power} of the $N$-congruence $\phi$, and we call $\phi$ an $(N,r)$-congruence. More precisely, $\ZNr{N}{r}$ is defined to be a certain compactification of this moduli space, see \Cref{sec:surf-ZNr}.

The problem of determining those pairs $(N,r)$ for which there exists an $(N,r)$-congruence between elliptic curves defined over $\bbQ$ remains difficult and open. It is a conjecture of Frey and Mazur that there exist no non-isogenous $N$-congruent elliptic curves over $\bbQ$ for sufficiently large integers $N$ (Fisher has conjectured that this should hold for prime $N > 17$~\cite[Conjecture~1.1]{F_OPO17CEC}). Progress has been obtained by Bakker and Tsimerman~\cite{BT_pTMROECOGFF} who proved a (non-effective) version of the Frey--Mazur conjecture for elliptic curves over geometric function fields when $N$ is prime.

The Hilbert modular surfaces $\YNr{N}{r}$ were placed in the Enriques--Kodaira classification by Kani--Schanz~\cite[Theorem~4]{KS_MDQS} and Hermann\footnote{As pointed out in \cite{KS_MDQS} Hermann's list of elliptic surfaces is incomplete, missing the case when $(N,r) = (10,3)$. This is remedied in \cite{KS_MDQS}.}~\cite[Satz~2]{H_MQD} who proved the following theorem. Note that if $k$ is an integer coprime to $N$, then the map $(E, E', \phi) \mapsto (E, E', k \phi)$ (which composes an $(N,r)$-congruence with the multiplication-by-$k$ map on $E'$) induces an isomorphism $\ZNr{N}{r} \cong \ZNr{N}{k^2r}$. Hence it suffices to consider $r$ up to a square in $(\bbZ/N\bbZ)^\times$. Accordingly, when listing cases $(N, r)$ we always choose $r \geq 1$ to be the least positive representative of its square class. The reader may find these cases listed (for $6 \leq N \leq 33$) in \Cref{sec:table-numer-invar}.

\begin{theorem}[{\cite[Theorem~4]{KS_MDQS}}]
  \label{thm:ZNr-geom}
  The surface $\ZNr{N}{r}$ is birational over $\bbC$ to a smooth projective algebraic surface which is:
  \begin{enumerate}[label=\eniii]
  \item
    rational if $N \leq 5$, or $(N,r) = (6,1)$, $(7,1)$, or $(8,1)$,
  \item
    an elliptic K3 surface if $(N,r) = (6,5)$, $(7,3)$, $(8,3)$, $(8,5)$, $(9,1)$, or $(12,1)$,
  \item
    an elliptic surface of Kodaira dimension $1$ if $(N,r) = (8,7)$, $(9,2)$, $(10,1)$, $(10,3)$, or $(11,1)$, and
  \item
    a surface of general type otherwise, i.e., if $N \geq 11$ and $(N,r) \neq (11,1)$ or $(12, 1)$.
  \end{enumerate}
\end{theorem}

In spite of \Cref{thm:ZNr-geom} infinite families of pairs of $(N,r)$-congruent elliptic curves are known for $N \leq 14$ (and for all $r$ when $N \neq 14$), even when $\ZNr{N}{r}$ is of general type (see~\cite{K_HMSFSDAESOG2FF,F_OFO13CEC,F_thesis,F_O12COEC} and the references therein). The computations in these works are greatly simplified by the fact that $\ZNr{N}{r}$ admits a canonical involution swapping the roles of the elliptic curves parametrised.

More precisely, let $\tau$ denote the involution of $\ZNr{N}{r}$ given by $(E, E', \phi) \mapsto (E', E, r\phi^{-1})$. We define the \emph{symmetric Hilbert modular surface} $\YNrSym{N}{r}$ to be the quotient of $\YNr{N}{r}$ by $\tau$. In \Cref{sec:nonsing-mod} we construct (when $\ZNr{N}{r}$ is not rational) a smooth projective algebraic surface $\WNr{N}{r}$ which is birational over $\bbC$ to $\YNrSym{N}{r}$ and use this to prove the following theorem extending \cite[Satz~1]{H_SMDDp2}.

\begin{theorem}
  \label{thm:WNr-KD}
  The surface $\ZNrSym{N}{r}$ is birational over $\bbC$ to a smooth projective algebraic surface which is:
  \begin{enumerate}[label=\eniii]
  \item \label{thm:WNr-KD-rat}
    rational if $N \leq 14$ or if $(N,r) = (15,1)$, $(15,2)$, $(15,11)$, $(16,1)$, $(16,3)$, $(16,7)$, $(18,5)$, $(20,11)$, or $(24,23)$,
  \item \label{thm:WNr-KD-K3}
    an elliptic K3 surface if $N = 17$, or if $(N,r) = (16,5)$, $(18,1)$, $(20,1)$, $(20,3)$, or $(21,2)$,
  \item \label{thm:WNr-KD-ell}
    an elliptic surface of Kodaira dimension $1$ if $N = 19$ or if $(N,r) = (15,7)$, $(21,5)$, $(22,1)$, or $(24,11)$, and
  \item \label{thm:WNr-KD-gen}
    of general type otherwise i.e., if $(N,r) = (20, 13)$, $(21, 1)$, $(21, 10)$, or $(22, 7)$, or if $N \geq 23$ and $(N,r) \neq (24, 11)$ or $(24, 23)$.
  \end{enumerate}
\end{theorem}

\begin{remark}
  \Cref{coro:humbert-geom} may be deduced from \Cref{thm:WNr-KD} as follows. If $\phi$ is an $(N,-1)$-congruence then it follows that the subgroup $G = \operatorname{Graph}(\phi) \subset E \times E'$ is maximal isotropic with respect to the $N$-Weil pairing on $E \times E'$ induced by the product polarisation. In particular the abelian surface $(E \times E')/G$ may be equipped with a principal polarisation which makes the quotient isogeny an $(N,N)$-isogeny. Using this observation it can be shown (see e.g., \cite[Corollary~1.8]{K_ECOAS} or \cite[Theorem~8.3]{K_MCOHSAOTN}) that the Humbert surface $\mathcal{H}_{N^2}$ is birational to the symmetric Hilbert modular surface $\YNrSym{N}{-1}$.
\end{remark}

\begin{remark}
  Several cases of \Cref{thm:WNr-KD} have been proved by explicitly computing the surfaces $\YNrSym{N}{r}$ for small integers $N$. Kumar showed that $\YNrSym{N}{-1}$ (which is birational to $\mathcal{H}_{N^2}$) is a rational surface for each $N \leq 11$ \cite{K_HMSFSDAESOG2FF}. In \cite{F_thesis,F_O12COEC} we showed that $\YNrSym{N}{r}$ is rational for $N=12,14$ and each $r$ and Fisher~\cite{F_OFO13CEC} showed that $\YNrSym{13}{r}$ is rational for each $r$ when $N = 13$.
  Remarkably, Fisher has shown that the elliptic K3 surfaces $\YNrSym{17}{1}$ and $\YNrSym{17}{3}$ are birational over $\bbQ$~\cite[Theorem~1.2]{F_OPO17CEC}. We discuss this further in \Cref{rmk:17-whats-going-on}.
\end{remark}

Given an integral projective variety $S/\bbC$ of dimension $n$ we write $\Omega_S^d$ for the sheaf of holomorphic $d$-forms on $S$, and write $p_g(S) = h^0(S, \Omega_S^n)$ and $p_a(S) = (-1)^n(\chi(\mathcal{O}_S) - 1)$ for the geometric and arithmetic genera of $S$ respectively. If $S$ is a smooth surface, we write $q(S) = p_g(S) - p_a(S) = h^0(S, \Omega_S^1)$ for the irregularity of $S$ and if $q(S) = 0$ we say that $S$ is a \emph{regular} surface. Let $\kappa(S)$ denote the Kodaira dimension of $S$ (where we adopt the convention that if $S$ is a rational surface, then $\kappa(S) = -1$).
We record the following interesting corollary of \Cref{thm:WNr-KD} whose analogue for the surfaces $\ZNr{N}{r}$ was noted by Kani--Schanz~\cite[Theorem~3]{KS_MDQS}. 

\begin{coro}
  \label{conj:kodaira-dim}
  Let $\ZNrSymtil{N}{r}$ be a smooth projective algebraic surface that is birational over $\bbC$ to $\ZNrSym{N}{r}$. For all pairs $(N,r)$ we have $\kappa(\ZNrSymtil{N}{r}) = \min(2, p_g(\ZNrSymtil{N}{r}) - 1)$. Indeed, the surface $\ZNrSymtil{N}{r}$ is:
  \begin{enumerate}[label=\eniii]
  \item 
    rational if and only if $p_g(\ZNrSymtil{N}{r}) = 0$,
  \item
    a (blown-up) elliptic K3 surface if and only if $p_g(\ZNrSymtil{N}{r}) = 1$,
  \item
    a (blown-up) elliptic surface of Kodaira dimension $1$ if $p_g(\ZNrSymtil{N}{r}) = 2$, and
  \item
    of general type otherwise.
  \end{enumerate}
\end{coro}

\begin{remark}
  The second part of \Cref{conj:kodaira-dim} is not surprising. It is likely that $\ZNrSymtil{N}{r}$ is simply connected (cf. \cite[Corollary~IV.6.2]{vdG_HMS}) and this would imply that $\ZNrSymtil{N}{r}$ cannot be an Enriques surface. Since $\ZNrSymtil{N}{r}$ is regular (see \Cref{lemma:W-nonsing}) the Enriques--Kodaira classification then yields that $\ZNrSymtil{N}{r}$ is rational if $\kappa(\ZNrSymtil{N}{r}) = -1$, (blown-up) K3 if $\kappa(\ZNrSymtil{N}{r}) = 0$, (blown-up) properly elliptic if $\kappa(\ZNrSymtil{N}{r})= 1$, and of general type if $\kappa(\ZNrSymtil{N}{r}) = 2$.
\end{remark}

In the case of fundamental (rather than square) discriminants $D$, work of Hirzebruch, Van de Ven, van der Geer, and Zagier \cite{H_THMGROTSATCARP,HV_HMSATCOAS,HZ_COHMS,vdG_HMS} places the Hilbert modular surface of discriminant $D$ within the Enriques--Kodaira classification, see \cite[Theorem~VII.3.3]{vdG_HMS}. It is an active area of research to address these questions when level structure is also imposed (see e.g., \cite{ABBCDHKKMSV_ADOBNIOHMS}).

The study of symmetric Hilbert modular surfaces of fundamental discriminant was initiated by Hirzebruch~\cite[\S5.8~Theorem]{H_HMS} when $D \equiv 1 \pmod{4}$ is a prime number. This work was continued by Hirzebruch, Zagier, Hausmann, and Bassendowski \cite{HZ_INOCOHMSAMFON,H_MUMZSHM,H_KAHM,H_TFPOTSHMGOARQFWAD,B_KHMZSHME} culminating in a classification theorem in the case when $D \equiv 1 \pmod{4}$ which may be found in \cite[Satz~4.1]{B_KHMZSHME}.

The analogue of \Cref{conj:kodaira-dim} holds for Hilbert modular surfaces of prime discriminants and square discriminants (\cite[{\S}I.1.7]{HZ_COHMS} and \cite[Theorem~3]{KS_MDQS}) and for symmetric Hilbert modular surfaces of prime discriminant~\cite[Satz~4.1]{B_KHMZSHME}. Note, however, that it \emph{does not} hold for more general discriminants. For example there are properly elliptic Hilbert modular surfaces with $p_g = 3$~\cite[{\S}I.1.7]{HZ_COHMS} and symmetric Hilbert modular surfaces which are properly elliptic with $p_g = 1$ and of general type with $p_g = 2$~\cite[Satz~4.1]{B_KHMZSHME}.

By \cite{F_OFO13CEC,F_thesis,F_O12COEC} the surface $\YNrSym{N}{r}$ is birational over $\bbQ$ to $\bbA^2$ for each $N = 12$, $13$, $14$ and $(N,r) = (15,2)$ and $(15,11)$. Indeed, in each case where $\YNrSym{N}{r}$ has been computed, if it is geometrically rational (by \Cref{thm:WNr-KD}), then it is in fact rational over $\bbQ$. This leads to a natural pair of conjectures.

\begin{conj}
  \label{conj:rational=Qrational}
  For each $(N,r)$ in \Cref{thm:WNr-KD}(i) the surfaces $\YNrSym{N}{r}$ are rational over $\bbQ$ (i.e., birational over $\bbQ$ to $\bbA^2$).
\end{conj}

\begin{conj}
  \label{conj:rat-fibrations}
  For each $(N,r)$ in \Cref{thm:WNr-KD}(ii)--(iii) there exists a Jacobian elliptic fibration $\YNrSym{N}{r} \dashrightarrow \bbP^1$ which is defined over $\bbQ$. Moreover, in each case there exists such a fibration with a $\bbQ$-rational section of infinite order.
\end{conj}

Note that it would follow from \Cref{conj:rational=Qrational,conj:rat-fibrations} that for each such pair $(N,r)$, there exist infinitely many points on $\YNr{N}{r}$ defined over quadratic fields. In particular, it would follow that there exist infinitely many pairs of $j$-invariants of non-isogenous $(N,r)$-congruent elliptic curves defined over quadratic extensions of $\bbQ$.

\subsection{An outline of the proof of \texorpdfstring{\Cref{thm:WNr-KD}}{Theorem 1.3}}
\label{sec:outline}
On a first reading, the reader is encouraged to specialise to the case when $N$ is odd. This allows one to avoid many subtle case distinctions which arise throughout this work (often depending on congruence classes of $r$ modulo $8$). Much of the analysis in \Cref{sec:kodaira-dimension} is greatly simplified if $N$ is further assumed to be coprime to $3$.

We begin by fixing notation and recalling a number of standard facts about regular surfaces (in \Cref{sec:surface-generalities}) and modular curves (in \Cref{sec:modular-curves}). In addition, given an element $g \in \GL_2(\bbZ/N\bbZ)$, we define certain modular curves $\Xg{g}$ and $\Xgplus{g}$ in \Cref{sec:cartan-nearly-cartan}. This allows us to prove \Cref{prop:Xgplus-moduli}, generalising an observation made in \cite{F_COECAFNSMNGR} which gives a moduli interpretation for the modular curves associated to an extended Cartan subgroup of $\GL_2(\bbZ/N\bbZ)$ (when $N$ is an odd prime power the extended Cartan subgroup is the normaliser of an appropriate split or non-split Cartan subgroup). We compute the genera of many such curves $\Xgplus{g}$ in \Cref{sec:genera} (this includes the genera of the modular curves associated to extended Cartan subgroups of $\GL_2(\bbZ/N\bbZ)$).

In \Cref{sec:background} we recall a number of facts about the surfaces $\ZNr{N}{r}$ and their minimal desingularisations $\ZNrtil{N}{r}$ from \cite{H_MQD,KS_MDQS}. In particular we give a precise definition of the (compactified) Hilbert modular surfaces $\ZNr{N}{r}$ following \cite{KS_MDQS}. A result of Hermann~\cite{H_MQD} which proves very useful to us is a formula for the numerical equivalence class of a canonical divisor on $\ZNrtil{N}{r}$ which we record in \eqref{eq:canonical-Z}.

In \Cref{sec:inters-modul-curv} we recall the construction of the Hirzebruch--Zagier divisors $F_m \subset \ZNr{N}{r}$ which arise as certain Hecke correspondences on $\ZNr{N}{r}$. When $N$ is an odd prime power the divisors $F_m$ are irreducible, though this is not the case when $N$ is divisible by multiple primes (or by $8$). The irreducible components of the curves $F_m$ are birational to $X_0(m)$ and their images under the natural forgetful map $\ZNr{N}{r} \to \Zone = X(1) \times X(1)$ are exactly the images of the Hecke correspondences $X_0(m) \to X(1) \times X(1)$. The behaviour of the divisors $F_{m}$ was studied by Hermann~\cite{H_MQD} and we recall several of their results in \Cref{lemma:some-FN-smooth}.

The proof of \Cref{thm:WNr-KD} roughly follows $3$ steps. This mirrors previous approaches taken by Hirzebruch, Bassendowski, and Hermann~\cite{H_HMS,B_KHMZSHME,H_SMDDp2}.

\subsubsection*{Step 1: Understand the fixed points of \texorpdfstring{$\tilde{\tau}$}{{\unichar{"1D70F}}}}
Let $\tilde{\tau}$ be the involution induced by $\tau$ on the minimal desingularisation $\ZNrtil{N}{r}$ of $\ZNr{N}{r}$. The quotient $\ZNrtil{N}{r}/\tilde{\tau}$ is not necessarily smooth since $\tilde{\tau}$ may act with isolated fixed points (i.e., components of the fixed point locus may have codimension $2$). To construct a non-singular model for $\ZNrSym{N}{r}$ we first construct a smooth birational model for $\ZNrtil{N}{r}$ on which the involution induced by $\tilde{\tau}$ acts without isolated fixed points.

To achieve this we first study the $1$-dimensional component $\Ramtil$ of the fixed point locus of $\tilde{\tau}$. In \Cref{sec:diag-HZ} we define curves $\Fgplus{g} \subset \ZNr{N}{r}$ for each $g \in \GL_2(\bbZ/N\bbZ)/\{\pm 1\}$ using the moduli interpretation for the modular curves $\Xgplus{g}$. We term the curves $\Fgplus{g}$ \emph{diagonal Hirzebruch--Zagier divisors}. These divisors generalise a previous construction of Hermann~\cite{H_SMDDp2} (and closely relate to divisors studied in detail by Hausmann~\cite{H_KAHM}). We prove in \Cref{sec:modul-curv-ramif} that $\Ramtil$ is equal to the union of those diagonal Hirzebruch--Zagier divisors for which $g^2 = \pm \det(g)$.

It then remains for us to identify the isolated fixed points of $\tilde{\tau}$, and we do so in \Cref{prop:action-aff-figs,prop:action-cusps-figs}. When $N$ is an odd integer these results bear a close resemblance to those previously obtained by Hirzebruch~\cite[Section~5]{H_HMS}, Hausmann~\cite{H_KAHM,H_TFPOTSHMGOARQFWAD}, and Hermann~\cite{H_SMDDp2}. Recall that the Hirzebruch--Zagier divisors $F_m$ are birational to a disjoint union of copies of $X_0(m)$. When $N$ is odd we show that (except for one) the isolated fixed points of $\tilde{\tau}$ are points on the strict transforms of the Hirzebruch--Zagier divisors $F_2$, $F_3$, and $F_4$ that are fixed by the action of the Fricke involution on $X_0(m)$. It turns out that for $m \leq 4$ the components of the strict transforms $\widetilde{F}_m$ are $(-1)$-curves and may be contracted onto $\Ramtil$, removing the isolated fixed points. When $N$ is even the analysis is more involved, but it remains true that we may eliminate fixed points in a similar manner.

In \Cref{sec:nonsing-mod} we use these results to construct (when $\ZNr{N}{r}$ is not rational) a smooth surface $\ZNro{N}{r}$ which is birational to $\ZNrtil{N}{r}$ and on which the involution $\tau^\circ$ induced by $\tilde{\tau}$ acts without isolated fixed points. This allows us to define a smooth surface $\WNr{N}{r} = \ZNro{N}{r}/\tau^\circ$ which is birational to $\ZNrSym{N}{r}$ for those pairs $(N,r)$ for which $\ZNr{N}{r}$ is not rational.

\subsubsection*{Step 2: Compute the geometric genus and Chern numbers of \texorpdfstring{$\WNr{N}{r}$}{W(N,r)}}
We compute explicit formulae for $p_g(\WNr{N}{r})$ and $K_W^2$ in \Cref{thm:geom-genus,lemma:chern}, where $K_W$ is a canonical divisor for $\WNr{N}{r}$. These results rely on two main observations. The first is due to Hirzebruch~\cite[Section~3, (16)]{H_HMS} which implies that if $K_{Z^\circ}$ is a canonical divisor for $\ZNro{N}{r}$ and $\Ramo$ is the locus of fixed points of the involution $\tau^\circ$ (by construction $\Ramo$ is a smooth, possibly reducible, curve)  then
\begin{equation*}
  2p_g(\WNr{N}{r}) = p_g(\ZNro{N}{r}) - \frac{1}{4} K_{Z^\circ} \cdot \Ramo - 1.
\end{equation*}
Since a closed formula for $p_g(\ZNrtil{N}{r}) = p_g(\ZNro{N}{r})$ was computed by Kani--Schanz~\cite[Theorem~2]{KS_MDQS} it remains for us to understand the intersection number $K_{Z^\circ} \cdot \Ramo$ using \eqref{eq:canonical-Z}, which we do in \Cref{lemma:Ktil-dot-Rtil}.

The second is the standard fact that if $\pi \colon \ZNro{N}{r} \to \WNr{N}{r}$ is the quotient morphism, then we have $\pi^* K_W = K_{Z^\circ} - \Ramo$. Since Kani--Schanz~\cite[Theorem~2]{KS_MDQS} determined the Chern numbers $c_1^2(\ZNrtil{N}{r}) = K_{\tilde{Z}}^2$ (which are easily related to $c_1^2(\ZNro{N}{r})$), to compute the Chern numbers $c_1^2(\WNr{N}{r}) = K_W^2$ it is only remains for us to compute the self-intersections $\Ram\!_{Z^\circ}^{\;\,2}$. To achieve this we argue with the adjunction formula, first computing the geometric genera of the irreducible components of $\Ramo$ in \Cref{sec:genera}. 

\subsubsection*{Step 3: Prove \texorpdfstring{\Cref{thm:WNr-KD}}{Theorem 1.2} by finding \texorpdfstring{$(-1)$ and $(-2)$-curves on $\WNr{N}{r}$}{(-1) and (-2)-curves on W(N,r)}}
This final step requires the most subtle analysis in this work, especially when $N$ is divisible by either $2$ or $3$. Within this analysis the most care must be taken in the cases when $\WNr{N}{r}$ is an elliptic fibration.

We first compute explicit lower bounds on $p_g(\WNr{N}{r})$ and $K_W^2$ in \Cref{sec:bounds}. In particular, combining \Cref{thm:geom-genus,lemma:pgW-bound} we deduce the following result.
\begin{theorem}
  \label{thm:p_g}
  We have:
  \begin{enumerate}[label=\eniii]
  \item \label{enum:rat-case-pg}
    $p_g(\WNr{N}{r}) = 0$ when $N \leq 14$, and when $(N,r) = (15,1)$, $(15,2)$, $(15,11)$, $(16,1)$, $(16,3)$, $(16,7)$, $(18,5)$, $(20,11)$, and $(24,23)$,
  \item 
    $p_g(\WNr{N}{r}) = 1$ when $N = 17$, and when $(N,r) = (16,5)$, $(18,1)$, $(20,1)$, $(20,3)$, and $(21,2)$,
  \item
    $p_g(\WNr{N}{r}) = 2$ when $N = 19$, and when $(N,r) = (15,7)$, $(21,5)$, $(22,1)$, and $(24,11)$, and
  \item 
    $p_g(\WNr{N}{r}) \geq 3$ otherwise, i.e., if $(N,r) = (20, 13)$, $(21, 1)$, $(21, 10)$, or $(22, 7)$, or if $N \geq 23$ and $(N,r) \neq (24, 11)$ or $(24, 23)$.
  \end{enumerate}
\end{theorem}

\case{The rational cases} It follows immediately from \Cref{thm:p_g} that $\WNr{N}{r}$ may only be rational in case \ref{enum:rat-case-pg}. Using rationality criteria of Hirzebruch~\cite[4.4]{H_HMS} (see \Cref{lemma:rat-crit}) we prove \Cref{thm:WNr-KD}\ref{thm:WNr-KD-rat} in \Cref{prop:rational-cases}.

\case{The cases of general type} If $S/\bbC$ is a non-rational smooth surface with canonical divisor $K_S$ which satisfies $K_S^2 > 0$, then $S$ is of general type. Using this criterion and the bounds in \Cref{lemma:pgW-bound} we deduce \Cref{thm:WNr-KD}\ref{thm:WNr-KD-gen} in \Cref{prop:gen-type-cases}. Note however that the bounds in \Cref{lemma:pgW-bound} are not enough to prove \Cref{thm:WNr-KD}\ref{thm:WNr-KD-gen} directly. We first describe modular curves (e.g., components of $F_m$ and $\Fgplus{g}$) whose images on $\WNr{N}{r}$ become $(-1)$-curves in \Cref{sec:special-curves}. Blowing down such curves we define a ``closer-to-minimal'' model $\WNro{N}{r}$ for $\WNr{N}{r}$ in \Cref{sec:gt-cases} whose canonical divisor $\KWo$ satisfies $\KWo^2 > 0$ (except in some cases when $N = 24$ where extra care is required).

\case{The elliptic fibrations} To complete the proof of \Cref{thm:WNr-KD} we exhibit elliptic fibrations of the surfaces $\WNr{N}{r}$. In \Cref{sec:elliptic-K3,sec:prop-ellipt-cases} for each $(N,r)$ in either \ref{thm:WNr-KD-K3} or \ref{thm:WNr-KD-ell} of \Cref{thm:WNr-KD}, we construct smooth surfaces birational to $\WNr{N}{r}$ which contain an \emph{elliptic configuration} $\mathscr{C}$. That is, $\mathscr{C}$ is a set of $(-2)$-curves which is isomorphic to (the reduced subscheme of) one of Kodaira's singular fibres of an elliptic fibration. It follows that $\WNr{N}{r}$ is an elliptic fibration, essentially because the morphism given by (a subsystem of) the complete linear system $| \mathscr{C} |$ is such a fibration (see e.g.,~\cite[Proposition I.9]{HV_HMSATCOAS} or \cite[Proposition~VII.2.9]{vdG_HMS}).

We conclude the article with a few remarks on minimal models for $\ZNrtil{N}{r}$ and $\WNr{N}{r}$ in \Cref{sec:rmk-min-models}. In particular in \Cref{coro:min-models} we prove that several surfaces constructed in \Cref{sec:elliptic-K3,sec:prop-ellipt-cases} are minimal (utilising the fact that an elliptic fibration $S$ with $K_S^2 = 0$ is necessarily minimal). We make a conjecture that for sufficiently large integers $N$ coprime to $6$ the surfaces $\ZNro{N}{r}$ and $\WNro{N}{r}$ are minimal (see \Cref{conj:min-mods}).

\subsection{Notation}
\label{sec:notation}
In \Cref{table:surface-defns,table:ZNr-notation} we record notation which is used repeatedly throughout this article. In \Cref{table:surface-defns} we record the definitions of several surfaces which are birational to either $\ZNr{N}{r}$ or $\WNr{N}{r}$ and in \Cref{table:ZNr-notation} we record notation for a number of divisors on these surfaces.

\begingroup
\addtolength{\textfloatsep}{8mm}
\renewcommand{\arraystretch}{1.3}
\begin{table}[t]
  \centering
  \begin{tabular}{c|c|c|c}
    Surface               & Definition                                    & Construction                                 & Conditions                   \\
    \hline
    $\Zone$               & $X(1) \times X(1)$                            & \Cref{sec:surf-ZNr}                          &                              \\
    $\ZNr{N}{r}$          & $\Delta_{\ee} \backslash (X(N) \times X(N))$  & \Cref{sec:surf-ZNr}                          &                              \\
    $\ZNrtil{N}{r}$       & The minimal desingularisation of $\ZNr{N}{r}$ & \Cref{sec:minim-resol-sing}                  &                              \\
    $\ZNro{N}{r}$         & A blow-down of $\ZNrtil{N}{r}$                & \Cref{constr:ZNro}                           & $\kappa(\ZNr{N}{r}) \neq -1$ \\
    $\WNr{N}{r}$          & $\ZNro{N}{r}/\tau^{\circ}$                    & \Cref{const:WNr}                             & $\kappa(\ZNr{N}{r}) \neq -1$ \\
    $\overbar{W}\!_{N,r}$ & A blow-down of $\WNr{N}{r}$                   & \Cref{sec:rational-cases}                    & $\kappa(\ZNr{N}{r}) \neq -1$ \\
    $\WNro{N}{r}$         & A blow-down of $\WNr{N}{r}$                   & \Cref{constr:WNro}                           & $\kappa(\WNr{N}{r}) \neq -1$ \\
    $\WNrmin{N}{r}$       & The minimal model of $\WNr{N}{r}$             & \Cref{sec:elliptic-K3,sec:prop-ellipt-cases} & See \Cref{coro:min-models}   \\
  \end{tabular}
  \caption{Surfaces appearing in this work.}
  \label{table:surface-defns}
\end{table}

\begin{table}[p]
  \centering
  \begin{tabular}{C{11mm}|C{18mm}|c|c}
    Surface                          & Divisor                        & Description                                                                                               & Image on $\Zone$                                         \\
    \hline
    $\Zone$ & $\DiagDiv$ & $\{(x,x) : x \in X(1)\}$ & \\
    \hdashline \rule{0pt}{1.5\normalbaselineskip}
                                     & $C_{\infty,1}$, $C_{\infty,2}$   & {\makecell[c]{The image of $\{\infty\} \times X(N)$ \\(resp. $X(N) \times \{\infty\}$) on $\ZNr{N}{r}$}}   & $\{\infty\} \times X(1)$, $X(1) \times \{\infty\}$       \\[3mm]
                                     & $C_{\infty}$                    & $C_{\infty,1} + C_{\infty,2}$                                                                                 & $(\{\infty\} \times X(1)) \cup (X(1) \times \{\infty\})$ \\[2mm]
                                     & $C_{\star,1}$, $C_{\star,2}$     & {\makecell[c]{The image of $\{j\} \times X(N)$ \\(resp. $X(N) \times \{j\}$) on $\ZNr{N}{r}$}}            & $\{j\} \times X(1)$, $X(1) \times \{j\}$                 \\[3mm]
    $\ZNr{N}{r}$                     & $F_{m,\lambda}$                 & {\makecell[c]{A Hirzebruch--Zagier divisor \\ on $\ZNr{N}{r}$ see \Cref{sec:inters-modul-curv}}}        & {\makecell[c]{The Hecke correspondence\\ $X_0(m) \dashrightarrow \Zone$}}         \\[3mm]
                                     & $\Fgplus{g}$                  & {\makecell[c]{A diagonal Hirzebruch--Zagier \\ divisor, see \Cref{sec:diag-HZ}}}                          & $\DiagDiv$                                               \\[3mm]
                                     & $F_{m \circ g}^+$                 & {\makecell[c]{A certain modular curve \\ on $\ZNr{N}{r}$ see \Cref{sec:isogeny-compose}}}             & {\makecell[c]{The Hecke correspondence\\ $X_0(m) \dashrightarrow \Zone$}}         \\[3mm]

                                     & $\Ram\!_{Z}$                   & {\makecell[c]{The one dimensional component of \\the set of fixed points of $\tau$}}                      & $\DiagDiv$                                                \\[3mm]
    \hdashline \rule{0pt}{1.5\normalbaselineskip}
                                     & $K_{\tilde{Z}}$                 & {\makecell[c]{A canonical divisor on \\ $\ZNrtil{N}{r}$, see \eqref{eq:canonical-Z}}}                       &                                 \\[3mm]
                                     & $\widetilde{D}$                & {\makecell[c]{Strict transform \\of a divisor $D$ on $\ZNr{N}{r}$}}                                     & $(\ffj \times \ffj')(D)$                                 \\[3mm]
                                     & $E_{2,1}$                      & {\makecell[c]{Total transform of singularities of \\type $(2,1)$ on $\ZNr{N}{r}$ above $(1728,1728)$}}      & $(1728,1728)$                                            \\[3mm]
                                     & $E_{2,1}(\phi)$                & {\makecell[c]{Total transform of $(E, E, \phi)$ of \\type $(2,1)$ on $\ZNr{N}{r}$ above $(1728,1728)$}}      & $(1728,1728)$                                            \\[3mm]

   \multirow{2}{*}{$\ZNrtil{N}{r}$}  & $E_{3,q}$                      & {\makecell[c]{Total transform of singularities of \\type $(3,q)$ on $\ZNr{N}{r}$ above $(0,0)$}}            & $(0,0)$                                                  \\[3mm]
                                     & $E_{3,q}(\phi)$                & {\makecell[c]{Total transform of $(E, E, \phi)$ of \\type $(3,q)$ on $\ZNr{N}{r}$ above $(0,0)$}}      & $(0,0)$                                            \\[3mm]

                                     & $E_{\infty,d,q}$               & {\makecell[c]{Total transform of singularities of \\type $(d,q)$ on $\ZNr{N}{r}$ above $(\infty,\infty)$}}  & $(\infty,\infty)$                                        \\[3mm]
                                     & $D_{\infty}$                   & {\makecell[c]{Total transform of $C_{\infty}$ \\equivalently $\widetilde{C}_\infty + \sum E_{\infty,d,q}$}} & $(\{\infty\} \times X(1)) \cup (X(1) \times \{\infty\})$ \\[3mm]
                                     & $\widetilde{C}_{\star}$        & $\widetilde{C}_{\star,1} + \widetilde{C}_{\star,2}$                                                          & $(\{j\} \times X(1)) \cup (X(1) \times \{j\})$           \\[2mm]
                                     & $\Ram\!_{\tilde{Z}}$             & {\makecell[c]{The one dimensional component of \\the set of fixed points of $\tilde{\tau}$}}               & $\DiagDiv$                                              \\
    \hdashline \rule{0pt}{1.5\normalbaselineskip}
                                     & $K_{Z^\circ}$                 & {A canonical divisor on $\ZNro{N}{r}$}                       &                                 \\[3mm]
    $\ZNro{N}{r}$                    & $D^\circ$                     & {\makecell[c]{The image of a divisor $\widetilde{D}$ on $\ZNrtil{N}{r}$ \\ under the blow-down $\ZNrtil{N}{r} \to \ZNro{N}{r}$}}    & $(\ffj \times \ffj')(D)$         \\[3mm]
                                     & $\Ramo$                      & {\makecell[c]{The (one dimensional component of) \\the set of fixed points of ${\tau}^\circ$}}               & $\DiagDiv$                                        \\
    \hdashline \rule{0pt}{1.5\normalbaselineskip}
    \multirow{2}{*}{$\WNr{N}{r}$}    & $K_{W}$                      & {A canonical divisor on $\WNr{N}{r}$}                       &                                 \\[3mm]
                                     & $D^*$                       & {\makecell[c]{The image of a divisor $D$ on $\ZNro{N}{r}$ \\ under the quotient $\ZNro{N}{r} \to \WNr{N}{r}$}}    & $(\ffj \times \ffj')(D)$         \\[3mm]
  \end{tabular}
  \caption{Our notation for certain divisors on $\Zone$, $\ZNr{N}{r}$, $\ZNrtil{N}{r}$, $\ZNro{N}{r}$, and $\WNr{N}{r}$. Here $j \in X(1)$ is chosen so that $j \not\in \{0,1728,\infty\}$. }
  \label{table:ZNr-notation}
\end{table}
\endgroup

\subsection{Code accompanying this article}
\label{sec:code-attached}
The code accompanying this article has been written primarily in \texttt{python} utilising \texttt{Matplotlib}~\cite{Matplotlib} with some checks performed in \texttt{Magma}~\cite{MAGMA}. Our code is publicly available under an MIT licence on the author's GitHub profile in the repository \cite{ME_ELECTRONIC_HERE} and may be accessed at\par
\centerline{\texttt{\url{https://github.com/SamFrengley/humbert.git}}.}\par
In particular, \texttt{python} code to generate \Cref{tab:invariants-WNr1} and to compute the outputs of \Cref{thm:geom-genus,lemma:chern,lemma:Ko2} for any pair $(N,r)$ is provided.

\subsection{Acknowledgements}
This work is based on a chapter from my PhD thesis~\cite[Chapter~3]{F_thesis}. I would like to thank my supervisor Tom Fisher for extensive comments on earlier versions of this work which greatly improved this article. My thanks to Ernst Kani for introducing to me the work of Hermann~\cite{H_SMDDp2}, and for several useful correspondences. I would also like to express my gratitude to Maria Corte-Real Santos, Patrick Kennedy-Hunt, Dhruv Ranganathan, Tony Scholl, Damiano Testa, and John Voight for useful comments and conversations. 

This work was conducted while I was supported by the Woolf Fisher and Cambridge Trusts and by C{\'e}line Maistret's Royal Society Dorothy Hodgkin Fellowship.

\section{Generalities on regular surfaces}
\label{sec:surface-generalities}
Let $S/\bbC$ be a smooth, projective, integral, algebraic variety of dimension $n$ and let $K_S$ be a canonical divisor. Recall that we write $p_g(S) = h^0(S, \Omega_S^n)$ and $p_a(S) = (-1)^n (\chi(\mathcal{O}_S) - 1)$ for the geometric and arithmetic genera of $S$ respectively. In particular if $S$ is a surface then $p_g(S) = h^0(S, \Omega_S^2)$ and $p_a(S) = \chi(\mathcal{O}_S) - 1$ and we write $q(S) = p_g(S) - p_a(S)$ for the irregularity of $S$. We say that $S$ is \emph{regular} if $q(S) = 0$.

If $C$ and $D$ are divisors on $S$ we write $C \cdot D$ for the intersection product of $C$ and $D$ and $C^2 = C \cdot C$ for the self-intersection of $C$. If $C$ is a (reduced, irreducible) curve on $S$ and $k > 0$ is an integer, we say that $C$ is a $(-k)$-curve if $C$ is smooth and rational, and $C^2 = -k$.

For convenience we record the following well known result.

\begin{lemma}
  \label{lemma:fixed-is-smooth}
  Let $\nu$ be an involution acting non-trivially on a smooth, projective, algebraic surface $S/\bbC$ (i.e., $\nu \colon S \to S$ is a morphism such that $\nu^2$ is the identity). Then the locus $S^\nu$ of fixed points of $\nu$ is smooth. In particular $S^\nu$ is the union of finitely many isolated fixed points and finitely many smooth curves which do not intersect.
\end{lemma}

It follows from the Enriques--Kodaira classification that if $S$ is a regular surface, then $S$ is either rational ($p_g(S) = 0$), an Enriques surface ($p_g(S) = 1$), a K3 surface ($p_g(S) = 1$), a properly elliptic surface ($p_g(S) \geq 0$), or a surface of general type ($p_g(S) \geq 0$). We now recall several well known facts which we use throughout this work. The first is a series of rationality criteria (which are based on Castelnuovo's criterion). 

\begin{lemma}[{\cite[4.4~I--III]{H_HMS}}]
  \label{lemma:rat-crit}
  Let $S/\bbC$ be a smooth, regular (i.e., $q(S) = h^0(S, \Omega_S^1) = 0$), projective algebraic surface. Let $K_S$ be a canonical divisor of $S$. If $C \subset S$ is an irreducible curve such that either:
  \begin{enumerate}[label=\eniii]
  \item \label{enum:rat-crit-1}
    $K_S \cdot C \leq -2$, or
  \item \label{enum:rat-crit-2}
    $K_S \cdot C \leq -1$ and $C$ is either singular, or is not rational
  \end{enumerate}
  then $S$ is a rational surface. Furthermore, if $C,D \subset S$ are distinct smooth, irreducible, rational $(-1)$-curves which meet, then $S$ is a rational surface.

  In particular, if $S$ is not rational and $K_S \cdot C \leq -1$, then $C$ is a non-singular rational curve, $K_S \cdot C = C^2 = -1$, and $C$ does not meet another $(-1)$-curve on $S$.
\end{lemma}

In order to prove a surface is either a (blown-up) elliptic K3 surface, or a properly elliptic surface we make use of the following standard results.

\begin{defn}
  \label{defn:ell-config}
  We say that a (possibly reducible) reduced curve $\mathscr{C} \subset S$ is an \emph{elliptic configuration} if its components have the same (arithmetic and geometric) genera and mutual intersection numbers as the (reduced subscheme of a) fibre of a (minimal) elliptic fibration. That is, $\mathscr{C}$ is either an irreducible curve of arithmetic genus $1$ with $\mathscr{C}^2 = K_S \cdot \mathscr{C} = 0$, or $\mathscr{C}$ is a union of $(-2)$-curves whose respective intersections are depicted in e.g., the table preceding \cite[Definition~VII.2.8]{vdG_HMS}.  
\end{defn}

\begin{prop}[{\cite[Proposition~I.9]{HV_HMSATCOAS}}]
  \label{prop:ell-cong-K3}
  Let $S/\bbC$ be a smooth, regular, projective algebraic surface with $p_g(S) \geq 1$. Suppose that there exists an elliptic configuration $\mathscr{C} \subset S$ which meets a $(-2)$-curve not contained in $\mathscr{C}$, then $S$ is a (blown-up) elliptic K3 surface.
\end{prop}

To save ourselves a subtle analysis of singularities on certain curves we state a slight weakening of the result in \cite[Proposition~VII.2.9]{vdG_HMS} which can be found in Bassendowski's PhD thesis~\cite[Satz~4.4]{B_KHMZSHME}. To this end we make the following definition.

\begin{defn}
  A \emph{pseudo-elliptic configuration} $\mathscr{D} \subset S$ is a (possibly reducible) reduced curve defined analogously to an elliptic configuration (see \Cref{defn:ell-config}), except that intersections may occur with multiplicity and components $D$ of $\mathscr{D}$ may be singular, provided that it remains true that $K_S \cdot D \leq 0$. That is, $\mathscr{D}$ is such that if its components are smooth and intersections transversal, then $\mathscr{D}$ is an elliptic configuration.
\end{defn}

\begin{lemma}
  \label{lemma:ell-conf-bound-enough}
  Let $S$ be a smooth, regular, projective algebraic surface with $p_g(S) \geq 1$ and let $K_S$ be a canonical divisor on $S$. Let $\mathscr{D} \subset S$ be a pseudo-elliptic configuration. Then $\mathscr{D}$ does not meet a $(-1)$-curve, and for each component $D$ of $\mathscr{D}$ we have $K_S \cdot D = 0$.
\end{lemma}

\begin{proof}
  Because $S$ is not rational, if any component $D$ of $\mathscr{D}$ has $K_S \cdot D \leq -1$, then equality holds and $D$ is a $(-1)$-curve by \Cref{lemma:rat-crit}. Blowing down $D$ (and further components of $\mathscr{D}$ as required) leads to a contradiction to \Cref{lemma:rat-crit}. Therefore $K_S \cdot D = 0$ for each component $D$ of $\mathscr{D}$. A similar contradiction follows if $\mathscr{D}$ meets a $(-1)$-curve.
\end{proof}

\begin{prop}[{\cite[Satz~4.4]{B_KHMZSHME}}]
  \label{prop:pseudo-ell}
  Let $S/\bbC$ be a smooth, regular, projective, algebraic surface with $p_g(S) \geq 2$. If there exists a pseudo-elliptic configuration $\mathscr{D} \subset S$, then $S$ is a properly elliptic surface.
\end{prop}

\begin{proof}
  This essentially follows from the argument of \cite[{\S}5.1~Proposition]{HZ_COHMS}. Let $K_S$ be a canonical divisor on $S$. Since $p_g(S) \geq 2$ the surface $S$ is not a rational, K3, or Enriques surface and it suffices to show that $S$ is not of general type. By \Cref{lemma:ell-conf-bound-enough} each component of $\mathscr{D}$ has $K_S \cdot D = 0$, and we are free to assume that $S$ is minimal.

  If $S$ is of general type Corollary~VII.2.3 and Proposition~VII.2.5(ii) of \cite{WHPV_CCS} imply that each component of $\mathscr{D}$ is a $(-2)$-curve, and the intersection matrix of $\mathscr{D}$ is negative definite. But the latter is clearly not the case for any pseudo-elliptic configuration.
\end{proof}

\FloatBarrier
\section{Modular curves}
\label{sec:modular-curves}
Let $N$ be a positive integer. We denote by $\Yred{N}/\bbQ$ the (non-compact) modular curve of full level $N$. That is, $\Yred{N}$ is the (coarse) moduli space for pairs $(E, \iota)$ where $E/K$ is an elliptic curve and $\iota \colon E[N] \to (\bbZ/N\bbZ)^2$ is an isomorphism of $\Gal(\Kbar/K)$-modules. Let $\Xred{N}$ be the smooth compactification of $\Yred{N}$. The natural action of the group $\GL_2(\bbZ/N\bbZ)$ on $(\bbZ/N\bbZ)^2$ induces a left action on $\Yred{N}$ via $(E, \iota) \mapsto (E, g \circ \iota)$, which extends to an action on $\Xred{N}$. The quotient $\GL_2(\bbZ/N\bbZ) \backslash \Xred{N}$ is isomorphic to $\bbP^1$ and the quotient morphism realises the natural forgetful morphism $\Xred{N} \to X(1) \cong \bbP^1$ given by $(E, \iota) \mapsto j(E)$.

Let $Y(N)$ and $X(N)$ denote the quotients of $\Yred{N}$, respectively $\Xred{N}$, by the subgroup $\left\{ \big( \begin{smallmatrix} 1 & 0 \\ 0 & * \end{smallmatrix} \big) \right\} \subset \GL_2(\bbZ/N\bbZ)$. We call $X(N)$ the modular curve with \emph{full symplectic level $N$ structure}, and for each field $K/\bbQ$ the $K$-points on $Y(N)$ are in bijective correspondence with isomorphism classes of pairs $(E, \iota)$ where $E/K$ is an elliptic curve and $\iota \colon E[N] \to \mu_N \times \bbZ/N\bbZ$ is an isomorphism of symplectic abelian groups (where $E[N]$ is equipped with the Weil pairing and the pairing on $\mu_N \times \bbZ/N\bbZ$ is given by $\apair{(\zeta,a)}{(\xi,b)} = \zeta^b \xi^{-a}$).

\begin{remark}
  The modular curve $\Xred{N}$ is not geometrically irreducible, it consists of $\varphi(N)$ geometrically irreducible components defined over $\bbQ(\zeta_N)$ and permuted transitively by $\Gal(\bbQ(\zeta_N)/\bbQ)$. Each geometrically irreducible component is isomorphic (over $\bbQ(\zeta_N)$) to $X(N)$. Moreover, the Riemann surface $Y(N)(\bbC)$ is isomorphic to $\Gamma(N) \backslash \bbH$, where $\bbH$ denotes the complex upper half plane and $\Gamma(N) \subset \SL_2(\bbZ)$ is the kernel of the reduction modulo $N$ homomorphism $\SL_2(\bbZ) \to \SL_2(\bbZ/N\bbZ)$.
\end{remark}

 The modular curve $Y(N)$ admits a natural action of the group $\UpsN{N}$ of symplectic automorphisms of $\mu_N \times \bbZ/N\bbZ$, which extends to an action on $X(N)$. Over $\bbQ(\zeta_N)$ the group $\UpsN{N}$ becomes isomorphic to $\SL_2(\bbZ/N\bbZ)$ and we use this identification to specify elements of $\UpsN{N}$.

 More generally, for each subgroup $H \subset \GL_2(\bbZ/N\bbZ)$ containing $-I$, we define the modular curve with \emph{level $H$-structure} to be the quotient $X(H) = H \backslash \Xred{N}$. The modular curve $X(H)$ has $[(\bbZ/N\bbZ)^\times : \det(H)]$ geometrically irreducible components defined over $\bbQ(\zeta_N)^{\det(H)}$, where we identify $\det(H) \subset (\bbZ/N\bbZ)^\times \cong \Gal(\bbQ(\zeta_N)/\bbQ)$.

By construction, the modular curve $X(H)$ depends only on the $\GL_2(\bbZ/N\bbZ)$-conjugacy class of $H$. We define the \emph{index} of a modular curve $X(H)$ to be the index $[\GL_2(\bbZ/N\bbZ) : H]$ which is equal to the degree of the morphism $X(H) \to X(1)$. We say that a point $x \in X(H)$ is a \emph{cusp} if it maps to $j = \infty$ under the forgetful morphism $X(H) \to X(1)$ and we define the \emph{width} of $x$ to be the ramification index of the morphism $X(H) \to X(1)$ at $x$.

\subsection{\texorpdfstring{$N$}{N}-gons and cusps}
\label{sec:N-gons-cus}
Let $N \geq 1$ be an integer and write $\mathcal{C}_N$ for the \emph{standard $N$-gon}. That is, if $N = 1$ we let $\mathcal{C}_1$ be the nodal cubic, and if $N \geq 2$ the curve $\mathcal{C}_N$ consists of $N$ copies of $\bbP^1$ (indexed by $\bbZ/N\bbZ$) with $\infty$ on the $i$-th component glued to $0$ on the $(i+1)$-th component for each $i \in \bbZ/N\bbZ$. A \emph{generalised elliptic curve} $\mathcal{E}$ is either a smooth elliptic curve or a standard $N$-gon for some $N \geq 1$ together with a group structure on the smooth locus $\mathcal{E}^{\sm} \subset \mathcal{E}$ which extends to an action of $\mathcal{E}^{\sm}$ on $\mathcal{E}$. We write $\Ngon[N]$ for the standard $N$-gon $\mathcal{C}_N$ equipped with the natural structure of a generalised elliptic curve.

Note that $\Ngon^{\sm} \cong \mathbb{G}_m \times \bbZ/N\bbZ$ and that such a choice of isomorphism induces an isomorphism
\begin{equation*}
  \Aut(\Ngon) \cong G_\infty = \left\{
    \begin{pmatrix}
      \pm 1 & \alpha \\ 0 & \pm 1
    \end{pmatrix}
    \in \UpsN{N} : \alpha \in \bbZ/N\bbZ \right\}
\end{equation*}
and an isomorphism ${\Ngon}[N] \colonequals \Ngon^{\sm}[N] \cong \mu_N \times \bbZ/N\bbZ$. In particular $\Ngon^{\sm}[N]$ comes equipped with a natural alternating pairing. 

This description allows one to give a moduli interpretation of the cusps on the modular curve $X(N)$. In particular, it follows from \cite[V.4.4]{DR_LSDMDCE} that the $K$-rational cusps on $X(N)$ correspond to pairs $(\Ngon, \iota)$ where $\iota \colon {\Ngon}[N] \to \mu_N \times \bbZ/N\bbZ$ is a $K$-isomorphism of symplectic abelian groups (defined up to pre-composition with an automorphism of $\Ngon$).

\subsection{The modular curves \texorpdfstring{$X_0(N)$, $X_1(N)$, and $X_\mu(N)$}{X\_0(N), X\_1(N), X\_\unichar{"03BC}(N)}}
\label{sec:modular-exs}
We write $X_0(N)$ for the modular curve associated to the Borel subgroup $\left\{ \mymat{*}{*}{0}{*} \right\} \subset \GL_2(\bbZ/N\bbZ)$. Explicitly, the non-cuspidal $K$-rational points on $X_0(N)$ parametrises pairs $(E,\psi)$ where $E/K$ is an elliptic curve (defined up to quadratic twist) and $\psi \colon E \to E'$ is a cyclic $N$-isogeny. The curve $X_0(N)$ comes equipped with the \emph{Fricke involution}, $w_N$, which acts via $(E, \psi) \mapsto (E', \widehat{\psi})$ where $\widehat{\psi}$ denotes the dual of $\psi$. We write $X_0^+(N)$ for the quotient of $X_0(N)$ by $w_N$.

The following result is well known, and we reproduce it for the convenience of the reader. If $p$ is a prime number we write $\big( \frac{a}{p} \big)$ for the Kronecker symbol.

\begin{prop}
  \label{prop:X0-facts}
  Let $N > 1$ be an integer. Let $\nu_2(N)$ and $\nu_3(N)$ be the number of elliptic points of order $2$ and $3$ on $X_0(N)$ respectively, let $\nu_\infty(N)$ be the number of cusps on $X_0(N)$, and let $\psi(N)$ be the index of $X_0(N)$. Then
  \begin{align*}
    \nu_2(N) &=
               \begin{cases}
                 0 & \text{if $4 \mid N$, and } \\
                 \prod_{p \mid N} \left( 1 + \left( \frac{-1}{p} \right) \right) & \text{otherwise,}
               \end{cases} \\
    \nu_3(N) &=
               \begin{cases}
                 0 & \text{if $9 \mid N$, and } \\
                 \prod_{p \mid N} \left( 1 + \left( \frac{-3}{p} \right) \right) & \text{otherwise,}
               \end{cases} \\
    \nu_\infty(N) &= \sum_{d \mid N, d > 0} \varphi(\gcd(d, N/d)), \\
    \psi(N) &= N \prod_{p \mid N} (1 + p^{-1}).
  \end{align*}
  
  More precisely, the cusps on $X_0(N)$ are in bijective correspondence with pairs $[c,d]$ where $d \mid N$ and $c \in (\bbZ/\gcd(d, N/d)\bbZ)^\times$. The width of the cusp $[c,d]$ is equal to $N/dt$ where $t = \gcd(d, N/d)$. The Fricke involution takes a cusp $[c,d]$ to $[c^{-1},N/d]$. 
\end{prop}

\begin{proof}
   A proof of the first claim can be found in \cite[Proposition~1.43]{S_ITTATOAF}. The calculation of the cusp widths may be seen by combining Propositions~1~and~2 of \cite{O_RPOCEMC} (see also \cite[Proposition~2.15]{RG_EOBMC} and \cite[Example~1.20]{S_MFACA}). The action of the Fricke involution can be found in \cite[Section~3]{O_RPOCEMC} (see also \cite[Corollary~3.5]{RG_EOBMC}).
\end{proof}

We write $X_1(N)$, respectively $X_\mu(N)$, for the modular curve whose non-cuspidal $K$-rational points parametrise pairs $(E, \iota)$, where $E/K$ is an elliptic curve equipped with an embedding $\iota \colon \bbZ/N\bbZ \hookrightarrow E[N]$ (respectively $\iota \colon \mu_N \hookrightarrow E[N]$) defined up to composition with an automorphism of $E$. Indeed, by \cite[Construction~4.14]{DR_LSDMDCE} the cusps of $X_\mu(N)$ correspond to embeddings $\iota \colon \mu_N \hookrightarrow {\Ngon}[N]$ up to composition with an automorphism of $\Ngon$. As curves $X_1(N)$ and $X_\mu(N)$ are canonically isomorphic. If $(E, \iota) \in X_1(N)$ then writing $G \subset E[N]$ for the image of $\iota$ the quotient $E' = E/G$ comes equipped with a natural embedding $\iota \colon \mu_N \hookrightarrow E'[N]$ (whose image is the kernel of the dual isogeny).

\subsection{Cartan subgroups and the subgroups \texorpdfstring{$\Hg{g}$ and $\Hgplus{g}$}{H\_g and H\_g{\textasciicircum}+}}
\label{sec:cartan-nearly-cartan}
Let $N$ be a positive integer and let $g$ be an element of $\GL_2(\bbZ/N\bbZ)$. We define the subgroup
\begin{equation*}
  \Hg{g} = \{ h \in \GL_2(\bbZ/N\bbZ) : g h g^{-1} = h \}
\end{equation*}
to be the centraliser of $g$ in $\GL_2(\bbZ/N\bbZ)$, and define
\begin{equation*}
  \Hgplus{g} = \{h \in \GL_2(\bbZ/N\bbZ) : g h g^{-1} = \pm h \}.
\end{equation*}
For brevity we write $\Xg{g}$ and $\Xgplus{g}$ for the modular curves $X(\Hg{g})$ and $X(\Hgplus{g})$ respectively.

\begin{remark}
  \label{rmk:not-irred-example}
  Note that it is not necessarily true in general that the determinant character is surjective on $\Hg{g}$ and $\Hgplus{g}$ (though this is the case if e.g., $g^2 = \pm \det(g)$). For example, consider the matrix $g = \big( \begin{smallmatrix} 1 & 1 \\ 0 & 1 \end{smallmatrix} \big) \in \GL_2(\bbZ/p\bbZ)$ where $p$ is an odd prime number. In this case $\Xg{g} \cong \Xgplus{g}$ is a geometrically reducible modular curve whose two geometrically irreducible components become isomorphic to $X_1(p)$ over $\bbQ(\sqrt{p^*})$ where $p^* = (-1)^{(p-1)/2} p$.
\end{remark}

The following result follows immediately from \cite[Lemma~3.1]{F_COECAFNSMNGR}, and describes the relation between the modular curves $\Xgplus{g}$ and congruences of elliptic curves.

\begin{prop}
  \label{prop:Xgplus-moduli}
  Let $K$ be a field of characteristic zero. The $K$-rational points on the modular curve $\Xgplus{g} \setminus j^{-1}\{0,1728,\infty\}$ parametrise triples $(E, E^d, \phi)$, defined up to quadratic twist. Here $E/K$ is an elliptic curve, $E^d/K$ is a quadratic twist of $E$, and $\phi \colon E[N] \to E^d[N]$ is an $N$-congruence such that, after choosing a basis for $E[N]$, the congruence $\phi$ is given by $t \circ g'$ where $g' \in \Aut(E[N]) \cong \GL_2(\bbZ/N\bbZ)$ is conjugate to $g$ and $t \colon E \to E^d$ is an isomorphism defined over $K(\sqrt{d})$. Moreover, in this case the power of the congruence $\phi$ is $\det(g)$.
\end{prop}

Though it will not be strictly required in this work, we briefly recall the definitions of Cartan subgroups of $\GL_2(\bbZ/N\bbZ)$ since they provide a helpful orientation. More complete expositions can be found in e.g.,~\cite{B_NONSCSMCATCNP} and \cite[\href{https://www.lmfdb.org/knowledge/show/gl2.cartan?timestamp=1648299705226585}{\texttt{gl2.cartan}}]{LMFDB}.

Let $N > 1$ be an integer and let $A$ be a free {\'e}tale $\bbZ/N\bbZ$-algebra of rank $2$ -- that is, $A$ is isomorphic to an algebra $(\bbZ/N\bbZ)[x]/( f(x) )$ where $f(x) \in (\bbZ/N\bbZ)[x]$ is a monic polynomial of degree $2$ with invertible discriminant. Let $\sigma$ be the natural $\bbZ/N\bbZ$-algebra automorphism of $A$ given by swapping the ``roots'' of $f(x)$.

If $N = p^k$ is a prime power, then $A$ is split (isomorphic to $(\bbZ/p^k\bbZ)^2$) or non-split (isomorphic to $\mathcal{O}/p^k\mathcal{O}$ where $\mathcal{O}$ is a quadratic order in which $p$ is inert). More generally let $N = \prod_{p \mid N} p^{k_p}$. The algebra $A$ is isomorphic to a direct sum of $\bbZ/p^{k_p}\bbZ$-algebras which are either split or non-split at each prime power $p^{k_p}$ dividing $N$. Fixing a basis for $A$ and letting $A^\times$ act on $A$ by multiplication yields the regular representation $\varrho \colon A^\times \hookrightarrow \GL_2(\bbZ/N\bbZ)$.

\begin{defn}
  A \emph{Cartan subgroup} of $\GL_2(\bbZ/N\bbZ)$ is any subgroup occurring as the image of such a representation $\varrho$. Given a Cartan subgroup of $\GL_2(\bbZ/N\bbZ)$ we define the associated \emph{extended Cartan subgroup} of $\GL_2(\bbZ/N\bbZ)$ to be the subgroup generated by a Cartan subgroup and $\sigma$.

  We say that a Cartan subgroup $H \subset \GL_2(\bbZ/N\bbZ)$ is split, respectively non-split, modulo a prime $p \mid N$ if $A$ is split, respectively non-split, modulo $p$. Moreover, we say that $H$ is split, respectively non-split, if it is split, respectively non-split, modulo every prime $p \mid N$.
\end{defn}

\begin{remark}
  When $N = p^k$ is an odd prime power an extended Cartan subgroup of $\GL_2(\bbZ/p^k\bbZ)$ is equal to the normaliser of the corresponding Cartan subgroup. This terminology is used by some authors when $N = 2^k$, though in this case the extended Cartan subgroup is strictly contained in the normaliser of the corresponding Cartan subgroup.
\end{remark}

It is clear that Cartan subgroups are determined up to conjugacy by their split and non-split primes. We write $X(\s N)$ and $X(\ns N)$ for the modular curves associated to a split, respectively non-split, Cartan subgroup and we write $X^+(\s N)$ and $X^+(\ns N)$ for the modular curves associated to their respective extended Cartan subgroups. 

The following lemma follows immediately from \cite[Propositions~3.4 and 3.6]{F_COECAFNSMNGR} (essentially by choosing $g$ to be the image of $\sigma$ in $\GL_2(\bbZ/N\bbZ)$) and demonstrates how one may realise Cartan subgroups as special cases of our subgroups $\Hg{g}$ and $\Hgplus{g}$.

\begin{lemma}
  \label{lemma:cartan-Xg}
  Let $N > 1$ be an integer, and define $\decorN{N} = N$ if $N$ is odd, and $\decorN{N} = N/2$ if $N$ is even. Let $H \subset \GL_2(\bbZ/\decorN{N}\bbZ)$ be a Cartan subgroup, and let $H^+ \subset \GL_2(\bbZ/\decorN{N}\bbZ)$ be the associated extended Cartan subgroup.
  Fix an integer $\xi \in \bbZ$ which is a quadratic non-residue modulo $p$ for each odd $p \mid N$. Define matrices $g_{\s} = \mymat{1}{0}{0}{-1} \in \GL_2(\bbZ/p^k\bbZ)$ and define $g_{\ns} = \mymat{0}{\xi}{1}{0} \in \GL_2(\bbZ/p^k\bbZ)$ if $p$ is odd, and $g_{\ns} = \mymat{1}{2}{-2}{-1} \in \GL_2(\bbZ/2^k\bbZ)$ if $p = 2$.

  Let $g \in \GL_2(\bbZ/N\bbZ)$ be a matrix for which $g \equiv g_{\s} \pmod{p^k}$ or $g \equiv g_{\ns} \pmod{p^k}$ for each prime power $p^k \mid N$, depending on whether $H$ is split, respectively non-split modulo $p^k$. Then the pre-images of $H$ and $H^+$ in $\GL_2(\bbZ/N\bbZ)$ are conjugate to $\Hg{g}$ and $\Hgplus{g}$ respectively.
\end{lemma}

\section{The surface \texorpdfstring{$\ZNr{N}{r}$}{Z(N,r)} and Hirzebruch--Zagier divisors}
\label{sec:background}
We recall several facts about the surfaces $\ZNr{N}{r}$ which we will use repeatedly in this work and fix notation for a number of interesting curves on $\ZNr{N}{r}$. The results stated here can be found in greater detail in \cite{KS_MDQS} and \cite{H_MQD}. The notation from this section is summarised in \Cref{table:ZNr-notation}.

\subsection{The surface \texorpdfstring{$\ZNr{N}{r}$}{Z(N,r)}}
\label{sec:surf-ZNr}
Let $\ee = \mymat{1}{0}{0}{r} \in \GL_2(\bbZ/N\bbZ)$ and define $\ZNr{N}{r}$ to be the quotient surface $\Delta_\ee \backslash (X(N) \times X(N))$ where
\begin{equation*}
  \Delta_\ee = \left\{ (g, \ee g \ee^{-1}) \in \SL_2(\bbZ/N\bbZ) \times \SL_2(\bbZ/N\bbZ) \right\}.
\end{equation*}
By construction we have morphisms
\begin{equation*}
  X(N) \times X(N) \to \Delta_\ee \backslash (X(N) \times X(N)) = \ZNr{N}{r} \to \Zone = X(1) \times X(1).
\end{equation*}
We write $\ffj \colon \ZNr{N}{r} \to X(1)$ and $\ffj' \colon \ZNr{N}{r} \to X(1)$ for the morphisms given by composing the forgetful map $\ZNr{N}{r} \to \Zone = X(1) \times X(1)$ with the projection maps onto the first and second factors respectively. We write $\DiagDiv$ for the diagonal in $\Zone = X(1) \times X(1)$.

If $j \in X(1)$ is a general point (i.e., $j \not\in \{0,1728,\infty\}$) then, following Hermann~\cite{H_MQD,H_SMDDp2} we define $C_{\star, 1} = \ffj^{-1}(j)$ and $C_{\star, 2} = (\ffj')^{-1}(j)$. Each fibre $C_{\star, i}$ is isomorphic to $X(N)$ over $\bbC$.

\subsection{A moduli interpretation for \texorpdfstring{$\ZNr{N}{r}$}{Z(N,r)}}
\label{sec:N-gons-cusps}
We define an \emph{$N$-congruence} of generalised elliptic curves to be an isomorphism $\phi \colon E[N] \cong E'[N]$ where $E$ and $E'$ are either (smooth) elliptic curves or $N$-gons. We say that $\phi$ is an \emph{$(N,r)$-congruence} or an $N$-congruence with \emph{power $r$} if $\phi$ raises the Weil pairing to the power of $r$ (as in \eqref{eqn:power-cong}).

We say that a pair of $N$-congruences $\phi_1, \phi_2 \colon E[N] \cong E'[N]$ are isomorphic if they define the same class in $\Aut(E) \backslash \Isom(E[N], E'[N]) / \Aut(E')$, that is, if they are equal up to composition with automorphisms. In particular, an isomorphism class of $(N,r)$-congruences over $\bbC$ is a choice of matrix $\phi \in \Aut(E) \backslash \GL_2(\bbZ/N\bbZ) / \Aut(E')$ of determinant $r$ (where we fix an isomorphism $\Isom(E[N],E'[N]) \cong \GL_2(\bbZ/N\bbZ)$ by choosing bases for $E[N]$ and $E'[N]$).

The following lemma is well known and follows immediately from the moduli interpretation of the modular curve $X(N)$ from \cite[V.4.4]{DR_LSDMDCE} (see \cite{KS_MDQS}, \cite[Section~2]{BT_pTMROECOGFF}, \cite[Lemma~3.2]{F_OFO13CEC}, or \cite[Proposition~2.3.10]{F_thesis}).
\begin{lemma}
  \label{lemma:MDQS}
  For each $N \geq 3$ the surface $\ZNr{N}{r}$ is the fine moduli space such that for each field $K/\bbQ$ the $K$-points on $\ZNr{N}{r}$ parametrise triples $(E, E', \phi)$, where $E/K$ and $E'/K$ are generalised elliptic curves defined up to simultaneous quadratic twist and $\phi$ is an isomorphism class of $(N,r)$-congruences.
\end{lemma}

It will later be useful for us to give an explicit description of the points on $\ZNr{N}{r}$ in the pre-image of $\infty$ under $\ffj$ and $\ffj'$. The supports of the fibres $\ffj^{-1}(\infty)$ and $(\ffj')^{-1}(\infty)$ are irreducible curves which we denote by ${C}_{\infty, 1}$ and ${C}_{\infty, 2}$ respectively. We write $C_\infty = C_{\infty,1} + C_{\infty,2}$. Over $\bbC$ these are isomorphic (as covers of $X(1)$ given by $\ffj'$, respectively $\ffj$) to $X_1(N)$ by \cite[Section~1.1]{KS_MDQS}. In fact the curves ${C}_{\infty, i}$ are naturally $\bbQ$-isomorphic to $X_\mu(N)$ as coverings of $X(1)$.

\begin{prop}
  \label{prop:Xmu-cusp}
  The map
  \begin{align*}
    C_{\infty, 1} &\to X_\mu(N) \\
    (\Ngon, E, \phi) &\mapsto (E, \phi \vert_{\mu_N \times \{0\}})
  \end{align*}
  is an isomorphism over $\bbQ$. In particular, after fixing an isomorphism ${\Ngon}[N] \cong \mu_N \times \bbZ/N\bbZ$, the cusps on $C_{\infty, 1}$ are the points $(\Ngon, \Ngon, \phi)$ where $\phi$ is given by either
  \begin{align}
    \label{eqn:typeA}
    \phi &= \begin{pmatrix} a & 0\\ bd & a^{-1} r\end{pmatrix}
    \intertext{where $d \mid N$ with $d \neq 1$, $0 < a < d$ and $b \in (\bbZ/(N/d)\bbZ)^\times$, or}
    \label{eqn:typeB}
    \phi &= \begin{pmatrix} 0 &-a^{-1} r \\ a & 0\end{pmatrix}
  \end{align}
  where $a \in (\bbZ/N\bbZ)^\times$.
\end{prop}

\begin{proof}
  We construct an explicit inverse. Fix a primitive $N^{\text{th}}$-root of unity $\zeta \in \mu_N$. For each non-cuspidal point $(E, \iota)$ on $X_\mu(N)$ let $Q \in E[N]$ be a point such that $e_{E,N}(\iota(\zeta), Q) = \zeta^r$. Consider the isomorphism
  \begin{equation*}
    \phi \colon {\Ngon}[N] \to E[N]
  \end{equation*}
  given by $(\zeta, 0) \mapsto \iota(\zeta)$ and $(1,1) \mapsto Q$. Clearly $\phi$ has power $r$, and $\phi$ is well defined up to composing with automorphisms of $\Ngon$ and $E$.

  By applying Cartier duality to \cite[Construction~4.14]{DR_LSDMDCE} the cusps on $X_\mu(N)$ correspond to pairs $(\Ngon, \iota)$ where $\iota \colon \mu_N \hookrightarrow {\Ngon}[N]$ is defined up to post-composition with an automorphism of $\Ngon$. Partition the cusps on $X_\mu(N)$ into two subsets:
  \begin{enumerate}[label=\enABC]
  \item \label{item:typeA}
    for each $d \mid N$ with $d \neq 1$ we have $\iota(\zeta) = (\zeta^a, bd)$ for some $0 < a < d$ and $b \in (\bbZ/(N/d)\bbZ)^\times$,
  \item \label{item:typeB}
    the cusps where $d = 1$ and $\iota(\zeta) = (1, a)$ for some $a \in (\bbZ/N\bbZ)^\times$.
  \end{enumerate}
  It is clear that \ref{item:typeA} and \ref{item:typeB} correspond to \eqref{eqn:typeA} and \eqref{eqn:typeB} respectively.
\end{proof}

\subsection{Atkin--Lehner type involutions}
\label{sec:atkin-lehn-invol}
Let $\tau$ denote the involution of $\ZNr{N}{r}$ given by $(E, E', \phi) \mapsto (E', E, r\phi^{-1})$. We define the \emph{symmetric Hilbert modular surface} $\ZNrSym{N}{r}$ to be the quotient of $\ZNr{N}{r}$ by the action of $\tau$.

Define the set $\LamNr{N}{r} = \{x \in (\bbZ/N\bbZ)^\times : x^2 = r \} / \{ \pm 1 \}$ and write $\LamN{N}$ for the group $\LamNr{N}{1}$. We have a natural action of $\LamN{N}$ on $\ZNr{N}{r}$ given by $(E, E', \phi) \mapsto (E, E', \lambda \phi)$ for each $\lambda \in \LamN{N}$.
For each integer $d > 1$ and each integer $m$ we define
\begin{equation}
  \label{eqn:function-rho}
  \rho(d,m) = \# \{ x \in (\bbZ/d\bbZ)^\times : x^2 = m \}.
\end{equation}
For brevity we also define $\rho(d) = \rho(d,1)$, by convention we take $\rho(1,m) = \rho(1) = 1$. Clearly we have $\# \LamNr{N}{r} = \frac{1}{2} \rho(N,r)$ if $N > 2$ and $\rho(d,m) = 0$ if $\gcd(d,m) \neq 1$. 

\begin{remark}
  \label{rmk:al-invs}
  The group $\langle \tau \rangle \times \Lambda(N) \subset \Aut(\ZNr{N}{r})$ is reminiscent of the group of Atkin--Lehner involutions of the modular curve $X_0(N)$ in that, when $N$ is odd, the number of such involutions is equal to $2^s$ where $s$ is the number of distinct prime factors of $N$. The involution $\tau$ is the natural analogue of the Fricke involution on $X_0(N)$.
\end{remark}

\subsection{The minimal resolution of singularities}
\label{sec:minim-resol-sing}
Let $d \geq 1$ be an integer and let $\mu_d \subset \bbC^\times$ denote the subgroup of $d^{\textnormal{th}}$-roots of unity. Let $\zeta \in \mu_d$ be a primitive $d^{\textnormal{th}}$-root of unity and let $q \in (\bbZ/d\bbZ)^\times$. Consider the action of $\mu_d$ on $\bbC^2$ given by $(z_1, z_2) \mapsto (\zeta z_1, \zeta^q z_2)$. If $S/\bbC$ is an algebraic surface, then a point $\ffz \in S(\bbC)$ is said to be a \emph{cyclic quotient singularity of type $(d,q)$} if locally analytically in a neighbourhood of $\ffz$ the surface $S$ is isomorphic to $\bbC^2 / \mu_d$. Choosing a representative $1 \leq q < d$ the \emph{Hirzebruch--Jung continued fraction} of $d/q$ is the (unique) sequence of integers $[[c_1,...,c_\ell]]$ such that
\begin{equation*}
  \frac{d}{q} = c_1 - \frac{1}{c_2 - \frac{1}{c_3 - ...}} .
\end{equation*}
The minimal resolution of the singularity $\ffz$ consists of a chain $E_{\ffz,1}, ..., E_{\ffz,\ell}$ of $\bbP^1$'s with self-intersections $E_{\ffz,i}^2 = -c_i$ determined by the Hirzebruch--Jung continued fraction of $d/q$ (see~\cite[Theorem~III.5.1]{WHPV_CCS}).

Points on $X(N)$ have stabilisers in $\SL_2(\bbZ/N\bbZ)$ whose images in $\SL_2(\bbZ/N\bbZ) / \{ \pm 1 \}$ are cyclic of order $1$, $2$, $3$, or $N$. It follows that all singularities on $\ZNr{N}{r}$ are cyclic quotient singularities (of order $2$, $3$, or $d$ where $d \mid N$) and map to one of $(0,0)$, $(1728,1728)$, or $(\infty, \infty)$ under the forgetful map $\ZNr{N}{r} \to \Zone$. Moreover one can show that the only singularities are of type $(2,1)$, $(3, 1)$, $(3,2)$, and $(d,q)$ for some $q \in (\bbZ/d\bbZ)^\times$ such that $qr$ is a square mod $d$.

More precisely we have the following theorem of Hermann~\cite[Hilfssatz~1]{H_MQD} and Kani--Schanz~\cite[Theorem~2.1]{KS_MDQS}. 

\begin{theorem}[{\cite[Hilfssatz~1]{H_MQD}, \cite[Theorem~2.1]{KS_MDQS}}]
  \label{lemma:sing-pts}
  Let $\rho(d,r)$ be as defined in \eqref{eqn:function-rho}, let $\varphi(d)$ be Euler's totient function, and let $h(-n)$ denote the class number of the imaginary quadratic order of discriminant $-n$.
  \begin{enumerate}[label=\eniii]
  \item
    There are exactly $h(-4N^2)$ singular points on $\ZNr{N}{r}$ above $(1728, 1728) \in \Zone$ all of which have type $(2,1)$.
  \item
    There are exactly $h(-3N^2)$ singular points on $\ZNr{N}{r}$ above $(0,0) \in \Zone$. If $3 \nmid N$ then $\frac{1}{2} h(-3N^2)$ are of type $(3,1)$ and $\frac{1}{2} h(-3N^2)$ are of type $(3,2)$, whereas if $3 \mid N$ then all $h(-3N^2)$ are of type $(3,q)$ where $q \equiv r \pmod{3}$.
  \item
    For each $d \mid N$ such that $d \neq 1$, the surface $\ZNr{N}{r}$ has exactly $\frac{1}{2} \rho(d, qr) \varphi(N/d)$ singular points of type $(d, q)$ above $(\infty,\infty) \in \Zone$.      
  \end{enumerate}  
\end{theorem}

The minimal resolution of singularities $\ZNrtil{N}{r} \to \ZNr{N}{r}$ is given by resolving the singularities in \Cref{lemma:sing-pts}.  In particular, the exceptional divisors of the resolution above each singular point of type $(n,q)$ (with $n,q \in \bbZ$ chosen so that $q \in \{1,...,n-1\}$) consists of chains of smooth rational curves with self-intersection determined by the Hirzebruch--Jung continued fraction of $d/q$. The surface $\ZNrtil{N}{r}$ is regular (see \cite[Satz~1]{H_MQD} and \cite[Theorem~1]{KS_MDQS}) and its geometric genus, Chern numbers, and Kodaira dimension are given by \cite[Theorems~2~and~3]{KS_MDQS}.

We write $E_{2,1}$ for the divisor consisting of curves in the resolution of the singularities of type $(2,1)$ above $(1728,1728) \in \Zone$. Similarly, for each $q = 1, 2$ we define $E_{3,q}$ to be the divisor consisting of curves in the resolution of the singularities of type $(3,q)$ above $(0,0) \in \Zone$. Finally, for each $d \mid N$ we define $E_{\infty, d, q}$ to be the divisor consisting of the curves in the resolution of the singularities of type $(d, q)$ above $(\infty,\infty) \in \Zone$.

Define $\tilde{\ffj}$ (resp. $\tilde{\ffj}'$) to be the morphisms given by composing the resolution of singularities $\ZNrtil{N}{r} \to \ZNr{N}{r}$ with $\ffj$ (resp. $\ffj')$. Hermann~\cite[p.~98]{H_MQD} and Kani--Schanz~\cite[Proposition~2.5]{KS_MDQS} give the following description of the multiplicities of the components of $\tilde{\ffj}^*(\infty)$.

\begin{lemma}[{\cite[p.~98]{H_MQD}, \cite[Proposition~2.5]{KS_MDQS}}]
  \label{lemma:j-mult-inf}
  Let $\ffz$ be a singular point of $\ZNr{N}{r}$ of type $(d,q)$ contained in $\ffj^*(\infty)$, let $\sum_{i=1}^\ell E_{\ffz,i}$ be the component of $E_{\infty,d,q}$ in the resolution of $\ffz$, and let $[[c_{1}, ..., c_{\ell}]]$ be the Hirzebruch--Jung continued fraction expansion of $d/q$. 

  Consider the integers $a_{i} \colonequals a_{i}(d,q)$ and $a_{i}' \colonequals a_{i}'(d,q) $ defined by the recurrence relations
  \begin{align*}
    a_{i+1} = c_{i}a_{i} - a_{i-1} \qquad \text{and} \qquad a_{i+1}' = c_{i}a_{i}' - a_{i-1}'
  \end{align*}
  (for $1 \leq i \leq \ell$) and the boundary conditions $a_{0} = a_{\ell+1}' = N$ and $a_{\ell+1} = a_{0}' = 0$. Let $E_{\ffz,0} = \widetilde{C}_{\infty,1}$ and $E_{\ffz,\ell + 1} = \widetilde{C}_{\infty,2}$, then for each $i = 0, ..., \ell + 1$ the multiplicity of the component of $E_{\ffz,i}$ in $\tilde{\ffj}^*(\infty)$ (respectively $(\tilde{\ffj'})^*(\infty)$) is equal to $a_{i}$ (respectively $a_{i}'$).
\end{lemma}

If $D$ is a divisor on $\ZNr{N}{r}$ we write $\widetilde{D}$ for the strict transform of $D$ under the morphism $\ZNrtil{N}{r} \to \ZNr{N}{r}$. The class of $\widetilde{C}_{\star,i}$ in the N{\'e}ron--Severi group of $\ZNrtil{N}{r}$ is unchanged by the choice of $j \in X(1) \setminus \{0,1728,\infty\}$, and we therefore obtain a well defined class $\widetilde{C}_{\star} = \widetilde{C}_{\star,1} + \widetilde{C}_{\star,2} \in \NS(\ZNrtil{N}{r})$.

Let $K_{\tilde{Z}}$ be a canonical divisor on $\ZNrtil{N}{r}$ and let $D_\infty = \widetilde{C}_\infty + \sum E_{\infty, d, q}$ where the sum is taken over all pairs $(d,q)$ with $d \mid N$ and $q \in r( (\bbZ/d\bbZ)^\times)^2$. By \cite[(4) p.~102]{H_MQD} we have an equivalence of rational homology classes (i.e., elements of $H_2(\ZNrtil{N}{r}, \bbQ)$)
\begin{equation}
  \label{eq:canonical-Z}
  K_{\tilde{Z}} \sim \frac{1}{6} \widetilde{C}_\star - D_\infty - \frac{1}{3}E_{3,1}.
\end{equation}

\begin{remark}
  The relation in \eqref{eq:canonical-Z} is the analogue of \cite[{4.3~(19)}]{H_HMS} and \cite[VI.~Proposition~2.8]{vdG_HMS} for Hilbert modular surfaces of fundamental discriminant. It is possible that this result can be obtained directly from their proof with the viewpoint in \cite{M_FOHMS}, \cite{K_HMSFSDAESOG2FF}, and \cite[Chapter~2]{F_thesis}.
\end{remark}

\subsection{The Hirzebruch--Zagier divisors \texorpdfstring{$F_{m,\lambda}$}{F\_(m,\unichar{"03BB})}}
\label{sec:inters-modul-curv}
Let $m > 0$ be an integer such that $mr$ is a square modulo $N$. Fix an integer $a \in \bbZ$ such that $a^2 m \equiv r \pmod{N}$. For each $\lambda \in \LamN{N}$ we obtain a morphism $X_0(m) \to \ZNr{N}{r}$ given on the generic point of $X_0(m)$ by $(\mathcal{E}, \Psi) \mapsto (\mathcal{E}, \mathcal{E}', a \lambda \Psi \vert_{\mathcal{E}[N]})$ where $\mathcal{E}/\bbQ(X_0(m))$ is a generic curve and $\Psi \colon \mathcal{E} \to \mathcal{E}'$ is a cyclic $m$-isogeny. We define $F_{m,\lambda}$ to be the image of this map, and we write $F_m = \sum_{\lambda} F_{m,\lambda}$. We refer to the curves $F_{m, \lambda}$ as \emph{Hirzebruch--Zagier divisors} or \emph{modular curves}. 

\begin{lemma}
  \label{lemma:which-cusp-met-comp}
  Let $n$ be a divisor of $m$ and let $x \in X_0(m)$ be a cusp corresponding to a pair $[c,n]$ as in \Cref{prop:X0-facts}. The image of $x$ on $F_{m,\lambda}$ is the point on $C_{\infty,i}$ of the type \eqref{eqn:typeA} from \Cref{prop:Xmu-cusp} for which
  \begin{equation}
    \label{eqn:Fm-cusps}
    \phi =
    \begin{pmatrix}
      a \lambda n & 0 \\
      0           & (a \lambda n)^{-1} r
    \end{pmatrix}
  \end{equation}
  which is a singular point of type $(N,q)$ where $m \equiv n^2q \pmod{N}$. Moreover, these are the only intersection points of $F_{m,\lambda}$ and $C_{\infty,i}$.
\end{lemma}

\begin{proof}
  This is \cite[(20)]{KS_MDQS}.
\end{proof}

\begin{remark}
  \label{rmk:indep-choice}
  Although our labelling for the components of $F_m$ depends on the choice of integer $a$, the divisor $F_m$ itself does not. The rational points on $F_m$ correspond to $(N,r)$-congruences which are trivial in the sense that they arise from the restriction of an isogeny.
\end{remark}

Finally, we recall a result of Hermann~\cite{H_MQD} which describes the self-intersections of the modular curves $\widetilde{F}_{m,\lambda}$ (cf. the proof of \cite[Proposition~2.14]{KS_MDQS}).

One would expect from the relation in \eqref{eq:canonical-Z} that the intersection number $\widetilde{F}_{m,\lambda} \cdot K_{\tilde{Z}}$ should be predicted by the index, number of cusps, and number of elliptic points of order $3$ on $X_0(m)$ (which we denote $\psi(m)$, $\nu_\infty(m)$, and $\nu_3(m)$, see \Cref{prop:X0-facts}). Indeed, the following result (which is due to Hermann~\cite[Hilfss{\"a}tze~12--16]{H_MQD}) shows that if $m < N$ then the intersections of $\widetilde{F}_{m,\lambda}$ with the divisors $\widetilde{C}_\star$, $D_\infty$, and $E_{3,1}$ occur with multiplicity one, and $\widetilde{F}_{m,\lambda} \cdot K_{\tilde{Z}} = \frac{1}{3} \psi(m) - \nu_\infty(m) - \frac{1}{3}\nu_3(m)$, as expected. However, if $m > N$ it can be the case (even when $\widetilde{F}_{m,\lambda}$ is smooth) that the curve $\widetilde{F}_{m,\lambda}$ meets $D_\infty$ with multiplicity. The error term $\err{N}{m}$ in the lemma counts the contribution of these additional intersections to $\widetilde{F}_{m,\lambda} \cdot K_{\tilde{Z}}$ (cf.~\cite[Satz~2.1]{H_MUMZSHM}).

\begin{lemma}
  \label{lemma:some-FN-smooth}
  Let $m > 0$ be an integer such that $mr$ is a square modulo $N$. For each divisor $n$ of $m$ let $q \in \{1,...,N\}$ be chosen such that $m = n^2q \pmod{N}$. Denote by $k \colonequals k(n)$ the unique index such that
  \begin{equation*}
    \frac{a_{k-1}(N,q)}{a'_{k-1}(N,q)} > \frac{m}{n^2} \geq \frac{a_{k}(N,q)}{a'_{k}(N,q)},
  \end{equation*}
  where $a_{i}(N,q)$ are the integers defined in \Cref{lemma:j-mult-inf}.
  Let $s(n) > 0$, $t(n) \geq 0$ be the unique integers such that $m/n = s(n) a_{k}(N,q) + t(n) a_{k-1}(N,q)$ and $n = s(n)a_k'(N,q) + t(n)a_{k-1}'(N,q)$. Define the function\footnote{This function is denoted $\rho(N,m)$ by Hermann~\cite{H_MQD}, we rename it to avoid notational clash.}
  \begin{align*}
    \err{N}{m} = \sum_{n|m} \varphi \big( \gcd(s(n), t(n)) \big) \left( \frac{s(n) + t(n)}{\gcd(s(n),t(n))} - 1 \right) .
  \end{align*}
  Then we have:
  \begin{enumerate}[label=\eniii]
  \item \label{enum:int-which-cusp}
    The modular curve $\widetilde{F}_{m}$ meets the divisor $D_{\infty}$ only at points in the pre-image of a singular point $\ffz \in \ZNr{N}{r}$ of type $(N,q)$ where $m \equiv n^{2}q \pmod{N}$ for some divisor $n$ of $m$. 
  \item
    If $\widetilde{F}_{m,\lambda}$ is an irreducible component of $\widetilde{F}_m$, then
    \begin{equation*}
      K_{\tilde{Z}} \cdot \widetilde{F}_{m,\lambda} = \frac{1}{3} \psi(m)  - \nu_\infty(m) - \frac{1}{3} \nu_3(m) - \err{N}{m}.
    \end{equation*}
  \item
    For each $m < N$ each of the components $\widetilde{F}_{m,\lambda}$ of the curves $\widetilde{F}_m$ are non-singular and $\err{N}{m} = 0$.

  \item
    If $m$ and $m'$ are integers such that $\widetilde{F}_{m,\lambda}$ and $\widetilde{F}_{m',\lambda'}$ are non-singular curves on $\ZNr{N}{r}$ then if $m \neq m'$ or $\lambda \neq \lambda'$ we have $\widetilde{F}_{m,\lambda} \cdot \widetilde{F}_{m',\lambda'} = 0$.

  \item \label{enum:F1-4-are-1}
    If $N \geq 5$, $m \in \{ 1, 2, 3, 4 \}$ is coprime to $N$, and $mr$ is a square modulo $N$, then the curves $\widetilde{F}_{m,\lambda}$ are $(-1)$-curves.
  \end{enumerate}

\end{lemma}

\begin{proof}
  The result is an abbreviation of Hilfss{\"a}tze~12--16 of \cite{H_MQD}.
\end{proof}

\begin{remark}
  \label{rmk:cusp-widths}
  In the code associated to this article we give functions which compute the data in \Cref{lemma:some-FN-smooth} for any pair of coprime integers $N$ and $m$ (see \cite{ME_ELECTRONIC_HERE}). By the projection formula (counting multiplicity) the intersection numbers of $\widetilde{F}_m$ with $\tilde{\ffj}^*(\infty)$ in an (analytic) neighbourhood of $E_{\infty, N, m/n^2}$ is exactly equal to the set of cusp widths of the cusp(s) of $X_0(m)$ corresponding to the divisor $n \mid m$ given in \Cref{prop:X0-facts}. That is, the ramification indices of the morphism $X_0(m) \to X(1)$ at these cusps. For the convenience of the reader we provide \texttt{python} code in \cite{ME_ELECTRONIC_HERE} which uses the formula for the cusp widths in \Cref{prop:X0-facts} to depict intersections of the modular curves $\widetilde{F}_m$ with the divisors $E_{\infty, d, q}$ on $\ZNrtil{N}{r}$.
\end{remark}

In \Cref{sec:elliptic-K3} we will need a slightly more refined condition for the curves $\widetilde{F}_{m,\lambda}$ to be non-singular. The following is essentially the content of \cite[Hilfsatz~12]{H_MQD}. Taking $N = 1$ is classical and picks out the singular points of the Hecke correspondence in $\Zone$ given by the vanishing of the modular polynomial (cf.~\cite[Section~2]{H_KADHMUK}).

\begin{prop}
  \label{lemma:extra-smooth}
  Let $m$ be a positive integer such that $mr$ is a square modulo $N$. Suppose that $\tilde{\ffz} \in \widetilde{F}_{m,\lambda} \subset \ZNrtil{N}{r}$ is a singular point of $\widetilde{F}_{m,\lambda}$ whose image is $(E, E', \phi)$ on $\ZNr{N}{r}$ with $j(E) \neq \infty$. Then $E$ has complex multiplication by an imaginary quadratic order of discriminant $-D$ for some integer $D > 0$, and there exist integers $a,b \in \bbZ$ with $b \neq 0$, such that $4m^2 - a^2 = b^2D$ (in particular $0 \leq a < 2m$) and $b \equiv 0 \pmod{N}$, and either
  \begin{enumerate}[label=\eniii]
  \item
    $-D \equiv 0 \pmod{4}$ and $a/2 \equiv \pm 1 \pmod{N}$, or
  \item
    $-D \equiv 1 \pmod{4}$ and $(a - b)/2 \equiv \pm 1 \pmod{N}$.
  \end{enumerate}
\end{prop}

\begin{proof}

    We begin with two lemmas. \Cref{lemma:Fm-ordinary} is a mild generalisation of the classical fact that the singularities on the curve $\Phi_p \subset \Zone$ given by the vanishing of the modular polynomial are ordinary double points (this follows from the Kronecker congruence relation, see e.g., the proof of \cite[Theorem~2]{I_KMOFOEMF}). \Cref{lemma:CM-pts} explicitly describes CM points on $\ZNr{N}{r}$.

  \begin{lemma}
    \label{lemma:Fm-ordinary}
    With the assumptions of \Cref{lemma:extra-smooth} the morphism $X_0(m) \to \widetilde{F}_{m,\lambda}$ is unramified away from the cusps on $X_0(m)$. In particular, away from $\widetilde{\ffj}^*(\infty)$, singular points on $\widetilde{F}_{m,\lambda} \subset \ZNrtil{N}{r}$ occur only where multiple branches meet.
  \end{lemma}

  \begin{proof}
    Let $\Phi_m \subset \Zone = X(1) \times X(1)$ be the curve given by the vanishing of the modular polynomial (i.e., the image of the Hecke correspondence $X_0(m) \to \Zone$). Moreover, the forgetful morphism $\ZNr{N}{r} \to \Zone$ yields a factorisation $X_0(m) \to F_{m,\lambda} \to \Phi_m$ of the morphism $X_0(m) \to \Phi_m$. It therefore suffices to prove the claim for $\Phi_m$ since $X_0(m) \to F_{m,\lambda}$ is unramified if $X_0(m) \to \Phi_m$ is.

    If $x \in X(m)$ has $j(x) \neq 0,1728,\infty$ the claim follows since the morphism $X_0(m) \to \Phi_m \to X(1)$ (projection onto the first factor in $\Zone = X(1) \times X(1)$) is unramified at $x$. By symmetry it remains to consider the points $(0,0)$ and $(1728,1728)$ on the diagonal in $\Phi_m \cap \DiagDiv \subset \Zone$. However, as proved in \cite[Section~2]{H_KADHMUK}, it is a consequence of the Hurwitz class number formula that away from $(\infty,\infty)$ the intersections of branches of $\Phi_m$ and $\DiagDiv$ occur with multiplicity $1$. The claim follows.
  \end{proof}

  \begin{lemma}
    \label{lemma:CM-pts}
    Let $E = \bbC/(\bbZ + \varphi \bbZ)$ be an elliptic curve with CM by the order $\bbZ[\varphi]$, where $\varphi = \sqrt{-d}$ or $(1 + \sqrt{-d})/2$. Suppose that $\psi = a + b\varphi \in \bbZ[\varphi]$ is chosen such that $\operatorname{Nm} \psi$ is coprime to $N$. Then there exists a choice of basis for $E[N]$ such that
    \begin{equation*}
      \psi\vert_{E[N]} = \left\{
      \begin{aligned}
        &\mymat{a}{-bd}{b}{a}         && \text{if $\varphi = \sqrt{-d}$, and}     \\[2mm]
        &\mymat{a}{-b(d+1)/4}{b}{a+b} && \text{if $\varphi = (1 + \sqrt{-d})/2$.} \\
      \end{aligned}
      \right.
    \end{equation*}
  \end{lemma}

  \begin{proof}
    The claim follows by considering the action of $\varphi$ on the basis $1/N, \varphi/N$ for $E[N]$.
  \end{proof}
  
  To prove the proposition we may assume without loss of generality that $r = m$ and $\lambda = 1$. By \Cref{lemma:Fm-ordinary} any singularity on $F_{m,1}$ occurs at the intersection of two branches. Let $(E, E', \phi) \in \ZNr{N}{r}(\bbC)$ be a point at which two branches of $F_{m,1}$ meet, then there exist a pair of distinct cyclic $m$-isogenies $\psi_1, \psi_2 \colon E \to E'$ for which $\psi_1 \vert_{E[N]} = \psi_2 \vert_{E[N]}$ (up to automorphisms of $E$ and $E'$). Then $\widehat{\psi_2} \psi_1 \in \End(E) \setminus \bbZ$ has degree $m^2$. Let $-D$ be the discriminant of the quadratic order $\End(E)$, and set $\eta = (a + b\sqrt{-D})/2 = \widehat{\psi_2} \psi_1$. Since $\widehat{\psi_2}$ and $\psi_1$ are not dual $b \neq 0$. Taking norms shows that $4m^2 = a^2 + b^2D$. Noting that $\eta \vert_{E[N]}$ acts as $\pm 1$, the dichotomy of (i) and (ii) follows from \Cref{lemma:CM-pts}.
\end{proof}

\subsection{The lifted Hirzebruch--Zagier divisors \texorpdfstring{$\Fliftnum{2}{\lambda}$}{F\_(2,\unichar{"03BB})\textasciicircum(lift)} when \texorpdfstring{$N \equiv 2 \pmod{4}$}{N=2 mod 4}}
\label{sec:lifted-modular}
In this section we define additional modular curves on $\ZNr{N}{r}$ which play a role only when $N \equiv 2 \pmod{4}$.

Let $\Xliftnum{2}$ denote the modular curve attached to the subgroup $\left\langle \big( \begin{smallmatrix} 3 & 1 \\ 2 & 3 \end{smallmatrix} \big), \bigl( \begin{smallmatrix} 3 & 3 \\ 2 & 1 \end{smallmatrix} \bigr) \right\rangle \subset \GL_2(\bbZ/4\bbZ)$ (this is the modular curve with LMFDB label \LMFDBLabelMC{4.6.0.d.1}).

\begin{lemma}
  \label{lemma:Xlift-moduli}
  The non-cuspidal $K$-rational points on $\Xliftnum{2}$ parametrise triples $(E, \delta, \psi)$ where $E/K$ is an elliptic curve (defined up to quadratic twist), $\psi \colon E \to E'$ is a $K$-rational $2$-isogeny, and $\delta \in K^\times$ satisfies $\delta^2 = \Delta(E) \Delta(E')$. In particular, the elliptic curves $E$ and $E'$ are $2$-congruent and $2$-isogenous over $K$.
\end{lemma}

\begin{proof}
  Using an explicit model for $\Xliftnum{2} \cong \bbP^1$ with $j$-invariant $\frac{(t^2 + 16)^3}{t^2}$ (see e.g.,~\cite{RZB_ECOQA2AIOG}), it is easy to verify the first claim. In this case, choose Weierstrass equations $y^2 = x f(x)$ and $y^2 = xf'(x)$ for $E/K$ and $E'/K$ so that $K(E[2]) = K(\alpha)$ and $K(E'[2]) = K(\beta)$ where $\alpha$ and $\beta$ are roots of $f(x)$ and $f'(x)$ respectively. It is clear that $E$ and $E'$ are $2$-congruent over $K$ if and only if the discriminants of $f(x)$ and $f'(x)$ are equal up to a square.
\end{proof}

Let $N = 2M$ where $M$ is an odd integer. Consider a generic curve $\mathcal{E}/\bbQ(\Xliftnum{2})$, let $\Psi \colon \mathcal{E} \to \mathcal{E}'$ be a $2$-isogeny defined over $\bbQ(\Xliftnum{2})$, and let $\Phi \colon \mathcal{E}[2] \to \mathcal{E}'[2]$ be a $2$-congruence (these exist by \Cref{lemma:Xlift-moduli}).

Similarly to \Cref{sec:inters-modul-curv}, fix an integer $a \in \bbZ$ such that $r \equiv 2 a^2 \pmod{M}$. Since we have an isomorphism $\mathcal{E}[N] \cong \mathcal{E}[2] \times \mathcal{E}[M]$, for each class $\lambda \in \LamN{N}$ we obtain a point $(\mathcal{E}, \mathcal{E}', \Phi \times (a\lambda\Psi)\vert_{\mathcal{E}[M]})$ on $\ZNr{N}{r}$. We define the \emph{lifted Hirzebruch--Zagier divisor} $\Fliftnum{2}{\lambda}$ to be the image of the morphism $\Xliftnum{2} \to \ZNr{N}{r}$ given by sending the generic point of $\Xliftnum{2}$ to $(\mathcal{E}, \mathcal{E}', \Phi \times (a\lambda\Psi)\vert_{\mathcal{E}[M]})$. By construction the morphism $\Xliftnum{2} \to \Fliftnum{2}{\lambda}$ is birational.

\begin{lemma}
  \label{lemma:Ktil-F2lift}
  Let $N = 2M$ where $M$ is an odd integer and suppose that $2r$ is a square modulo $M$. Then we have $K_{\tilde{Z}} \cdot \Fliftnumtil{2}{\lambda} \leq 0$.
\end{lemma}

\begin{proof}
  By \eqref{eq:canonical-Z} we have a numerical equivalence $K_{\tilde{Z}} \sim \tfrac{1}{6} \widetilde{C}_{\star} - D_{\infty} - \tfrac{1}{3} E_{3,1}$. The index of $\Xliftnum{2}$ is $6$ and the number of cusps is $2$. By the projection formula $\widetilde{C}_{\star,i} \cdot \Fliftnumtil{2}{\lambda} = 6$ for each $i = 1,2$ and $D_{\infty} \cdot \Fliftnumtil{2}{\lambda} \geq 2$. The claim follows.
\end{proof}

\section{Diagonal Hirzebruch--Zagier divisors}
\label{sec:diag-HZ}
In this section we introduce divisors which play a central role in this work. In particular they make up the (one dimensional component of the) locus of fixed points of $\tau$ acting on $\ZNr{N}{r}$ (see \Cref{sec:modul-curv-ramif}).

\begin{defn}
  Let $g \in \GL_2(\bbZ/N\bbZ)$ be an element of determinant $r$. We define the \emph{diagonal Hirzebruch--Zagier divisor} $\Fgplus{g} \subset \ZNr{N}{r}$ to be the Zariski closure of the subset of $\ZNr{N}{r}(\bbC)$ consisting of the points $(E,E,\phi) \in \ZNr{N}{r}$ where $\phi$ ranges over the $\GL_2(\bbZ/N\bbZ)$-conjugacy class of $g$.   
\end{defn}

We begin by showing that the curves $\Fgplus{g}$ are birational to the modular curves $\Xgplus{g}$ defined in \Cref{sec:cartan-nearly-cartan}.

\begin{lemma}
  \label{lemma:Fgp-birational}
  The curve $\Fgplus{g}$ is $\bbQ$-birational to the modular curve $\Xgplus{g}$. 
\end{lemma}

\begin{proof}
  Let $K \subset \bbQ(\zeta_N)$ be a subfield over which $\Xgplus{g}$ decomposes into irreducible components and let $\calXgplus{g}/K$ be such an irreducible component. Consider a generic elliptic curve $\mathcal{E}/K(\calXgplus{g})$, and fix a basis for $\mathcal{E}[N]$ such that on this basis $\Phi = t \circ g \colon \mathcal{E}[N] \to \mathcal{E}^d[N]$ is an $N$-congruence for some choice of quadratic twist $t \colon \mathcal{E} \to \mathcal{E}^d$ of $\mathcal{E}$. Such a basis exists by \Cref{prop:Xgplus-moduli}. Consider the point $(\mathcal{E}, \mathcal{E}^d, \Phi) \in \Xgplus{g}$. By \Cref{lemma:MDQS} it is clear that $(\mathcal{E}, \mathcal{E}^d, \Phi)$ is a point on $\Fgplus{g}$, thus by taking its Zariski closure we obtain a rational map $\calXgplus{g} \dashrightarrow \Fgplus{g}$ defined over $K$.

  The irreducible components of $\Xgplus{g}$ are permuted by $\Gal(K/\bbQ)$ and therefore repeating the above for each irreducible component, we obtain a dominant rational map $\Xgplus{g} \dashrightarrow \Fgplus{g}$ defined over $\bbQ$. We therefore have a commutative diagram
  \begin{equation*}
    \begin{tikzcd}
      \Xgplus{g} && \Fgplus{g} \\
      & X(1)
      \arrow[from=1-1, to=2-2]
      \arrow[from=1-3, to=2-2]
      \arrow[dashed,from=1-1, to=1-3]
    \end{tikzcd}
  \end{equation*}
  where the map on the right factors through the morphism $\Fgplus{g} \to \DiagDiv$, where $\DiagDiv$ is the diagonal in $\Zone = X(1) \times X(1)$. It therefore suffices to show that the morphism $\Fgplus{g} \to \DiagDiv$ has degree $[\GL_2(\bbZ/N\bbZ) :  \Hgplus{g} ]$. But this is clear since $\Hgplus{g}$ is by definition the pre-image in $\GL_2(\bbZ/N\bbZ)$ of the stabiliser of $g$ when $\GL_2(\bbZ/N\bbZ)/\{\pm 1\}$ acts on itself by conjugation. 
\end{proof}

\begin{remark}
  If $I \in \GL_2(\bbZ/N\bbZ)$ denotes the identity matrix, it is simple to check that the diagonal Hirzebruch--Zagier divisor $\sum_{\lambda} \Fgplus{\lambda I}$ is exactly the ``usual'' Hirzebruch--Zagier divisor $F_1 = \sum_\lambda F_{1, \lambda}$.
\end{remark}

\subsection{Elements of \texorpdfstring{$\GL_2(\bbZ/N\bbZ)$}{GL\_2(\unichar{"2124}/N\unichar{"2124})} and notation for diagonal Hirzebruch--Zagier divisors}
\label{sec:notat-elements}

Throughout this work it will be necessary to specify elements of $\GL_2(\bbZ/N\bbZ)$ (in order to specify certain diagonal Hirzebruch--Zagier divisors). The elements of $\GL_2(\bbZ/N\bbZ)$ which play the greatest role are those with $g^2 = \pm \det(g)$ since in this case $\Fgplus{g}$ forms part of the $1$-dimensional component of the fixed locus of the involution $\tau$ on $\ZNr{N}{r}$ (see \Cref{sec:modul-curv-ramif}). Of secondary importance will be those $g \in \GL_2(\bbZ/N\bbZ)$ when $N$ is divisible by $2$ or $3$ for which the order of the $\GL_2(\bbZ/N\bbZ)$-conjugacy class is small, since $\Fgplus{g}$ gives rise to $(-1)$ or $(-2)$-curves a surface birational to $\ZNr{N}{r}$ or $\ZNrSym{N}{r}$ (see \Crefrange{lemma:flat-smooth}{lemma:flat-smooth-3} and \ref{lemma:Fb4-dot-K}--\ref{lemma:F-3b}). To this end we make the following definitions.

\begin{defn}
  \label{def:defining-the-g}
  Let $p$ be a prime number and $k \geq 1$ be an integer. Let $r$ be an integer coprime to $p$ and if $p$ is odd fix an integer $\xi \in \bbZ$ which is a quadratic non-residue modulo $p$. For each label $\bullet$ in \Cref{table:order-2-notation} we choose $g_{\bullet}(p^k, r) \in \GL_2(\bbZ/p^k\bbZ)$ to be a matrix with $\det ( g_{\bullet}(p^k, r) ) = r$ and with projective image $\bbP(g_{\bullet}(p^k)) \in \PGL_2(\bbZ/p^k\bbZ)$ recorded in \Cref{table:order-2-notation}.
  
  \begingroup
  \renewcommand{\arraystretch}{2.05}
  \begin{table}[t]
    \centering
    \begin{tabular}{c|c|c|c}
      $p^k$                & Label $\bullet$ & $r$           & $\bbP( g_\bullet(p^k) ) \in \PGL_2(\bbZ/p^k\bbZ)$ \\
      \hline
                           & $I$             & $1$           & $\myspecialmat{\hspace{0.5mm}1\hspace{0.5mm}}{0}{0}{\hspace{0.5mm}1\hspace{0.5mm}}$                         \\[0.7mm]
      $*$                  & $\borel$        & $1$           & $\myspecialmat{1}{p^{k-1}}{0}{1}$                   \\[0.9mm]
                           & $\borel p^\ell$ & $1$           & $\myspecialmat{1}{p^{k-\ell}}{0}{1}$                \\[2mm]
      \hdashline
      \multirow{2}{*}{odd} & $\ns$           & $-\xi$        & $\myspecialmat{0}{\hspace{0.5mm}\xi\hspace{0.5mm}}{\hspace{0.5mm}1\hspace{0.5mm}}{0}$                       \\
                           & $\s$            & $-1$          & $\myspecialmat{\hspace{0.5mm}1\hspace{0.5mm}}{0}{0}{-1\vspace{-0.5mm}}$                        \\[1mm]
      \hdashline
                           & $\Isharp$       & {\small$\begin{cases}1 & \text{if $k=1$} \\[0.5mm] 2^{k-1} + 1 & \text{if $k \neq 1$}\end{cases}$} &  $\myspecialmat{1}{2^{k-1}}{2^{k-1}}{2^{k-1}+1}$     \\[0.5em]
                           & $\borelsharp$   & $2^{k-1} + 1$ & $\myspecialmat{\hspace{1mm}1\hspace{1mm}}{0}{0}{2^{k-1}+1}$                 \\
      $2^k$                & $\ns$           & $3$           & $\myspecialmat{1}{2}{-2}{-1}$                       \\
                           & $\s$            & $-1$          & $\myspecialmat{\hspace{0.5mm}1\hspace{0.5mm}}{0}{0}{-1}$                        \\
                           & $\antidiag$     & $r$           & $\myspecialmat{0}{-r}{\hspace{0.5mm}1\hspace{0.5mm}}{0}$                        \\[1mm]
      \hdashline
      $3$                  & $\nsalt$        & $-1$           & $\myspecialmat{1}{1\hspace{0.5mm}}{-1}{1\hspace{0.5mm}}$                        \\ 
    \end{tabular}
    \caption{Our notation for specific elements in $\PGL_2(\bbZ/p^k\bbZ)$ of determinant $r$ modulo $((\bbZ/p^k\bbZ)^\times)^2$. Here $\xi \in \bbZ$ is a quadratic non-residue modulo $p$. Note that $\bbP( g_{\borelsharp}(2^k))$ is defined only for $k \geq 2$.}
    \label{table:order-2-notation}
  \end{table}
  \endgroup
\end{defn}

\begin{defn}
  \label{def:g-weyl}
  If $p$ is an odd prime number let $\bigl( \tfrac{\cdot}{p} \bigr)$ denote the Legendre symbol. We define $g_{\weyl}(p^k, r) \in \GL_2(\bbZ/p^k\bbZ)$ to be given by
  \begin{equation}
    \label{eqn:g-weyl}
    g_{\weyl}(p^k,r) =
    \begin{cases}
      g_{\s}(p^k, r) & \text{if $\big( \frac{-r}{p} \bigr) = 1$, and} \\[2mm]
      g_{\ns}(p^k, r) & \text{if $\big(\frac{-r}{p} \bigr) = -1$.} \\
    \end{cases}
  \end{equation}
  
  If $M$ is an odd integer we define $g_{\weyl}(M,r) \in \GL_2(\bbZ/M\bbZ)$ to be the unique element such that $g_{\weyl}(M,r) \equiv g_{\weyl}(p^k,r) \pmod{p^k}$ for each prime power $p^k \mid M$.
\end{defn}

\begin{remark}
  The property which characterises $g_{\weyl}(M,r)$, and which we make repeated use of in this work, is that $g_{\weyl}(M,r)$ is a representative for the unique conjugacy class of elements of $\GL_2(\bbZ/M\bbZ)$ of determinant $r$ for which $g^2 = - \det(g)$ (see \Cref{lemma:fixer-g-N}). When $M = p^k$ the matrix $g_{\weyl}(M,r)$ is a representative for the Weyl group of an appropriate torus in $\GL_2(\bbZ/p^k\bbZ)$.
\end{remark}

\begin{remark}
  Similarly to \Cref{rmk:indep-choice}, the choice of lift $g_{\bullet}(p^k, r)$ of $\bbP(g_{\bullet}(p^k))$ will not play a role in this work since we deal collectively with the matrices $\lambda g_{\bullet}(p^k, r)$ for each $\lambda \in \LamN{p^k}$.
\end{remark}

When the integers $p^k$ and $r$ are implicit, we often write $g_{\bullet} = g_{\bullet}(p^k, r)$ for brevity, where $\bullet$ is one of the labels in \Cref{def:defining-the-g,def:g-weyl}.
For convenience we introduce the following notation.

\begin{notation}
  We adopt the following conventions:
  \begin{enumerate}[label=\eniii]
  \item
    If $N = 2^k M$ where $M$ is an odd integer we fix an isomorphism $\GL_2(\bbZ/N\bbZ) \cong \GL_2(\bbZ/2^k\bbZ) \times \GL_2(\bbZ/M\bbZ)$ and write $$(g,h) \in \GL_2(\bbZ/2^k \bbZ) \times \GL_2(\bbZ/M\bbZ) \cong \GL_2(\bbZ/N\bbZ).$$
  \item
    If $N = 3 Q$ where $Q$ is an integer coprime to $3$ we write $$\three{g}{h} \colonequals (g,h) \in \GL_2(\bbZ/3\bbZ) \times \GL_2(\bbZ/Q\bbZ) \cong \GL_2(\bbZ/N\bbZ).$$
  \end{enumerate}
\end{notation}

\begin{notation}
  \label{notation:X-bullet}
  Write $N = 2^k M$  where $k \geq 0$ and $M$ is an odd integer.
  \begin{enumerate}[label=\eniii]
  \item
    If $g_\bullet \in \GL_2(\bbZ/N\bbZ)$ is an element with label $\bullet$ we abbreviate the notation for the curves
    \begin{equation*}
      \Xgplus{\bullet} \colonequals \Xgplus{g_\bullet} \quad \text{and} \quad \Fgplus{\bullet} \colonequals \Fgplus{g_\bullet} .
    \end{equation*}

  \item
    If $\lambda \in \LamN{N}$ we write
    \begin{equation*}
      \Xgplus{\lambda\bullet} \colonequals \Xgplus{\lambda g_\bullet} \quad \text{and} \quad \Fgplus{\lambda \bullet} \colonequals \Fgplus{\lambda g_\bullet}. 
    \end{equation*}

  \item
    If $k \neq 0$ and if $g_\bullet \in \GL_2(\bbZ/2^k\bbZ)$ and $g_{\tblacksquare} \in \GL_2(\bbZ/M\bbZ)$ are elements with labels $\bullet$ and $\tblacksquare$ we write
    \begin{equation*}
     \Xgplus{\bullet, \tblacksquare} \colonequals \Xgplus{g_\bullet, g_\tblacksquare} \quad \text{and} \quad \Fgplus{\bullet,\tblacksquare} \colonequals \Fgplus{g_\bullet,g_{\tblacksquare}}. 
    \end{equation*}
    
  \item
    If $k \neq 0$ and $\lambda \in \LamN{N}$ we write
    \begin{equation*}
      \Xgplus[\lambda]{\bullet, \tblacksquare} \colonequals \Xgplus{\lambda g_\bullet, \lambda g_\tblacksquare} \quad \text{and} \quad \Fgplus[\lambda]{\bullet, \tblacksquare} \colonequals \Fgplus{\lambda g_\bullet,\lambda g_{\tblacksquare}}.
    \end{equation*}

  \item
    If $k \neq 0$, $\omega \in \LamN{2^k}$, and $\lambda \in \LamN{M}$ we write
    \begin{equation*}
      \Xgplus{\omega\bullet, \lambda\tblacksquare} \colonequals \Xgplus{\omega g_\bullet, \lambda g_\tblacksquare} \quad \text{and} \quad \Fgplus{\omega\bullet, \lambda\tblacksquare} \colonequals \Fgplus{\omega g_\bullet, \lambda g_\tblacksquare}.
    \end{equation*}
    
  \item
    When $N = 3Q$ we analogously define $\Xgplus{\three{\bullet}{\tblacksquare}}$, $\Fgplus{\three{\bullet}{\tblacksquare}}$, $\Xgplus{\lambda\three{\bullet}{\tblacksquare}}$, $\Fgplus{\lambda\three{\bullet}{\tblacksquare}}$, $\Xgplus{\three{\omega\bullet}{\lambda\tblacksquare}}$, and $\Fgplus{\three{\omega\bullet}{\lambda\tblacksquare}}$ for elements $g_{\bullet} \in \GL_2(\bbZ/3\bbZ)$ and $g_{\tblacksquare} \in \GL_2(\bbZ/Q\bbZ)$.
  \end{enumerate}

  We will take, for example, $\bullet$ and $\tblacksquare$ to be any of the labels $I$, $\Isharp$, $\s$, $\ns$, $\nsalt$, $\borel$, $\borelsharp$, $\antidiag$, or $\weyl$ from \Cref{def:defining-the-g,def:g-weyl}.
\end{notation}

\begin{remark}
  \label{remark:FF-and-Fp}
  When $N = p$ is an odd prime number the curve $\Fgplus{\weyl}$ is equal to the curve which is denoted by Hermann as $F_{p^2}$~\cite{H_SMDDp2}. When $g \in \GL_2(\bbZ/N\bbZ)$ satisfies $g^2 = \pm \det(g)$, the curves $\Fgplus{g}$ are closely related to the Hirzebruch--Zagier divisors $F_{d}$ on the Hilbert modular surface of fundamental discriminant $D$ for each $d \mid D$ which are studied by Hausmann~\cite{H_KAHM,H_TFPOTSHMGOARQFWAD} (see also \cite[Section~V.9]{vdG_HMS}). Indeed, the reader should compare \Cref{coro:M-odd-components} and \cite[\S2.2 Proposition]{H_TFPOTSHMGOARQFWAD} (see also \cite[Proposition~V.9.1]{vdG_HMS} and, when $D$ is prime, \cite{P_DFDSHMGZERQZMP}). We do not know of an analogous construction for arbitrary $g \in \GL_2(\bbZ/N\bbZ)$. It would be interesting to provide a construction of curves on a Hilbert modular surface of arbitrary discriminant which specialises to $\Fgplus{g}$ when the discriminant is a square.
\end{remark}

The constructions in the following \Crefrange{sec:examples-with-small}{sec:isogeny-compose} play a role only in the (more subtle) cases when $N$ is divisible by $2$ or $3$.

\subsection{Examples of curves \texorpdfstring{$\Fgplus{g}$}{F\_g{\textasciicircum}+} with small genera}
\label{sec:examples-with-small}

For later use, we now record several diagonal Hirzebruch--Zagier divisors with small genera, and whose self-intersection on $\ZNrtil{N}{r}$ is small in absolute value. We will see in \Cref{sec:special-curves} that these curves give rise to either $(-1)$ or $(-2)$-curves on $\ZNrtil{N}{r}$ or (a smooth surface birational to) its quotient by the involution induced by ${\tau}$.

\begin{lemma}
  \label{lemma:flat-smooth}
  Let $N = 2^k M$ where $M$ is an odd integer and suppose that $\ZNrtil{N}{r}$ is not rational. For each $g$ recorded in \Cref{table:small-genus} the curve $\Fgplustil{g}$ is non-singular, and the intersection number $K_{\tilde{Z}} \cdot \Fgplustil{g}$ is given by the final column of \Cref{table:small-genus}.
\end{lemma}

\begingroup
\renewcommand{\arraystretch}{1.3}
\begin{table}[H]
  \centering
  \begin{tabular}{c|cccccccc}
    $g \in \GL_2(\bbZ/N\bbZ)$      & $d_{g}$ & LMFDB label of $\calXgplus{g}$ & $\eta^+$ & $\epsilon_{\infty}^+$ & $\epsilon_2^+$ & $\epsilon_3^+$ & $p_g(\calXgplus{g})$ & $K_{\tilde{Z}} \cdot \Fgplustil{g}$ \\
    \hline
    $(g_{\Isharp}, g_I)$            & $1$ & \LMFDBLabelMC{2.2.0.a.1} & $2$   & $1$          & $0$   & $2$   & $0$               & $-1$                                \\
    $(g_{\borel}, g_I)$             & $1$ & \LMFDBLabelMC{2.3.0.a.1} & $3$   & $2$          & $1$   & $0$   & $0$               & $-1$                                \\
    $(g_{\borelsharp}, g_I)$        & $1$ & \LMFDBLabelMC{2.6.0.a.1} & $6$   & $3$          & $0$   & $0$   & $0$               & $-1$                                \\
    $(g_{\borel \! 4}, g_I)$        & $2$ & \LMFDBLabelMC{4.6.0.c.1} & $6$   & $3$          & $0$   & $0$   & $0$               & $-2$                                \\
  \end{tabular}
  \caption{Numerical invariants for modular curves $\Xgplus{g}$. Here $\eta^+$, $\epsilon_{\infty}^+$, $\epsilon_2^+$, and $\epsilon_3^+$ are respectively the index, number of cusps, and number of elliptic points of order $2$ and $3$ on a \emph{geometrically irreducible component} $\calXgplus{g}$ of $\Xgplus{g}$ and $p_g$ denotes the geometric genus. The column $d_g$ records the index of $\det(\Hgplus{g})$ in  $(\bbZ/N\bbZ)^\times$. The third column records the LMFDB label of a modular curve $X(H)$ such that $-I \in H$ and $H \cap \SL_2(\bbZ/N\bbZ) = \Hgplus{g} \cap \SL_2(\bbZ/N\bbZ)$.}
  \label{table:small-genus}
\end{table}
\endgroup

\begin{proof}
  Let $\calXgplus{g}/\bbQ(\zeta_N)$ be a geometrically irreducible component of $\Xgplus{g}$. The index, number of cusps, and number of elliptic points of order $2$ and $3$ on the modular curves $\calXgplus{g}$ (which we denote $\eta^+$, $\epsilon_{\infty}^+$, $\epsilon_2^+$, and $\epsilon_3^+$) are simple to compute, and are recorded in \Cref{table:small-genus}. The geometric genera follow from the standard formula $p_g(\calXgplus{g}) = 1 + \frac{1}{12} \eta^+ - \frac{1}{4} \epsilon_2^+ - \frac{1}{3} \epsilon_3^+ - \frac{1}{2} \epsilon_\infty^+$, see e.g., \cite[Theorem~3.1.1]{DS_AFCIMF}. 
  
  Recall from \eqref{eq:canonical-Z} that we have a numerical equivalence
  \begin{equation*}
    K_{\tilde{Z}} \sim \frac{1}{6} \widetilde{C}_{\star} - D_\infty - \frac{1}{3} E_{3, 1}.
  \end{equation*}
  Note that by the projection formula, for each irreducible component $\calFgplustil{g}$ of $\Fgplustil{g}$, we have $\widetilde{C}_{\star} \cdot \calFgplustil{g} = 2 \eta^+$ and that $D_\infty \cdot \calFgplustil{g} \geq \epsilon_\infty^+$ and $E_{3,1} \cdot \calFgplustil{g} \geq  0$ (cf. \Cref{lemma:sing-to-genus}). In particular $K_{\tilde{Z}} \cdot \calFgplustil{g} \leq -1$.

  Since $\ZNrtil{N}{r}$ is assumed to be non-rational, \Cref{lemma:rat-crit} implies that the geometric components of $\Fgplustil{g}$ are non-singular $(-1)$-curves which do not intersect. Since $d_g = [(\bbZ/N\bbZ)^\times : \det(\Hgplus{g})]$ is the number of geometrically irreducible components of $\Fgplus{g}$ it follows that $K_{\tilde{Z}} \cdot \Fgplustil{g} = -d_g$ is equal to the integer listed in \Cref{table:small-genus}.
\end{proof}

\begin{lemma}
  \label{lemma:flat-smooth-3}
  Let $N = 3 Q$ where $Q \neq 1,2$ is an integer coprime to $3$ and suppose that $\ZNrtil{N}{r}$ is not rational. For each $g$ recorded in \Cref{table:small-genus-3} the intersection number $K_{\tilde{Z}} \cdot \Fgplustil{g}$ is bounded by (or given by, if $g = \three{g_{\borel}}{g_I}$) the final column of \Cref{table:small-genus-3}.
\end{lemma}

\begingroup
\renewcommand{\arraystretch}{1.3}
\begin{table}[H]
  \centering
  \begin{tabular}{c|cccccccc}
    $g \in \GL_2(\bbZ/N\bbZ)$ & $d_{g}$ & LMFDB label of $\calXgplus{g}$ & $\eta^+$ & $\epsilon_{\infty}^+$ & $\epsilon_2^+$ & $\epsilon_3^+$ & $p_g(\calXgplus{g})$ & $K_{\tilde{Z}} \cdot \Fgplustil{g}$ \\
    \hline
    $\three{g_{\borel}}{g_I}$ & $2$     & \LMFDBLabelMC{3.4.0.a.1}       & $4$      & $2$                   & $0$            & $1$            & $0$                  & $-2$                                \\
    $\three{g_{\ns}}{g_I}$    & $1$     & \LMFDBLabelMC{3.6.0.a.1}       & $6$      & $2$                   & $2$            & $0$            & $0$                  & $\leq 0$                            \\
    $\three{g_{\s}}{g_I}$     & $1$     & \LMFDBLabelMC{3.12.0.a.1}      & $12$     & $4$                   & $0$            & $0$            & $0$                  & $\leq 0$                            \\
    $\three{g_{\nsalt}}{g_I}$ & $1$     & \LMFDBLabelMC{3.6.0.a.1}       & $6$      & $2$                   & $2$            & $0$            & $0$                  & $\leq 0$                            \\
  \end{tabular}
  \caption{Numerical invariants for modular curves $\Xgplus{g}$. Here $\eta^+$, $\epsilon_{\infty}^+$, $\epsilon_2^+$, and $\epsilon_3^+$ are respectively the index, number of cusps, and number of elliptic points of order $2$ and $3$ on a \emph{geometrically irreducible component} $\calXgplus{g}$ of $\Xgplus{g}$ and $p_g$ denotes the geometric genus. The column $d_g$ records the index of $\det(\Hgplus{g})$ in  $(\bbZ/N\bbZ)^\times$. The third column records the LMFDB label of a modular curve $X(H)$ such that $-I \in H$ and $H \cap \SL_2(\bbZ/N\bbZ) = \Hgplus{g} \cap \SL_2(\bbZ/N\bbZ)$.}
  \label{table:small-genus-3}
\end{table}
\endgroup

\begin{proof}
  This follows from an identical argument to \Cref{lemma:flat-smooth}.
\end{proof}

\subsection{Composing with an isogeny}
\label{sec:isogeny-compose}
Finally, we define some modular curves on $\ZNr{N}{r}$ which will form either $(-1)$ or $(-2)$-curves on (a surface birational to) $\ZNrSym{N}{r}$ for some integers $N$ divisible by either $2$ or $3$. These curves combine the ``usual'' Hirzebruch--Zagier divisors from \Cref{sec:inters-modul-curv}, with the diagonal version from this section.

\begin{defn}
  Let $m$ be an integer coprime to $N$ and let $g \in \GL_2(\bbZ/N\bbZ)$ be an element such that $m \det(g) \equiv r \pmod{N}$. We define $F_{m \circ g}^+$ to be the Zariski closure in $\ZNr{N}{r}$ of the set of points $(E, E', \phi (\psi\vert_{E[N]}) ) \in \ZNr{N}{r}(\bbC)$ where $\phi \in \GL_2(\bbZ/N\bbZ)$ runs through the conjugacy class of $g$ and  $\psi \colon E \to E'$ is a cyclic $N$-isogeny.
\end{defn}

\begin{remark}
  Note that the curves $F_{m \circ g}^+$ are strictly more general than both $F_{m,\lambda}$ and $\Fgplus{g}$. The former may be recovered by taking $g$ to be a scalar matrix, and the latter by taking $m = 1$.
\end{remark}

The following lemmas are proved analogously to \Cref{lemma:Fgp-birational,lemma:flat-smooth} respectively.

\begin{lemma}
  \label{lemma:Fmg-bir}
  The curve $F_{m \circ g}^+$ is birational over $\bbQ$ to the fibre product $X_0(m) \times_{X(1)} \Xgplus{g}$.
\end{lemma}

\begin{lemma}
  \label{lemma:SI-of-F3bI}
  Let $N = 2^k M$ where $M$ is odd and $k \geq 1$. The curve $\widetilde{F}_{3 \circ (\borel, I)}^+$ is birational to $X_0(6)$ and we have $K_{\tilde{Z}} \cdot \widetilde{F}_{3 \circ (\borel, I)}^+ \leq 0$.
\end{lemma}

\FloatBarrier
\section{The fixed points of \texorpdfstring{$\tau$}{\unichar{"1D70F}}}
\label{sec:fixed-points-tau}
Recall that by definition for a point $\ffz = (E, E', \phi) \in \ZNr{N}{r}$ we have $\tau(\ffz) = (E', E, r \phi^{-1})$. Let $\tilde{\tau}$ be the involution of $\ZNrtil{N}{r}$ induced by the action of $\tau$ on $\ZNr{N}{r}$. Since $\tau$ is a morphism and $\ZNrtil{N}{r}$ is the minimal resolution of singularities of $\ZNr{N}{r}$ the involution $\tilde{\tau}$ is also a morphism. 

Our first task is to understand the locus of fixed points of the action of $\tilde{\tau}$ on the resolution of singularities $\ZNrtil{N}{r}$ (cf. \cite[Chapter~5]{H_HMS}, \cite{HZ_COHMS}, \cite{H_KAHM}, \cite{H_TFPOTSHMGOARQFWAD}, \cite{B_KHMZSHME}, \cite[V.7]{vdG_HMS}, \cite{H_SMDDp2}). We begin with a lemma which restricts the points on $\ZNr{N}{r}$ whose total transforms we need to consider.

\begin{lemma}
  \label{lemma:fixed-pts}
  If a point $\ffz = (E,E',\phi) \in \ZNr{N}{r}(\bbC)$ is fixed by $\tau$ then $j(E) = j(E')$. Moreover, fix an isomorphism $E \cong E'$ (and hence an isomorphism $\Isom(E[N], E'[N]) \cong \GL_2(\bbZ/N\bbZ)$), then either
  \begin{enumerate}[label=\eniii]
  \item \label{enum:fp-iso}
    the point $\ffz$ is singular (hence $j(E) = j(E') \in \{0, 1728, \infty\}$), or
  \item \label{enum:fp-1dim}
    the congruence $\phi$ satisfies $\phi^2 = \pm r$ as an element of $\GL_2(\bbZ/N\bbZ)$.
  \end{enumerate}
\end{lemma}

\begin{proof}
  Recall that we define $\ZNr{N}{r}$ to be the quotient $\Delta_\ee \backslash ( X(N) \times X(N) )$ where $\ee = \mymat{1}{0}{0}{r}$ and $\Delta_\ee$ is defined in \Cref{sec:surf-ZNr}. Since $\ffz$ is fixed by $\tau$ it is immediate that $j(E) = j(E')$. If $\ffz$ is non-singular, choose a pre-image $\mathfrak{Z} = ((E, \iota), (E', \iota')) \in X(N) \times X(N)$, that is, so that $\phi = (\iota')^{-1} \ee \iota$. Since $\ffz$ is non-singular the point $\mathfrak{Z}$ is stabilised only by $\pm 1 \in \Delta_\ee$. In particular, the assumption that $\ffz = \tau(\ffz)$ implies that $\phi = \pm r \phi^{-1}$ as elements of $\GL_2(\bbZ/N\bbZ)$. 
\end{proof}

We will see that in case \ref{enum:fp-1dim} of \Cref{lemma:fixed-pts} the point $\ffz$ is part of the one dimensional fixed locus of $\tau$. We will classify both this fixed locus and the fixed points $\ffz \in \ZNr{N}{r}$ above $(0,0)$, $(1728,1728)$, and $(\infty, \infty)$. Indeed, we show that all the fixed points contained in the resolutions of singular points of type \ref{enum:fp-iso} in \Cref{lemma:fixed-pts} are isolated fixed points of $\tilde{\tau}$ (see \Cref{prop:action-aff-figs,prop:action-cusps-figs}).

\subsection{The one dimensional fixed locus}
\label{sec:modul-curv-ramif}
The points in \Cref{lemma:fixed-pts}\ref{enum:fp-1dim} appear to, in some sense, be continuously varying and parametrised by a finite number of elements of $\GL_2(\bbZ/N\bbZ)$. In fact, it is the case that all of these points lie on a finite union of irreducible curves fixed by the action of $\tau$.  This section concludes with \Cref{coro:M-odd-components}, which describes the irreducible components of the one dimensional component of the locus of fixed points of $\tau$, which we denote $\Ram\!_Z$. In particular, when $N$ is an odd integer we show that (except for the Hirzebruch--Zagier divisor $F_1$) the only component of $\Ram\!_Z$ is the irreducible diagonal Hirzebruch--Zagier divisor $\Fgplus{\weyl}$. In this case, the curve $\Fgplus{\weyl}$ is birational to the modular curve attached to an extended Cartan subgroup of $\GL_2(\bbZ/N\bbZ)$ (when $N$ is a prime number this is \cite[Hilfssatz~1]{H_SMDDp2}).

\begin{theorem}
  \label{thm:ramification-locus}
  Let $\Ram\!_Z \subset \ZNr{N}{r}$ be the union of the one dimensional irreducible components of the fixed point locus of $\tau$. Then $\Ram\!_Z = \sum_g \Fgplus{g}$ where $g$ ranges over representatives for the conjugacy classes of $g \in \GL_2(\bbZ/N\bbZ)/\{\pm 1\}$ such that $\det(g) = r$ and $g^2 = \pm \det(g)$ for some choice of sign.
\end{theorem}

\begin{proof}
  The claim follows immediately from \Cref{lemma:fixed-pts,lemma:Fgp-birational}.
\end{proof}

\begin{coro}
  \label{coro:ram-locus-smooth}
  Let $g,h \in \GL_2(\bbZ/N\bbZ)$ be elements such that $g^2 = \pm \det(g)$, $h^2 = \pm \det(h)$ and $\det(g) = \det(h) = r$. The curves $\Fgplustil{g} \subset \ZNrtil{N}{r}$ are non-singular and therefore isomorphic to the modular curve $\Xgplus{g}$. Moreover, if $\Fgplustil{g} \neq \Fgplustil{h}$ then $\Fgplustil{g} \cdot \Fgplustil{h} = 0$.
\end{coro}

\begin{proof}
  Both claims follow by combining \Cref{lemma:fixed-is-smooth,thm:ramification-locus}.
\end{proof}

By \Cref{thm:ramification-locus}, in order to explicitly determine all the irreducible components of $\Ram\!_Z$ it suffices to determine the conjugacy classes of matrices $g \in \GL_2(\bbZ/N\bbZ)$ such that $g^2 = \pm \det(g)$. When $N$ is an odd prime power the only such conjugacy classes are the matrices $aI$ for some $a \in (\bbZ/N\bbZ)^\times$ and the elements normalising a Cartan subgroup (i.e., the elements $g_{\s}$ and $g_{\ns}$), see e.g.,~\cite[Lemma~3.5]{F_COECAFNSMNGR}. Since $g_{\s}^2 = -\det(g_{\s})$ and $g_{\ns}^2 = -\det(g_{\ns})$ it follows from the Chinese remainder theorem that if $N$ is odd and $g \in \GL_2(\bbZ/N\bbZ)$ satisfies $g^2 = \pm \det(g)$, then either $g = aI$ for some $a \in (\bbZ/N\bbZ)^\times$ or $g$ is conjugate to $g_{\weyl}(N,r)$ for some $r$. This yields the case of \Cref{lemma:fixer-g-N} when $N$ is odd. When $N$ is even the analysis is elementary, but more computational, and we relegate the proof to \Cref{sec:order-2-app}.

\begin{lemma}
  \label{lemma:fixer-g-N}
  Let $N > 1$ be an integer, and let $N = 2^k M$ where $M$ is an odd integer. The conjugacy classes of elements $g \in \GL_2(\bbZ/N\bbZ) / \{ \pm 1 \}$ of determinant $r$ for which $g^2 = \pm \det(g)$ are as recorded in \Cref{table:order2}.
\end{lemma}

\begingroup
\renewcommand{\arraystretch}{1.3}
\begin{table}[t]
  \centering
  \begin{tabular}{c|c|c|c}
    $k$                       & $r \mod{2^e}$        & $r \mod{M}$ & \makecell[c]{Representatives for conjugacy classes of                                           \\ $g \in \GL_2(\bbZ/N\bbZ) / \{\pm 1\}$ with $g^2 = \pm \det(g)$} \\
    \hline
    \multirow{2}{*}{$0$}      & \multirow{2}{*}{$*$} & square      & $\lambda g_I$, \, $g_{\weyl}$                                                                     \\
                              &                      & non-square  & $g_{\weyl}$                                                                                     \\
    \hdashline
    \multirow{2}{*}{$1$}      & \multirow{2}{*}{$*$} & square      & $\lambda g_I$, \, $\lambda (g_{\borel}, g_I)$, \, $(g_I, g_{\weyl})$, \, $(g_{\antidiag}, g_{\weyl})$ \\
                              &                      & non-square  & $(g_I, g_{\weyl})$, \, $(g_{\antidiag}, g_{\weyl})$                                               \\
    \hdashline
    \multirow{3}{*}{$2$}      & \multirow{2}{*}{$1$} & square      & $\lambda g_I$, \, $\lambda (g_{\borel}, g_I)$, \, $(g_{\antidiag}, g_{\weyl})$                      \\
                              &                      & non-square  & $(g_{\antidiag}, g_{\weyl})$                                                                    \\
    \cdashline{2-4}[0.6pt/2pt]
                              & $3$                  & $*$         & $(g_{\s},g_{\weyl})$, \, $(g_{\ns},g_{\weyl})$, \, $(g_{\antidiag}, g_{\weyl})$                 \\
    \hdashline
    \multirow{5}{*}{$\geq 3$} & \multirow{2}{*}{$1$} & square      & $\lambda g_I$, \, $\lambda (g_{\borel}, g_I)$, \, $(g_{\antidiag}, g_{\weyl})$                      \\
                              &                      & non-square  & $(g_{\antidiag}, g_{\weyl})$                                                                    \\
    \cdashline{2-4}[0.6pt/2pt]
                              & $3$                  & $*$         & $(g_{\ns}, g_{\weyl})$, \, $(\omega g_{\ns},g_{\weyl})$, \, $(g_{\antidiag}, g_{\weyl})$        \\
                              & $5$                  & $*$         & $(g_{\antidiag}, g_{\weyl})$                                                                    \\
                              & $7$                  & $*$         & $(g_{\ns}, g_{\weyl})$, \, $(\omega g_{\ns},g_{\weyl})$, \, $(g_{\antidiag}, g_{\weyl})$        \\

  \end{tabular}
  \caption{Conjugacy classes of elements $g \in \GL_2(\bbZ/N\bbZ) / \{\pm 1\}$ for which $g^2 = \pm \det(g)$. Here $N = 2^k M$ where $M$ is odd, $r = \det(g)$, and $2^e = \min(8, 2^k)$. The element $\lambda$ ranges over representatives for $\LamNr{N}{r}$ and $\omega = 2^{k-1} + 1$.}
  \label{table:order2}
\end{table}
\endgroup

The following corollary of \Cref{thm:ramification-locus} can be seen as a generalisation of an observation of Hermann~\cite[Hilfssatz~1]{H_SMDDp2} who notes that when $N=p$ is a prime number the irreducible components of the one dimensional fixed locus of $\tau$ consists of either one or two irreducible components, one of which is birational to the modular curve associated to an extended Cartan subgroup\footnote{As defined therein \cite[Hilfssatz 1]{H_SMDDp2} is false when $-1$ is a square modulo $p$. The subgroup $O_{\epsilon}(p)$ on p.~171 of \emph{loc. cit.} should vary across $\epsilon \in \{1,\xi\}$ where $\xi \in \bbF_p^\times \setminus \left(\bbF_p^\times\right)^2$ and $\left(\frac{\epsilon}{p}\right) = \left(\frac{N}{p}\right)\left(\frac{\gamma_\nu}{p}\right)$. This ensures that $O_\epsilon(p) \backslash \bbH$ is the modular curve associated to the correct extended Cartan subgroup of $\GL_2(\bbF_p)$.  This is presumably a typographical error since the genera in \cite[(7)]{H_SMDDp2} agree with those of the \emph{correct} modular curves (cf. \Cref{lem:genus-Xg-odd}).}.
The latter observation is also made by Fisher~\cite[Sections~2.3~and~7]{F_OFO13CEC} in the case when $p = 13$, though their argument applies more generally to all odd prime numbers.

\begin{coro}
  \label{coro:M-odd-components}
  Let $N = 2^k M$ where $M$ is an odd integer and $k \geq 0$. Then $\Ram\!_Z$ decomposes into irreducible components as:
  \begin{equation*}
    \Ram\!_Z = 
    \begin{cases}
      \Fgplus{\weyl} + \sum_\lambda \Fgplus[\lambda]{I}                                                                         &  \text{if $k = 0$,}                             \\
      \Fgplus{\antidiag,\weyl} + \Fgplus{I, \weyl} + \sum_\lambda \Big(\Fgplus[\lambda]{\borel, I} + \Fgplus[\lambda]{I,I}\Big) & \text{if $k = 1$,}                              \\
      \Fgplus{\antidiag,\weyl}  + \sum_\lambda \Big(\Fgplus[\lambda]{\borel, I} + \Fgplus[\lambda]{I,I}\Big)                    & \text{if $k = 2$ and $r \equiv 1 \pmod{4}$,}    \\
      \Fgplus{\antidiag,\weyl} + \Fgplus{\s, \weyl} + \Fgplus{\ns, \weyl}                                                       & \text{if $k = 2$ and $r \equiv 3 \pmod{4},$}    \\
      \Fgplus{\antidiag,\weyl} + \sum_\lambda \Big(\Fgplus[\lambda]{\borel, I} + \Fgplus[\lambda]{I,I}\Big)                     & \text{if $k \geq 3$ and $r \equiv 1 \pmod{8}$,} \\
      \Fgplus{\antidiag,\weyl} + \Fgplus{\ns, \weyl} + \Fgplus{\omega \!\ns, \weyl}                                             & \text{if $k \geq 3$ and $r \equiv 3 \pmod{8},$} \\
      \Fgplus{\antidiag,\weyl}                                                                                                  & \text{if $k \geq 3$ and $r \equiv 5 \pmod{8}$,} \\
      \Fgplus{\antidiag,\weyl} + \Fgplus{\s, \weyl} + \Fgplus{\omega \!\s, \weyl}                                               & \text{if $k \geq 3$ and $r \equiv 7 \pmod{8},$} \\
    \end{cases}
  \end{equation*}
  where $\omega = 2^{k-1} + 1$ and in the sums $\lambda$ ranges over elements of $\LamNr{N}{r}$.
\end{coro}

\begin{proof}
  By \Cref{thm:ramification-locus} it suffices to classify those $g \in \GL_2(\bbZ/N\bbZ)/\{\pm 1\}$ such that $g^2 = \pm \det(g)$. The claim then follows from \Cref{lemma:fixer-g-N}, noting that in each case $\det(\Hgplus{g}) = (\bbZ/N\bbZ)^\times$.
\end{proof}

\subsection{The fixed points above \texorpdfstring{$0$ and $1728$}{0 and 1728}}
\label{sec:fixed-points-0-1728}
We now determine the singular points of $\ZNr{N}{r}$ fixed by the involution $\tau$. We first give a description of the singular points of $\ZNr{N}{r}$ which lends itself to direct calculations with matrices (cf. \cite[Corollary~2.4]{KS_MDQS}).

\begin{lemma}
  \label{lem:singular-matrix}
  Let $E/\bbC$ be a (generalised) elliptic curve and let $\phi \colon E[N] \to E[N]$ be an $(N,r)$-congruence. Let $G$ be the image of the homomorphism $\Aut(E) \to \Aut(E[N])$ and let
  \begin{equation*}
   G_\ffz = \{g \in G : g \phi = \phi g' \textnormal{ for some } g' \in G\} \subset G.
  \end{equation*}
  Then the group $\mathrm{P} G_\ffz \colonequals G_\ffz / \{\pm 1\}$ is cyclic.

  The point $\ffz = (E,E,\phi) \in \ZNr{N}{r}$ is singular if and only if $\mathrm{P}G_\ffz \neq 1$. In this case, if $g \in G_\ffz$ represents a generator for $\mathrm{P}G_\ffz$ then $\ffz$ is a cyclic quotient singularity of type $(d, \ell)$ where $d = | \mathrm{P}G_\ffz |$ and $\ell \in \bbZ$ is such that $g \phi = \pm \phi g^\ell$.
\end{lemma}

\begin{proof}
  First note that for every (generalised) elliptic curve $E/\bbC$ the group $G/\{ \pm 1 \}$ is cyclic (of order $1$, $2$, $3$, or $N$), hence $\mathrm{P} G_\ffz$ is cyclic.
  
  Let $\mathfrak{Z} = ((E, \iota), (E, \iota')) \in X(N) \times X(N)$ be a pre-image of $\ffz$, that is $\phi = (\iota')^{-1} \ee \iota$. The point $\ffz$ is singular if and only if there exists $h \in \SL_2(\bbZ/N\bbZ) \setminus \{\pm 1\}$ acting trivially on $\mathfrak{Z}$. Equivalently $h \in G$ and both $h \iota = \iota g$ and $\ee h \ee^{-1} \iota' = \iota' g'$ for some $g, g' \in G$. Thus if $\ffz$ is singular we have $\phi = (\iota')^{-1} \ee \iota= (\ee h \ee^{-1} \iota')^{-1} \ee (h \iota) = (\iota' g')^{-1} \ee (\iota g) = (g')^{-1} \phi g $.

  Conversely assume there exist $g, g' \in G \setminus \{\pm 1\}$ such that $g' \phi = \phi g$. By choosing a basis for $E[N]$ (equivalently, a pre-image $\mathfrak{Z}$ of $\ffz$) we identify $\Aut(E[N]) \cong \SL_2(\bbZ/N\bbZ)$ and are free to assume that $\iota$ is given on this basis by $\big( \begin{smallmatrix} 1 & 0 \\ 0 & 1 \end{smallmatrix} \big)$. In particular we have $g \iota = \iota g$ and $\ee g \ee^{-1} \iota' = \iota' g'$. Hence $\ffz$ is singular.

  Let $T_{(E, \iota)}(X(N))$ be the tangent space to $X(N)$ at the point $(E, \iota)$. Note that (again choosing a basis for $E[N]$) the element $g$ acts on the tangent space $T_{(E, \iota)}(X(N)) \cong \bbC$ by $z \mapsto \zeta z$ where $\zeta \in \mu_d$ is a primitive $N^{\mathrm{th}}$-root of unity. Thus fixing an isomorphism $\mathrm{P}G_\ffz \cong \mu_d$ given by $g \mapsto \zeta$ we see that $\mathrm{P}G_\ffz$ acts on $T_{\mathfrak{Z}}(X(N) \times X(N)) \cong \bbC^2$ by $(z_1, z_2) \mapsto (\zeta z_1, \zeta^\ell z_2)$ after choosing $\ell$ such that $g' = \pm g^\ell$. In particular $\ffz$ is a cyclic quotient singularity of type $(d, \ell)$.
\end{proof}

With the help of \Cref{lem:singular-matrix} we are now able to classify all the fixed points of the action of $\tau$ on $\ZNr{N}{r}$ away from the cusps. In the case when $N = p$ is an odd prime this is proved in \cite[Section~4]{H_SMDDp2}. Define the subgroups $G_{0} = \left\langle \big( \begin{smallmatrix} 0 & -1 \\ 1 & 1 \end{smallmatrix} \big) \right\rangle \subset \GL_2(\bbZ/N\bbZ)$ and  $G_{1728} = \left\langle \big( \begin{smallmatrix} 0 & -1 \\ 1 & 0 \end{smallmatrix} \big) \right\rangle \subset \GL_2(\bbZ/N\bbZ)$. Recall that if $E/\bbC$ is an elliptic curve, then by definition a pair of congruences $\phi,\phi' \in \Aut(E[N])$ are isomorphic if they are contained in the same double coset class in $\Aut(E) \backslash \Aut(E[N]) / \Aut(E)$ (see \Cref{sec:N-gons-cusps}).

\begin{lemma}
  \label{lemma:fix-1728}
  Let $E/\bbC$ be an elliptic curve of $j$-invariant $1728$. Fix a basis for $E[N]$ so that  the image of $\Aut(E) \to \Aut(E[N]) \cong \GL_2(\bbZ/N\bbZ)$ is equal to $G_{1728}$. A singular point $\ffz = (E, E, \phi) \in \ZNr{N}{r}$ is fixed by $\tau$ if and only if either
  \begin{enumerate}[label=\eniii]
  \item \label{lemma:fix-1728-iso}
    the isomorphism class of the congruence $\phi$ contains either $\big( \begin{smallmatrix} a & 0 \\ 0 & a \end{smallmatrix} \big)$ or $\big( \begin{smallmatrix} a & -a \\ a & a \end{smallmatrix} \big)$, or if $N$ is even $\big( \begin{smallmatrix} a & N/2 \\ N/2 & a \end{smallmatrix} \big)$, for some $a \in \bbZ/N\bbZ$, or
  \item \label{lemma:fix-1728-crv}
    the isomorphism class of the congruence $\phi$ contains $\big( \begin{smallmatrix} a & b \\ b & -a \end{smallmatrix} \big)$ for some $a,b \in \bbZ/N\bbZ$.
  \end{enumerate}
\end{lemma}

\begin{proof}
  Let $\phi = \big( \begin{smallmatrix} a & b \\ c & d \end{smallmatrix} \big) \in \GL_2(\bbZ/N\bbZ)$. By \Cref{lem:singular-matrix} the point $\ffz$ is singular if and only if there exist $g, g' \in G_{1728} \setminus \{\pm 1\}$ such that $g \phi = \phi g'$. It follows from a direct calculation that either $\phi = \big( \begin{smallmatrix} a & -b \\ b & a \end{smallmatrix} \big)$ or $\big( \begin{smallmatrix} a & b \\ b & -a \end{smallmatrix} \big)$ (cf. \cite[p.173]{H_SMDDp2}).

  Similarly note that $\ffz$ is fixed by $\tau$ if and only if there exists $g,g' \in G_{1728}$ such that $g \phi = \det(\phi) \phi^{-1} g'$ and it is simple to check that this occurs if and only if the isomorphism class of $\phi$ (i.e., the double coset $G_{1728} \, \phi \, G_{1728}$) contains an element of the form \ref{lemma:fix-1728-iso} or \ref{lemma:fix-1728-crv}. 
\end{proof}

\begin{lemma}
  \label{lemma:fix-0}
  Let $E/\bbC$ be an elliptic curve of $j$-invariant $0$. Fix a basis for $E[N]$ so that the image of $\Aut(E) \to \Aut(E[N]) \cong \GL_2(\bbZ/N\bbZ)$ is equal to $G_{0}$. A singular point $(E, E, \phi) \in \ZNr{N}{r}$ is fixed by $\tau$ if and only if either
  \begin{enumerate}[label=\eniii]
  \item \label{lemma:fix-0-31}
    the isomorphism class of the congruence $\phi$ contains either $\big( \begin{smallmatrix} a & 0 \\ 0 & a \end{smallmatrix} \big)$ or $\big( \begin{smallmatrix} a & 2a \\ -2a & -a \end{smallmatrix} \big)$ for some $a \in (\bbZ/N\bbZ)^\times$, or
  \item \label{lemma:fix-0-32}
    the isomorphism class of the congruence $\phi$ contains $\big( \begin{smallmatrix} -a & b \\ a + b & a \end{smallmatrix} \big)$ for some $a,b \in \bbZ/N\bbZ$.
  \end{enumerate}
  Moreover, the singularities in \ref{lemma:fix-0-31} are of type $(3,1)$ and those in \ref{lemma:fix-0-32} are of type $(3,2)$.
\end{lemma}

\begin{proof}
  Let $\phi = \big( \begin{smallmatrix} a & b \\ c & d \end{smallmatrix} \big) \in \GL_2(\bbZ/N\bbZ)$. As above by \Cref{lem:singular-matrix} the point $\ffz$ is singular if and only if there exist $g, g' \in G_{0} \setminus \{\pm 1\}$ such that $g \phi = \phi g'$. A calculation shows that either $\phi = \big( \begin{smallmatrix} a & b \\ -b & -b + a \end{smallmatrix} \big)$ or $\big( \begin{smallmatrix} -a & b \\ a + b & a \end{smallmatrix} \big)$. Similarly to \Cref{lemma:fix-1728} we see that $\ffz$ is fixed by $\tau$ if and only if the isomorphism class of $\phi$ contains a congruence of the form \ref{lemma:fix-0-31} or \ref{lemma:fix-0-32}.

  The second claim follows immediately from \Cref{lem:singular-matrix} by noting that in case \ref{lemma:fix-0-31} we have $\big( \begin{smallmatrix} 0 & -1 \\ 1 & 1 \end{smallmatrix} \big) \phi = \phi \big( \begin{smallmatrix} 0 & -1 \\ 1 & 1 \end{smallmatrix} \big)$ and in case \ref{lemma:fix-0-32} we have $\big( \begin{smallmatrix} 0 & -1 \\ 1 & 1 \end{smallmatrix} \big) \phi = \phi \big( \begin{smallmatrix} 0 & -1 \\ 1 & 1 \end{smallmatrix} \big)^2$.
\end{proof}

It remains to translate the above description of the fixed \emph{singular points} on $\ZNr{N}{r}$ to a classification of fixed points on the resolution of singularities $\ZNrtil{N}{r}$. The following lemma allows us to pass easily between both descriptions. Historically it has proved very useful for computing the action of involutions on Hilbert modular surfaces and this remains true for us (a proof of this lemma may be found in \cite[Lemma~7.8]{vdG_HMS}).

\begin{lemma}[{\cite[Section~5.3]{H_HMS}}]
  \label{lemma:hirz-inv-lemma}
  Let $S$ be a compact complex manifold, let $C$ be a non-singular rational curve on $S$, and let $\tau$ be an involution of $S$ such that $\tau(C) = C$. Then either $\tau$ acts as the identity on $C$ or $\tau$ fixes exactly two points on $C$. In the latter case either
  \begin{enumerate}[label=\eniii]
  \item
    $C \cdot C$ is odd, one of the fixed points of $\tau$ on $C$ is an isolated fixed point of $\tau$, and the other is a transversal intersection of $C$ with a curve fixed pointwise by $\tau$, or
  \item
    $C \cdot C$ is even and either both of the fixed points of $\tau$ on $C$ are isolated fixed points of $\tau$, or both are transversal intersection points of $C$ with a curve fixed pointwise by $\tau$.
  \end{enumerate}
\end{lemma}

Combining the descriptions in \Cref{lemma:fix-1728,lemma:fix-0} with \Cref{lemma:hirz-inv-lemma}, the following proposition gives a complete description of the fixed point set of the action of $\tau$ on $\ZNrtil{N}{r}$ away from the cusps. We write $h(-n)$ for the class number of the imaginary quadratic order of discriminant $-n$.

\begin{prop}
  \label{prop:action-aff-figs}
  The fixed points of $\tilde{\tau}$ on the open subvariety of $\ZNrtil{N}{r}$ outside a neighbourhood of $\tilde{\ffj}^{-1}(\infty) \cup (\tilde{\ffj}')^{-1}(\infty)$ is given by \Crefrange{fig:non-cusp-tau-odd}{fig:non-cusp-tau-4mod8-r1}.

  Moreover, if $N \geq 5$ then all curves depicted are smooth, all intersections are transversal, all intersections between components of $\Ramtil$ and $E_{2,1}$, $E_{3,1}$, and $E_{3,2}$ are shown, and the only isolated fixed points of $\tilde{\tau}$ (outside $\tilde{\ffj}^{-1}(\infty) \cup (\tilde{\ffj}')^{-1}(\infty)$) are the points $P_{1,\lambda}$, $P_{2,\lambda}$, $P_{2, \lambda}'$, and $P_{3,\lambda}$. Curves which are not shown to intersect do not intersect.
\end{prop}

\begin{proof}
  \Cref{coro:M-odd-components} determines the one dimensional component of the fixed points of $\tau$. By \Cref{lemma:fixed-pts,thm:ramification-locus} the isolated fixed points of $\tilde{\tau}$ can only occur in the exceptional divisor of the resolution of singularities $\ZNrtil{N}{r} \to \ZNr{N}{r}$. Clearly the isolated fixed points of $\tilde{\tau}$ may only occur in components of the resolution above points of $\ZNr{N}{r}$ which are fixed by $\tau$. Outside $\ffj^{-1}(\infty) \cup (\ffj')^{-1}(\infty)$ every such point is classified by \Cref{lemma:fix-1728,lemma:fix-0}.
  
  We proceed with a series of lemmas, each concerning a sub-configuration in \Crefrange{fig:non-cusp-tau-odd}{fig:non-cusp-tau-4mod8-r1}. We iteratively consider each of the singular points on $\ZNr{N}{r}$ that are fixed by $\tau$ according to the classification in \Cref{lemma:fix-1728,lemma:fix-0} thereby determining the fixed points of $\tilde{\tau}$ away from $\tilde{\ffj}^{-1}(\infty) \cup (\tilde{\ffj}')^{-1}(\infty)$.

  We first exhibit the configurations consisting of $\widetilde{F}_{1,\lambda}$, a $(-2)$-curve, and a $(-3)$-curve.
  
  \begin{lemma}
    \label{lemma:F1-config}
    We have:
    \begin{enumerate}[label=\eniii]
    \item \label{enum:F1-conf-2}
      The $\tfrac{1}{2} \rho(N,r)$ points with $\phi = \mymat{a}{0}{0}{a}$ from \Cref{lemma:fix-1728}(i) lie on
      \begin{enumerate}[label=\enabc]
      \item
        both $F_{1}$ and $\Fgplus{\weyl}$ when $N$ is odd, and
      \item
        both $F_{1}$ and $\Fgplus{\antidiag,\weyl}$ when $N$ is even.
      \end{enumerate}
      Each such point resolves to a $(-2)$-curve which is not pointwise fixed by $\tilde{\tau}$, transversally intersects $\widetilde{F}_{1,\lambda}$ and $\Fgplustil{\weyl}$ (resp. $\Fgplustil{\antidiag, \weyl}$) when $N$ is odd (resp. $N$ is even), and does not meet any other component of $\Ramtil$.
    \item
      The $\tfrac{1}{2} \rho(N,r)$ points with $\phi = \mymat{a}{0}{0}{a}$ from \Cref{lemma:fix-0}(i) lie on $F_{1}$. Each such point resolves to a $(-3)$-curve which is not pointwise fixed by $\tilde{\tau}$, transversally intersects $\widetilde{F}_{1,\lambda}$, contains an isolated fixed point $P_{1,\lambda}$, and does not meet any other component of $\Ramtil$.
    \end{enumerate}
  \end{lemma}

  \begin{proof}    
    We begin with (i), and the $\# \LamNr{N}{r} = \tfrac{1}{2}\rho(N,r)$ points with $\phi = \mymat{a}{0}{0}{a}$ from \Cref{lemma:fix-1728}\ref{lemma:fix-1728-iso} where $\det(\phi) = r$. Let $N = 2^k M$ where $k \geq 0$ and $M$ is odd. Clearly each such point lies on $F_{1,\lambda}$ for some $\lambda \in \LamNr{N}{r}$ (this set may be empty). The isomorphism class of $\phi$ (i.e., the double coset $G_{1728} \, \phi \,  G_{1728}$) contains $\phi' = \mymat{0}{-a}{a}{0}$. We have $(\phi')^2 = -\det(\phi')$, so by \Cref{lemma:fixer-g-N} the matrix $\phi'$ must be conjugate to $g_{\weyl}$ modulo $M$. It is clear that $\phi'$ is conjugate to $g_{\antidiag}$ modulo $2^k$. The cases (a) and (b) follow.

    Resolving these singular points therefore gives a $(-2)$-curve which meets both $\widetilde{F}_{1,\lambda}$ and $\Fgplus{\weyl}$ (or $\Fgplus{\antidiag,\weyl}$ when $N$ is even). If the $(-2)$-curve were fixed pointwise by $\tilde{\tau}$ we would obtain a contradiction to \Cref{lemma:fixed-is-smooth}. The claims in (i) follow from \Cref{lemma:hirz-inv-lemma}.

    We turn to case (ii), and the $\# \LamNr{N}{r} = \tfrac{1}{2}\rho(N,r)$ points with $\phi = \mymat{a}{0}{0}{a}$ from \Cref{lemma:fix-0}\ref{lemma:fix-0-31}. Each point clearly lies on $F_{1,\lambda}$ for some $\lambda \in \LamNr{N}{r}$. On $\ZNrtil{N}{r}$ each such singular point gives rise to a $(-3)$-curve meeting $\widetilde{F}_{1,\lambda}$. By \Cref{lemma:hirz-inv-lemma} these intersections are transversal, and each $(-3)$-curve has a unique isolated fixed point of $\tilde{\tau}$, which we denote $P_{1,\lambda}$.
  \end{proof}

  We now exhibit the configurations with $\widetilde{F}_{2,\lambda}$ and a $(-2)$-curve (note that these only occur when $N$ is odd).

  \begin{lemma}
    \label{lemma:F2-config}
    The $\tfrac{1}{2} \rho(N,2r)$ points with $\phi = \mymat{a}{-a}{a}{a}$ from \Cref{lemma:fix-1728}(i) lie on $F_{2}$ and no component of $\Ram_{\!Z}$.  Each such point resolves to a $(-2)$-curve which is not pointwise fixed by $\tilde{\tau}$, contains a pair of isolated fixed points $P_{2,\lambda}$ and $P_{2,\lambda}'$, and which transversally intersects $\widetilde{F}_{2,\lambda}$ at $P_{2,\lambda}$. The curve $\widetilde{F}_{2,\lambda}$ meets $\Fgplustil{\weyl}$ transversally at a point.
  \end{lemma}

  \begin{proof}
    Let $E/\bbC$ be an elliptic curve with $j(E) = 1728$. By direct computation one sees that no element $\phi' \in G_{1728} \, \phi \, G_{1728}$ has trace zero and therefore $(\phi')^2 \neq \pm \det(\phi')$. In particular $\phi'$ is not conjugate to a scalar multiple of $I$ or $g_{\weyl}$, and hence the points with $\phi = \mymat{a}{-a}{a}{a}$ do not lie on $\Ram_{\!Z}$ (by \Cref{thm:ramification-locus}). Let $\varphi = 1 + i \in \End(E)$, and note that $\varphi$ is a $2$-isogeny. Then by \Cref{lemma:CM-pts} there exists a choice of basis for $E[N]$ such that $a \varphi \vert_{E[N]} = \phi$ and therefore $(E, E, \phi) \in F_{2}$.

    Let $E'/\bbC$ be the elliptic curve with CM by the order $\bbZ[\sqrt{-2}]$. The endomorphism $\varphi = \sqrt{-2}$ of $E'$ is a $2$-isogeny and by \Cref{lemma:CM-pts} we may choose a basis for $E'[N]$ so that $a \varphi\vert_{E'[N]} = \mymat{0}{-2a}{a}{0}$ which is conjugate to $g_{\weyl}$ by \Cref{lemma:fixer-g-N}. Thus by \Cref{lemma:some-FN-smooth,lemma:hirz-inv-lemma} each curve $\widetilde{F}_{2,\lambda}$ meets $\Fgplustil{\weyl}$ transversally at a point. Hence $\widetilde{F}_{2,\lambda}$ meets the $(-2)$-curve in the resolution of $(E',E',\phi) \in \ZNr{N}{r}$ at an isolated fixed points of $\tilde{\tau}$. By \Cref{lemma:hirz-inv-lemma} there exists a second isolated fixed point on the $(-2)$-curve. The intersection with $\widetilde{F}_{2,\lambda}$ occurs transversally, otherwise blowing down $\widetilde{F}_{2,\lambda}$ gives a contradiction to \Cref{lemma:hirz-inv-lemma}.
  \end{proof}

  Next we exhibit the configurations consisting of $\widetilde{F}_{3,\lambda}$ and a $(-3)$-curve.
  \begin{lemma}
    \label{lemma:F3-config}
    Let $N = 2^kM$ where $k \geq 0$ and $M$ is odd. The $\tfrac{1}{2} \rho(N,3r)$ points with $\phi = \mymat{a}{-2a}{2a}{-a}$ from \Cref{lemma:fix-1728}(i) lie on $F_{3}$. Moreover:
    \begin{enumerate}[label=\eniii]
    \item
      if $k = 0$ they lie on $\Fgplus{\weyl}$,
    \item
      if $k = 1$ they lie on $\Fgplus{I, \weyl}$, and
    \item
      if $k \geq 2$ they lie on $\Fgplus{\omega\! \ns, \weyl}$ for some $\omega \in \LamN{2^k}$.
    \end{enumerate}
    Each such point resolves to a $(-3)$-curve which is not pointwise fixed by $\tilde{\tau}$, contains an isolated fixed point $P_{3,\lambda}$, and which transversally intersects $\widetilde{F}_{3,\lambda}$ at $P_{3,\lambda}$. The curve $\widetilde{F}_{3,\lambda}$ meets $\Ramtil$ transversally at the component $\Fgplustil{\weyl}$ (resp. $\Fgplustil{\antidiag, \weyl}$) if $N$ is odd (resp. if $N$ is even).
  \end{lemma}  

  \begin{proof}
    Note that $\phi^2 = -\det(\phi)$ modulo $M$, so by \Cref{lemma:fixer-g-N} the matrix $\phi$ is conjugate to $g_{\weyl}$ modulo $M$. Clearly if $k = 1$ then $\phi$ is conjugate $I$ modulo $2$, and is conjugate to a multiple of $g_{\ns}$ modulo $2^k$ for $k \geq 2$. The claims in (i)--(iii) follow. Resolving this singularity gives a $(-3)$-curve which by \Cref{lemma:fixed-is-smooth} is not fixed by $\tilde{\tau}$.

    Let $E/\bbC$ be an elliptic curve with $j(E) = 0$, and consider the $3$-isogeny $\varphi = 1 + 2\zeta_3 \in \End(E)$. By \Cref{lemma:CM-pts} there exists a choice of basis for $E[N]$ such that $a \varphi \vert_{E[N]} = \phi$ and therefore $(E, E, \phi) \in F_{3}$.

    Similarly if $E'/\bbC$ has CM by $\bbZ[\sqrt{-3}]$ then there exists a choice of basis for $E'[N]$ such that the $3$-isogeny $\varphi = \sqrt{-3} \in \End(E')$ has $\phi' = a \varphi \vert_{E'[N]} = \mymat{0}{-3a}{a}{0}$. By \Cref{lemma:fixer-g-N}, because $(\phi')^2 = -\det(\phi')$ modulo $M$ the matrix $\phi'$ is conjugate to $g_{\weyl}$ modulo $M$. It follows from the definition that if $N$ is even then $\phi'$ is conjugate to $g_{\antidiag}$ modulo $2^k$. Thus $\widetilde{F}_{3,\lambda}$ meets $\Fgplustil{\weyl}$ (resp. $\Fgplustil{\antidiag, \weyl}$) at the point $(E', E', \phi')$ if $N$ is odd (resp. if $N$ is even). Such intersections are transversal by \Cref{lemma:hirz-inv-lemma} and $\widetilde{F}_{3,\lambda}$ does not meet any other component of $\Ramtil$. 

    Finally the intersection of $\widetilde{F}_{3,\lambda}$ and the $(-3)$-curve must be transversal (otherwise blowing down $\widetilde{F}_{3,\lambda}$ gives a contradiction to \Cref{lemma:hirz-inv-lemma}) and must occur at the isolated fixed point on $\widetilde{F}_{3,\lambda}$, which is therefore $P_{3,\lambda}$. 
  \end{proof}

  We now exhibit the configurations consisting of $\Fgplustil[\lambda]{\borel,I}$ and a $(-2)$-curve. These only occur when $N$ is even and are the last of the cases from \Cref{lemma:fix-1728}\ref{lemma:fix-1728-iso} and \Cref{lemma:fix-0}\ref{lemma:fix-0-31}.

  \begin{lemma}
    \label{lemma:Fb-config}
    Let $N = 2^kM$ where $k \geq 1$ and $M$ is odd. The $\tfrac{1}{2} \rho(N,r)$ points with $\phi = \mymat{a}{N/2}{N/2}{a}$ from \Cref{lemma:fix-1728}(i) lie on:
    \begin{enumerate}[label=\eniii]
    \item
      both $\Fgplus[\lambda]{\borel,I}$ and $\Fgplus{I, \weyl}$ when $k = 1$, and
    \item
      both $\Fgplus[\lambda]{\borel,I}$ and $\Fgplus{\antidiag,\weyl}$ when $k \geq 2$.
    \end{enumerate}
    Each such point resolves to a $(-2)$-curve which is not pointwise fixed by $\tilde{\tau}$, transversally intersects $\Fgplustil[\lambda]{\borel,I}$ and $\Fgplustil{I, \weyl}$ (resp. $\Fgplustil{\antidiag, \weyl}$) when $k = 1$ (resp. $k \geq 2$), and which does not meet any other component of $\Ramtil$.
  \end{lemma}  

  \begin{proof}
    Let $E/\bbC$ be an elliptic curve with $j(E) = 1728$. First note that $\phi$ is a multiple of $I$ modulo $M$. Moreover modulo $2^k$ it is easy to check that $\phi$ is conjugate to $\lambda g_{\borel}$ for some $\Lambda \in \LamN{2^k}$. It follows that the singular point $(E, E, \phi)$ lies on $\Fgplus[\lambda]{\borel, I}$. The double coset $G_{1728} \, \phi \, G_{1728}$ contains $\phi' = \mymat{N/2}{-a}{a}{N/2}$. As above, note $(\phi')^2 = - \det(\phi')$ modulo $M$, so by \Cref{lemma:fixer-g-N} the matrix $\phi'$ is conjugate to $g_{\weyl}$ modulo $M$. When $k = 1$ we have $a \equiv 0 \pmod{2}$ and if $k \geq 2$ we have $a \equiv 1 \pmod{2}$ (since $\phi'$ is invertible). In the former case $\phi'$ is equal to $I$ modulo $2$ and in the latter is conjugate to $g_{\antidiag}$ modulo $2^k$. Hence $(E,E,\phi)$ lies on $\Fgplus[\lambda]{I, \weyl}$ when $k = 1$ and on $\Fgplus[\lambda]{\antidiag, \weyl}$ when $k \geq 2$. The remainder of the proof is identical to \Cref{lemma:F1-config}\ref{enum:F1-conf-2}.
  \end{proof}

  We now turn to the singular points in \Cref{lemma:fix-1728}\ref{lemma:fix-1728-crv}. These result in the components of $E_{2,1}$ which meet the components of $\Ramtil$ which are not $\widetilde{F}_{1,\lambda}$ or $\Fgplustil[\lambda]{\borel,I}$. These are the ``circles'' on $\Fgplus{\weyl}$ in \Cref{fig:non-cusp-tau-odd}, and the ``links'' between components of $\Ramtil$ in \Cref{fig:non-cusp-tau-2mod4,fig:non-cusp-tau-4mod8-r1}.

  \begin{lemma}
    \label{lemma:circles}
    Let $N = 2^kM$ where $k \geq 0$ and $M$ is odd. There are $s_{2,1}(N,r)$ singular points with $\phi = \mymat{a}{b}{b}{-a}$ given in \Cref{lemma:fix-1728}(ii), where
    \begin{equation*}
      s_{2,1}(N,r) =
      \begin{cases}
      \frac{1}{2} h(-4N^2) & \text{if $k \leq 1$,} \\[2mm]
      h(-4N^2) & \text{if $k \geq 2$ and $r \equiv 3 \hspace{-0.5em}\pmod{4}$, and} \\[2mm]
      0 & \text{if $k \geq 2$ and $r \equiv 1 \hspace{-0.5em}\pmod{4}$.} \\
      \end{cases}
    \end{equation*}
    More precisely:
    \begin{enumerate}[label=\eniii]
    \item
      if $k = 0$ all such points lie on $\Fgplus{\weyl}$ and no other component of $\Ram_{\!Z}$,
    \item
      if $k = 1$ all such points lie on intersections of $\Fgplus{I,\weyl}$ and $\Fgplus{\antidiag,\weyl}$,
    \item
      if $k = 2$ and $r \equiv 3 \pmod{4}$ then 
      \begin{enumerate}[label=\enabc]
      \item
        exactly $h(-4(N/2)^2)$ lie on intersections of $\Fgplus{\ns,\weyl}$ and $\Fgplus{\antidiag,\weyl}$,
      \item
        exactly $h(-4(N/2)^2)$ lie on intersections of $\Fgplus{\s,\weyl}$ and $\Fgplus{\antidiag,\weyl}$,
      \end{enumerate}
    \item
      if $k \geq 3$ and $r \equiv 3 \pmod{8}$ for each $\omega \in \LamN{2^k}$ exactly $h(-4(N/2)^2)$ lie on intersections of $\Fgplus[\omega]{\ns,\weyl}$ and $\Fgplus{\antidiag,\weyl}$,
    \item
      if $k \geq 3$ and $r \equiv 7 \pmod{8}$ for each $\omega \in \LamN{2^k}$ exactly $h(-4(N/2)^2)$ lie on intersections of $\Fgplus[\omega]{\s,\weyl}$ and $\Fgplus{\antidiag,\weyl}$.
    \end{enumerate}
    The resolution of each such singular point consists of a $(-2)$-curve which is not pointwise fixed by $\tilde{\tau}$ and meets $\Ramtil$ transversally at two points.
  \end{lemma}

  \begin{proof}
    Let $A_{2,1}(N,r) = \{ a^2 + b^2 = -r : a,b \in \bbZ/N\bbZ \}$. There are $\#A_{2,1}(N,r)$ matrices of the form in \Cref{lemma:fix-1728}\ref{lemma:fix-1728-crv}. For odd $M$ we have
    \begin{equation*}
      \# A_{2,1}(M,r) = M \prod_{p | M} \left( 1 - \tfrac{1}{p} \left(\tfrac{-1}{p}\right) \right).
    \end{equation*}
    Note that for any $a,b$ we have $(\phi')^2 = - \det(\phi')$ for all $\phi' \in G_{1728} \, \phi \, G_{1728}$.
    
    Suppose first that $k = 0$.  By \Cref{lemma:fixer-g-N} each such element is conjugate to $g_{\weyl}$. There are $\tfrac{1}{4} \# A_{2,1}(M,r) = \tfrac{1}{2} h(-4M^2)$ classes $G_{1728} \, \phi \, G_{1728}$ where $a,b \in A_{2,1}(M,r)$ and (i) follows.

    If $k = 1$ clearly either $a \equiv 0 \pmod{2}$ or $b \equiv 0 \pmod{2}$ (since $\phi$ is invertible). Thus $G_{1728} \, \phi \, G_{1728}$ contains a conjugate of both $(g_I, g_{\weyl})$ and $(g_{\antidiag}, g_{\weyl})$. There are $\tfrac{1}{4} \# A_{2,1}(N,r) = \tfrac{1}{2} h(-4N^2)$ classes $G_{1728} \, \phi \, G_{1728}$ where $a,b \in A_{2,1}(N,r)$ and (ii) follows.

    Now suppose $k \geq 2$. Clearly if $r \equiv 1 \pmod{4}$ there are no such matrices $\phi$, since $A_{2,1}(N,r)$ is empty. Note that $g_{\s}, g_{\ns} \in \GL_2(\bbZ/2^k\bbZ)$ (and all of their conjugates) are congruent to $I$ modulo $2$. Conversely no conjugate of $g_{\antidiag}$ are is congruent to $I$ modulo $2$. So either $a \equiv 0 \pmod{2}$ or $b \equiv 0 \pmod{2}$. By \Cref{lemma:fixer-g-N} the double coset $G_{1728} \, \phi \, G_{1728}$ contains a conjugate of $(g_{\antidiag}, g_{\weyl})$ and either $(\omega g_{\ns}, g_{\weyl})$ or $(\omega g_{\s}, g_{\weyl})$ for some $\omega \in \LamN{2^k}$. There are $\tfrac{1}{4} \cdot 2^{k+1} \cdot \#A_{2,1}(M,r) = h(-4N^2) = 2 h(-4(N/2)^2)$ classes $G_{1728} \, \phi \, G_{1728}$ where $a,b \in A_{2,1}(N,r)$.
  \end{proof}

  To conclude the proof of \Cref{prop:action-aff-figs} it remains to consider the singular points given in \Cref{lemma:fix-0}\ref{lemma:fix-0-32}. The resolutions of these points consist of chains of two $(-2)$-curves meeting one another and $\Ramtil$ transversally at a point. These are the ``crosses'' in \Crefrange{fig:non-cusp-tau-odd}{fig:non-cusp-tau-4mod8-r1}.

  \begin{lemma}
    \label{lemma:crosses}
    Let $N = 2^kM$ where $k \geq 0$ and $M$ is odd. There are $s_{3,2}(N,r)$ singular points with $\phi = \mymat{-a}{b}{a+b}{a}$ given in \Cref{lemma:fix-0}(ii), where
    \begin{equation*}
    s_{3,2}(N,r) =
      \begin{cases}
        \frac{1}{2} h(-3N^2) & \text{if $3 \nmid N$,} \\[2mm]
        h(-3N^2) & \text{if $3 \mid N$ and $r \equiv 2 \hspace{-0.5em}\pmod{3}$, and} \\[2mm]
        0 & \text{if $3 \mid N$ and $r \equiv 1 \hspace{-0.5em}\pmod{3}$}.
      \end{cases}
    \end{equation*}
    More precisely:
    \begin{enumerate}[label=\eniii]
    \item
      if $k = 0$ all such points lie on $\Fgplus{\weyl}$ and no other component of $\Ram_{\!Z}$, and
    \item
      if $k \geq 1$ all such points lie on $\Fgplus{\antidiag,\weyl}$ and no other component of $\Ram_{\!Z}$.
    \end{enumerate}
    The resolution of each such singular point consists of a chain of two $(-2)$-curves which are not pointwise fixed by $\tilde{\tau}$ and which meet each other and $\Ramtil$ transversally at a point.
  \end{lemma}

  \begin{proof}
    Let $A_{3,2}(N) = \{ a^2 + ab + b^2 = -r : a,b \in \bbZ/N\bbZ \}$. There are $\#A_{3,2}(N)$ matrices of the form in \Cref{lemma:fix-0}\ref{lemma:fix-0-32}. We have
    \begin{equation*}
      \# A_{3,2}(N) =
      \begin{cases}
        N \prod_{p | N} \left( 1 - \frac{1}{p} \left(\frac{-3}{p}\right) \right)  & \text{if $3 \nmid N$},\\[2mm]
        2 N \prod_{p | N} \left( 1 - \frac{1}{p} \left(\frac{-3}{p}\right) \right)  & \text{if $3 \mid N$ and $r \equiv 2 \hspace{-0.5em}\pmod{3}$, and} \\[2mm]
        0 & \text{if $3 \mid N$ and $r \equiv 1 \hspace{-0.5em}\pmod{3}$.}
      \end{cases}
    \end{equation*}
    Here $\big( \tfrac{\cdot}{p} \big)$ denotes the Kronecker symbol (in particular $\big( \tfrac{-3}{2} \big) = -1$ and $\big( \tfrac{-3}{3} \big) = 0$). The number of double cosets $G_0 \, \phi \, G_0$ is equal to $\tfrac{1}{6} \# A_{3,2}(N) = s_{3,2}(N,r)$.

    If $k = 0$ note that each matrix $\phi$ as in \Cref{lemma:fix-0}\ref{lemma:fix-0-32} satisfies $\phi^2 = - \det(\phi)$, so by \Cref{lemma:fixer-g-N} is conjugate to $g_{\weyl}$. Hence each such singular point lies on $\Fgplus{\weyl}$. If $k \geq 1$ note that modulo $2$ the matrix $\phi$ is not conjugate to $I$, so by the same argument as \Cref{lemma:circles} each such singular point lies on $\Fgplus{\antidiag,\weyl}$.

        By \Cref{lemma:fix-0} each such singular point has type $(3,2)$ and therefore its resolution consists of a chain of two $(-2)$-curves. By \cite[Proposition~2.5]{KS_MDQS} this chain joins the strict transforms on $\ZNrtil{N}{r}$ of the curves $\ffj^{-1}(0)$ and $(\ffj')^{-1}(0)$, which are therefore swapped by the action of $\tilde{\tau}$. The unique fixed point of $\tilde{\tau}$ on this chain is the intersection point of the pair of $(-2)$-curves. If $k = 0$ the curve $\Fgplustil{\weyl}$ intersects both $(-2)$-curves transversally at this point, and if $k=1$ the curve $\Fgplustil{\antidiag,\weyl}$ intersects both curves transversally (forming a ``cross'' in \Crefrange{fig:non-cusp-tau-odd}{fig:non-cusp-tau-4mod8-r1}).
  \end{proof}

  Note that if any additional intersections occur, other than those depicted in \Crefrange{fig:non-cusp-tau-odd}{fig:non-cusp-tau-4mod8-r1}, a contradiction to either \Cref{lemma:fixed-is-smooth} or \ref{lemma:hirz-inv-lemma} would be obtained (possibly after blowing-down a $(-1)$-curve).  The proposition follows by combining \Crefrange{lemma:F1-config}{lemma:crosses}.
\end{proof}

\begin{remark}
  It is no coincidence that the numbers of fixed points occurring in \Cref{prop:action-aff-figs} are given by class numbers of imaginary quadratic orders. Indeed, the isolated fixed points lying on a modular curve $F_{m,\lambda}$ may be described in terms of equivalence classes of binary quadratic forms. In the case of fundamental discriminants Hirzebruch--Zagier~\cite[1.1]{HZ_INOCOHMSAMFON} (cf. \cite[V.6]{vdG_HMS}) showed how one may attach a positive definite binary quadratic form of discriminant divisible by $D$ to each ``special'' elliptic point of a Hilbert modular surface. It can furthermore be shown that an isolated elliptic point fixed by $\tilde{\tau}$ must correspond to an \emph{ambiguous} binary quadratic form. This approach is utilised by Hausmann~\cite{H_KAHM,H_TFPOTSHMGOARQFWAD} to classify the fixed points of the action of $\tilde{\tau}$ on Hilbert modular surfaces of any fundamental discriminant $D$.

  Hermann~\cite[Section~4]{H_SMDDp2} describes how this generalises to the surfaces $\ZNr{p}{r}$ when $p$ is an odd prime number. In their setting all the fixed points of the type in \Cref{lemma:fix-1728}\ref{lemma:fix-1728-iso} and \Cref{lemma:fix-0}\ref{lemma:fix-0-31} correspond to ambiguous, primitive, binary quadratic forms with content coprime to $p$ and discriminant $-4p^2$ and $-3p^2$ respectively. These $\frac{1}{2}(2\rho(p,r) + \rho(p,2r) + \rho(p,3r))$ fixed points correspond to the forms $[1,0,p^2], [1,1,\frac{3p^2+1}{4}], [2,2,\frac{p^2+1}{2}], [3,3,\frac{p^2 + 3}{4}]$ given in \cite[p.~173]{H_SMDDp2}.

  Ernst Kani has shared with us his preprint~\cite{K_MCOHSAOTN} where he attaches to a modular curve $F_{m,\lambda} \subset \ZNr{N}{-1}$ a binary quadratic form (cf.~\cite[Section~2]{H_KADHMUK} when $N = 1$). It may be possible to prove \Cref{prop:action-aff-figs} using this, together with the analogue of Hausmann and Hermann's approach.
\end{remark}

The reader is encouraged to compare the following figures (especially \Cref{fig:non-cusp-tau-odd}) with those of Hirzebruch--Van de Ven~\cite[p.~18]{HV_HMSATCOAS} and Hirzebruch \cite[5.4~(8)]{H_HMS} for Hilbert modular surfaces of fundamental discriminants. 
{\vfill
\begin{figure}[H]
  \begin{center}
    \begin{tikzpicture}[scale=0.77]
      \draw[line width=1.1mm] (2, -10.5) -- (2,-1) node[anchor=south]{$\Fgplustil{\weyl}$};
      \draw[line width=1.1mm] (-1,-8) -- (-1,-5) node[anchor=south]{$\widetilde{F}_{1,\lambda}$};

      \draw (-1.5,-5.5) -- (2.5,-5.5); \node[anchor=south] at (0.5,-5.5) {\small $-2$};
      \draw (-1.5,-7.5) -- (1,-7.5); \node[anchor=south] at (-0.25,-7.5) {\small $-3$};
      \draw (2,-4) circle[radius=0.75]; \node[anchor=west] at (2.75,-4) {\small $-2$};
      \draw (1.25, -1.25) -- (2.75, -2.75); \draw (2.75, -1.25) -- (1.25, -2.75); \node[anchor=north east] at (1.35,-1.25) {\small $-2$}; \node[anchor=north west] at (2.5,-1.25) {\small $-2$};
      \draw (1.5,-6.5) -- (4.5,-6.5); \node[anchor=south] at (3,-6.5) {\small $-3$};
      \draw (4,-8) -- (4,-10.5); \node[anchor=west] at (4,-8) {\small $-2$};
      
      \filldraw[black] (0.5,-7.5) circle (0.5mm) node[anchor=north]{$P_{1,\lambda}$};
      \filldraw[black] (4,-6.5) circle (0.5mm) node[anchor=north west]{$P_{3,\lambda}$};
      \filldraw[black] (4,-9) circle (0.5mm) node[anchor=north west]{$P_{2,\lambda}$};
      \filldraw[black] (4,-10) circle (0.5mm) node[anchor=north west]{$P_{2,\lambda}'$};

      \draw (4,-6.5) -- (4,-5.5) node[anchor=south]{$\widetilde{F}_{3,\lambda}$}; \draw (4,-6.5) arc (0:-90: 2 and 1.5); \draw (1.5,-8) -- (2,-8); 
      \draw (1,-9) -- (4.5,-9); \node[anchor=east] at (1,-9) {$\widetilde{F}_{2,\lambda}$}; 
    \end{tikzpicture}
  \end{center}
  \captionsetup{singlelinecheck=off}
  \caption[foo]{The action of the involution $\tilde{\tau}$ on the non-cuspidal locus when $N$ is odd. Curves and points in bold are fixed by $\tilde{\tau}$ and numbers indicate self-intersections of resolution curves. Here:
    \begin{itemize}
    \item
      the curves $\widetilde{F}_{1,\lambda}$, $\widetilde{F}_{2,\lambda}$, and $\widetilde{F}_{3,\lambda}$ occur $\frac{1}{2}\rho(N,r)$, $\frac{1}{2}\rho(N,2r)$, and $\frac{1}{2}\rho(N,3r)$ times respectively,
    \item
      the circles occur $s_{2,1}(N,r)$ times, and
    \item
      the crosses occur $s_{3,2}(N,r)$ times.
    \end{itemize}
    }
  \label{fig:non-cusp-tau-odd}
\end{figure}
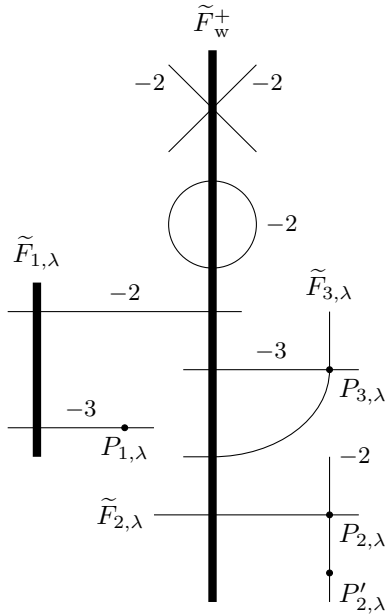

\vfill}
\begingroup
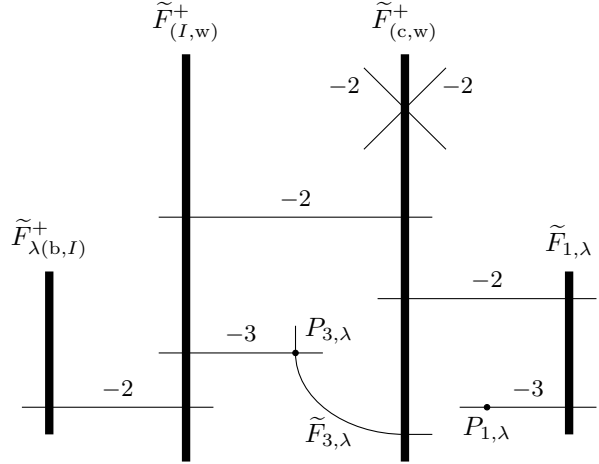
\begin{figure}[p]
  \begin{center}
    \begin{tikzpicture}[scale=0.72]
      \draw[line width=1.1mm] (1, -8.5) -- (1,-1) node[anchor=south]{$\Fgplustil{I,\weyl}$};
      \draw[line width=1.1mm] (5, -8.5) -- (5,-1) node[anchor=south]{$\Fgplustil{\antidiag,\weyl}$};
      \draw (0.5,-4) -- (5.5,-4); \node[anchor=south] at (3,-4) {\small $-2$};
      \draw (4.25, -1.25) -- (5.75, -2.75); \draw (5.75, -1.25) -- (4.25, -2.75); \node[anchor=north east] at (4.35,-1.25) {\small $-2$}; \node[anchor=north west] at (5.5,-1.25) {\small $-2$};
      %
      \draw[line width=1.1mm] (8,-8) -- (8,-5) node[anchor=south]{$\widetilde{F}_{1,\lambda}$};
      \draw (4.5,-5.5) -- (8.5,-5.5); \node[anchor=south] at (6.5,-5.5) {\small $-2$};
      \draw (8.5,-7.5) -- (6,-7.5); \node[anchor=south] at (7.25,-7.5) {\small $-3$};
      \filldraw[black] (6.5,-7.5) circle (0.5mm) node[anchor=north]{$P_{1,\lambda}$};
      %
      \draw[line width=1.1mm] (-1.5,-8) -- (-1.5,-5) node[anchor=south]{$\Fgplustil[\lambda]{\borel, I}$};
      \draw (1.5,-7.5) -- (-2,-7.5); \node[anchor=south] at (-0.25,-7.5) {\small $-2$};
      %
      \draw (3,-6.5) -- (3,-6);
      \draw (3,-6.5) arc (-180:-90: 2 and 1.5);
      \draw (5,-8) -- (5.5,-8);
      \node[anchor=north west] at (3,-7.5) {$\widetilde{F}_{3,\lambda}$};
      \draw (0.5,-6.5) -- (3.5,-6.5); \node[anchor=south] at (2,-6.5) {\small $-3$};
      \filldraw[black] (3,-6.5) circle (0.5mm) node[anchor=south west]{$P_{3,\lambda}$};
    \end{tikzpicture}
  \end{center}
  \captionsetup{singlelinecheck=off}
  \caption[foo]{The action of the involution $\tilde{\tau}$ on the non-cuspidal locus when $N \equiv 2 \pmod{4}$. Curves and points in bold are fixed by $\tilde{\tau}$ and numbers indicate self-intersections of resolution curves. Here: 
    \begin{itemize}
    \item
      the curves $\widetilde{F}_{1,\lambda}$ and $\widetilde{F}_{3,\lambda}$ occur $\frac{1}{2}\rho(N,r)$ and $\frac{1}{2}\rho(N,3r)$ times respectively,
    \item
      the $(-2)$-curves linking $\Fgplustil{I,\weyl}$ and $\Fgplustil{\antidiag,\weyl}$ occur $s_{2,1}(N,r)$ times, and
    \item
      the cross occurs $s_{3,2}(N,r)$ times.
    \end{itemize}  
  }
  \label{fig:non-cusp-tau-2mod4}
\end{figure}

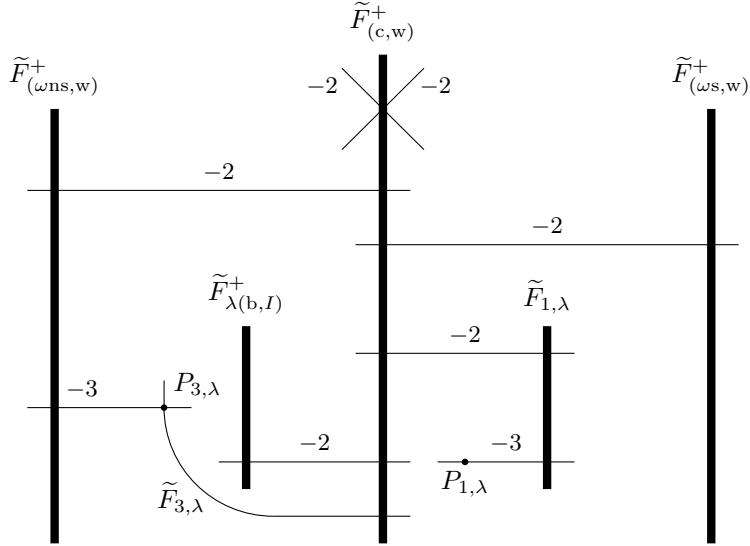
\begin{figure}[p]
  \begin{center}
    \begin{tikzpicture}[scale=0.72]
      \draw[line width=1.1mm] (1, -9) -- (1,0) node[anchor=south]{$\Fgplustil{\antidiag,\weyl}$};
      \draw (0.25, -0.25) -- (1.75, -1.75); \draw (1.75, -0.25) -- (0.25, -1.75); \node[anchor=north east] at (0.35,-0.25) {\small $-2$}; \node[anchor=north west] at (1.5,-0.25) {\small $-2$};
      %
      \draw[line width=1.1mm] (4,-8) -- (4,-5) node[anchor=south]{$\widetilde{F}_{1, \lambda}$};
      \draw (0.5,-5.5) -- (4.5,-5.5); \node[anchor=south] at (2.5,-5.5) {\small $-2$};
      \draw (2,-7.5) -- (4.5,-7.5); \node[anchor=south] at (3.25,-7.5) {\small $-3$};
      \filldraw[black] (2.5,-7.5) circle (0.5mm) node[anchor=north]{$P_{1,\lambda}$};
      %
      \draw[line width=1.1mm] (-1.5,-8) -- (-1.5,-5) node[anchor=south]{$\Fgplustil[\lambda]{\borel, I}$};
      \draw (-2,-7.5) -- (1.5,-7.5); \node[anchor=south] at (-0.25,-7.5) {\small $-2$};
      %
      \draw[line width=1.1mm] (-5, -9) -- (-5,-1) node[anchor=south]{$\Fgplustil{\omega\!\ns,\weyl}$};
      %
      \draw (-3,-6.5) -- (-3,-6);
      \draw (-3,-6.5) arc (-180:-90: 2 and 2);
      \draw (-1,-8.5) -- (1.5,-8.5);
      \node[anchor=north west] at (-3.3,-7.7) {$\widetilde{F}_{3,\lambda}$};
      %
      \draw (-5.5,-6.5) -- (-2.5,-6.5); \node[anchor=south] at (-4.5,-6.5) {\small $-3$};
      %
      \filldraw[black] (-3,-6.5) circle (0.5mm) node[anchor=south west]{$P_{3,\lambda}$};
      %
      \draw (-5.5,-2.5) -- (1.5,-2.5); \node[anchor=south] at (-2,-2.5) {\small $-2$};
      \draw[line width=1.1mm] (7, -9) -- (7,-1) node[anchor=south]{$\Fgplustil{\omega\!\s,\weyl}$};
      %
      \draw (0.5,-3.5) -- (7.5,-3.5); \node[anchor=south] at (4,-3.5) {\small $-2$};
    \end{tikzpicture}
  \end{center}
  \captionsetup{singlelinecheck=off}
  \caption[foo]{The action of the involution $\tilde{\tau}$ on the non-cuspidal locus when $N \equiv 0 \pmod{4}$. Curves and points in bold are fixed by $\tilde{\tau}$ and numbers indicate self-intersections of resolution curves. Here:
    \begin{itemize}
    \item
      the curves $\widetilde{F}_{1,\lambda}$ and $\widetilde{F}_{3,\lambda}$ occur $\frac{1}{2}\rho(N,r)$ and $\frac{1}{2}\rho(N,3r)$ times respectively,
    \item
      $\omega$ ranges over elements of the group $\LamN{2^k}$,
    \item
      the curves $\Fgplustil{\omega\!\s,\weyl}$ occur if and only if $7r$ is a square modulo $2^k$,
    \item
      the curves $\Fgplustil{\omega\!\ns,\weyl}$ occur if and only if $3r$ is a square modulo $2^k$,
    \item
      the $(-2)$-curves linking $\Fgplustil{\omega\!\ns,\weyl}$ and $\Fgplustil{\antidiag,\weyl}$ occur $h(-4(N/2)^2)$ times for each $\omega \in \LamN{2^k}$,
    \item
      the $(-2)$-curves linking $\Fgplustil{\omega\!\s,\weyl}$ and $\Fgplustil{\antidiag,\weyl}$ occur $h(-4(N/2)^2)$ times for each $\omega \in \LamN{2^k}$,
    \item
      the cross occurs $s_{3,2}(N,r)$ times.
    \end{itemize}
  }
  \label{fig:non-cusp-tau-4mod8-r1}
\end{figure}

\endgroup

\FloatBarrier
\subsection{The action of \texorpdfstring{$\tau$}{\unichar{"1D70F}} on the cusps}
\label{sec:action-tau-cusps}

We now classify the fixed cusps of $\ZNr{N}{r}$ under the action of $\tau$. Let $G_\infty$ denote the subgroup $\left\{ \big( \begin{smallmatrix} \pm 1 & \alpha \\ 0 & \pm 1 \end{smallmatrix} \big) : \alpha \in \bbZ/N\bbZ \right\} \subset \GL_2(\bbZ/N\bbZ)$. Recall that we write $\Ngon[N]$ for the standard $N$-gon viewed as a generalised elliptic curve (see \Cref{sec:N-gons-cus}).

\begin{lemma}
  \label{lemma:fixed-cusps-even}
  Fix a basis for ${\Ngon[N]}[N]$ such that the image of $\Aut(\Ngon) \to \Aut(\Ngon[N] [N]) \cong \GL_2(\bbZ/N\bbZ)$ is equal to $G_{\infty}$.  A cusp $\ffz = (\Ngon, \Ngon, \phi)$ of $C_{\infty, 1}$ (and $C_{\infty, 2}$) is fixed if and only if either:
  \begin{enumerate}[label=\eniii]
  \item \label{enum:fixed-cusps-A}
    the congruence $\phi$ is contained in the isomorphism class of the congruence $\big( \begin{smallmatrix} a & 0 \\ bd & a^{-1}r \end{smallmatrix} \big)$ of type \eqref{eqn:typeA} in \Cref{prop:Xmu-cusp} and one of the following conditions holds:
    \begin{enumerate}[label=\enabc]
    \item \label{enum:case-a}
      $d = N$ and $a^2 \equiv r \pmod{d}$, or
    \item \label{enum:case-b}
      $N \neq 2$ is divisible by $2$, $d = N/2$, and $a^2 \equiv r \pmod{d}$, or
    \item \label{enum:case-c}
      $a^2 \equiv -r \pmod{d}$.
    \end{enumerate}
  \item \label{enum:fixed-cusps-B}
    the congruence $\phi$ is contained in the isomorphism class of the congruence $\bigl( \begin{smallmatrix} 0 &-a^{-1} r \\ a & 0\end{smallmatrix}\bigr)$ of type \eqref{eqn:typeB} in \Cref{prop:Xmu-cusp}.
  \end{enumerate}

  Moreover in case \ref{enum:fixed-cusps-A} the point $\ffz$ is a cyclic quotient singularity of type $(d, a^{-2} r)$ and in case \ref{enum:fixed-cusps-B} the point $\ffz$ is a non-singular point of $\ZNr{N}{r}$ at which $C_{\infty,1}$ and $C_{\infty,2}$ intersect transversally.
\end{lemma}

\begin{proof}
  The point $\ffz$ is fixed by $\tau$ if and only if there exist $g,g' \in G_\infty$ such that $g \phi = r \phi^{-1} g'$.
  
  In case \ref{enum:fixed-cusps-A} let $\phi$ be given by $\big( \begin{smallmatrix} a & 0 \\ bd & a^{-1}r \end{smallmatrix} \big)$ where $d \vert N$, $d \neq 1$, $1 \leq a \leq d$, and $b \in (\bbZ/(N/d)\bbZ)^\times$, as in \Cref{prop:Xmu-cusp}. If \ref{enum:case-a} or \ref{enum:case-b} holds then taking $g = g' = I$ we see that $g \phi = r \phi^{-1} g'$ as required. In case (c) we may take  $g = \big( \begin{smallmatrix} -1 & \alpha \\ 0 & -1 \end{smallmatrix} \big)$ and $g' = \big( \begin{smallmatrix} 1 & \beta \\ 0 & 1 \end{smallmatrix} \big)$ where $\alpha = \beta = b^{-1}(a^{-1}r + a)/d$.

  Conversely let $\phi = \big( \begin{smallmatrix} a & b \\ c & d \end{smallmatrix} \big)$. If $\ffz$ is fixed by $\tau$ then there are $g, g' \in G_\infty$ such that $g \phi = r \phi^{-1} g'$. Letting $g = \big( \begin{smallmatrix} u & \alpha \\ 0 & u \end{smallmatrix} \big)$ and $g' = \big( \begin{smallmatrix} v & \beta \\ 0 & v \end{smallmatrix} \big)$ we have
  \begin{equation}
    \label{eq:fixed-case-A}
    \begin{pmatrix}
      au + \alpha bd & \alpha a^{-1} r \\
      bdu            &  a^{-1} r u
    \end{pmatrix} =
    \begin{pmatrix}
      a^{-1}rv & \beta a^{-1} r \\
      -bdv     & av - \beta bd
    \end{pmatrix}
  \end{equation}
  and therefore $a^2 u \equiv rv \pmod{d}$ and $bd(u + v) \equiv 0 \pmod{N}$. Since $u + v \in \{0, \pm 2\}$ and $b$ is invertible modulo $N$ it follows that \ref{enum:case-a}, \ref{enum:case-b}, or \ref{enum:case-c} respectively hold in the cases where \ref{enum:case-a} $d = N$, \ref{enum:case-b} $2$ divides $N$ and $d = N/2$, and \ref{enum:case-c} $u = -v$.

  Let $g = \big( \begin{smallmatrix} 1 & \alpha \\ 0 & 1 \end{smallmatrix} \big)$. If $\ell \in \bbZ$ is such that $g \phi = \phi g^\ell$ then a direct calculation shows that $\alpha d = 0$ and $\ell = a^{-2} r$. By \Cref{lem:singular-matrix} it follows that $\ffz$ is a singular point of type $(d, a^{-2}r)$.

  In case \ref{enum:fixed-cusps-B} we may take $\phi = \bigl( \begin{smallmatrix} 0 &-a^{-1} r \\ a & 0\end{smallmatrix}\bigr)$ and therefore the cusp $\ffz$ is fixed since we may take $g = I$ and $g' = -I$. Moreover in this case if $g,g' \in G_\infty$ satisfy $g \phi = \phi g'$ then $g = g' \in \{\pm I\}$, i.e., $\ffz$ is non-singular.
\end{proof}

With the aid of \Cref{lemma:hirz-inv-lemma} we now classify the cusps fixed by $\tilde{\tau}$.

\begin{prop}
  \label{prop:action-cusps-figs}
  Let $N > 1$ be an integer. The fixed divisor $\Ramtil$ meets the curves $\widetilde{C}_{\infty,i}$ only at the $\frac{1}{2} \varphi(N)$ intersection points of $\widetilde{C}_{\infty,1}$ and $\widetilde{C}_{\infty,2}$. The intersections of $\Ramtil$ and $\widetilde{C}_{\infty, i}$ are transversal, when $N$ is odd these occur on the component $\Fgplustil{\weyl}$ and when $N$ is even these occur on the component $\Fgplustil{\antidiag, \weyl}$.

  If $N \neq 2,4$ and $\ZNrtil{N}{r}$ is not rational the involution $\tilde{\tau}$ acts on the surface $\ZNrtil{N}{r}$ in a neighbourhood of $\tilde{\ffj}^{-1}(\infty) \cup (\tilde{\ffj}')^{-1}(\infty)$ as described in \Crefrange{table:cusp-1mod2}{table:cusp-0mod8}. In particular all curves depicted are smooth, all intersections are transversal, and the only isolated fixed points of $\tilde{\tau}$ are the points $P_{\infty,\lambda}$ and $P_{\infty,\lambda}^-$
\end{prop}

\begin{proof}
  The involution $\tilde{\tau}$ takes $\widetilde{C}_{\infty,1}$ to $\widetilde{C}_{\infty,2}$. In particular if a point on $\widetilde{C}_{\infty,i}$ is fixed by $\tilde{\tau}$ then it is contained in their intersection. Such intersection points are precisely the cusps of type \ref{enum:fixed-cusps-B} in \Cref{lemma:fixed-cusps-even} and meet the divisor $\Ramtil$ at the components indicated. Since the intersections of $\widetilde{C}_{\infty,1}$ and $\widetilde{C}_{\infty,2}$ occur with multiplicity one \cite[Proposition~2.5(d)]{KS_MDQS} so do the intersections of $\Ramtil$ and $\widetilde{C}_{\infty,1}$ (respectively $\widetilde{C}_{\infty,2}$).

  The intersections of the curves $\Fgplus{g}$ with the points on $\ZNr{N}{r}$ above the point $(\infty,\infty) \in \Zone = X(1) \times X(1)$ follows from the construction of $\Fgplus{g}$ together with the explicit description of the fixed cusps in \Cref{lemma:fixed-cusps-even}. That is, a point $(\Ngon, \Ngon, \phi)$ lies on $\Fgplus{g}$ if and only if $\phi$ is conjugate to an element of the double coset $G_\infty \, g \, G_\infty$.

  The remaining fixed points of $\tau$ acting on $C_\infty$ are those singular points of the type $(d,\pm 1)$ with the case $(d,1)$ occurring if and only if $d = N, N/2$. The number of such singularities is given in \cite[Theorem~2.1]{KS_MDQS} (see \Cref{lemma:sing-pts}). Precisely, on the surface $\ZNr{N}{r}$ there are exactly $\frac{1}{2} \rho(d,qr) \phi(N/d)$ singularities of type $(d, q)$ at the cusp for each $q \in (\bbZ/d\bbZ)^\times$.

  Note that the Hirzebruch--Jung continued fraction of $\frac{d}{1}$ is equal to $[[d]]$, and the continued fraction of $\frac{d}{d-1}$ is $[[2,2,...,2,2]]$ where the length of this sequence is $d - 1$.  In particular the resolution of a singularity of type $(d,1)$ consists of a single rational curve of self-intersection $-d$, and the resolution of a singularity of type $(d,-1)$ consists of a chain of $d - 1$ rational curves of self-intersection $-2$. The number and configuration of fixed points on these resolutions follows from \Cref{lemma:hirz-inv-lemma,lemma:fixed-cusps-even}. 
  
  It remains to describe the curves which pass through the isolated fixed points. We first recall from \Cref{lemma:some-FN-smooth} that when $N$ is odd the modular curves $\widetilde{F}_{4,\lambda}$ are $(-1)$-curves on $\ZNrtil{N}{r}$. Moreover, the involution $\tilde{\tau}$ acts as the Fricke involution on $\widetilde{F}_{4, \lambda} \cong X_0(4)$, which has a unique fixed cusp. The curve $\widetilde{F}_{4, \lambda}$ meets the $1$-dimensional fixed locus $\Ram\!_Z \subset \ZNr{N}{r}$ at the image of the CM point of $X_0(4)$ corresponding to the order of discriminant $-16$. By \Cref{lemma:hirz-inv-lemma} the curve $\widetilde{F}_{4,\lambda}$ therefore meets the cusp resolution at an isolated fixed point, which must be $P_{\infty,\lambda}$.

  We now turn to the case when $N = 2^kM$ is even. If $k \geq 1$ then $g_{(\Isharp,I)}$ is conjugate to $r g_{(\Isharp,I)}^{-1}$ and if $k \geq 2$ then $g_{(\borelsharp,I)}$ is conjugate to $r g_{(\borelsharp,I)}^{-1}$. In particular, the involution $\tilde{\tau}$ takes the curves $\Fgplustil[\lambda]{\Isharp, I}$ and $\Fgplustil[\lambda]{\borelsharp, I}$ to themselves.

  By \Cref{lemma:flat-smooth} the curves $\Fgplustil[\lambda]{\Isharp, I}$ and $\Fgplustil[\lambda]{\borelsharp, I}$ are geometrically irreducible, smooth, and have self-intersection $-1$. Since, by assumption, $\ZNrtil{N}{r}$ is not rational by \ratcrit{??} the curves $\Fgplustil[\lambda]{\Isharp, I}$ and $\Fgplustil[\lambda]{\borelsharp, I}$ do not intersect each other or $\Fgplustil[\lambda]{\borel, I}$. Applying \Cref{lemma:hirz-inv-lemma} it follows that $\Fgplustil[\lambda]{\Isharp, I}$ (when $k \geq 1$) and $\Fgplustil[\lambda]{\borelsharp, I}$ (when $k \geq 2$) meet the isolated fixed points contained in the resolutions of the cusps of type $(N/2, 1)$ (as indicated in \Crefrange{table:cusp-2mod4}{table:cusp-0mod8}).  
\end{proof}

{
\vfill
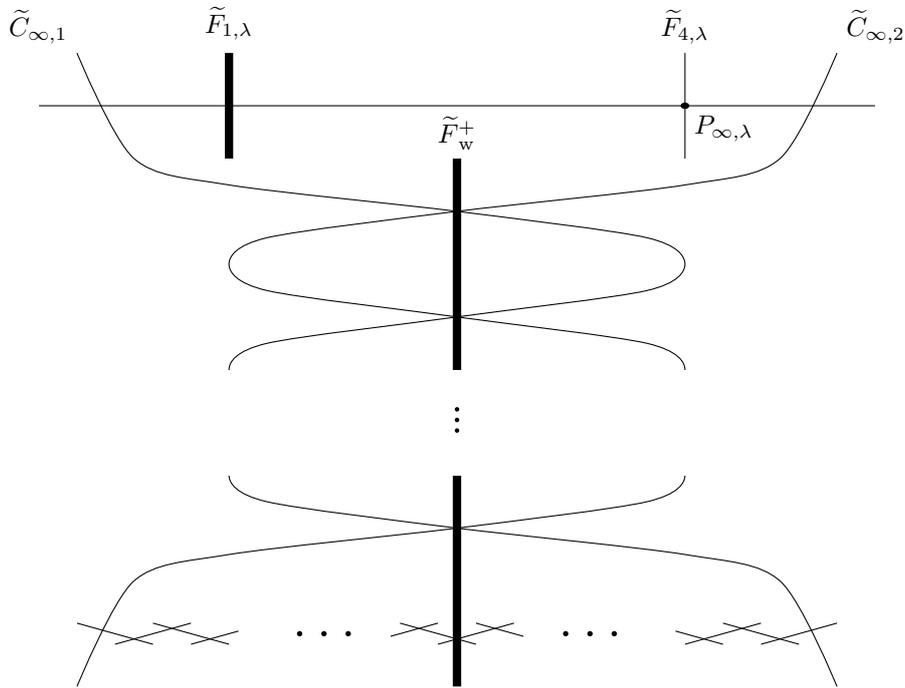
\begin{figure}[H]
  \begin{center}
    \begin{tikzpicture}[yscale=0.7]
      \draw[line width=1.1mm] (-3,10) -- (-3,12) node[anchor=south]{$\widetilde{F}_{1,\lambda}$}; 
      \draw (3,10) -- (3,12) node[anchor=south]{$\widetilde{F}_{4,\lambda}$}; 
      \draw[line width=1.1mm] (0,0) -- (0,10); \node[anchor=south] at (0,10) {$\Fgplustil{\weyl}$}; 

      \draw plot [smooth] coordinates {(5,0) (4.25,2.00) (3,2.50) (0,3.00) (-2.5,3.50) (-3,4.00) (-2.5,4.50) (0,5.00) (2.5,5.50) (3,6.00) (2.5,6.50) (0,7.00) (-2.5,7.50) (-3,8.00) (-2.5,8.50) (0,9.00) (3,9.5) (4.25, 10) (5,12)} node[anchor=south west]{$\widetilde{C}_{\infty,2}$};
      
      \draw plot [smooth] coordinates {(-5,0) (-4.25,2.00) (-3,2.50) (0,3.00) (2.5,3.50) (3,4.00) (2.5,4.50) (0,5.00) (-2.5,5.50) (-3,6.00) (-2.5,6.50) (0,7.00) (2.5,7.50) (3,8.00) (2.5,8.50) (0,9.0) (-3,9.5) (-4.25, 10) (-5,12) } node[anchor=south east]{$\widetilde{C}_{\infty,1}$};

      \fill[white] (-5,6) rectangle (5,4);
      \node at (0,5.2) (a) {\Huge $\vdots$};
      
      \draw (-5.5,11) -- (5.5,11);

      \draw (-5, 1.2) -- (-4,0.8); \draw (-4.5, 0.8) -- (-3.5,1.2); \draw (-4, 1.2) -- (-3,0.8); \draw (-3.5, 0.8) -- (-2.875,1.05);
      \node at (-1.75,1) (a) {\Huge $\ldots$};
      \draw (-0.75, 1.2) -- (0.25,0.8); \draw (-0.25, 0.8) -- (0.75, 1.2); \draw (-0.875, 0.95) -- (-0.25,1.2); \draw (0.875, 0.95) -- (0.25,1.2);
      \node at (1.75,1) (a) {\Huge $\ldots$};
      \draw (5, 1.2) -- (4,0.8); \draw (4.5, 0.8) -- (3.5,1.2); \draw (4, 1.2) -- (3,0.8); \draw (3.5, 0.8) -- (2.875,1.05);
      
      \filldraw[black] (3,11) circle (0.5mm) node[anchor=north west]{$P_{\infty,\lambda}$};
    \end{tikzpicture}
  \end{center}
  \caption{The action of the involution $\tilde{\tau}$ in a neighbourhood of the cusps when $N$ is odd. Lines and points in bold are fixed by $\tilde{\tau}$. The curves $\widetilde{C}_{\infty, 1}$ and $\widetilde{C}_{\infty,2}$ intersect transversally at $\frac{1}{2}\phi(N)$ points. All horizontal curves are resolution divisors occuring with multiplicity recorded in \Cref{table:cusp-1mod2}.}

  \label{fig:cusp-tau-odd}
\end{figure}


}

\begingroup
\renewcommand*{\arraystretch}{1.5}
\setlength{\tabcolsep}{0pt}

\begin{table}[p]
  \centering
  \scriptsize{
    \begin{tabular}{C{13mm}|C{32mm}|C{74mm}p{30mm}}
      Type                              & {\makecell[c]{Number of \\occurrences}}                           & Diagram                                  & \\
      \hline
      \multirow[c]{2}{*}[1mm]{$(N,1)$}  & \multirow[c]{2}{*}[1mm]{$\frac{1}{2} \rho(N, r)$}               & {

\begin{tikzpicture}[baseline=(current bounding box.center), yscale=0.8]
  \draw[line width=1.1mm] (-2,0.25) -- (-2,1.75); \node at (-2,1.95)[anchor=east]{$\widetilde{F}_{1,\lambda}$}; 
  \draw (2,0.25) -- (2,1.75); \node at (2,1.95)[anchor=west]{$\widetilde{F}_{4,\lambda}$}; 
  \draw (-3.5,1) -- (3.5,1);

  \filldraw[black] (2,1) circle (0.5mm) node[anchor=north west]{$P_{\infty,\lambda}$};
\end{tikzpicture}

}  & \\[-4mm]
      \hdashline
      \multirow[c]{2}{*}[1mm]{$(d,-1)$} & \multirow[c]{2}{*}[1mm]{$\frac{1}{2} \rho(d, -r) \varphi(N/d)$} & {

\begin{tikzpicture}[baseline=(current bounding box.center), yscale=0.8]
  \draw[line width=1.1mm] (0,0.25) -- (0,1.75); \node at (0,1.75) [anchor=south]{$\Fgplustil{\weyl}$};

  \draw (-3, 1.2) -- (-2,0.8); \draw (-2.5, 0.8) -- (-1.875,1.05);
  \node at (-1.375,1) (a) {\large $\ldots$};
  \draw (-0.75, 1.2) -- (0.25,0.8); \draw (-0.25, 0.8) -- (0.75, 1.2); \draw (-0.875, 0.95) -- (-0.25,1.2); \draw (0.875, 0.95) -- (0.25,1.2);
  \node at (1.375,1) (a) {\large $\ldots$};
  \draw (3, 1.2) -- (2,0.8); \draw (2.5, 0.8) -- (1.875,1.05);
\end{tikzpicture}

} & \\[-4mm]
      \hdashline
      \multirow[c]{2}{*}[1mm]{}         & \multirow[c]{2}{*}[1mm]{$\frac{1}{2} \varphi(N)$}               & {

\begin{tikzpicture}[baseline=(current bounding box.center), yscale=0.8]
  \draw[line width=1.1mm] (0,0.25) -- (0,1.75); \node at (0,1.75) [anchor=south]{$\Fgplustil{\weyl}$}; 
  \draw plot [smooth] coordinates {(-3,0.25) (-2.5, 0.65) (0,1) (2.5,1.35) (3, 1.75)};
  \node at (-3,1.75) [anchor=south]{$\widetilde{C}_{\infty,1}$};
  \node at (3,1.75) [anchor=south]{$\widetilde{C}_{\infty,2}$};
  \draw plot [smooth] coordinates {(3,0.25) (2.5, 0.65) (0,1) (-2.5,1.35) (-3, 1.75)};

\end{tikzpicture}

}  & \\[-4mm]
    \end{tabular}
  }
  \caption{The action of $\tilde{\tau}$ in a neighbourhood of the resolutions of the fixed cusps when $N \equiv 1 \pmod{2}$. Here $d \neq 1$ ranges over divisors of $N$. Horizontal curves are exceptional divisors of the resolution of sinuglarities $\ZNrtil{N}{r} \to \ZNr{N}{r}$ and their self-intersections are determined by Hirzebruch--Jung continued fractions.}
  \label{table:cusp-1mod2}
\end{table}
\endgroup



\begingroup
\renewcommand*{\arraystretch}{1.5}
\setlength{\tabcolsep}{0pt}

\begin{table}[p]
  \centering
  \scriptsize{
  \begin{tabular}{C{13mm}|C{32mm}|C{74mm}C{30mm}}
    Type      & {\makecell[c]{Number of\\ occurrences}} & Diagram                                    &\\
    \hline
    \multirow[c]{2}{*}[1mm]{$(N,1)$}   & \multirow[c]{2}{*}[1mm]{$\frac{1}{2} \rho(M, r)$}                   & {

\begin{tikzpicture}[baseline=(current bounding box.center), yscale=0.8]
  \draw[line width=1.1mm] (-2,0.25) -- (-2,1.75); \node at (-2,1.95)[anchor=east]{$\widetilde{F}_{1, \lambda}$}; 
  \draw[line width=1.1mm] (2,0.25) -- (2,1.75); \node at (2,1.95)[anchor=west]{$\Fgplustil[\lambda]{\borel,I}$}; 
  \draw (-3.5,1) -- (3.5,1);
\end{tikzpicture}

}    &\\[-4mm]
    \hdashline
    \multirow[c]{2}{*}[1mm]{$(M,1)$}   & \multirow[c]{2}{*}[1mm]{$\frac{1}{2} \rho(M, r)$}                   & {

\begin{tikzpicture}[baseline=(current bounding box.center), yscale=0.8]
  \draw[line width=1.1mm] (-2,0.25) -- (-2,1.75); \node at (-2,1.95)[anchor=east]{$\Fgplustil[\lambda]{\borel,I}$}; 
  \draw (2,0.25) -- (2,1.75); \node at (2,1.95)[anchor=west]{$\Fgplustil[\lambda]{\Isharp, I}$};
  \filldraw[black] (2,1) circle (0.5mm) node[anchor=north west]{$P_{\infty,\lambda}$};
  \draw (-3.5,1) -- (3.5,1);
\end{tikzpicture}

}   &\\[-4mm]
    \hdashline
    \multirow[c]{2}{*}[1mm]{$(d,-1)$}  & \multirow[c]{2}{*}[1mm]{$\frac{1}{2} \rho(d, -r) \varphi(M/d)$}      & {

\begin{tikzpicture}[baseline=(current bounding box.center), yscale=0.8]
  \draw[line width=1.1mm] (0,0.25) -- (0,1.75); \node at (0,1.75)[anchor=south]{$\Fgplustil{\antidiag,\weyl}$}; 

  \draw (-3, 1.2) -- (-2,0.8); \draw (-2.5, 0.8) -- (-1.875,1.05);
  \node at (-1.375,1) (a) {\large $\ldots$};
  \draw (-0.75, 1.2) -- (0.25,0.8); \draw (-0.25, 0.8) -- (0.75, 1.2); \draw (-0.875, 0.95) -- (-0.25,1.2); \draw (0.875, 0.95) -- (0.25,1.2);
  \node at (1.375,1) (a) {\large $\ldots$};
  \draw (3, 1.2) -- (2,0.8); \draw (2.5, 0.8) -- (1.875,1.05);
\end{tikzpicture}

}  &\\[-4mm]
    \hdashline
    \multirow[c]{2}{*}[1mm]{$(2d,-1)$} & \multirow[c]{2}{*}[1mm]{$\frac{1}{2} \rho(d, -r) \varphi(M/d)$}      & {

\begin{tikzpicture}[baseline=(current bounding box.center), yscale=0.8]
  \draw[line width=1.1mm] (-0.5,0.25) -- (-0.5,1.75); \node at (-0.5,1.95)[anchor=east]{$\Fgplustil{I,\weyl}$};
  \draw[line width=1.1mm] (0.5,0.25) -- (0.5,1.75); \node at (0.5,1.95)[anchor=west]{$\Fgplustil{\antidiag,\weyl}$}; 

  \draw (-3.5, 1.2) -- (-2.5,0.8); \draw (-2.9, 0.8) -- (-2.275,1.05); 
  \node at (-1.8,1) {\large $\ldots$};
  \draw (-1.5, 1.2) -- (-0.75,0.9); 
  \draw (-1.25,1) -- (1.25,1);
  \draw (1.5, 1.2) -- (0.75,0.9);
  \node at (1.8,1) {\large $\ldots$};
  \draw (2.9, 0.8) -- (2.275,1.05); \draw (3.5, 1.2) -- (2.5,0.8);
\end{tikzpicture}

} &\\[-4mm]
    \hdashline
    \multirow[c]{2}{*}[1mm]{}         & \multirow[c]{2}{*}[1mm]{$\frac{1}{2} \varphi(N)$}               & {

\begin{tikzpicture}[baseline=(current bounding box.center), yscale=0.8]
  \draw[line width=1.1mm] (0,0.25) -- (0,1.75); \node at (0,1.75) [anchor=south]{$\Fgplustil{\antidiag, \weyl}$}; 
  \draw plot [smooth] coordinates {(-3,0.25) (-2.5, 0.65) (0,1) (2.5,1.35) (3, 1.75)};
  \node at (-3,1.75) [anchor=south]{$\widetilde{C}_{\infty,1}$};
  \node at (3,1.75) [anchor=south]{$\widetilde{C}_{\infty,2}$};
  \draw plot [smooth] coordinates {(3,0.25) (2.5, 0.65) (0,1) (-2.5,1.35) (-3, 1.75)};

\end{tikzpicture}

}  & \\[-4mm]
  \end{tabular}
  }
  \caption{The action of $\tilde{\tau}$ in a neighbourhood of the resolutions of the fixed cusps when $N = 2M \equiv 2 \pmod{4}$ (where $M >1$). The integer $d$ ranges ranges over divisors of $M$ such that $d \neq 1$ if the singularity is of type $(d, -1)$.}
  \label{table:cusp-2mod4}
\end{table}
\endgroup



\begingroup
\renewcommand*{\arraystretch}{1.5}
\setlength{\tabcolsep}{0pt}

\begin{table}[p]
  \centering
  \scriptsize{
    \begin{tabular}{C{13mm}|C{32mm}|C{74mm}p{30mm}}
      Type                                  & {\makecell[c]{Number of\\occurrences}}                         & Diagram                                                       &                                                                             \\
      \hline
      \multirow[c]{2}{*}[1mm]{$(N,1)$}      & \multirow[c]{2}{*}[1mm]{$\frac{1}{2}\rho(N,r)$}                            & {

\begin{tikzpicture}[baseline=(current bounding box.center), yscale=0.8]
  \draw[line width=1.1mm] (-2,0.25) -- (-2,1.75); \node at (-2,1.95)[anchor=east]{$\widetilde{F}_{1, \lambda}$}; 
  \draw[line width=1.1mm] (2,0.25) -- (2,1.75); \node at (2,1.95)[anchor=west]{$\Fgplustil[\lambda]{\borel, I}$}; 
  \draw (-3.5,1) -- (3.5,1);
\end{tikzpicture}

}                       &                                                                             \\[-4mm]
      \hdashline
      \multirow[c]{2}{*}[-15mm]{{$(2M,1)$}} & \multirow[c]{1}{*}[-8mm]{$\frac{1}{2} \rho(M, r)$}              & \multirow[c]{2}{*}{{

\begin{tikzpicture}[baseline=(current bounding box.center), yscale=0.8]
  \draw[line width=1.1mm] (-2,0.25) -- (-2,1.75); \node at (-2,1.95)[anchor=east]{$\Fgplustil{\borel, \lambda I}$}; 
  \draw[line width=1.1mm] (2,0.25) -- (2,1.75); \node at (2,1.95)[anchor=west]{$\Fgplustil{\shortminus\!\borel, \lambda I}$};
  \draw (-3.5,1) -- (3.5,1);

  \draw (-3.5,-1.25) -- (3.5,-1.25);
  \filldraw[black] (-2,-1.25) circle (0.5mm) node[anchor=north east]{$P_{\infty,\lambda}$};
  \filldraw[black] (2,-1.25) circle (0.5mm) node[anchor=north west]{$P_{\infty,\lambda}^{-}$};
  \draw (-2,-0.5) -- (-2,-2); \node at (-2,-0.3)[anchor=east]{$\Fgplustil[\lambda]{\Isharp, I}$};
  \draw (2,-0.5) -- (2,-2); \node at (2,-0.3)[anchor=west]{$\Fgplustil[\lambda]{\borelsharp, I}$};
\end{tikzpicture}

}} & \multirow[c]{1}{*}[-8mm]{$r \equiv 1 \pmod{4}$}   \\[0mm]
                                            & \multirow[c]{1}{*}[-22mm]{$\frac{1}{2} \rho(M, r)$}               &                                                               & \multirow[c]{1}{*}[-22mm]{$r \equiv 3 \pmod{4}$} \\[27mm]
    \hdashline
    \multirow[c]{2}{*}[1mm]{$(d,-1)$}     & \multirow[c]{2}{*}[1mm]{$\frac{1}{2} \rho(d, -r) \varphi(N/d)$}               & {

\begin{tikzpicture}[baseline=(current bounding box.center), yscale=0.8]
  \draw[line width=1.1mm] (0,0.25) -- (0,1.75); \node at (0,1.75)[anchor=south]{$\Fgplustil{\antidiag,\weyl}$}; 

  \draw (-3, 1.2) -- (-2,0.8); \draw (-2.5, 0.8) -- (-1.875,1.05);
  \node at (-1.375,1) (a) {\large $\ldots$};
  \draw (-0.75, 1.2) -- (0.25,0.8); \draw (-0.25, 0.8) -- (0.75, 1.2); \draw (-0.875, 0.95) -- (-0.25,1.2); \draw (0.875, 0.95) -- (0.25,1.2);
  \node at (1.375,1) (a) {\large $\ldots$};
  \draw (3, 1.2) -- (2,0.8); \draw (2.5, 0.8) -- (1.875,1.05);
\end{tikzpicture}

}                     &                                                                             \\[-4mm]
    \hdashline
    \multirow[c]{2}{*}[1mm]{$(2d,-1)$}    & \multirow[c]{2}{*}[1mm]{$\frac{1}{2} \rho(2d, -r) \varphi(2M/d)$}   & {

\begin{tikzpicture}[baseline=(current bounding box.center), yscale=0.8]
  \draw[line width=1.1mm] (-0.5,0.25) -- (-0.5,1.75); \node at (-0.5,1.95)[anchor=east]{$\Fgplustil{\antidiag,\weyl}$};
  \draw[line width=1.1mm] (0.5,0.25) -- (0.5,1.75); \node at (0.5,1.95)[anchor=west]{$\Fgplustil{\antidiag,\weyl}$}; 

  \draw (-3.5, 1.2) -- (-2.5,0.8); \draw (-2.9, 0.8) -- (-2.275,1.05); 
  \node at (-1.8,1) {\large $\ldots$};
  \draw (-1.5, 1.2) -- (-0.75,0.9); 
  \draw (-1.25,1) -- (1.25,1);
  \draw (1.5, 1.2) -- (0.75,0.9);
  \node at (1.8,1) {\large $\ldots$};
  \draw (2.9, 0.8) -- (2.275,1.05); \draw (3.5, 1.2) -- (2.5,0.8);
\end{tikzpicture}

}                    &                                                                             \\[-4mm]
    \hdashline
    \multirow[c]{2}{*}[-15mm]{$(4d,-1)$}  & \multirow[c]{1}{*}[-8mm]{$\frac{1}{4} \rho(4d, -r) \varphi(M/d)$} & \multirow[c]{2}{*}{

\begin{tikzpicture}[baseline=(current bounding box.center), yscale=0.8]
  \draw[line width=1.1mm] (-0.5,0.25) -- (-0.5,1.75); \node at (-0.5,1.95)[anchor=east]{$\Fgplustil{\ns,{\weyl}}$};
  \draw[line width=1.1mm] (0.5,0.25) -- (0.5,1.75); \node at (0.5,1.95)[anchor=west]{$\Fgplustil{\s,{\weyl}}$}; 

  \draw (-3.5, 1.2) -- (-2.5,0.8); \draw (-2.9, 0.8) -- (-2.275,1.05); 
  \node at (-1.8,1) {\large $\ldots$};
  \draw (-1.5, 1.2) -- (-0.75,0.9); 
  \draw (-1.25,1) -- (1.25,1);
  \draw (1.5, 1.2) -- (0.75,0.9);
  \node at (1.8,1) {\large $\ldots$};
  \draw (2.9, 0.8) -- (2.275,1.05); \draw (3.5, 1.2) -- (2.5,0.8);

  \draw[line width=1.1mm] (-0.5,-2) -- (-0.5,-0.5); \node at (-0.5,-0.3)[anchor=east]{$\Fgplustil{\s,{\weyl}}$};
  \draw[line width=1.1mm] (0.5,-2) -- (0.5,-0.5); \node at (0.5,-0.3)[anchor=west]{$\Fgplustil{\s,{\weyl}}$}; 

  \draw (-3.5, -1.05) -- (-2.5,-1.45); \draw (-2.9, -1.45) -- (-2.275,-1.2); 
  \node at (-1.8,-1.25) {\large $\ldots$};
  \draw (-1.5, -1.05) -- (-0.75,-1.35); 
  \draw (-1.25,-1.25) -- (1.25,-1.25);
  \draw (1.5, -1.05) -- (0.75,-1.35);
  \node at (1.8,-1.25) {\large $\ldots$};
  \draw (2.9, -1.45) -- (2.275,-1.2); \draw (3.5, -1.05) -- (2.5,-1.45);
\end{tikzpicture}

}  &                                                                             \\[0mm]
                                            & \multirow[c]{1}{*}[-22mm]{$\frac{1}{4} \rho(4d, -r) \varphi(M/d)$} &                                                               &                                                                             \\[27mm]
      \hdashline
      \multirow[c]{2}{*}[1mm]{}         & \multirow[c]{2}{*}[1mm]{$\frac{1}{2} \varphi(N)$}               & {

\begin{tikzpicture}[baseline=(current bounding box.center), yscale=0.8]
  \draw[line width=1.1mm] (0,0.25) -- (0,1.75); \node at (0,1.75) [anchor=south]{$\Fgplustil{\antidiag, \weyl }$}; 
  \draw plot [smooth] coordinates {(-3,0.25) (-2.5, 0.65) (0,1) (2.5,1.35) (3, 1.75)};
  \node at (-3,1.75) [anchor=south]{$\widetilde{C}_{\infty,1}$};
  \node at (3,1.75) [anchor=south]{$\widetilde{C}_{\infty,2}$};
  \draw plot [smooth] coordinates {(3,0.25) (2.5, 0.65) (0,1) (-2.5,1.35) (-3, 1.75)};

\end{tikzpicture}

}  & \\[-4mm]
  \end{tabular}
  }
  \caption{The action of $\tilde{\tau}$ in a neighbourhood of the resolutions of the fixed cusps when $N = 4M \equiv 4 \pmod{8}$ (where $M > 1$). The integer $d$ ranges over divisors of $M$ such that $d \neq 1$ when the singularity is of type $(d, -1)$. In several cases the behaviour depends on the congruence class of $r$ modulo $4$, which is listed on the right.}
  \label{table:cusp-4mod8}
\end{table}
\endgroup



\begingroup
\renewcommand*{\arraystretch}{1.5}
\setlength{\tabcolsep}{0pt}

\begin{table}[p]
  \centering
  \vspace{-3.5pt}
  \scriptsize{
    \begin{tabular}{C{13mm}|C{32mm}|C{74mm}p{30mm}}
      Type                                   & {\makecell[c]{Number of                                                                                                                                                                                                           \\occurrences}}                                  & Diagram                                                          &                                                                                                    \\
      \hline
      \multirow[c]{2}{*}[1mm]{$(N,1)$}       & \multirow[c]{2}{*}[1mm]{$2 \rho(M, r)$}                                 & {

\begin{tikzpicture}[baseline=(current bounding box.center), yscale=0.8]
  \draw[line width=1.1mm] (-2,0.25) -- (-2,1.75); \node at (-2,1.95)[anchor=east]{$\widetilde{F}_{1,\lambda}$}; 
  \draw[line width=1.1mm] (2,0.25) -- (2,1.75); \node at (2,1.95)[anchor=west]{$\Fgplustil[\lambda]{\borel, I}$}; 
  \draw (-3.5,1) -- (3.5,1);
\end{tikzpicture}

}                          &                                                                                      \\[-4mm]
      \hdashline
      \multirow[c]{2}{*}[-15mm]{{$(N/2,1)$}} & \multirow[c]{1}{*}[-8mm]{$\rho(M, r)$}                                  & \multirow[c]{2}{*}{{

\begin{tikzpicture}[baseline=(current bounding box.center), yscale=0.8]
  \draw[line width=1.1mm] (-2,0.25) -- (-2,1.75); \node at (-2,1.95)[anchor=east]{$\Fgplustil[\lambda]{\borel, I}$}; 
  \draw[line width=1.1mm] (2,0.25) -- (2,1.75); \node at (2,1.95)[anchor=west]{$\Fgplustil[\lambda]{\shortminus\!\borel, I}$};
  \draw (-3.5,1) -- (3.5,1);

  \draw (-3.5,-1.25) -- (3.5,-1.25);
  \filldraw[black] (-2,-1.25) circle (0.5mm) node[anchor=north east]{$P_{\infty,\lambda}$};
  \filldraw[black] (2,-1.25) circle (0.5mm) node[anchor=north west]{$P_{\infty,\lambda}^{-}$};
  \draw (-2,-0.5) -- (-2,-2); \node at (-2,-0.3)[anchor=east]{$\Fgplustil[\lambda]{\Isharp, I}$};
  \draw (2,-0.5) -- (2,-2); \node at (2,-0.3)[anchor=west]{$\Fgplustil[\lambda]{\borelsharp, I}$};
\end{tikzpicture}

}} & \multirow[c]{1}{*}[-8mm]{$r \equiv 1 \pmod{8}$}                                      \\[0mm]
                                             & \multirow[c]{1}{*}[-22mm]{$\rho(M, r)$}                                 &                                                                  & \multirow[c]{1}{*}[-22mm]{{\makecell[c]{$k \geq 4 \text{ and }  r \equiv 1 \!\pmod{8}$ \\ $k = 3 \text{ and }  r \equiv 5 \!\pmod{8}$}}} \\[25mm]
      \hdashline
      \multirow[c]{2}{*}[1mm]{$(d,-1)$}      & \multirow[c]{2}{*}[1mm]{$2^{k-2} \rho(d, -r) \varphi(M/d)$}             & {

\begin{tikzpicture}[baseline=(current bounding box.center), yscale=0.8]
  \draw[line width=1.1mm] (0,0.25) -- (0,1.75); \node at (0,1.75)[anchor=south]{$\Fgplustil{\antidiag,\weyl}$}; 

  \draw (-3, 1.2) -- (-2,0.8); \draw (-2.5, 0.8) -- (-1.875,1.05);
  \node at (-1.375,1) (a) {\large $\ldots$};
  \draw (-0.75, 1.2) -- (0.25,0.8); \draw (-0.25, 0.8) -- (0.75, 1.2); \draw (-0.875, 0.95) -- (-0.25,1.2); \draw (0.875, 0.95) -- (0.25,1.2);
  \node at (1.375,1) (a) {\large $\ldots$};
  \draw (3, 1.2) -- (2,0.8); \draw (2.5, 0.8) -- (1.875,1.05);
\end{tikzpicture}

}                        &                                                                                      \\[-4mm]
      \hdashline
      \multirow[c]{2}{*}[1mm]{$(2d,-1)$}     & \multirow[c]{2}{*}[1mm]{$2^{k-3} \rho(d, -r) \varphi(M/d)$}             & {

\begin{tikzpicture}[baseline=(current bounding box.center), yscale=0.8]
  \draw[line width=1.1mm] (-0.5,0.25) -- (-0.5,1.75); \node at (-0.5,1.95)[anchor=east]{$\Fgplustil{\antidiag,\weyl}$};
  \draw[line width=1.1mm] (0.5,0.25) -- (0.5,1.75); \node at (0.5,1.95)[anchor=west]{$\Fgplustil{\antidiag,\weyl}$}; 

  \draw (-3.5, 1.2) -- (-2.5,0.8); \draw (-2.9, 0.8) -- (-2.275,1.05); 
  \node at (-1.8,1) {\large $\ldots$};
  \draw (-1.5, 1.2) -- (-0.75,0.9); 
  \draw (-1.25,1) -- (1.25,1);
  \draw (1.5, 1.2) -- (0.75,0.9);
  \node at (1.8,1) {\large $\ldots$};
  \draw (2.9, 0.8) -- (2.275,1.05); \draw (3.5, 1.2) -- (2.5,0.8);
\end{tikzpicture}

}                       &                                                                                      \\[-4mm]
      \hdashline
      \multirow[c]{2}{*}[-15mm]{$(4d,-1)$}   & \multirow[c]{1}{*}[-8mm]{$2^{k-4} \rho(4d, -r) \varphi(M/d)$}           & \multirow[c]{2}{*}{

\begin{tikzpicture}[baseline=(current bounding box.center), yscale=0.8]
  \draw[line width=1.1mm] (-0.5,0.25) -- (-0.5,1.75); \node at (-0.5,1.95)[anchor=east]{$\Fgplustil{\ns,\weyl}$};
  \draw[line width=1.1mm] (0.5,0.25) -- (0.5,1.75); \node at (0.5,1.95)[anchor=west]{$\Fgplustil{\omega\!\ns,{\weyl}}$}; 

  \draw (-3.5, 1.2) -- (-2.5,0.8); \draw (-2.9, 0.8) -- (-2.275,1.05); 
  \node at (-1.8,1) {\large $\ldots$};
  \draw (-1.5, 1.2) -- (-0.75,0.9); 
  \draw (-1.25,1) -- (1.25,1);
  \draw (1.5, 1.2) -- (0.75,0.9);
  \node at (1.8,1) {\large $\ldots$};
  \draw (2.9, 0.8) -- (2.275,1.05); \draw (3.5, 1.2) -- (2.5,0.8);

  \draw[line width=1.1mm] (-0.5,-2) -- (-0.5,-0.5); \node at (-0.5,-0.3)[anchor=east]{$\Fgplustil{\s,{\weyl}}$};
  \draw[line width=1.1mm] (0.5,-2) -- (0.5,-0.5); \node at (0.5,-0.3)[anchor=west]{$\Fgplustil{\omega\!\s,{\weyl}}$}; 

  \draw (-3.5, -1.05) -- (-2.5,-1.45); \draw (-2.9, -1.45) -- (-2.275,-1.2); 
  \node at (-1.8,-1.25) {\large $\ldots$};
  \draw (-1.5, -1.05) -- (-0.75,-1.35); 
  \draw (-1.25,-1.25) -- (1.25,-1.25);
  \draw (1.5, -1.05) -- (0.75,-1.35);
  \node at (1.8,-1.25) {\large $\ldots$};
  \draw (2.9, -1.45) -- (2.275,-1.2); \draw (3.5, -1.05) -- (2.5,-1.45);
\end{tikzpicture}

}     & \multirow[c]{1}{*}[-8mm]{$r \equiv 3 \pmod{8}$\hspace{-3mm}}                         \\[0mm]
                                             & \multirow[c]{1}{*}[-22mm]{$2^{k-4} \rho(4d, -r) \varphi(M/d)$}          &                                                                  & \multirow[c]{1}{*}[-22mm]{$r \equiv 7 \pmod{8}$\hspace{-3mm}}                        \\[27mm]
      \hdashline
      \multirow[c]{2}{*}[1mm]{$(2^ld,-1)$}   & \multirow[c]{2}{*}[1mm]{$2^{k-l-2} \rho(2^ld, -r) \varphi(M/d)$}        & {

\begin{tikzpicture}[baseline=(current bounding box.center), yscale=0.8]
  \draw[line width=1.1mm] (-0.5,0.25) -- (-0.5,1.75); \node at (-0.5,1.95)[anchor=east]{$\Fgplustil{\s,{\weyl}}$};
  \draw[line width=1.1mm] (0.5,0.25) -- (0.5,1.75); \node at (0.5,1.95)[anchor=west]{$\Fgplustil{\omega\!\s,{\weyl}}$}; 

  \draw (-3.5, 1.2) -- (-2.5,0.8); \draw (-2.9, 0.8) -- (-2.275,1.05); 
  \node at (-1.8,1) {\large $\ldots$};
  \draw (-1.5, 1.2) -- (-0.75,0.9); 
  \draw (-1.25,1) -- (1.25,1);
  \draw (1.5, 1.2) -- (0.75,0.9);
  \node at (1.8,1) {\large $\ldots$};
  \draw (2.9, 0.8) -- (2.275,1.05); \draw (3.5, 1.2) -- (2.5,0.8);
\end{tikzpicture}

}                      &                                                                                      \\[-4mm]
      \hdashline
      \multirow[c]{2}{*}[1mm]{}         & \multirow[c]{2}{*}[1mm]{$\frac{1}{2} \varphi(N)$}               & {

\begin{tikzpicture}[baseline=(current bounding box.center), yscale=0.8]
  \draw[line width=1.1mm] (0,0.25) -- (0,1.75); \node at (0,1.75) [anchor=south]{$\Fgplustil{\antidiag, \weyl }$}; 
  \draw plot [smooth] coordinates {(-3,0.25) (-2.5, 0.65) (0,1) (2.5,1.35) (3, 1.75)};
  \node at (-3,1.75) [anchor=south]{$\widetilde{C}_{\infty,1}$};
  \node at (3,1.75) [anchor=south]{$\widetilde{C}_{\infty,2}$};
  \draw plot [smooth] coordinates {(3,0.25) (2.5, 0.65) (0,1) (-2.5,1.35) (-3, 1.75)};

\end{tikzpicture}

}  & \\[-4mm]
    \end{tabular}
}
\caption{The action of $\tilde{\tau}$ in a neighbourhood of the resolutions of the fixed cusps when $N = 2^kM \equiv 0 \pmod{8}$ where $M$ is an odd integer. Here $\omega = 2^{k-1} + 1$ and $l$ ranges among integers $3 \leq l \leq k$. The integer $d$ ranges over divisors of $M$ such that $d \neq 1$ when the singularity is of type $(d, -1)$. In several cases the behaviour depends on the congruence class of $r$ modulo $8$, which is listed on the right.}
  \label{table:cusp-0mod8}
\end{table}
\endgroup


\FloatBarrier

\section{The genera of the modular curves \texorpdfstring{$\Xgplus{g}$}{X\_g{\textasciicircum}+}}
\label{sec:genera}
A concrete application of the results in \Cref{sec:fixed-points-tau} is that we are able to compute the genera of the modular curves $\Xgplus{g}$. There are also a number of benefits to these calculations.

Firstly, we will need these genera in order to compute the Chern numbers of the surfaces $\WNr{N}{r}$ (the surfaces $\WNr{N}{r}$ are smooth projective surfaces birational to $\ZNrSym{N}{r}$ which we construct in \Cref{sec:nonsing-mod}). 

Secondly, the following results give us a convenient way of checking the (rather complicated) calculations in \Cref{sec:fixed-points-0-1728,sec:action-tau-cusps}. Indeed, after fixing $N$ and choosing an element $g \in \GL_2(\bbZ/N\bbZ)$ it is simple to compute the genus of $\Xgplus{g} = X(\Hgplus{g})$ by e.g., first computing the subgroup $\Hgplus{g}$ explicitly (by applying the definition in \Cref{sec:cartan-nearly-cartan}) and then using the function \texttt{GL2Genus} provided by Rouse, Sutherland, and Zureick-Brown~\cite{RSZB_LAIOGFECOQ,RSZB_Electronic}.

Finally, these results are likely of independent interest. A special case of \Cref{lem:genus-Xg-odd} gives the genus of both the modular curves $X^+(\s {p^k})$ and $X^+(\ns {p^k})$ for every odd prime power $p^k$ (when $p = 2$ the analogous result follows from \Cref{lem:genus-Xg-0mod4}). Our approach gives a uniform treatment of the split and non-split cases. More generally, \Cref{lem:genus-Xg-odd,lem:genus-Xg-0mod4}\ref{enum:genus-Xns-w-0-mod4}--\ref{enum:genus-Xs-w-0-mod4} give explicit genus formulae for the modular curves associated to every extended Cartan subgroup of $\GL_2(\bbZ/N\bbZ)$ (see \Cref{remark:extended-genus}).

\Crefrange{lem:genus-Xg-odd}{lem:genus-Xg-0mod4} generalise many cases which have previously appeared in the literature. Let $\decorN{N} = N$ if $N$ is odd, and $\decorN{N} = N/2$ if $N$ is even. 
\begin{enumerate}[label=\enABC]
\item \label{enum:ns-genus}
  The genus of the modular curves $X^+(\ns {p^k})$ (and their fibre products) are computed by Baran in Theorems~7.2~and~7.3 of \cite{B_NONSCSMCATCNP}. 

\item \label{enum:s-genus}
  There exists a well known isomorphism $X(\s N) \cong X_0(N^2)$. Under this isomorphism it is simple to check that $X_{\s}^+$ is isomorphic to $X_0^+(\decorN{N}^2)$, the quotient of $X_0(\decorN{N}^2)$ by the Fricke involution. A genus formula for these modular curves is given in e.g.,~\cite[(1.6d)]{H_ANSOTROMCGBFI}.
 
\item \label{enum:hermann-genus}
  When $N = p$ is an odd prime number Hermann~\cite[(7)]{H_SMDDp2} gives a genus formula for the modular curve associated to the normaliser of a Cartan subgroup of $\GL_2(\bbF_p)$ which treats both the split and non-split cases simultaneously.
\end{enumerate}

We recover \ref{enum:ns-genus} from \Crefrange{lem:genus-Xg-odd}{lem:genus-Xg-0mod4}. Note that when $p^k$ is an odd prime power and $-r$ is a quadratic non-residue we have an isomorphism $\Xgplus{\weyl} \cong X^+(\ns \decorN{p^{k}})$ (by \Cref{lemma:cartan-Xg}). We similarly recover \ref{enum:s-genus} from \Crefrange{lem:genus-Xg-odd}{lem:genus-Xg-0mod4} since, if $-r$ is a square modulo $N$, we have an isomorphism $\Xgplus{\s,\weyl} \cong X_0^+(\decorN{N}^2)$ (by \Cref{lemma:cartan-Xg}). Restricting the result in \Cref{lem:genus-Xg-odd} to the case when $N = p$ is an odd prime number we recover Hermann's result \ref{enum:hermann-genus} giving the genera of the modular curve associated to an extended Cartan subgroup of $\GL_2(\bbF_p)$.

The genera of these modular curves will follow immediately by combining the following lemma with \Cref{prop:action-aff-figs,prop:action-cusps-figs}. 

\begin{lemma}
  \label{lemma:sing-to-genus}
  Let $g \in \GL_2(\bbZ/N\bbZ)$ be chosen so that $\Fgplustil{g}$ is a component of $\Ramtil$. Then the elliptic points of order $2$ and $3$ on $\Xgplus{g}$ are in bijective correspondence with intersection points of $\Fgplustil{g}$ with the resolutions of quotient singularities of type $(2,1)$ and $(3,1)$ respectively.
\end{lemma}

\begin{proof}
  First note that the elliptic points of order $2$ (resp. $3$) on a modular curve $X$ are exactly the points in the fibre of the morphism $X \to X(1)$ above $j = 1728$ (resp. $j = 0$) for which this morphism is unramified.

  Let $\DiagDiv$ be the diagonal in $\Zone = X(1) \times X(1)$. Note that $\Fgplus{g}$ has at worst nodal singularities (this may be seen from \Cref{prop:action-aff-figs,prop:action-cusps-figs}), in particular since $\Fgplustil{g}$ is non-singular the normalisation morphism $\widetilde{F}_g^{+} \to \Fgplus{g}$ is unramified. Consider a point $\ffz = (E,E,\phi) \in \Fgplus{g}$, let $G$ be the image of $\Aut(E) \to \Aut(E[N])$, and let $\mathrm{P}G = G/\{\pm 1\}$.

  If $\mathfrak{Z} \in X(N) \times X(N)$ is a pre-image of $\ffz$, then the morphism $\Fgplus{g} \to \DiagDiv$ is unramified at $\ffz$ if and only if the diagonal of $\mathrm{P}G \times \mathrm{P}G$ fixes $\mathfrak{Z}$. It follows from \Cref{lem:singular-matrix} that $\ffz$ is a cyclic quotient singularity of type $(d, 1)$ where $d = | \mathrm{P} G |$. The elliptic points of order $2$ and $3$ on $\Xgplus{g} \cong \Fgplustil{g}$ are therefore in one-to-one correspondence with the points which map to the singular points of type $(2,1)$ and $(3,1)$ above $j = 1728$ and $0$ respectively.
\end{proof}

For $g \in \GL_2(\bbZ/N\bbZ)$ let $\epsilon_2^+(g)$ and $\epsilon_3^+(g)$ be the number of elliptic points of order $2$ and $3$ on $\Xgplus{g}$ respectively and let $\epsilon_\infty^+(g)$ be the number of cusps and let $\eta^+(g) = [\GL_2(\bbZ/N\bbZ) : \Hgplus{g}]$. 
\begingroup
\allowdisplaybreaks
\begin{coro}
  \label{lem:genus-Xg-odd}
  Let $N$ be an odd integer and let $r$ be coprime to $N$. Let $\varphi$ denote Euler's totient function and let $\rho(N,m)$ be the function defined in \eqref{eqn:function-rho}. Define
  \begin{align*}
    e_2(N,r) &= \rho(N,r), \\
    e_3(N,r) &= \rho(N,3r), \\
    e_\infty(N,r) &= \prod_{p \mid N} \left( \varphi(p^{v_p(N)}) + \sum_{l = 0}^{v_p(N) - 1} \left( 1 + \left(\frac{-r}{p}\right) \right) \varphi( p^l ) \right), \\ 
    \mu(N,r) &= N^2 \prod_{p \mid N} \left( 1 + \frac{1}{p} \left(\frac{-r}{p}\right) \right)
  \end{align*}
  (cf. \cite[Theorem~7.2]{B_NONSCSMCATCNP} in the non-split case when $-r$ is a quadratic non-residue modulo $p$). We further define
  \begin{align*}
    e_2^+(N,r)                   &= \frac{1}{2} e_2(N,r) + h(-4N^2), &e_3^+(N,r)              &= \frac{1}{2} e_3(N,r), \\
    e_\infty^+(N,r)              &= \frac{1}{2} e_\infty(N,r),       &\mu^+(N,r)              &= \frac{1}{2} \mu(N,r). \\
    \intertext{If $N \neq 1$ then}
    \epsilon_2^+(g_{\weyl})      &= e_2^+(N, r),                     &\epsilon_3^+(g_{\weyl}) &= e_3^+(N,r),           \\
    \epsilon_\infty^+(g_{\weyl}) &= e_\infty^+(N,r),                 &\eta^+(g_{\weyl})       &= \mu^+(N,r).
  \end{align*}
  In particular, the genus of the modular curve $\Xgplus{\weyl}$ is given by
  \begin{equation*}
    \label{eq:genus-formula}
    1 + \frac{1}{12} \bigg( \mu^+(N,r) - 3e_2^+(N,r) - 4e_3^+(N,r) - 6e_\infty^+(N,r) \bigg).
  \end{equation*}
\end{coro}

\begin{proof}
  By a general result for modular curves (see e.g., \cite[Theorem~3.1.1]{DS_AFCIMF}) we have that $p_g(\Xgplus{g}) = 1 + \frac{1}{12} ( \eta^+(g) - 3 \epsilon_2^+(g) - 4 \epsilon_3^+(g) - 6 \epsilon_\infty^+(g))$.
  
  Let $g_{\weyl} \in \GL_2(\bbZ/N\bbZ)$ be the element from \Cref{lemma:fixer-g-N}. It follows immediately from \Cref{lemma:sing-to-genus,prop:action-cusps-figs} that $\epsilon_2^+(g_{\weyl}) = e_2^+(N,r)$ and $\epsilon_3^+(g_{\weyl}) = e_3^+(N,r)$. Similarly it follows from \Cref{prop:action-aff-figs} that $\epsilon_\infty^+(g_{\weyl}) = e_\infty^+(N,r)$. It remains to compute the index $\eta^+(g_{\weyl})$. But this is equal to the number of points on $\Fgplus{\weyl}$ above a general point on the diagonal in $X(1) \times X(1)$. By construction this is exactly one half of the number of elements of $\GL_2(\bbZ/N\bbZ)$ which are conjugate to $g_{\weyl}$. Each such element is of the form $\big( \begin{smallmatrix} a & b \\ c & -a \end{smallmatrix} \big)$ where $a^2 + bc = -r$, and there are therefore $\mu(N,r)$ such elements, and therefore $\eta^+(g_{\weyl}) = \mu^+(N,r)$. 
\end{proof}

If $N = 2^kM$ and $g \in \GL_2(\bbZ/N\bbZ)$ the index $\eta^+(g) = [\GL_2(\bbZ/N\bbZ) : \Hgplus{g}]$ is equal to one half of the order of the $\GL_2(\bbZ/N\bbZ)$-conjugacy class of $g$. This latter number is a multiplicative function of $N$, so in particular if $g \equiv g_{\weyl} \pmod{M}$ then $\eta^+(g) = \mu^+(M,r) \mu'$ where $\mu'$ denotes the order of the $\GL_2(\bbZ/2^k\bbZ)$-conjugacy class of $g$ modulo $2^k$. A counting argument yields the following corollaries.

\begin{coro}
  \label{lem:genus-Xg-2mod4}
  Let $N = 2M$ where $M$ is an odd integer and let $r$ be coprime to $N$. Then:
  \begin{enumerate}[label=\eniii]
  \item
    The modular curve $\Xgplus{I,\weyl}$ is isomorphic to the modular curve $\Xgplus{\weyl}$.
  \item
    We have
    \begin{align*}
      \epsilon_2^+ \left(g_{(\antidiag,\weyl)} \right)      &= e_2^+(M, r),     &\epsilon_3^+ \left( g_{(\antidiag,\weyl)} \right) &= 0, \\
      \epsilon_\infty^+\left( g_{(\antidiag,\weyl)} \right) &= 2 e_\infty^+(M,r), &\eta^+ \left( g_{(\antidiag,\weyl)} \right)       &= 3 \mu^+(M,r).
    \end{align*}
    In particular the genus of the modular curve $\Xgplus{\antidiag,\weyl}$ is given by
    \begin{equation*}
      1 + \frac{1}{12} \bigg( 3 \mu^+(M,r) - 3 e_2^+(M,r) - 6 \cdot 2 e_\infty^+(M,r) \bigg).
    \end{equation*}
  \end{enumerate}
\end{coro}

\begin{coro}
  \label{lem:genus-Xg-0mod4}
  Let $N = 2^kM$ where $M$ is an odd integer and $k \geq 2$, and let $r$ be coprime to $N$. Then:
  \begin{enumerate}[label=\eniii]
  \item \label{enum:genus-Xns-w-0-mod4}
    We have
    \begin{align*}
      \epsilon_2^+ \left(g_{(\ns,\weyl)} \right)      &= h(-4 (N/2)^2),     &\epsilon_3^+ \left( g_{(\ns,\weyl)} \right) &= 2 e_3^+(M,r), \\
      \epsilon_\infty^+\left( g_{(\ns,\weyl)} \right) &= 2^{k-2} e_\infty^+(M,r), &\eta^+ \left( g_{(\ns,\weyl)} \right)       &= 2^{2k-3} \mu^+(N,r).
    \end{align*}
    In particular the genus of the modular curve $\Xgplus{\ns, \weyl}$ is given by
    \begin{equation*}
      1 + \frac{1}{12} \bigg( 2^{2k-3} \mu^+(M,r) - 3h(-4 (N/2)^2) - 4 \cdot 2 e_3^+(M,r) - 6 \cdot 2^{k-2} e_\infty^+(M,r) \bigg).
    \end{equation*}
  \item \label{enum:genus-Xs-w-0-mod4}
    We have
    \begin{align*}
      \epsilon_2^+ \left(g_{(\s,\weyl)} \right)      &= h(-4 (N/2)^2),     &\epsilon_3^+ \left( g_{(\s,\weyl)} \right) &= 0, \\
      \epsilon_\infty^+\left( g_{(\s,\weyl)} \right) &= 3 \cdot 2^{k-2} e_\infty^+(M,r), &\eta^+ \left( g_{(\s,\weyl)} \right) &= 3 \cdot 2^{2k-3} \mu^+(N,r).
    \end{align*}
    In particular the genus of the modular curve $\Xgplus{\s, \weyl}$ is given by
    \begin{equation*}
      1 + \frac{1}{12} \bigg( 3 \cdot 2^{2k-3} \mu^+(M,r) - 3h(-4 (N/2)^2) - 6 \cdot 3 \cdot 2^{k-2} e_\infty^+(M,r) \bigg).
    \end{equation*}
  \item
    We have
    \begin{align*}
      \epsilon_2^+ \left(g_{(\antidiag,\weyl)} \right)      &=
                                                       \begin{cases}
                                                         \rho(N,1) & \text{if $r \equiv 1 \pmod{4}$} \\
                                                         h(-4N^2) & \text{if $r \equiv 3 \pmod{4}$,}
                                                       \end{cases}
                      &\epsilon_3^+ \left( g_{(\antidiag,\weyl)} \right) &= 0, \\
      \epsilon_\infty^+\left( g_{(\antidiag,\weyl)} \right) &= 2^{k} e_\infty^+(M,r), &\eta^+ \left( g_{(\antidiag,\weyl)} \right) &= 3 \cdot 2^{2k-2} \mu^+(N,r).
    \end{align*}
    In particular the genus of the modular curve $\Xgplus{\antidiag, \weyl}$ is given by
    \begin{equation*} 
      \begin{dcases}
        1 + \frac{1}{12} \bigg( 3 \cdot 2^{2k-2} \mu^+(M,r) - 3 \rho(N,1) - 6 \cdot 2^k e_\infty^+(M,r)  \bigg) & \text{if $r \equiv 1 \hspace{-2mm} \pmod{4}$} \\[3mm]
        1 + \frac{1}{12} \bigg( 3 \cdot 2^{2k-2} \mu^+(M,r) - 3h(-4N^2) - 6 \cdot 2^k e_\infty^+(M,r)  \bigg)            & \text{if $r \equiv 3 \hspace{-2mm} \pmod{4}$.}
      \end{dcases}
    \end{equation*}
  \end{enumerate}
\end{coro}
\endgroup

\vspace{1mm}
\begin{remark}
  \label{remark:extended-genus}
  We note once again that combining \Cref{lemma:cartan-Xg,lem:genus-Xg-odd} gives a formula for the genus of each extended Cartan subgroup of $\GL_2(\bbZ/N\bbZ)$ when $N$ is odd. In particular \Cref{lem:genus-Xg-odd} specialises, when $-r$ is a quadratic residue modulo $p$ for every prime $p \mid N$ (resp. when $-r$ is a quadratic non-residue modulo $p$ for every prime $p \mid N$), to a formula for the genus of the modular curve attached to the extended split (resp. non-split) Cartan subgroup of $\GL_2(\bbZ/N\bbZ)$. When $N$ is even the same follows from \Cref{lem:genus-Xg-0mod4}\ref{enum:genus-Xns-w-0-mod4} and \ref{enum:genus-Xs-w-0-mod4} for extended Cartan subgroups of $\GL_2(\bbZ/\decorN{N}\bbZ)$.
\end{remark}

\section{Intersection numbers on the surface \texorpdfstring{$\ZNrtil{N}{r}$}{\~Z(N,r)}}
\label{sec:intersections-ZNrtil}
Let $\Ramtil \colonequals \widetilde{\Ram}\!_{Z}$ denote the one dimensional component of the locus of fixed points of $\tilde{\tau}$, and recall that we write $K_{\tilde{Z}}$ for a canonical divisor of $\ZNrtil{N}{r}$.

To understand the intersection theory of the images of the modular curves $\widetilde{F}_{m,\lambda}$ on (a smooth model for) $\ZNrSym{N}{r}$ we need to control the intersections $\widetilde{F}_{m,\lambda} \cdot \Ramtil$. The following result is a direct generalisation of \cite[Satz~3]{H_SMDDp2} in the case when $N = p$ (cf. \cite[Satz~1.6]{H_MUMZSHM}). When $N = 1$ this relation is due to Kronecker and may be found in \cite[Theorem~10.11]{L_EF}.

\begin{lemma}
  \label{prop:Fm-intersect-R}
  Let $N \geq 5$ be an integer. If $m > 1$ is an integer coprime to $N$ then $\widetilde{F}_{m,\lambda} \cdot \Ramtil \geq f(m)$ where
  \begin{equation*}
    f(m) = 
    \begin{cases}
      h(-4m) + h(-m) &\text{ if } m \equiv 3 \pmod{4} \\
      h(-4m)         &\text{ otherwise}.
    \end{cases}   
  \end{equation*}
  and, moreover, equality holds if $\widetilde{F}_{m,\lambda}$ is non-singular.
\end{lemma}

\begin{proof}
  The Fricke involution, $w_m$, on $X_0(m)$ induces an involution on the non-singular points of $F_{m,\lambda}$ which agrees with $\tau$.
  Suppose $m > 4$. It is well known that a point $(E, \phi) \in X_0(m)$ is fixed by $w_m$ if and only if it is non-cuspidal, $E$ has complex multiplication, and $\phi^2 = -m$. This occurs only when $E$ has CM by an order of discriminant $D = -m$ or $-4m$ (the former only occurring if $m \equiv 3 \pmod{4}$). The number of $j$-invariants of such elliptic curves is given by the formula in the statement. Moreover if $\widetilde{F}_{m,\lambda}$ is non-singular, by considering the action of $\tilde{\tau}$ on the tangent space at the intersection we see that the intersections occur transversally.

  When $m \in \{2,3,4\}$ by \Cref{lemma:some-FN-smooth} the curves $F_{m,\lambda}$ are $(-1)$-curves. Applying \Cref{lemma:hirz-inv-lemma} exactly one of the two fixed points of $\tilde{\tau}$ on $\widetilde{F}_{m,\lambda}$ is a transversal intersection of $\widetilde{F}_{m,\lambda}$ and $\Ramtil$. The claim follows.
\end{proof}

Similarly, to understand the self-intersection of a canonical divisor on (our smooth birational model for) $\ZNrSym{N}{r}$ we will need the following lemma.

\begin{lemma}
  \label{lemma:Ktil-dot-Rtil}
  Let $N = 2^k M$ where $k \geq 0$ and $M$ is an odd integer. If $M \neq 1$ on $\ZNrtil{M}{r}$ we have
  \begin{equation*}
    K_{\tilde{Z}} \cdot \Fgplustil{\weyl} = \frac{1}{3} \mu^+(M,r) - 2 e_\infty^+(M,r) - \frac{1}{3} e_3^+(M,r).
  \end{equation*}
  If $k = 1$ then on $\ZNrtil{N}{r}$ we have
  \begin{align*}
    K_{\tilde{Z}} \cdot \Fgplustil{I,\weyl}         &= \frac{1}{3} \mu^+(M,r) - e_\infty^+(M,r) - \frac{1}{3} e_3^+(M,r), \\
    K_{\tilde{Z}} \cdot \Fgplustil{\antidiag,\weyl} &= \mu^+(M,r) - 3 e_\infty^+(M,r).
  \end{align*}
  If $k \geq 2$ then we have
  \begin{align*}
    K_{\tilde{Z}} \cdot \Fgplustil{\omega\!\ns, \weyl} &= \frac{2^{2k-3}}{3} \mu^+(M,r) - 2^{k-2}e_\infty^+(M,r) - \frac{2}{3} e_3^+(M,r), \\
    K_{\tilde{Z}} \cdot \Fgplustil{\omega\!\s, \weyl}  &= {2^{2k-3}} \mu^+(M,r) - 3 \cdot 2^{k-2}e_\infty^+(M,r),                          \\
    K_{\tilde{Z}} \cdot \Fgplustil{\antidiag, \weyl}   &= 2^{2k-2} \mu^+(M,r) - 3 \cdot 2^{k-1} e_\infty^+(M,r).
  \end{align*}
\end{lemma}

\begin{proof}
  Recall that by \eqref{eq:canonical-Z} we have a numerical equivalence of divisors 
  \begin{equation*}
    K_{\widetilde{Z}} \sim \frac{1}{6} \widetilde{C}_\star - D_\infty - \frac{1}{3}E_{3,1}
  \end{equation*}
  where $\widetilde{C}_{\star}$, $D_\infty$, and $E_{3,1}$ are the divisors defined in \Cref{sec:surf-ZNr,sec:minim-resol-sing} (see \Cref{table:ZNr-notation}).

  By the projection formula for a modular curve $F$ on $\ZNr{N}{r}$ (i.e., the curves $F_{m,\lambda}$, $\Fgplus{g}$, or $F_{m \circ g}^+$) we have ${C}_\star \cdot F = 2 \,[\GL_2(\bbZ/N\bbZ) : H]$ where $H \subset \GL_2(\bbZ/N\bbZ)$ is the subgroup (unique up to conjugacy) such that $-I \in H$ and $F$ is birational (as a cover of $X(1)$) to $X(H)$. In the cases in the lemma these indices are computed in \Cref{lem:genus-Xg-odd,lem:genus-Xg-2mod4,lem:genus-Xg-0mod4}. The intersection numbers $D_\infty \cdot \widetilde{F}$ and $E_{3,1} \cdot \widetilde{F}$ are simple to compute from \Cref{prop:action-aff-figs,prop:action-cusps-figs}.
\end{proof}

\section{A non-singular model \texorpdfstring{$\WNr{N}{r}$ for $\ZNrSym{N}{r}$}{W(N,r) for Z(N,r){\textasciicircum}(sym)}}
\label{sec:nonsing-mod}
Recall that if $S/\bbC$ is a smooth, projective, algebraic surface we define the irregularity $q(S) = p_g(S) - p_a(S) = h^1(S, \mathcal{O}_S) = h^0(S, \Omega_S^1)$, where $\Omega_S^1$ is the sheaf of holomorphic $1$-forms on $S$. We recall from \cite[Satz~1]{H_MQD} and \cite[Theorem~1]{KS_DQS} that the surface $\ZNrtil{N}{r}$ is \emph{regular} i.e., $q(\ZNrtil{N}{r}) = 0$ (this is the case for any Hilbert modular surface, see the discussion after Proposition~VI.6.2 in \cite{vdG_HMS}).

Let $N = 2^k M$ where $k \geq 0$ and $M$ is an odd integer. Let $r$ be coprime to $N$. We suppose from now that $\ZNrtil{N}{r}$ is \emph{not} a rational surface. 
\begin{construction}
  \label{constr:ZNro}
  Since $\ZNrtil{N}{r}$ is regular and not rational, no two $(-1)$-curves on (any smooth surface birational to) $\ZNrtil{N}{r}$ may intersect by \ratcrit{}. Recalling \Cref{prop:action-aff-figs,prop:action-cusps-figs} (combined with with \Cref{lemma:some-FN-smooth}\ref{enum:F1-4-are-1} and \Cref{lemma:flat-smooth}) we may define a smooth surface, which we denote by $\ZNro{N}{r}$, from $\ZNrtil{N}{r}$ by sequentially blowing down:
  \begin{enumerate}[label=(\arabic*)]
  \item \label{step:ZNro-1}
    the $\frac{1}{2} \rho(N,r)$ configurations consisting of $\widetilde{F}_{1,\lambda}$, the component of $E_{2,1}$, and the component of $E_{3,1}$ which they meet,
  \item
    the $\frac{1}{2} \rho(N,2r)$ configurations consisting of $\widetilde{F}_{2,\lambda}$ and the component of $E_{2,1}$ which they meet,
  \item
    the $\frac{1}{2} \rho(N,3r)$ curves $\widetilde{F}_{3,\lambda}$,
  \item
    the $\frac{1}{2} \rho(N,4r)$ curves $\widetilde{F}_{4,\lambda}$,
  \item
    if $k \geq 1$ the $\frac{1}{2} \rho(N,r)$ configurations consisting of $\Fgplustil[\lambda]{\borel,I}$ and the component of $E_{2,1}$ which they meet,
  \item
    if $k = 1$ the $\frac{1}{2} \rho(M,r)$ curves $\Fgplustil[\lambda]{\Isharp, I}$,
  \item
    if $k = 2$ and $r = 3 \pmod{4}$ the $\frac{1}{2} \rho(M,r)$ pairs of curves $\Fgplustil[\lambda]{\Isharp, I}$ and $\Fgplustil[\lambda]{\borelsharp,I}$,
  \item
    if $k = 3$ and $r \equiv 5 \pmod{8}$ the $\rho(M,r)$ pairs of curves $\Fgplustil[\lambda]{\Isharp, I}$ and $\Fgplustil[\lambda]{\borelsharp,I}$, and
  \item \label{step:ZNro-7}
    if $k \geq 4$ and $r \equiv 1 \pmod{8}$ the $\rho(M,r)$ pairs of curves $\Fgplustil[\lambda]{\Isharp, I}$ and $\Fgplustil[\lambda]{\borelsharp,I}$.
  \end{enumerate}
\end{construction}

\begin{construction}
  \label{const:WNr}
  The involution $\tilde{\tau}$ induces an involution $\tau^\circ$ on $\ZNro{N}{r}$. We define the surface $\WNr{N}{r}$ to be the quotient of $\ZNro{N}{r}$ by $\tau^\circ$.  
\end{construction}

\begin{lemma}
  \label{lemma:to-no-isolated}
  The involution $\tau^\circ$ acts on the surface $\ZNro{N}{r}$ without isolated fixed points.
\end{lemma}

\begin{proof}
  The isolated fixed points of $\tilde{\tau}$ are classified in \Cref{prop:action-aff-figs,prop:action-cusps-figs} (see also \Crefrange{fig:non-cusp-tau-odd}{fig:non-cusp-tau-4mod8-r1} and \Crefrange{table:cusp-1mod2}{table:cusp-0mod8}). In particular every isolated fixed point meets a curve blown down in steps~\ref{step:ZNro-1}--\ref{step:ZNro-7} of \Cref{constr:ZNro} above. By \Cref{lemma:hirz-inv-lemma} a $(-1)$-curve taken to itself by an involution on a smooth surface is either fixed or meets the one dimensional component of the fixed locus at a point. The claim follows by studying the configurations in \Cref{prop:action-aff-figs,prop:action-cusps-figs}.
\end{proof}

\begin{theorem}
  \label{lemma:W-nonsing}
  The surface $\WNr{N}{r}$ is a smooth, regular (i.e., $q(\WNr{N}{r}) = 0$), projective, algebraic surface which is birational to $\ZNrSym{N}{r}$.
\end{theorem}

\begin{proof}
  Clearly $\WNr{N}{r}$ is birational to $\ZNrSym{N}{r}$. Since $\tau^\circ$ acts without isolated fixed points the quotient $\WNr{N}{r}$ is smooth. Let $\pi \colon \ZNro{N}{r} \to \WNr{N}{r}$ be the quotient morphism. The pull back of a non-zero holomorphic $1$-form on $\WNr{N}{r}$ is non-zero. Since both $\WNr{N}{r}$ and $\ZNr{N}{r}$ are smooth, the sequence $0 \to \pi^*\Omega_{W/\bbC}^1 \to \Omega_{Z^\circ/\bbC}^1 \to \Omega_{Z^\circ/W}^1 \to 0$ is exact. Because the irregularity of the surface $\ZNro{N}{r}$ is zero taking global sections we see that $q(\WNr{N}{r}) = h^0(\WNr{N}{r}, \Omega_{W}^1) = 0$ as required.
\end{proof}

\subsection{The numerical invariants}
\label{sec:geometric-genus}
For convenience we first state the following lemma which is a simple consequence of the projection formula.

\begin{lemma}
  \label{lemma:coro-of-proj}
  Let $S/\bbC$ be a smooth, projective, algebraic surfaces, let $C \subset S$ be a (possibly reducible) curve. Let $f \colon S' \to S$ be the blow-up of $S$ at a point, let $E$ be the exceptional divisor, let $C'$ be the strict transform of $C$, and let $K_S$ and $K_S'$ be canonical divisors for $S$ and $S'$. Then:
  \begin{enumerate}[label=\eniii]
  \item \label{enum:KK-blowup}
   $K_S^2 = (K_S')^2 + 1$, and
  \item \label{enum:KC-blowup}
    $K_S \cdot C = (K_S' - E) \cdot C'$.
  \end{enumerate}
\end{lemma}

\begin{proof}
  There exists an integer $n$ such that $f^*C = C' - nE$ and we have $f^* K_S = K_S' - E$. The claim follows immediately from the fact that $E^2 = K_S' \cdot E = -1$ and the projection formula.
\end{proof}

Let $\pi \colon \ZNro{N}{r} \to \WNr{N}{r}$ be the quotient morphism. If $D$ is a divisor on $\ZNro{N}{r}$ we write $\widetilde{D}$ for the strict transform of $D$ on $\ZNrtil{N}{r}$. For consistency with previous literature (e.g.,~\cite{H_HMS}, \cite{H_SMDDp2}) we write $D^*$ for the image of $D$ under $\pi$. 

Given a divisor $D \subset \ZNrtil{N}{r}$ we write $D^\circ$ for the image of $D$ under the blow-down $\ZNrtil{N}{r} \to \ZNro{N}{r}$. If $D$ is a divisor on $\ZNr{N}{r}$ we write $D^\circ \colonequals \widetilde{D}^\circ$. In either case, we write $D^* \colonequals (D^\circ)^*$ by abuse of notation.

Let $\Ramo \subset \ZNro{N}{r}$ denote the divisor consisting of points fixed by $\tau^\circ$ (i.e., $\Ramo = (\Ramtil)^\circ$). Note that $\widetilde{\Ram}\!_{Z^\circ}$ is the divisor on $\ZNrtil{N}{r}$ given by removing $\widetilde{F}_{1, \lambda}$ and $\Fgplustil[\lambda]{\borel, I}$ from $\Ramtil$, for each $\lambda \in \LamNr{N}{r}$.

Let $K_{Z^\circ}$ and $K_W$ be canonical divisors for $\ZNro{N}{r}$ and $\WNr{N}{r}$ respectively. Recall that we write $K_{\tilde{Z}}$ for a canonical divisor for $\ZNrtil{N}{r}$. 

\begin{lemma}
  \label{lemma:Ko-Fo}
  Let $N = 2^k M$ where $M$ is an odd integer, and let $r$ be an integer coprime to $N$ such that $\ZNrtil{N}{r}$ is not a rational surface. Then
  \begingroup
  \allowdisplaybreaks
  \begin{align*}
    K_{Z^\circ}^2 - K_{\tilde{Z}}^2 &=
                                      \begin{cases}
                                        2 \rho(N,r) + \rho(N,2r) + \frac{1}{2}\rho(N,3r) & \text{if $k=0$,}                                \\
                                        3\rho(M,r) + \frac{1}{2} \rho(M,3r)              & \text{if $k=1$,}                                \\
                                        5 \rho(M,r)                                      & \text{if $k=2$ and $r \equiv 1 \pmod{4}$,}      \\
                                        \rho(M,r) + \rho(M,3r)                           & \text{if $k=2$ and $r \equiv 3 \pmod{4}$,}      \\
                                        10\rho(M,r)                                      & \text{if $k=3$ and $r \equiv 1 \pmod{8}$,}      \\
                                        2\rho(M,3r)                                      & \text{if $k \geq 3$ and $r \equiv 3 \pmod{8}$,} \\
                                        2\rho(M,r)                                       & \text{if $k=3$ and $r \equiv 5 \pmod{8}$,}      \\
                                        12\rho(M,r)                                      & \text{if $k \geq 4$ and $r \equiv 1 \pmod{8}$,} \\
                                        0                                                & \text{otherwise,}                               \\
                                      \end{cases}    
    \intertext{and}
    K_{\tilde{Z}} \cdot \Ramotil - K_{Z^\circ} \cdot \Ramo  &=
                                                              \begin{cases}
                                                                \frac{3}{2} \rho(N,r) + \rho(N,2r) + \frac{1}{2}\rho(N,3r) & \text{if $k=0$,}                                \\
                                                                2\rho(M,r) + \frac{1}{2} \rho(M,3r)                        & \text{if $k=1$,}                                \\
                                                                3\rho(M,r)                                                 & \text{if $k=2$ and $r \equiv 1 \pmod{4}$,}      \\
                                                                \rho(M,r) + \rho(M,3r)                                     & \text{if $k=2$ and $r \equiv 3 \pmod{4}$,}      \\
                                                                6\rho(M,r)                                                 & \text{if $k=3$ and $r \equiv 1 \pmod{8}$,}      \\
                                                                2\rho(M,3r)                                                & \text{if $k \geq 3$ and $r \equiv 3 \pmod{8}$,} \\
                                                                2\rho(M,r)                                                 & \text{if $k=3$ and $r \equiv 5 \pmod{8}$,}      \\
                                                                8\rho(M,r)                                                 & \text{if $k \geq 4$ and $r \equiv 1 \pmod{8}$,} \\
                                                                0                                                          & \text{otherwise.}                               \\
                                                              \end{cases}    
  \end{align*}
  \endgroup
\end{lemma}

\begin{proof}
  For the first claim, note that by \Cref{lemma:coro-of-proj}\ref{enum:KK-blowup} the difference $K_{Z^\circ}^2 - K_{\tilde{Z}}^2$ is equal to the number of components of the exceptional divisor of $\ZNrtil{N}{r} \to \ZNro{N}{r}$ (i.e., the number of curves blown down in \Cref{constr:ZNro}). 
  
  For the second claim, observe that the curve being blown-down at each step when constructing $\ZNro{N}{r}$ from $\ZNrtil{N}{r}$ meets (the image of) $\Ramotil$ at either $0$ or $1$ points. Moreover, the intersection number is $1$, except for when the curves $\widetilde{F}_{1,\lambda}$ and (when $k \geq 1$) $\Fgplustil[\lambda]{\borel, I}$ are blown-down. In particular by \Cref{lemma:coro-of-proj}\ref{enum:KC-blowup} we have
  \begin{equation*}
    (K_{Z^\circ}^2 - K_{\tilde{Z}}^2) - (K_{\tilde{Z}} \cdot \Ramotil - K_{Z^\circ} \cdot \Ramo) =
    \begin{cases}
      \frac{1}{2} \rho(N,r) & \text{if $k=0$,} \\
      \rho(N,r)             & \text{if $k \geq 1$.} \\
    \end{cases}
  \end{equation*}
  The result follows by a simple calculation.
\end{proof}

We are now able to prove the following theorem which generalises a result of Hermann when $N = p$ is a prime number~\cite[(10)]{H_SMDDp2}.

\vspace{1em}
\begin{theorem}
  \label{thm:geom-genus}
  Let $N = 2^kM$ be an odd integer where $k \geq 0$ and $M$ is an odd integer such that $N \geq 5$. Let $r$ be an integer coprime to $N$ such that $\ZNrtil{N}{r}$ is not rational. Let $\mu^+(M,r)$, $e_{\infty}^+(M,r)$, $e_2^+(M,r)$, and $e_3^+(M, r)$ be as defined in \Cref{lem:genus-Xg-odd}. The geometric genus of $\WNr{N}{r}$ is given as follows:
  \vspace{0.3em}
  \begin{enumerate}[label=\eniii]
  \item    
    If $k = 0$ the geometric genus is given by
    \vspace{0.2em}
    \begin{equation*}
      2p_g(\WNr{N}{r}) = p_g(\ZNr{N}{r}) + \frac{1}{4} \left( \frac{3}{2} \rho(N,r) + \rho(N,2r) + \frac{2}{3} \rho(N,3r) - \frac{1}{3}\mu^+(N,r) + 2 e_\infty^+(N,r) \right)  - 1 .
    \end{equation*}
  \item
    If $k = 1$ the geometric genus is given by
    \vspace{0.2em}
    \begin{equation*}
      2p_g(\WNr{N}{r}) = p_g(\ZNr{N}{r}) + \frac{1}{4} \left(2 \rho(M,r) + \frac{2}{3} \rho(M,3r) - \frac{4}{3} \mu^+(M,r) + 4 e_\infty^+(M,r) \right) - 1 .
    \end{equation*}
  \item
    If $k = 2$ and $r \equiv 1 \pmod{4}$ the geometric genus is given by
    \vspace{0.2em}
    \begin{equation*}
      2p_g(\WNr{N}{r}) = p_g(\ZNr{N}{r}) + \frac{1}{4} \bigg( 3 \rho(M,r) + 6 e_\infty^+(M,r) - 4 \mu^+(M,r) \bigg) - 1.
    \end{equation*}
  \item 
    If $k = 2$ and $r \equiv 3 \pmod{4}$ the geometric genus is given by
    \vspace{0.2em}
    \begin{equation*}
      2p_g(\WNr{N}{r}) = p_g(\ZNr{N}{r}) + \frac{1}{4} \left( \rho(M,r) + \frac{4}{3} \rho(M,3r) + 10 e_\infty^+(M,r) - \frac{20}{3} \mu^+(M,r) \right) - 1.
    \end{equation*}
  \item
    If $k \geq 3$ and $r \equiv 1 \pmod{8}$ the geometric genus is given by
    \vspace{0.2em}
    \begin{equation*}
      2p_g(\WNr{N}{r}) = \begin{dcases}
                           p_g(\ZNr{N}{r}) + \frac{1}{4} \bigg(6\rho(M,r) + 3 \cdot 2^{k-1} e_\infty^+(M,r) - 2^{2k-2}\mu^+(M,r) \bigg) - 1 &\!\!\text{if $k = 3$} \\[0.5em]
                           p_g(\ZNr{N}{r}) + \frac{1}{4} \bigg(8\rho(M,r) + 3 \cdot 2^{k-1} e_\infty^+(M,r) - 2^{2k-2}\mu^+(M,r) \bigg) - 1 &\!\!\text{if $k \geq 4$.}
                         \end{dcases}
    \end{equation*}
  \item
    If $k \geq 3$ and $r \equiv 3 \pmod{8}$ the geometric genus is given by
    \vspace{0.2em}
    \begin{equation*}
      2p_g(\WNr{N}{r}) = p_g(\ZNr{N}{r}) + \frac{1}{4} \left(\frac{8}{3} \rho(M,3r) + 2^{k+1} e_\infty^+(M,r) - \frac{2^{2k}}{3} \mu^+(M,r) \right) - 1.
    \end{equation*}
  \item
    If $k \geq 3$ and $r \equiv 5 \pmod{8}$ the geometric genus is given by
    \vspace{0.2em}
    \begin{equation*}
      2p_g(\WNr{N}{r}) = \begin{dcases}
                           p_g(\ZNr{N}{r}) + \frac{1}{4} \bigg(2\rho(M,r) + 3 \cdot 2^{k-1} e_\infty^+(M,r) - 2^{2k-2}\mu^+(M,r) \bigg) - 1 &\!\!\text{if $k = 3$} \\[0.5em]
                           p_g(\ZNr{N}{r}) + \frac{1}{4} \bigg(3 \cdot 2^{k-1} e_\infty^+(M,r) - 2^{2k-2}\mu^+(M,r) \bigg) - 1              &\!\!\text{if $k \geq 4$.}
                         \end{dcases}
    \end{equation*}
  \item
    If $k \geq 3$ and $r \equiv 7 \pmod{8}$ the geometric genus is given by
    \vspace{0.2em}
    \begin{equation*}
      2p_g(\WNr{N}{r}) = p_g(\ZNr{N}{r}) + \frac{1}{4} \bigg(2^{k+1} e_\infty^+(M,r) - {2^{2k-1}} \mu^+(M,r) \bigg) - 1.
    \end{equation*}
  \end{enumerate}
\end{theorem}

\begin{proof}
  Recall that we assumed $\ZNrtil{N}{r}$ is not a rational surface, in particular no $(-1)$-curves intersect on $\ZNrtil{N}{r}$ by \ratcrit{??}. By \cite[Section~3,~(16)]{H_THMGASAS} the Euler characteristics of the structure sheaves of $\WNr{N}{r}$ and $\ZNro{N}{r}$ are related by
  \begin{equation*}
    2\chi(\WNr{N}{r}) = \chi(\ZNro{N}{r}) - \frac{1}{4} K_{Z^\circ} \cdot \Ramo . 
  \end{equation*}
  We proved in \Cref{lemma:W-nonsing} that the surface $\WNr{N}{r}$ is regular and therefore we have a relation of geometric genera
  \begin{equation*}
    2 p_g(\WNr{N}{r}) = p_g(\ZNro{N}{r}) - \frac{1}{4} K_{Z^\circ} \cdot \Ramo - 1 .
  \end{equation*}
  The claim follows immediately from \Cref{lemma:Ktil-dot-Rtil,lemma:Ko-Fo}.
\end{proof}

Similarly, from the explicit formula for the class of $K_{\tilde{Z}}$ in \eqref{eq:canonical-Z} we are able to compute the self-intersection of $K_W$.

\begin{lemma}
  \label{lemma:chern}
  Suppose that $\WNr{N}{r}$ is not a rational surface (e.g., if $p_g(\WNr{N}{r}) > 0$). The Chern number $c_1^2(\WNr{N}{r}) = K_{W}^2$ is given by
  \begin{equation*}
    2K_{W}^2 - K_{\tilde{Z}}^2 + 2K_{\tilde{Z}} \cdot \Ramotil - \Ramotil^{\,2}  =
    \begin{cases}
      \frac{13}{2}\rho(N,r) + 4\rho(N,2r) + 2\rho(N,3r) & \text{if $k=0$,}                                \\
      9\rho(M,r) + 2\rho(M,3r)                          & \text{if $k=1$,}                                \\
      14\rho(M,r)                                       & \text{if $k=2$ and $r \equiv 1 \!\!\pmod{4}$,}      \\
      4\rho(M,r) + 4\rho(M,3r)                          & \text{if $k=2$ and $r \equiv 3 \!\!\pmod{4}$,}      \\
      28\rho(M,r)                                       & \text{if $k=3$ and $r \equiv 1 \!\!\pmod{8}$,}      \\
      8\rho(M,3r)                                       & \text{if $k \geq 3$ and $r \equiv 3 \!\!\pmod{8}$,} \\
      8\rho(M,r)                                        & \text{if $k=3$ and $r \equiv 5 \!\!\pmod{8}$,}      \\
      36\rho(M,r)                                       & \text{if $k \geq 4$ and $r \equiv 1 \!\!\pmod{8}$,} \\
      0                                                 & \text{otherwise.}
    \end{cases}                                                                                         
  \end{equation*}
  where $c_1^2(\ZNrtil{N}{r}) = K_{\tilde{Z}}^2$ is given in \cite[Theorem~2.6]{KS_MDQS} and $\Ramotil$ denotes the strict transform of $\Ramo$ under the blow-up $\ZNrtil{N}{r} \to \ZNro{N}{r}$.
\end{lemma}

\begin{proof}
  First note that by the adjunction formula if $\Fgplustil{g}$ is a component of $\Ramotil$, then because $\Fgplustil{g}$ and $(\Fgplustil{g})^\circ$ are smooth (by \Cref{lemma:fixed-is-smooth}) we have
  \begin{equation*}
    (K_{\tilde{Z}} + \Fgplustil{g}) \cdot \Fgplustil{g} = 2p_g(\Xgplus{g}) - 2 = (K_{Z^\circ} + (\Fgplus{g})^\circ) \cdot (\Fgplus{g})^{\circ} .
  \end{equation*}
  But none of the components of $\Ramotil$ or $\Ramo$ intersect (again by \Cref{lemma:fixed-is-smooth}), so we have
  \begin{equation*}
    \label{eq:Ramtil-Ramo}
    (K_{\tilde{Z}} + \Ramotil) \cdot \Ramotil = (K_{Z^\circ} + \Ramo) \cdot \Ramo.
  \end{equation*}
  Recall that $\pi^* K_W = K_{Z^\circ} - \Ramo$, and therefore by the projection formula
  \begin{equation*}
    2K_W^2 = (\pi^* K_W)^2 = K_{Z^\circ}^2 + (K_{Z^\circ} + \Ramo) \cdot \Ramo - 3 K_{Z^\circ} \cdot \Ramo.
  \end{equation*}
  Combining the previous two equations we obtain
  \begin{equation}
    \label{eqn:2KW-sq}
    2K_W^2 - K_{\tilde{Z}}^2 + 2 K_{\tilde{Z}}\cdot \Ramotil - \Ramotil^{\,2} = (K_{Z^\circ}^2 - K_{\tilde{Z}}^2) + 3(K_{\tilde{Z}}\cdot \Ramotil - K_{Z^\circ}\cdot {\Ramo}).
  \end{equation}
  Combining \Cref{lemma:Ko-Fo} and \eqref{eqn:2KW-sq} the claim follows.
\end{proof}

\subsection{Explicit computations of \texorpdfstring{$p_g(\WNr{N}{r})$ and $K_W^2$}{p\_g and K{\textasciicircum}2}}
\label{sec:comp-pg-K}
We note that with the formulae in \Cref{thm:geom-genus,lemma:chern} we may compute $p_g(\WNr{N}{r})$ and $c_1^2(\WNr{N}{r}) = K_{W}^2$ for any pair $(N,r)$.

For $p_g(\WNr{N}{r})$ this is easy to see. Kani--Schanz~\cite[Theorem~2]{KS_MDQS} give an explicit formula for $p_g(\ZNrtil{N}{r})$ which may be computed for any given pair $(N,r)$. By \Cref{thm:geom-genus} we may compute $p_g(\WNr{N}{r})$ from $p_g(\ZNrtil{N}{r})$ using only the functions $\rho(N,r)$, $\rho(N,2r)$, $\rho(N,3r)$, $\mu^+(N,r)$, and $e_\infty^+(N,r)$  -- all of which can be computed directly from the definitions for any given pair $(N,r)$.

Similarly $K_W^2$ may be computed from \Cref{lemma:chern}. Again Kani--Schanz~\cite[Theorem~2]{KS_MDQS} give an explicit formula for $K_{\tilde{Z}}^2$. It remains to show that we may compute $2K_{\tilde{Z}} \cdot \Ramotil - \Ramotil^{\,2}$. Define
\begin{equation*}
  S = \mathrm{Supp}(\Ramotil) =
  \begin{cases}
    \mathrm{Supp}(\Ramtil)                                                                                                & \text{if $r$ is not a square modulo $N$,}              \\
    \mathrm{Supp}(\Ramtil) \setminus \left\{\widetilde{F}_{1, \lambda} \right\}_{\lambda}                                 &\text{if $N$ is odd and $r$ is a square modulo $N$, or} \\
    \mathrm{Supp}(\Ramtil) \setminus \left\{\widetilde{F}_{1, \lambda}, \Fgplustil[\lambda]{\borel, I} \right\}_{\lambda} &\text{if $N$ is even and $r$ is a square modulo $N$.}
  \end{cases}
\end{equation*}
to be the support of the divisor $\Ramotil$, where $\lambda$ ranges over $\LamN{N}$. 
By \Cref{coro:ram-locus-smooth} the components of $\Ramtil$ do not intersect and therefore
\begin{equation*}
  2K_{\tilde{Z}} \cdot \Ramotil - \Ramotil^{\,2} = \sum_{F \in S} (2 K_{\tilde{Z}} \cdot F - F^2).
\end{equation*}
By the adjunction formula we compute
\begin{equation}
  \label{eq:compute-2KF-F2}
  2K_{\tilde{Z}} \cdot \Ramotil - \Ramotil^{\,2} = \sum_{F \in S} (3 K_{\tilde{Z}} \cdot F - 2p_g(F) + 2).
\end{equation}
But for each $F \in S$ the genus $p_g(F)$ is computed in \Crefrange{lem:genus-Xg-odd}{lem:genus-Xg-0mod4} and $K_{\tilde{Z}} \cdot F$ is computed in \Cref{lemma:Ktil-dot-Rtil}.

For the convenience of the reader, we have prepared accompanying \texttt{python} code in \cite{ME_ELECTRONIC_HERE} which may be used to compute $p_g(\ZNrtil{N}{r})$, $p_g(\WNr{N}{r})$, $K_{\tilde{Z}}^2$, and $K_{W}^2$ using the above descriptions. For each $6 \leq N \leq 33$ we include these invariants in the tables in \Cref{sec:table-numer-invar}. 

\subsection{Bounds on \texorpdfstring{$p_g(\WNr{N}{r})$ and $K_W^2$}{p\_g(W(N,r)) and K{\textasciicircum}2} and the proof of \texorpdfstring{\Cref{thm:p_g}}{Theorem 1.10}}
\label{sec:bounds}
In order to place the surfaces $\WNr{N}{r}$ within the Enriques--Kodaira classification we need some loose bounds on $p_g(\WNr{N}{r})$ and $K_W^2$, which we now give for sufficiently large $N$.

\Cref{thm:p_g} follows immediately from the following lemma, together with the computations of $p_g(\WNr{N}{r})$ for $6 \leq N \leq 33$ which are recorded in \Cref{sec:table-numer-invar}.

\begin{lemma}
  \label{lemma:pgW-bound}
  For each $N \geq 9$ we have
  \begin{align*}
    2p_g(\WNr{N}{r}) &\geq p_g(\ZNrtil{N}{r}) - \frac{1}{4} \cdot 2^{2k-1} \mu^+(M,r) - 1\\
                     &> \frac{1}{480} N(N - 1)(N - 23) - \frac{\zeta(2)}{16}N^{5/2} - 1
  \end{align*}
  and in particular $p_g(\WNr{N}{r}) \geq 3$ for those $(N,r)$ not appearing in cases (i)--(iii) of \Cref{thm:p_g}.
\end{lemma}

\begin{proof}
  The first claim is immediate from \Cref{thm:geom-genus}. By \cite[Proposition~2.9]{KS_MDQS} for each $N \geq 9$ we have the inequality $p_g(\ZNrtil{N}{r}) > \frac{1}{480} N(N - 1)(N - 23)$. To prove the second claim it therefore suffices to show that $2^{2k}\mu(M,r) \leq \zeta(2) N^{5/2}$. To see this, note that
  \begin{equation*}
    2^{2k} \mu(M,r) \leq 2^{2k} M^2 \prod_{p|M} \left(1 + \frac{1}{p} \right) = 2^{2k} \frac{M^3}{\varphi(M)}  \prod_{p|M} \left(1 - \frac{1}{p^2} \right).
  \end{equation*}
  Certainly $\prod_{p|M} \left(1 - 1/p^2 \right) < \prod_{p} \left(1 - 1/p^2\right) = \zeta(2)$. Moreover, since $M$ is odd, we have
  \begin{equation*}
    \varphi(M) = \prod_{p|M} p^{v_p(M)} \left(1 - \frac{1}{p}\right) \geq \prod_{p|M} p^{v_p(M)/2} = M^{1/2}.
  \end{equation*}
  Since $2^{2k} M^{5/2} \leq N^{5/2}$ we have
  \begin{equation}
    \label{eq:22k-mu-bound}
    2^{2k}\mu(M,r) < \zeta(2) N^{5/2}
  \end{equation}
  and the second claim follows.

  Define $\alpha(N) = \frac{1}{2} \left( \frac{1}{480} N(N - 1)(N - 23) - \frac{\zeta(2)}{16}N^{5/2} - 1 \right) - 3$. Note that $\alpha(2500) > 0$, and for $N \geq 2500$ the function $\alpha$ is increasing. Thus $p_g(\WNr{N}{r}) > 3$ for each $N > 2500$, and with the aid of a computer we check that
  \begin{equation}
    \frac{1}{2} \left( \frac{1}{480} N(N - 1)(N - 23) - \frac{1}{4} \cdot 2^{2k-1} \mu^+(M,r) - 1 \right) \geq 3 
  \end{equation}
  for each $67 \leq N \leq 2500$. The result follows by computing $p_g(\WNr{N}{r})$ for each $N < 67$ and $r \in (\bbZ/N\bbZ)^\times$ (using \Cref{thm:geom-genus} together with the formula for $p_g(\ZNrtil{N}{r})$ given in \cite[Theorem~2]{KS_MDQS}, as described in \Cref{sec:comp-pg-K}). 
\end{proof}

\begin{lemma}
  \label{lemma:K2-bound}
  For each $N \geq 5$ such that $\WNr{N}{r}$ is not rational we have
  \begin{align*}
    2 K_W^2 &\geq K_{\tilde{Z}}^2 - {3} \cdot 2^{2k - 1}\mu^+(M,r) - 6 \\
            &> \frac{1}{60} N(N-1)(N-30) - \frac{3}{4} \zeta(2) N^{5/2} - 6,
  \end{align*}
  so in particular $K_W^2 > 0$ for all $N > 105$.
\end{lemma}

\begin{proof}
  From \Cref{lemma:chern} we see that $2K_W^2 \geq K_{\tilde{Z}}^2 - 2 K_{\tilde{Z}} \cdot {\Ramotil} + {\widetilde{\Ram}_{Z^\circ}^2}$. Hence by \eqref{eq:compute-2KF-F2} we have
  \begin{equation*}
    2K_W^2 \geq K_{\tilde{Z}}^2 + \sum_{F} (2p_g(F) - 2) - 3K_{\tilde{Z}} \cdot {\Ramotil}
  \end{equation*}
  where the sum ranges over irreducible components of $\Ramotil$. But $\Ramotil$ has at most three components hence
  \begin{equation*}
    2K_W^2 \geq K_{\tilde{Z}}^2 - 3K_{\tilde{Z}} \cdot {\Ramotil} - 6. 
  \end{equation*}
  We compute from \Cref{lemma:Ktil-dot-Rtil} that $K_{\tilde{Z}} \cdot {\Ramotil} \leq 2^{2k-1} \mu^+(M,r)$ and the first claim follows.

  By \Cref{lemma:pgW-bound} when $N > 24$ the surface $\WNr{N}{r}$ is not rational, and by \cite[Proposition~2.8]{KS_MDQS} for $N \geq 5$ we have $K_{\tilde{Z}}^2 > \frac{1}{60}N(N-1)(N-30)$. As in \Cref{lemma:pgW-bound}, the second claim then follows from \eqref{eq:22k-mu-bound}.

  Finally, the function $\beta(N) = \frac{1}{60} N(N-1)(N-30) - \frac{3}{4} \zeta(2) N^{5/2} - 6$ is increasing for $N > 6000$ and noting that $\beta(6000) > 0$ is sufficient to see that $K_W^2 > 0$ for all $N \geq 6000$. With the aid of a computer we check that
  \begin{equation}
    \frac{1}{60} N(N-1)(N-30) - {3} \cdot 2^{2k - 1}\mu^+(M,r) - 6 > 0
  \end{equation}
  for each $106 \leq N \leq 6000$.
\end{proof}

\section{Modular curves on \texorpdfstring{$\WNr{N}{r}$}{W(N,r)}}
\label{sec:special-curves}
To prove \Cref{thm:WNr-KD} we use the standard criteria in \Cref{sec:surface-generalities} which rely on producing a (singular) fibre of an elliptic fibration. In particular, we are required to produce $(-1)$ and $(-2)$-curves on $\WNr{N}{r}$. We have three sources for producing such curves. Firstly the irreducible components of the exceptional divisors of the resolution of singularities $\ZNrtil{N}{r} \to \ZNr{N}{r}$, secondly the modular curves $F_{m,\lambda}$ defined on $\ZNr{N}{r}$ via Hecke correspondences, and finally the diagonal Hirzebruch--Zagier divisors $\Fgplus{g}$ (and their variants) defined in \Cref{sec:diag-HZ}.

Recall that for a curve $D \subset \ZNr{N}{r}$ we write $D^\circ$ for the image of $\widetilde{D}$ under the morphism $\ZNrtil{N}{r} \to \ZNro{N}{r}$. We write $D^*$ for the image of $D^\circ$ under the quotient morphism $\pi \colon  \ZNro{N}{r} \to \WNr{N}{r}$. Also note that $\pi^* K_W = K_{Z^\circ} - \Ramo$.

\subsection{Rational Hirzebruch--Zagier divisors on \texorpdfstring{$\WNr{N}{r}$}{W(N,r)}}
\label{sec:known-crvs}
Let $m \geq 5$ be an integer coprime to $N$ such that $rm$ is a square modulo $N$. The modular curve $F_{m,\lambda}^\circ$ is naturally birational to the modular curve $X_0(m)$, and the action on $F_{m,\lambda}^\circ$ induced by $\tau^\circ$ agrees with the Fricke involution on $X_0(m)$. It follows that for any $m \geq 5$ the image $F_{m,\lambda}^*$ is birational to $X_0^+(m)$ -- the quotient of $X_0(m)$ by the Fricke involution. The genus of $X_0^+(m)$ is well known (see e.g., \cite[(1.6d)]{H_ANSOTROMCGBFI} or \cite[Remark~2]{FH_HQOMCZX0N}). For our purposes it will be sufficient to recall (from \cite[p.~178]{H_NUDGGAG}) that $X_0^+(m)$ (and therefore $F_{m,\lambda}^*$ when $m \geq 5$), is a rational curve if and only if
\begin{equation}
  \label{eqn:X0mp-rat}
  m \in \Bigg \{
  \begin{aligned}
    &2, 3, 4, 5,6,7,8,9,10,11,12,13,14,15,16,17,18,19,20,21,\\
    &23,24,25,26,27,29,31,32,35,36,39,41,47,49,50,59,71
  \end{aligned}
  \Bigg\}.
\end{equation}
The reader is encouraged to compare the following lemma with \cite[Satz~2.2]{H_MUMZSHM}.
\begin{lemma}
  \label{lemma:KW-dot-Fm}
  Let $m \geq 5$ be an integer such that $rm$ is a square modulo $N$. Then
  \begin{equation*}
    K_W \cdot F_{m,\lambda}^* \leq \frac{1}{2} \left( \frac{1}{3} \psi(m)  - \nu_\infty(m) - \frac{1}{3} \nu_3(m) - \err{N}{m} - f(m) \right).
  \end{equation*}
  Here $\psi(m)$, $\nu_\infty(m)$, and $\nu_3(m)$ are defined in \Cref{prop:X0-facts}, $\err{N}{m}$ is defined in \Cref{lemma:some-FN-smooth}, and $f(m)$ is defined in \Cref{prop:Fm-intersect-R}.
\end{lemma}

\begin{proof}
  Since $\pi^* K_W = K_{Z^\circ} - \Ramo$ we have $K_W \cdot F_{m,\lambda}^* = \frac{1}{2} (K_{Z^\circ} - \Ramo) \cdot F_{m,\lambda}^\circ$ from the projection formula. Moreover, we compute that $(K_{Z^\circ} - \Ramo) \cdot F_{m,\lambda}^\circ \leq (K_{\tilde{Z}} - \Ramotil) \cdot \widetilde{F}_{m,\lambda}$ (equality holds if $\widetilde{F}_{m,\lambda}$ does not meet the exceptional divisor of $\ZNrtil{N}{r} \to \ZNro{N}{r}$). The claim follows from \Cref{lemma:some-FN-smooth,prop:Fm-intersect-R}.
\end{proof}

\subsection{Additional rational curves on \texorpdfstring{$\WNr{N}{r}$}{W(N,r)} when \texorpdfstring{$2 \mid N$ or $3 \mid N$}{2|N or 3|N}}
\label{sec:diag-HZ-on-W}
Again, the results in this section are only necessary to deal with the cases when $N$ is divisible by $2$ or $3$. In these cases there exist diagonal Hirzebruch--Zagier divisors which become $(-1)$ or $(-2)$-curves on $\WNr{N}{r}$, and, when $N \equiv 2 \pmod{4}$, the curves $(\Fliftnum{2}{\lambda})^*$ (as defined in \Cref{sec:lifted-modular}) are $(-1)$-curves. 

\begin{lemma}
  \label{lemma:Flift-star}
  Suppose that $N = 2M$ where $M$ is an odd integer and suppose that $2r$ is a square modulo $M$. If $\WNr{N}{r}$ is not a rational surface, then $(\Fliftnum{2}{\lambda})^*$ is a $(-1)$-curve on $\WNr{N}{r}$ which intersects a single component of $E_{2,1}^*$.
\end{lemma}

\begin{proof}
  By \Cref{lemma:Ktil-F2lift} we have $K_{\tilde{Z}} \cdot \Fliftnumtil{2}{\lambda} \leq 0$. Note that $\Fliftnumtil{2}{\lambda}$ meets $\Ram\!_{\tilde{Z}}$ at two distinct points (corresponding to the two pre-images of the CM point of discriminant $-8$ on $\widetilde{F}_{2,\lambda} \subset \ZNrtil{M}{r}$). Therefore $K_{W} \cdot (\Fliftnum{2}{\lambda})^* \leq -1$ and $(\Fliftnum{2}{\lambda})^*$ is a $(-1)$-curve by \ratcrit{??}.

  Let $E$ be an elliptic curve with $j(E)=1728$ and $\phi \equiv \big( \begin{smallmatrix} 1 & -1 \\ 1 & M + 1 \end{smallmatrix} \big) \in \PGL_2(\bbZ/N\bbZ)$. The curve $\Fliftnumtil{2}{\lambda}$ meets the component of $E_{2,1}$ given by resolving the singular point $(E, E, \phi) \in \ZNr{N}{r}$ and its image under ${\tau}$. The claim follows.
\end{proof}

\begin{lemma}
  \label{lemma:Fb4-dot-K}
  Let $N = 2^k M$ where $k \geq 2$ is an integer and $M$ is odd and suppose that $\WNr{N}{r}$ is not rational. For each $\omega \in \LamN{M}$ the curves $(\Fgplus{\borel\!4, \omega I})^*$ are distinct, geometrically irreducible $(-1)$-curves.
\end{lemma}

\begin{proof}
  Let $g = (g_{\borel\!4}, I)$. The matrix $r g^{-1}$ is conjugate to $g$, so $\tau$ restricts to an involution on $\Fgplus{\borel\!4, \omega I}$. By \Cref{lemma:flat-smooth} the two components of $\Fgplustil{\borel\!4, \omega I}$ are $(-1)$-curves which are swapped by $\tilde{\tau}$ (if they were not, their images on $\ZNro{N}{r}$ would contain an isolated fixed point of $\tau^\circ$, which is impossible). The claim follows.
\end{proof}

\begin{lemma}
  \label{lemma:calF-3-dot-K}
  Let $N = 3Q$ where $Q \neq 1, 2$ is an integer coprime to $3$ and suppose that $\WNr{N}{r}$ is not rational. For each $\lambda \in \LamN{N}$ and $\omega \in \LamN{Q}$ the curves $(\Fgplus[\lambda]{\three{\borel}{I}})^*$, $(\Fgplus{\three{\ns}{\omega I}})^*$, $(\Fgplus[\lambda]{\three{\nsalt}{I}})^*$, and $(\Fgplus{\three{\s}{\omega I}})^*$ are distinct, geometrically irreducible $(-1)$-curves.
  Moreover, the curves $(\Fgplus{\three{\ns}{\omega I}})^*$ and $(\Fgplus[\lambda]{\three{\nsalt}{I}})^*$ each intersect a component of $E_{2,1}^*$.
\end{lemma}

\begin{proof}
  First note that for each $g \in \{ \three{g_{\borel}}{I}, \three{g_{\ns}}{I}, \three{g_{\nsalt}}{I}, \three{g_{\s}}{I} \}$ the matrix $r g^{-1}$ is conjugate to $g$, so $\tau$ restricts to an involution on $\Fgplus[\lambda]{g}$. By \Cref{lemma:flat-smooth-3} and the same argument as \Cref{lemma:Fb4-dot-K} the claim holds for $(\Fgplustil[\lambda]{\three{\borel}{I}})^*$. For each $g \in \{\three{g_{\ns}}{I}, \three{g_{\nsalt}}{I}, \three{g_{\s}}{I} \}$ we have $K_{\tilde{Z}} \cdot \Fgplustil{g} \leq 0$ and $\Fgplustil{g}$ meets $\Ram_{\tilde{Z}}$ at two distinct points. Therefore $K_W \cdot (\Fgplustil{g})^* \leq -1$ and the first claim follows from \ratcrit.

  To see the second claim, note that if $E/\bbC$ has $j(E) = 1728$ then for some $a \in (\bbZ/N\bbZ)^\times$ the singular point $(E, E, a\phi) \in \ZNr{N}{r}$ lies on $\Fgplus{\lambda\three{\ns}{I}}$, respectively $\Fgplus{\lambda\three{\nsalt}{I}}$, where $\phi = \three{\mymat{0}{-1}{1}{0}}{I}$, respectively $\phi = \three{\mymat{1}{-1}{1}{1}}{I}$.
\end{proof}

\begin{lemma}
  \label{lemma:F-3b}
  Let $N = 2^kM$ where $M$ is an odd integer coprime to $3$ and suppose that $\WNr{N}{r}$ is not rational. For each $\lambda \in \LamN{N}$ the curves $(F_{3 \circ \lambda (\borel, I)}^+)^*$ are distinct, geometrically irreducible $(-1)$-curves.
\end{lemma}

\begin{proof}
  This follows from \Cref{lemma:SI-of-F3bI} analogously to \Cref{lemma:calF-3-dot-K}.
\end{proof}

\section{The Kodaira dimension and \texorpdfstring{\Cref{thm:WNr-KD}}{Theorem 1.3}}
\label{sec:kodaira-dimension}
For later use, we first record the following lemmas. Throughout this section the reader is encouraged to consult the code accompanying this paper \cite{ME_ELECTRONIC_HERE} which can assist computing numerical invariants and in depicting intersections in a neighbourhood of the divisors $E_{\infty,d,q}$ on $\ZNrtil{N}{r}$.

\begin{lemma}
  \label{lemma:C-inf-no-exceptional}
  We have
  \begin{equation*}
    K_W \cdot C_\infty^* =\, 2p_g(\widetilde{C}_{\infty,1}) - 2 - \widetilde{C}_{\infty,1}^2  - \frac{1}{2} \varphi(N),
  \end{equation*}
  where $C_\infty^*$ is the image of $\widetilde{C}_{\infty,1}$ (and $\widetilde{C}_{\infty, 2}$) on $\WNr{N}{r}$, and 
  \begin{align*}
    p_g(\widetilde{C}_{\infty,1}) &= 1 + \frac{N^2}{24} \prod_{p \mid N} \left(1 - \frac{1}{p^2} \right) - \frac{1}{4} \sum_{d \mid N} \varphi(d) \varphi(N/d), \\
    C_{\infty,1}^2 &= -\frac{1}{2} \sum_{\nu = 1}^{N-1} \varphi(\gcd(\nu,N)) \left\langle \frac{\nu^2r}{N \gcd(\nu, N)} \right\rangle
  \end{align*}
  are the genus and self-intersection of $\widetilde{C}_{\infty,1} \cong X_1(N)$. Here $\langle x \rangle = x - \lfloor x \rfloor$ denotes the fractional part of $x \in \bbQ$ and $\varphi$ denotes Euler's totient function.
\end{lemma}

\begin{proof}
  The genus of $\widetilde{C}_{\infty,1}$ (which is isomorphic to $X_\mu(N) \cong X_1(N)$ by \Cref{prop:Xmu-cusp}) is well known (see e.g., \cite[Remark~1.4]{KS_MDQS}). The self-intersection of the curve $\widetilde{C}_{\infty,1}$ is given in \cite[Proposition~2.5]{KS_MDQS} (cf. \cite[Hilfssatz~5]{H_MQD}, note the missing factor of $\frac{1}{2}$ in \cite[Proposition~2.5]{KS_MDQS}). By the adjunction formula we have $K_{\tilde{Z}} \cdot \widetilde{C}_{\infty,1} = 2p_g(\widetilde{C}_{\infty,1}) - 2 - \widetilde{C}_{\infty,1}^2$.
  
  Recall that by assumption $\ZNrtil{N}{r}$ is not rational (and therefore $N \geq 7$ or $(N,r) = (6,5)$). By \Cref{lemma:j-mult-inf} the multiplicity of $\widetilde{C}_{\infty,1}$ in $\ffj^*(\infty)$ is equal to $N$. Each modular curve in the exceptional divisor of $\ZNrtil{N}{r} \to \ZNro{N}{r}$ has index $\leq 6$ (and no curves are blown down when $(N,r)=(6,5)$). Thus $\widetilde{C}_{\infty,1}$ does not meet the exceptional divisor of $\ZNrtil{N}{r} \to \ZNro{N}{r}$ (otherwise a violation to the projection formula would occur). Since $\pi^* K_W = K_{Z^\circ} - \Ramo$, by the projection formula
  \begin{equation*}
    2 K_W \cdot C_{\infty}^* = (K_{Z^\circ} - \Ramo) \cdot (C_{\infty, 1}^\circ + C_{\infty, 2}^\circ) = 2 (K_{\tilde{Z}} - \Ramotil) \cdot \widetilde{C}_{\infty, 1}.
  \end{equation*}
  By \Cref{prop:action-cusps-figs} we have $\widetilde{C}_{\infty,1} \cdot \Ramotil = \frac{1}{2} \varphi(N)$ and the claim follows.
\end{proof}

\begin{lemma}
  \label{lemma:d-1-except}
  For each $1 \neq d \mid N$ the divisor $E_{\infty, d, -1}^*$ on $\WNr{N}{r}$ (where $E_{\infty, d, -1}$ is as defined in \Cref{sec:minim-resol-sing}) consists of $\frac{1}{2} \rho(d, -r) \varphi(N/d)$ disjoint chains of $\lfloor d/2 \rfloor$ smooth rational curves. The first curve in each chain meets $C_\infty^*$ transversally at one point, the first $\lfloor d/2 \rfloor - 1$ curves in each chain are $(-2)$-curves, and the last curve in each chain is a $(-1)$-curve (cf. \Cref{fig:24-23}).
\end{lemma}

\begin{proof}
  The Hirzebruch--Jung continued fraction of $1/d$ is equal to $[[2,...,2]]$ of length $d-1$, so by \Cref{lemma:sing-pts} the divisor $E_{\infty, d, -1}$ consists of $\frac{1}{2} \rho(d, -r) \varphi(N/d)$ chains of $d - 1$ smooth $(-2)$-curves joining $\widetilde{C}_{\infty,1}$ and $\widetilde{C}_{\infty,2}$.

  Now note that $E_{\infty, d, -1}$ does not meet the exceptional divisor of the blow-up $\ZNrtil{N}{r} \to \ZNro{N}{r}$. Indeed, if it did then on $\ZNro{N}{r}$ the component of $E_{\infty, d, -1}^\circ$ meeting the exceptional divisor is a curve $T^\circ$ which meets $\Ramo$ and such that $K_{Z^\circ} \cdot T^\circ \leq -1$. By \Cref{lemma:rat-crit} the curve $T^\circ$ is a $(-1)$-curve and successively blowing down $T^\circ$ and further components of $E_{\infty, d, -1}^\circ$ we can form a smooth surface on which the image of $\Ramo$ is singular, contradicting \Cref{lemma:fixed-is-smooth}.

  By \Cref{lemma:fixed-cusps-even} the involution $\tilde{\tau}$ acts on each chain by swapping the ends and it follows that $E_{\infty, d, -1}^*$ consists of $\frac{1}{2} \rho(d, -r) \varphi(N/d)$ chains of $\lfloor d/2 \rfloor$ smooth rational curves. Let $\Gamma$ be the $\lfloor d/2 \rfloor$-th component of such a chain. If $d$ is odd \Cref{prop:action-cusps-figs} implies that $\pi^* \Gamma$ consists of two components, each intersecting $\Ramtil$ transversally. Similarly if $d$ is even \Cref{prop:action-cusps-figs} implies that $\pi^* \Gamma$ is a single $(-2)$-curve intersecting $\Ramtil$ transversally at two points. Hence by the projection formula 
  \begin{equation*}
    2 K_W \cdot \Gamma = \pi^* K_W \cdot \pi^* \Gamma = (K_{Z^\circ} - \Ramo) \cdot \pi^* \Gamma = -2
  \end{equation*}
  as required. Since no other components of $E_{\infty, d, -1}$ meet $\Ramtil$ (again by \Cref{prop:action-cusps-figs}) the claim follows.
\end{proof}

\subsection{The rational cases}
\label{sec:rational-cases}

We now prove \Cref{thm:WNr-KD}\ref{thm:WNr-KD-rat}, which we re-state as follows.

\begin{prop}
  \label{prop:rational-cases}
  The symmetric Hilbert modular surface $\WNr{N}{r}$ is rational if and only if either $N \leq 14$ or $(N,r) \in \{(15,1), (15,2), (15,11), (16,1), (16,3), (16,7), (18,5), (20,11), (24,23)\}$.
\end{prop}

\begin{proof}
  A smooth rational surface $S/\bbC$ has geometric genus $p_g(S) = 0$. In particular, by \Cref{thm:p_g}, the surfaces $\WNr{N}{r}$ may only be rational for those $(N,r)$ in the statement of the proposition. To see that whenever $p_g(\WNr{N}{r}) = 0$ the surface $\WNr{N}{r}$ is rational, we argue as follows.

  By \Cref{lemma:d-1-except}, for each $1 \neq d \mid N$ the divisors $E_{\infty, d, -1}^*$ consist of $\frac{1}{2}\rho(d,-r)\varphi(N/d)$ disjoint chains of $\lfloor d/2 \rfloor$ smooth rational curves, the first $\lfloor d/2 \rfloor -1 $ of which are $(-2)$-curves and the last is a $(-1)$-curve. We successively blown down each such chain to form a surface $\overbar{W}_{N,r}$.

  The first curve in each chain in $E_{\infty,d,-1}^*$ meets $C_\infty^*$. Therefore if $K_{\overbar{W}}$ is a canonical divisor on $\overbar{W}_{N,r}$ we have
  \begin{equation}
    \label{eqn:K-Cinfbar}
    K_{\overbar{W}} \cdot \overbar{C}_{\infty} = K_{W} \cdot C_{\infty}^* - \frac{1}{2} \sum_{d \mid N, d \neq 1} \rho(d,-r) \varphi(N/d),
  \end{equation}
  where here $\overbar{C}_{\infty}$ denotes the image of $C_{\infty}^*$ on $\overbar{W}\!_{N,r}$.

  Suppose that $(N,r)$ is one of the pairs in the statement of the proposition and that $(N,r) \neq (24,23)$. We compute from \Cref{lemma:C-inf-no-exceptional} and \eqref{eqn:K-Cinfbar} that if $N \leq 10$ then $K_{\overbar{W}} \cdot \overbar{C}_\infty \leq -2$ and if $N \geq 11$ then $K_{\overbar{W}} \cdot \overbar{C}_\infty \leq -1$ (see \Cref{tab:invariants-WNr1}). In particular, by \Cref{lemma:rat-crit}\ref{enum:rat-crit-1} the surface $\WNr{N}{r}$ is rational for $N \leq 10$, and by \Cref{lemma:rat-crit}\ref{enum:rat-crit-2} the surface $\WNr{N}{r}$ is rational when $N \geq 11$ (if $N \geq 11$ then the curve $\overbar{C}_{\infty,1}$ is not rational, since $X_1(N)$ is not). We consider the case $(N,r) = (24,23)$ separately.

  \case{The case $(N,r) = (24,23)$} \pdfbookmark[3]{The case (N,r)=(24,23)}{24-23}
  On $\ZNrtil{24}{23}$ for each $\lambda \in \LamN{24}$ the smooth curve $\widetilde{F}_{23,\lambda}$ meets $E_{\infty, 24, 23}$ (in fact using the same argument with cusp widths which we utilise in \Cref{sec:elliptic-K3,sec:prop-ellipt-cases} one can show that $\widetilde{F}_{23, \lambda}$ meets components of $E_{\infty, 24, 23}$ which intersect $\widetilde{C}_{\infty}$, as is implicit in \Cref{fig:24-23}). By \Cref{lemma:KW-dot-Fm} we have $K_{W} \cdot {F}_{23,\lambda}^* \leq 0$. Let $\overbar{F}_{23,\lambda}$ be the image of $F_{23,\lambda}^*$ on $\overbar{W}_{N,r}$. Because the curve ${F}_{23,\lambda}^*$ meets $E_{\infty, 24, 23}^*$ (whose components are blown-down in forming $\overbar{W}_{24,23}$) the curve $\overbar{F}_{23,\lambda}$ meets $\overbar{C}_{\infty}$ and $K_{\overbar{W}} \cdot \overbar{F}_{23,\lambda} \leq -1$. It follows from \Cref{lemma:rat-crit} that either $\overbar{W}_{24,23}$ is rational (as required), or $\overbar{F}_{23,\lambda}$ is a $(-1)$-curve.

  We may therefore assume that the curves $\overbar{F}_{23,\lambda}$ are $(-1)$-curves which do not intersect (if they intersect then $\overbar{W}_{24,23}$ is rational by \Cref{lemma:rat-crit}). By \Cref{lemma:C-inf-no-exceptional} and \eqref{eqn:K-Cinfbar} we have $K_{\overbar{W}} \cdot \overbar{C}_\infty = 0$. Let $\overbar{W}_{24,23}'$ be the surface obtained from $\overbar{W}_{24,23}$ by blowing down each of the curves $\overbar{F}_{23,\lambda}$. Then if $\overbar{C}_{\infty}'$ is the image of $\overbar{C}_{\infty}$ on $\overbar{W}_{24,23}'$, we have $K_{\overbar{W}'} \cdot \overbar{C}_{\infty}' \leq -4$ (where $K_{\overbar{W}'}$ is a canonical divisor for $\overbar{W}_{24,23}'$) and it follows from \ratcrit{??}\ref{enum:rat-crit-1} that $\overbar{W}_{24,23}'$ is rational.
\end{proof}
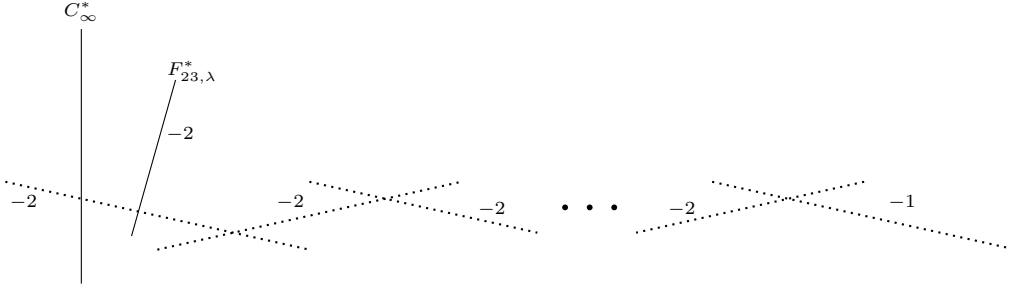
\begin{figure}[H]
  \centering
  \begin{tikzpicture}[x=1.0cm,y=1cm,rotate=270,xscale=0.9]
    \draw (1.25,0)-- (5,0);
    \draw[dotted,thick] (3.5,-1)-- (4.5,3);
    \draw[dotted,thick] (4.5,1)-- (3.5,5);
    \draw[dotted,thick] (3.5,8.3)-- (4.5,12.3);
    \draw[dotted,thick] (3.5,3)-- (4.25,6);
    \draw[dotted,thick] (3.5,10.3)-- (4.25,7.3);
    \draw (2,1.24)-- (4.3,0.66);
    \begin{scriptsize}
      \draw[anchor=south,color=black] (1.25,0) node {$C_{\infty}^*$};
      \draw[color=black] (3.8,-0.76) node {$-2$};
      \draw[color=black] (3.8,2.75) node {$-2$};
      \draw[color=black] (3.9,7.9) node {$-2$};
      \draw[color=black] (3.9,5.4) node {$-2$};
      \draw[color=black] (3.9,6.75) node {\Huge{$\cdots$}};
      \draw[color=black] (3.8,10.8) node {$-1$};
      \draw[anchor=south,color=black] (2.16,1.46) node {$F_{23,\lambda}^*$};
      \draw[color=black] (2.8,1.3) node {$-2$};
    \end{scriptsize}
  \end{tikzpicture}
  \caption{A configuration of curves on $\WNr{24}{23}$ in a neighbourhood of $E_{\infty,24,23}^*$. Dotted lines are the $12$ components of (one of the four chains in) $E_{\infty,24,23}^*$ which are blown down in forming $\overbar{W}_{24,23}$.}
  \label{fig:24-23}
\end{figure}

\subsection{The cases of general type}
\label{sec:gt-cases}

Let $N = 2^kM$ where $M$ is odd. We now suppose, in light of \Cref{prop:rational-cases}, that $\WNr{N}{r}$ is not a rational surface. Recall that we write
\begin{align*}
  s_{2,1}(N,r) &=
                 \begin{cases}
                   \parbox[h]{2cm}{$\frac{1}{2} h(-4N^2)$} & \parbox[h]{7cm}{if $N \not\equiv 0 \hspace{-0.5em} \pmod{4}$,}                         \\
                   h(-4N^2)             & \parbox[h]{7cm}{if $N \equiv 0 \hspace{-0.5em}\pmod{4}$ and $r \equiv 3 \hspace{-0.5em}\pmod{4}$, and } \\
                   0                    & \parbox[h]{7cm}{if $N \equiv 0 \hspace{-0.5em}\pmod{4}$ and $r \equiv 1 \hspace{-0.5em}\pmod{4}$,}
                 \end{cases}                                                                             
  \intertext{and}
  s_{3,2}(N,r) &=
                 \begin{cases}
                   \parbox[h]{2cm}{$\frac{1}{2} h(-3N^2)$} & \parbox[h]{7cm}{if $N \not\equiv 0 \hspace{-0.5em}\pmod{3}$,\hfill}                         \\
                   h(-3N^2)             & \text{if $N \equiv 0 \hspace{-0.5em}\pmod{3}$ and $r \equiv 2 \hspace{-0.5em}\pmod{3}$, and } \\
                   0                    & \text{if $N \equiv 0 \hspace{-0.5em}\pmod{3}$ and $r \equiv 1 \hspace{-0.5em}\pmod{3}$.}
                 \end{cases}                                                                              
  \intertext{Also define the indicator functions}
  \indicator{d \mid N}   &=
                 \begin{cases}
                   \parbox[h]{2cm}{$1$}           & \parbox[h]{7cm}{if $N \equiv 0 \hspace{-0.5em}\pmod{d}$, and}              \\
                   0                              & \text{otherwise,}
                 \end{cases}                                                                              \\
  \indicator{d \emid N} &=
                \begin{cases}
                  \parbox[h]{2cm}{$1$}          & \text{if $N \equiv 0 \hspace{-0.5em}\pmod{d}$, and $N \not\equiv 0 \hspace{-0.5em}\pmod{d^2}$} \\
                  0                             & \text{otherwise, }
                \end{cases}                                                                                                     \\                
  \intertext{and define}
  \mathfrak{m}(N,r) &= \frac{1}{2}  \sum_{m} \rho(N, mr),
\end{align*}
where in the definition of $\mathfrak{m}(N,r)$ the sum ranges over those $m \geq 5$ for which $F_{m,\lambda}^*$ is rational (see Equation~\ref{eqn:X0mp-rat}) and $K_{W} \cdot F_{m,\lambda}^* \leq -1$ (see \Cref{lemma:KW-dot-Fm}).

We now define a smooth projective surface, which we denote $\WNro{N}{r}$. Note that if $N$ is not divisible by $2$ or $3$ only steps \ref{enum:e21-FF}--\ref{enum:e21-5} in \Cref{constr:WNro} play a role.
\begin{construction}
  \label{constr:WNro}
  Suppose that the surface $\WNr{N}{r}$ is not rational (so that by \Cref{lemma:rat-crit} no pair of $(-1)$-curves may intersect on any smooth projective surface birational to $\WNr{N}{r}$). The surface $\WNro{N}{r}$ is constructed from $\WNr{N}{r}$ by successively blowing down:
  \begin{enumerate}[label=(\arabic*)]
  \item \label{enum:-1-cusps}
    the $\frac{1}{2} \lfloor d/2 \rfloor \rho(d,-1)\varphi(N/d)$ components of $E_{\infty,d,-1}^*$ for each $d \mid N$ with $d \neq 1$,
  \item \label{enum:e21-FF}
    the $s_{2,1}(N,r)$ components of $E_{2,1}^*$ whose strict transforms meet $\Ramtil$,
  \item \label{enum:e32-FF}
    the $s_{3,2}(N,r)$ components of $E_{3,2}^*$ whose strict transforms meet $\Ramtil$,
  \item \label{enum:e31-3}
    the $\frac{1}{2}\rho(N,3r)$ components of $E_{3,1}^*$ whose strict transforms on $\ZNrtil{N}{r}$ met $\widetilde{F}_{3,\lambda}$,
  \item \label{enum:frak-m}
    the $\mathfrak{m}(N,r)$ rational modular curves $F_{m,\lambda}^*$ for which $K_{W} \cdot F_{m,\lambda}^* \leq -1$,
  \item \label{enum:e21-5}
    the $\frac{1}{2}\rho(N,5r)$ components of $E_{2,1}^*$ which meet the curves $F_{5,\lambda}^*$,
  \item \label{enum:e31-sharp}
    the $\frac{1}{4}\rho(M,r) \big( 2 \indicator{2 \emid N} + \indicator{4 \emid N} \rho(4, 3r) + \indicator{8 \emid N} \rho(8, 5r) +  \indicator{16 \mid N} \rho(8, r) \big)$ components of $E_{3,1}^*$ whose strict transforms on $\ZNrtil{N}{r}$ met $\Fgplustil[\lambda]{\Isharp, I}$,
  \item \label{emum:Flift}
    the $\frac{\indicator{2 \emid N}}{2}\rho(M,2r)$ modular curves $(\Fliftnum{2}{\lambda})^*$ from \Cref{lemma:Flift-star} and the components of $E_{2,1}^*$ which they meet,
  \item \label{enum:Fb4}
    the $\frac{1}{4} \rho(N,r) \big( 2 \indicator{4 \emid N} + \indicator{8 \mid N} \big) $ curves $( \Fgplus{\borel\! 4, I} )^*$ from \Cref{lemma:Fb4-dot-K},
  \item \label{enum:F3b}
    the $\frac{\indicator{2 \mid N}}{2} \rho(N,3r)$ curves $(F_{3 \circ (\borel, I)}^+)^*$ from \Cref{lemma:F-3b},
  \item \label{enum:bl-down-Fg3}
    the $\frac{\indicator{3 \emid N}}{4} \rho(N/3, r) (3\rho(3,r) + 3\rho(3,2r) )$ curves $(\Fgplus{\three{\bullet}{I}})^*$ from \Cref{lemma:calF-3-dot-K},
  \item \label{enum:e21-Fg3}
    the $\frac{\indicator{3 \emid N}}{4} \rho(N/3, r) (\rho(3,r) + 2\rho(3,2r))$ components of $E_{2,1}^*$ meeting any of the curves $(\Fgplus{\three{\ns}{\omega I}})^*$ or $(\Fgplus[\lambda]{\three{\nsalt}{I}})^*$ blown down in step~\ref{enum:bl-down-Fg3}.
  \end{enumerate}  
\end{construction}
If $D$ is a divisor on $\ZNr{N}{r}$, $\ZNrtil{N}{r}$, $\ZNro{N}{r}$, or $\WNr{N}{r}$ we write $D^\oo$ for the image of $D$ on the surface $\WNro{N}{r}$. We let $\KWo$ be a canonical divisor for $\WNro{N}{r}$.

\begin{prop}
  The surface $\WNro{N}{r}$ is smooth.
\end{prop}

\begin{proof}
  Since $\WNr{N}{r}$ is not rational by assumption, no two $(-1)$-curves intersect on $\WNr{N}{r}$ (by \Cref{lemma:rat-crit}). It therefore suffices to show that each curve being blown down is a $(-1)$-curve. For steps~\ref{enum:-1-cusps} and \ref{emum:Flift}--\ref{enum:e21-Fg3} this follows from \Cref{lemma:d-1-except} and \Crefrange{lemma:Flift-star}{lemma:F-3b}. Since $\WNr{N}{r}$ is not rational the curves in steps~\ref{enum:frak-m} and \ref{enum:e21-5} are $(-1)$-curves by \Cref{lemma:rat-crit}. Any curve $\Gamma$ from steps~\ref{enum:e21-FF}--\ref{enum:e31-3} and \ref{enum:e31-sharp} has $\Ramo \cdot \pi^* \Gamma = 2$ and $K_{Z^\circ} \cdot \pi^* \Gamma \leq 0$ and therefore $K_W \cdot \Gamma \leq -1$, as required. 
\end{proof}

It is then simple to compute the self-intersection of a canonical divisor for $\WNro{N}{r}$ using the projection formula, which yields the following.

\begin{lemma}
  \label{lemma:Ko2}
  Let $\KWo$ be a canonical divisor for $\WNro{N}{r}$. We have 
  \begin{align*}
    \KWo^2 =  K_W^2 &+ \sum_{{d \mid N, \, d \neq 1}} \frac{\lfloor d/2 \rfloor \rho(d,-r)\varphi(N/d)}{2}
                      + s_{2,1}(N,r)
                      +  s_{3,2}(N,r)
                      + \frac{1}{2}\rho(N,3r)
                      + \mathfrak{m}(N,r)\\
                    &+ \frac{1}{2} \rho(N,5r)
                      + \frac{1}{4}\rho(M,r) \Big( 2 \indicator{2 \emid N} + \indicator{4 \emid N} \rho(4, 3r) + \indicator{8 \emid N} \rho(8, 5r) +  \indicator{16 \mid N} \rho(8, r) \Big) \\
                    &+ \indicator{2 \emid N} \rho(M,2r)
                      + \frac{1}{4} \rho(N,r) \Big( 2 \indicator{4 \emid N} + \indicator{8 \mid N} \Big)
                    + \frac{\indicator{2 \mid N}}{2} \rho(N,3r) \\
                    &+ \frac{\indicator{3 \emid N}}{4} \rho(N/3, r) \Big(4\rho(3,r) + 5\rho(3,2r) \Big). \\
  \end{align*}
  In particular if $(N,r) \not\in \{ (24,1), (24,5), (24,7), (24,17)\}$, and $(N,r)$ is as in \Cref{thm:WNr-KD}(iv) we have $\KWo^2 > 0$.
\end{lemma}

\begin{proof}
  The first claim follows immediately from the construction since, by \Cref{lemma:coro-of-proj}\ref{enum:KK-blowup}, $\KWo^2$ is equal to the sum of $K_W^2$ and the number of components of the exceptional divisor of the blow-up $\WNr{N}{r} \to \WNro{N}{r}$. Noting that $\KWo^2 > K_W^2$ it follows from \Cref{lemma:K2-bound} that $\KWo^2 > 0$ for each $N > 105$. The remaining cases follow from a calculation, where we compute $K_{W}^2$ by applying \Cref{lemma:chern} as described in \Cref{sec:comp-pg-K} (these values are recorded in \Cref{tab:invariants-WNr1} when $6 \leq N \leq 33$).
\end{proof}

\begin{remark}
  \label{rmk:hermann-typo}
  The surface $\WNro{N}{r}$ is a generalisation of the surface which Hermann denotes $W_{p}^\nu$~\cite[p.~177]{H_SMDDp2}. We take this opportunity to correct a small mistake in \cite{H_SMDDp2} where it is claimed there that no known curves (i.e., modular curves and curves arising from the resolution of singularities on $\ZNr{N}{r}$) are exceptional on $W_{p}^\nu$ for $p > 19$. However in their construction of $W_{p}^\nu$ they do not contract the component of $E_{2,1}^*$ which meets $F_{5}^*$ when $5r$ is a square modulo $p$.
\end{remark}

From \Cref{lemma:Ko2} we are now able to deduce \Cref{thm:WNr-KD}\ref{thm:WNr-KD-gen}. 

\begin{prop}
  \label{prop:gen-type-cases}
  If $(N,r)$ is as in \Cref{thm:WNr-KD}(iv) then $\WNr{N}{r}$ is a surface of general type.
\end{prop}

\begin{proof}
  By \cite[Proposition~X.1]{B_CAS} if $\KWo^2 > 0$, then $\WNr{N}{r}$ is a surface of general type. The claim then follows immediately from \Cref{lemma:Ko2}, except when $(N,r) = (24,1)$, $(24,5)$, $(24,7)$, or $(24,17)$ where we have $\KWo^2 = 0$. It therefore suffices to find a single $(-1)$-curve on $\WNro{24}{1}$, $\WNro{24}{5}$, $\WNro{24}{7}$, and $\WNro{24}{17}$. By \Cref{lemma:rat-crit} it suffices to find an irreducible curve $F^*$ on $\WNr{24}{r}$ with $K_W \cdot F^* \leq -1$, since such a curve is necessarily a $(-1)$-curve.

  \case{The case $(N,r) = (24,1)$} \pdfbookmark[3]{The case (N,r)=(24,1)}{24-1}
  Consider the curve $(\Fgplus{g})^*$ where $g = \mymat{1}{6}{12}{1}$. It is simple to compute that $\Xgplus{g} \cong \bbP^1$ has index $12$ and $4$ cusps (it has LMFDB label \LMFDBLabelMC{4.12.0.c.1}). In particular $K_{\tilde{Z}} \cdot \Fgplustil{g} \leq 0$ (by the same argument as \Cref{lemma:flat-smooth,lemma:flat-smooth-3}). Note that $g$ is conjugate to its inverse, so $\tilde{\tau}$ acts as an involution on $\Fgplustil{g}$. The involution of $\Xgplus{g}$ induced by $\tilde{\tau}$ has two fixed points and therefore $(\Fgplus{g})^\circ \cdot \Ramo \geq 2$. Thus $K_W \cdot (\Fgplus{g})^* \leq -1$ as required.
  
  \case{The case $(N,r) = (24,5)$} \pdfbookmark[3]{The case (N,r)=(24,5)}{24-5}
  Consider the curve $(\Fgplus{\Isharp,\nsalt})^*$. The curve $\Fgplus{\Isharp,\nsalt}$ is birational to the modular curve $X(\ns 6)$. Moreover, the quotient of $\Fgplus{\Isharp,\nsalt}$ by $\tau$ is birational to $X^+(\ns 6) \cong \bbP^1$. As in \Cref{lemma:flat-smooth,lemma:flat-smooth-3} we see that $K_{\tilde{Z}} \cdot \Fgplustil{\Isharp,\nsalt} \leq 2$. Moreover $(\Fgplus{\Isharp,\nsalt})^\circ \cdot \Ramo \geq 4$ and therefore $K_W \cdot (\Fgplus{\Isharp,\nsalt})^* \leq -1$.

  \case{The case $(N,r) = (24,7)$} \pdfbookmark[3]{The case (N,r)=(24,7)}{24-7}
  Consider the curve $(\Fgplus{g})^*$ where $g = \mymat{1}{6}{6}{-5}$. Note that $g$ is conjugate to its inverse and that $g = 19 g^{-1}$. In particular $\Fgplustil{g}$ is taken to itself by $\tilde{\tau}$ and by \Cref{lemma:fixed-is-smooth} it is smooth since it is fixed by $19 \tilde{\tau}$ (as defined in \Cref{rmk:al-invs}). It is simple to compute that  $\Fgplustil{g} \cong X(\ns 4)$ and therefore $K_{\tilde{Z}} \cdot \Fgplustil{g} \leq 0$. By \Cref{lemma:hirz-inv-lemma} we have $\Ramtil \cdot (\Fgplus{g})^\circ = 2$ and thus $K_W \cdot (\Fgplus{g})^* \leq -1$.

  \case{The case $(N,r) = (24,17)$} \pdfbookmark[3]{The case (N,r)=(24,17)}{24-17}
  Consider the curve $(F_{5 \circ (\Isharp, I)}^+)^*$ (as defined in \Cref{sec:isogeny-compose}). The curve $F_{5 \circ (\Isharp, I)}^+$ is birational to the genus $1$ modular curve $X_0(5) \times_{X(1)} X(\ns 2)$ (by \Cref{lemma:Fmg-bir}). Moreover $\tau$ acts non-trivially on each component of the fibre product, and the quotient has geometric genus $0$. Note that $K_{Z^\circ} \cdot (F_{5 \circ (\Isharp, I)}^+)^\circ \leq 2$ and $\tau^\circ$ has $4$ fixed points on $(F_{5 \circ (\Isharp, I)}^+)^\circ$. Therefore we have $K_W \cdot (F_{5 \circ (\Isharp, I)}^+)^* \leq -1$ as required.
\end{proof}

\subsection{The elliptic K3 surfaces}
\label{sec:elliptic-K3}
We now prove \Cref{thm:WNr-KD}(ii) by exhibiting elliptic configurations on (surfaces birational to) $\WNr{N}{r}$. More specifically, we check the conditions of \Cref{prop:ell-cong-K3} (noting that by \Cref{lemma:ell-conf-bound-enough} it suffices to show that the components of a putative elliptic configuration on a surface $S$ have $K_S \cdot D \leq 0$, rather than equality). In what follows we illustrate arguments when $(N,r) = (16,5)$ and $(17,1)$, leaving the details in further cases to the reader (see also the code associated to this article \cite{ME_ELECTRONIC_HERE}).

For each $\phi \in \GL_2(\bbZ/N\bbZ)$ of determinant $r$ we write $E_{2,1}(\phi)$ for the component of $E_{2,1}$ which maps to the singular point $(E, E, \phi) \in \ZNr{N}{r}$ where $j(E) = 1728$. Similarly we write $E_{3,1}(\phi)$ for the component of $E_{3,1}$ above $(E, E, \phi) \in \ZNr{N}{r}$ where $j(E) = 0$.

\case{The case $(N,r) = (16,5)$} \pdfbookmark[3]{The case (N,r)=(16,5)}{16-5}
We claim that the image of the curve
\begin{equation*}
  \mathscr{C}_{16,5} = F_{21}^* + E_{\infty,16,5}^*
\end{equation*}
on $\WNro{16}{5}$ contains an elliptic configuration of type $\mathrm{I}_8$ meeting a $(-2)$-curve.

Note that the Hirzebruch--Jung continued fraction of $16/5$ is equal to $[[4,2,2,2,2]]$ and therefore $E_{\infty,16,5} \subset \ZNrtil{16}{5}$ consists of a pair of $(-4,-2,-2,-2,-2)$-chains (by \Cref{lemma:sing-pts} there are two singular points of type $(16,5)$ on $\ZNr{16}{1}$ and these correspond to the matrices $\mymat{1}{0}{0}{5}$ and $\mymat{7}{0}{0}{3}$ in \Cref{prop:Xmu-cusp}).
Let $\Gamma$ be one such chain and let $\Delta$ be the other. Let $\Gamma_i$ and $\Delta_i$ be the $i$-th components of $\Gamma$ and $\Delta$ respectively for $i = 1, ..., 4$. Let $\Gamma_0 = \widetilde{C}_{\infty,1}$, $\Gamma_5 = \widetilde{C}_{\infty,2}$. By \Cref{lemma:j-mult-inf} the multiplicities of the components of $\Gamma_i$ in $\tilde{\ffj}^*(\infty)$ and $(\tilde{\ffj}')^*(\infty)$ are $(16,5,4,3,2,1,0)$ and $(0,1,4,7,10,13,16)$.

We first claim that $\Gamma^\oo$ is a $(-2,-2,-2,-2,-2)$-chain. The curve $\widetilde{F}_{5,1}$ meets $\Gamma$ transversally at the first component. To see this, note that by \Cref{lemma:which-cusp-met-comp} the image on $\widetilde{F}_{5,1}$ of the cusp $[1,1]$ on $X_0(5)$ (in the notation of \Cref{prop:X0-facts}) lies on $\Gamma$. The width of this cusp is $5$ and its image under the Fricke involution has width $1$ (by \Cref{lemma:which-cusp-met-comp}). Let $m_i$ be the multiplicity with which $\widetilde{F}_{5,1}$ meets $\Gamma_i$ at this cusp. Then
\begin{align*}
  &16 m_0 + 5m_1 + 4m_2 + 3m_3 + 2m_2 + m_1 = 5, \quad \text{and} \\
  &m_1 + 4m_2 + 7m_3 + 10m_4 + 13m_5 + 16m_6 = 1
\end{align*}
by the projection formula. But this pair of equations has a unique solution in positive integers $m_i$, where $m_1 = 1$ and $m_i = 0$ for each $i \neq 1$. Now recall that $F_{5,1}^*$ and the $(-2)$-curve $E_{2,1} \!\mymat{2}{-1}{1}{2}$ which it meets are both blown down in the map $\WNr{16}{5} \to \WNro{16}{5}$, therefore we compute that $\KWo \cdot \Gamma_1^\oo \leq (K_W \cdot \Gamma_1^*) - 2 \leq 0$, as required. 

Note that $\Delta$ is the image of $\Gamma$ under the involution given by the non-trivial element of $\LamN{16}$ (as defined in \Cref{sec:atkin-lehn-invol}). Therefore $\Delta^\oo$ is also a $(-2,-2,-2,-2,-2)$-chain.

Finally, we claim that $F_{21,1}^*$ and $F_{21,7}^*$ are smooth curves, that $F_{21,1}^*$ meets $\Gamma_1^*$, and that $F_{21,7}^*$ meets $\Gamma_3^*$. By \Cref{lemma:KW-dot-Fm} we have $K_W \cdot F_{21,7}^* \leq 0$. Therefore if the claim holds, it follows that the curves $F_{21,1}^\oo$, $\Gamma_1^\oo$, $\Gamma_2^\oo$, $\Gamma_3^\oo$, $F_{21,7}^\oo$, $\Delta_1^\oo$, $\Delta_2^\oo$, and $\Delta_3^\oo$ form an elliptic configuration of type $\mathrm{I}_8$ on $\WNro{16}{5}$ (by the symmetry of $\Gamma$ and $\Delta$ and of $F_{21,1}$ and $F_{21,7}$ under $\LamN{16}$).

Take $a = 1$ in the definition of $F_{21,\lambda}$ (see \Cref{sec:inters-modul-curv}). We argue similarly to above. By \Cref{lemma:which-cusp-met-comp} the image on $\widetilde{F}_{21,7}$ of the cusp $[1,7]$ on $X_0(21)$ lies on $\Gamma$. Again, let $m_i$ be the multiplicity with which $\widetilde{F}_{21,7}$ meets $\Gamma_i$ at this cusp. This cusp has width $3$ and its image under the Fricke involution has width $7$ (by \Cref{prop:X0-facts}). Therefore 
\begin{align*}
  &16 m_0 + 5m_1 + 4m_2 + 3m_3 + 2m_2 + m_1 = 3, \quad \text{and} \\
  &m_1 + 4m_2 + 7m_3 + 10m_4 + 13m_5 + 16m_6 = 7
\end{align*}
which has a unique solution $m_3 = 1$ and $m_i = 0$ for $i \neq 3$. Therefore $\widetilde{F}_{21,7}$ intersects $\Gamma_3$ transversally at a point. The analogous argument with $\widetilde{F}_{21,1}$ and the cusp $[1,1]$ on $X_0(21)$ shows that $\widetilde{F}_{21,1}$ meets both $\Gamma_1$ and $\widetilde{C}_{\infty,1}$ transversally at this cusp. Hence $F_{21,1}^\oo \cdot \Gamma_1^\oo = F_{21,7} \cdot \Gamma_3^\oo = 1$ as required. 

It remains to show that the curves $\widetilde{F}_{21,\lambda}^\oo$ are smooth. Repeating the above for each cusp shows that the points on $\widetilde{F}_{21,\lambda}$ above $(\infty,\infty) \in \Zone$ do not intersect and each meet distinct components of $D_{\infty}$ with multiplicity $1$. The curve $\widetilde{F}_{21,\lambda}$ is therefore smooth where it meets $D_{\infty}$. The only solution to $4 \cdot 21^2 - a^2 \equiv 0 \pmod{16^2}$ with $0 \leq a < 2 \cdot 21$ is when $a = 22$, in which case $4 \cdot 21^2 - a^2 = 16^2 \cdot 5$. But $-5 \not\equiv 0,1 \pmod{4}$ so by \Cref{lemma:extra-smooth} the curve $\widetilde{F}_{21,\lambda}$ is smooth. It is immediate that $F_{21,\lambda}^\oo$ is smooth because $\widetilde{F}_{21,\lambda}$ does not meet the exceptional divisor of $\ZNrtil{16}{5} \to \WNro{16}{5}$. If it did then $\KWo \cdot F_{21,\lambda}^\oo \leq -1$ and blowing down $F_{21,\lambda}^\oo$ (which is a $(-1)$-curve by \Cref{lemma:rat-crit}) gives a smooth surface on which two $(-1)$-curves intersect -- a contradiction to \Cref{lemma:rat-crit}.
\begin{figure}[H]
  \centering
  \begin{tikzpicture}[x=1.1cm,y=0.65cm]
    \draw  (-1.7677669529663689,1.767766952966369)-- (0.3535533905932738,-0.35355339059327384);
    \draw  (-1.4142135623730956,0.9142135623730957)-- (-1.414213562373095,3.914213562373092);
    \draw  (-1.7677669529663693,3.060660171779822)-- (0.35355339059327423,5.181980515339464);
    \draw  (-0.5,4.82842712474619)-- (2.5,4.82842712474619);
    \draw  (1.6464466094067207,5.181980515339469)-- (3.7677669529663658,3.060660171779823);
    \draw  (3.4142135623730856,0.9142135623731027)-- (3.4142135623730874,3.914213562373103);
    \draw  (1.646446609406717,-0.3535533905932662)-- (3.76776695296636,1.7677669529663764);
    \draw  (-0.5,0)-- (2.5,0);
    \draw  (1,4.328427124746184)-- (1,6.828427124746193);
    \draw  (1,0.5)-- (1,-2);
    \draw [dotted,thick] (-0.9142135623730834,2.414213562373096)-- (-3.914213562373126,2.4142135623730967);
    \draw [dotted,thick] (2.914213562373072,2.414213562373095)-- (5.9142135623731145,2.4142135623730936);
    \draw [dotted,thick] (-3.4142135623731242,3.4142135623730967)-- (-3.4142135623731242,1.4142135623730958);
    \draw [dotted,thick] (5.414213562373115,3.4142135623730945)-- (5.414213562373115,1.4142135623730936);
    \begin{scriptsize}
      \draw [anchor=north]  (0.35,-0.35) node {$F_{21,1}^*$};
      \draw [anchor=south]  (1.75,5.17) node {$F_{21,9}^*$};
      \draw [anchor=south] (-1.5,3.914213562373092) node {$-4$};
      \draw [anchor=north] (3.5,1) node {$-4$};
      \draw [anchor=east] (-3.4142135623731242,3.4142135623730967) node {$E_{2,1}\!\mymat{2}{-1}{1}{2}$};
      \draw [anchor=west] (5.414213562373115,1.4142135623730936) node {$E_{2,1}\!\mymat{2}{-9}{9}{2}$};
      \draw [anchor=west] (-3.4142135623731242,1.81) node {$-2$};
      \draw [anchor=east] (5.414213562373115,3.01) node {$-2$};
      \draw [anchor=south] (-2.4142135623731242,2.41) node {$-1$};
      \draw [anchor=north] (4.414213562373115,2.41) node {$-1$};
      \draw [anchor=west] (-0.9142135623730834,2.414213562373096) node {$F_{5,1}^*$};
      \draw [anchor=east] (2.914213562373072,2.414213562373095) node {$F_{5,9}^*$};
    \end{scriptsize}
  \end{tikzpicture}
  \caption{A configuration of curves on $\WNr{16}{5}$ whose image on $\WNro{16}{5}$ is an elliptic configuration of type $\mathrm{I}_{8}$ meeting (two) $(-2)$-curves. Unlabelled lines are components of $E_{\infty,16,5}^*$ and dotted lines are blown down in forming $\WNro{16}{5}$. Components without a recorded self-intersection are $(-2)$-curves. The involution on $\WNr{16}{5}$ induced by the action of $\LamN{16}$ on $\ZNr{16}{5}$ is given by reflection through the centre.}
  \label{fig:16-5}
\end{figure}
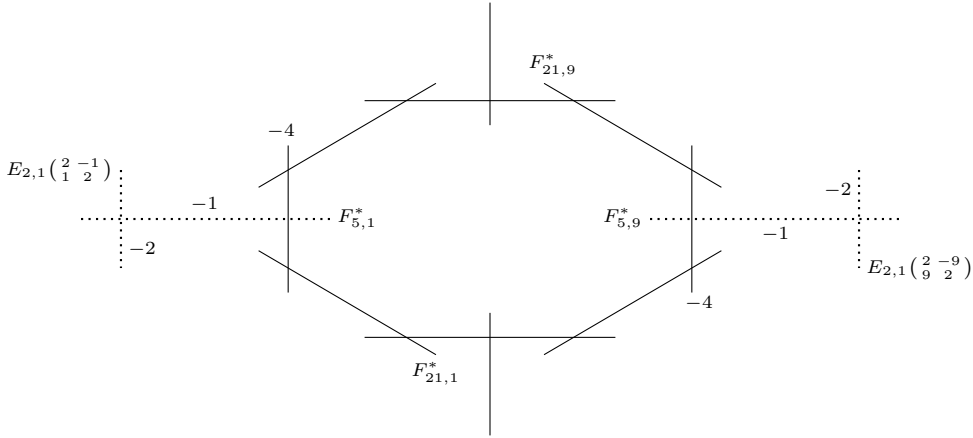


\case{The case $(N,r) = (17,1)$} \pdfbookmark[3]{The case (N,r)=(17,1)}{17-1}
This case was treated by Hermann~\cite[p.~178]{H_SMDDp2}, we exhibit a different elliptic configuration of type $\mathrm{I}_{11}$ for which the details are easier to check. We construct a surface $\WNrmin{17}{1}$ on which the image of the curve
\begin{equation*}
  \mathscr{C}_{17,1} = F_{15}^* + F_{21}^* + E_{\infty,17,4}^* + E_{\infty,17,8}^*
\end{equation*}
on $\WNr{17}{1}$ contains an elliptic configuration of type $\mathrm{I}_{11}$ meeting a $(-2)$-curve. The notation $\WNrmin{17}{1}$ is chosen in anticipation of \Cref{coro:min-models} where we show that $\WNrmin{17}{1}$ is a minimal surface.

We first claim that on $\WNro{17}{1}$ the image of $F_{16}^*$ and a component of $E_{\infty,17,9}^*$ are $(-1)$-curves. By \Cref{lemma:pgW-bound} we have $K_W \cdot F_{16}^* \leq 0$. By \Cref{lemma:some-FN-smooth}\ref{enum:int-which-cusp} the curve $F_{16}^*$ meets $E_{\infty,17,16}^*$. Note that in passing from $\WNr{17}{1}$ to $\WNro{17}{1}$ we have blown down the components of $E_{\infty,17,16}^*$, so the image $F_{16}^{\oo}$ of $F_{16}^*$ on $\WNro{17}{1}$ has $\KWo \cdot F_{16}^{\oo} \leq -1$ and is therefore a $(-1)$-curve by \Cref{lemma:rat-crit}.

For brevity write $\Gamma = E_{\infty,17,9}$ and note that the Hirzebruch--Jung continued fraction of $17/9$ is $[[2,9]]$, so $\Gamma$ is a $(-2,-9)$-chain. Let $\Gamma_i$ be the $i$-th component of $\Gamma$ and write $\Gamma_0 = \widetilde{C}_{\infty,1}$ and let $\Gamma_3 = \widetilde{C}_{\infty,2}$. Let $a_i$ and $a_i'$ be the multiplicity of $\Gamma_i$ in $\tilde{\ffj}^*(\infty)$ and $(\tilde{\ffj}')^*(\infty)$ respectively. By \Cref{lemma:j-mult-inf} we have $(a_0, a_1, a_2, a_3) = (17,9,1,0)$ and $(a_0', a_1', a_2', a_3') = (0,1,2,17)$. Consider the cusp $x \in X_0(9)$ corresponding to the pair $[1,1]$ from \Cref{prop:X0-facts}. By \Cref{prop:X0-facts} the width of $x$ is $9$ and the width of its image under the Fricke involution is $1$. By \Cref{lemma:which-cusp-met-comp} the image of $x$ on $\widetilde{F}_9$ lies on $\Gamma$. For each $0 \leq i \leq 3$, let $m_i$ be the multiplicity with which $\widetilde{F}_{9}$ meets $\Gamma_i$. By the projection formula we have
\begin{equation*}
  17 m_0 + 9 m_1 + m_2 = 9 \qquad \text{and} \qquad m_1 + 2m_2 + 17m_3 = 1
\end{equation*}
which, since for each $i$ the multiplicity $m_i$ is a positive integer, has a unique solution where $m_1 = 1$ and $m_0 = m_2 =m_3 = 0$. In particular, $\widetilde{F}_9$ meets $\Gamma_1$ transversally at a point. Since $F_9^*$ is blown down in forming $\WNro{17}{1}$, the image $\Gamma_1^\oo$ of $\Gamma_1$ on $\WNro{17}{1}$ is a $(-1)$-curve.

We define $\WNrmin{17}{1}$ to be the surface obtained from $\WNro{17}{1}$ by blowing down $F_{16}^{\oo}$ and $\Gamma_1^\oo$. We now show that the image of $\mathscr{C}_{17,1}$ on $\WNrmin{17}{1}$ contains an elliptic configuration. Let $K_{\mathsf{min}}$ be a canonical divisor on $\WNrmin{17}{1}$.

Let $\Delta = E_{\infty,17,4}$ and let $\Theta = E_{\infty,17,8}$. Note that $\Delta$ is a $(-5,-2,-2,-2)$-chain and $\Theta$ is a $(-3,-2,-2,-2,-2,-2,-2,-2)$-chain. Let $\Delta_i$ and $\Theta_i$ be the $i$-th components of $\Delta$ and $\Theta$ respectively. The components of $\Delta$ have multiplicities $(4,3,2,1)$ and $(1,5,9,13)$ in $\tilde{\ffj}^*(\infty)$ and $(\tilde{\ffj}')^*(\infty)$ respectively, by \Cref{lemma:j-mult-inf}.  Consider the point on $\widetilde{F}_{21}$ corresponding to the cusp $[1,1]$ on $X_0(21)$ from \Cref{prop:X0-facts}. By \Cref{lemma:which-cusp-met-comp} this point lies on $\Delta$ and, as above, if its multiplicity on the component $\Delta_i$ is $m_i$ (where $\Delta_0 = \widetilde{C}_{\infty,1}$ and $\Delta_5 = \widetilde{C}_{\infty,2}$) then
\begin{equation*}
  17 m_0 + 4 m_1 + 3 m_2 + 2 m_3 + m_4 = 21 \qquad \text{and} \qquad m_1 + 5m_2 + 9m_3 + 13m_4 + 17m_5 = 1
\end{equation*}
which has a unique solution in positive integers $m_i$ with $m_0=m_1=1$ (i.e., $\widetilde{F}_{21}$ meets $\Delta_1$ with multiplicity $1$ at the point it intersects $\widetilde{C}_{\infty,1}$). Similar arguments imply that $\widetilde{F}_{21}$ meets $\Theta_2$, that $\widetilde{F}_{4}$ and $\widetilde{F}_{16}$ meet $\Delta_1$, and that $\widetilde{F}_{15}$ meets both $\Delta_2$ and $\Theta_8$. All occur with multiplicity $1$ (in particular $\widetilde{F}_{21}$ is non-singular above $(\infty,\infty) \in \Zone$).

By \Cref{lemma:KW-dot-Fm} we have $K_W \cdot F_{15}^* \leq 0$ and $K_W \cdot F_{21}^* \leq 0$. We claim that the image $\Delta_1^{\mathsf{min}}$ of $\Delta_1^*$ on $\WNrmin{17}{1}$ satisfies $K_{\mathsf{min}} \cdot \Delta_1^{\mathsf{min}} \leq 0$. Granting the claim, we have an elliptic configuration of type $\mathrm{I}_{11}$ consisting of the images of the curves $F_{21}^*$, $\Theta_2^*$, ..., $\Theta_8^*$, $F_{15}^*$, $\Delta_2^*$, and $\Delta_1^*$, provided all the components are smooth. This configuration meets the $(-2)$-curve $\Delta_3$.

We now show that $K_{\mathsf{min}} \cdot \Delta_1^{\mathsf{min}} \leq 0$. Passing from $\ZNrtil{17}{1}$ to $\ZNro{17}{1}$ (see \Cref{constr:ZNro}) we blow down $\widetilde{F}_4$. Therefore $K_{Z^\circ} \cdot \Delta_1^\circ \leq 2$ and $\Ramo \cdot \Delta_1^\circ \geq 1$ and the same is true for the image of $\Delta_1$ under $\tilde{\tau}$. By the projection formula $K_W \cdot \Delta_1^* \leq \tfrac{1}{2} \cdot 2 \cdot (2 - 1) = 1$. Since $\Delta_1^*$ meets $F_{16}^*$ (which is blown-down in forming $\WNrmin{17}{1}$) its image $\Delta_1^{\mathsf{min}}$ on $\WNrmin{17}{1}$ has $K_{\mathsf{min}} \cdot \Delta_1^{\mathsf{min}} \leq 0$, as required.

The components of $\Delta$ and $\Theta$ are smooth, and the modular curve $\widetilde{F}_{15}$ is smooth by \Cref{lemma:some-FN-smooth}. As noted above, $\widetilde{F}_{21}$ is non-singular above $(\infty,\infty) \in \Zone$. By \Cref{lemma:extra-smooth} the curve $\widetilde{F}_{21}$ is non-singular, since there does not exist an integer $0 \leq a < 2 \cdot 21$ such that $4 \cdot 21^2 - a^2 \equiv 0 \pmod{17^2}$. None of the components of our putative elliptic configuration meet the exceptional divisor of $\ZNrtil{17}{1} \to \WNrmin{17}{1}$, except for $\Delta_1$ which meets only $\widetilde{F}_4$ and $\widetilde{F}_{16}$ (if another intersection occurred, the image of that component on $\WNrmin{17}{1}$ would be a $(-1)$-curve yielding a contradiction to \Cref{lemma:rat-crit}, cf. \Cref{lemma:ell-conf-bound-enough}).
\begin{figure}[H]
  \centering
  \begin{tikzpicture}[x=0.7cm,y=0.7cm,scale=0.9]
    \draw  (0,1)-- (3,0);
    \draw  (2,0)-- (5,1);
    \draw  (4,1)-- (7,0);
    \draw  (6,0)-- (9,1);
    \draw  (8,1)-- (11,0);
    \draw  (10,0)-- (13,1);
    \draw  (12,1)-- (15,0);
    \draw  (15,3)-- (6.5,5);
    \draw  (0,3)-- (8.5,5);
    \draw  (0.75,0)-- (0.75,4);
    \draw  (14.25,0)-- (14.25,4);
    \draw  (10.635480333137227,3.513291415833218)-- (11.208078667451094,5.946834336667151);
    \draw[dotted,thick]  (4.364519666862774,3.5132914158332116)-- (3.7919213325489056,5.946834336667152);
    \begin{scriptsize}
      \draw[color=black] (6,4.15) node {$-3$};
      \draw[color=black,anchor=south] (0.75,4) node {$F_{21}^*$};
      \draw[color=black,anchor=south] (14.25,4) node {$F_{15}^*$};
      \draw[color=black,anchor=south] (3.7919213325489056,5.946834336667152) node {$F_{16}^*$};
      \draw[color=black,anchor=east] (4,4.9) node {$-2$};
    \end{scriptsize}
  \end{tikzpicture}
  \caption{A configuration of curves on $\WNr{17}{1}$ whose image on $\WNrmin{17}{1}$ is an elliptic configuration of type $\mathrm{I}_{11}$ meeting a $(-2)$-curve. The unlabelled lines at the top of the picture are components of $E_{\infty,17,4}^*$ and those at the bottom are components of $E_{\infty,17,8}^*$. Lines without a recorded self-intersection are $(-2)$-curves. Dotted lines are blown down in forming $\WNrmin{17}{1}$.}
  \label{fig:17-1}
\end{figure}
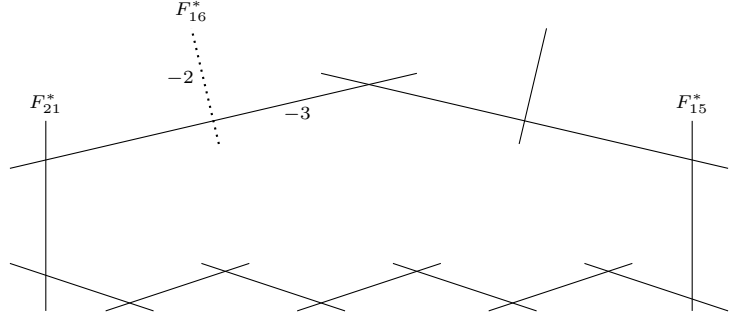


\begin{figure}[b]
  \centering
  \begin{tikzpicture}[x=0.75cm,y=0.7cm]
    \draw[dotted,thick] (0,3)-- (5.5,4);
    \draw (3.5,4)-- (9,3);
    \draw[dotted,thick] (2.8394427190999916,3.008065044950046)-- (2.3922291236000337,5.4677398201998155);
    \draw[dotted,thick] (0.75,3.6363636363636362)-- (0.75,0.25);
    \draw (8.25,0.25)-- (8.25,3.6363636363636416);
    \draw (9,1)-- (5.5,0);
    \draw (0,1)-- (3.5,0);
    \draw (2.0288570392693837,0.2060408459230335)-- (6.971142960730616,0.2060408459230332);
    \draw (1.8873605639486892,0.9807619738204119)-- (1.200557744205244,-1.4230478952816457);
    \begin{scriptsize}
      \draw[color=black,anchor=south] (8.25,3.6363636363636416) node {$F_{12}^*$};
      \draw[color=black,anchor=south] (0.75,3.6363636363636362) node {$F_{6}^*$};
      \draw[color=black,anchor=south] (2.3922291236000337,5.4677398201998155) node {$F_{23}^*$};
      \draw[color=black] (1.8,3.6) node {$-3$};
      \draw[color=black] (2.9,4.7) node {$-1$};
      \draw[color=black,anchor=east] (0.75,2) node {$-1$};
      \draw[color=black] (0.2,0.7) node {$-4$};
    \end{scriptsize}
  \end{tikzpicture}
  \caption{A configuration of curves on $\WNr{17}{3}$ whose image on $\WNrmin{17}{3}$ is an elliptic configuration of type $\mathrm{I}_{5}$ meeting a $(-2)$-curve. Unlabelled lines along the bottom are components of $E_{\infty,17,6}^*$ and those at the top are components of $E_{\infty,17,10}^*$. Lines without a recorded self-intersection are $(-2)$-curves. Dotted lines are blown down in forming $\WNrmin{17}{3}$.}
  \label{fig:17-3}
\end{figure}
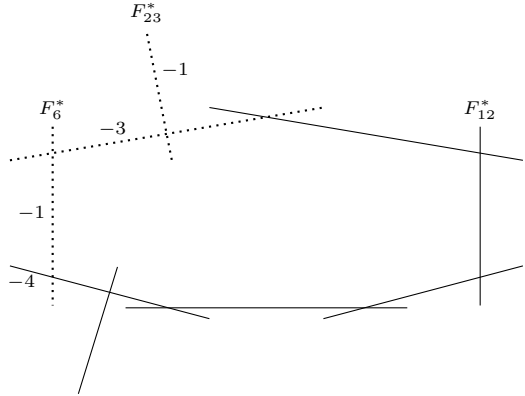

\case{The case $(N,r) = (17,3)$} \pdfbookmark[3]{The case (N,r)=(17,3)}{17-3}
This case was treated by Hermann~\cite[p.~178]{H_SMDDp2} and we repeat their elliptic configuration of type $\mathrm{I}_5$. We construct a surface $\WNrmin{17}{3}$ on which the image of the curve
\begin{equation*}
  \mathscr{C}_{17,3} = F_{12}^* + E_{\infty,17,6}^* + E_{\infty, 17,10}^*
\end{equation*}
contains an elliptic configuration of type $\mathrm{I}_5$ meeting a $(-2)$-curve.

For brevity write $\Gamma = E_{\infty,17,6}$, $\Delta = E_{\infty,17,10}$, and $\Theta = E_{\infty,17,11}$. Computing continued fractions we see that $\Gamma$ is a $(-3,-6)$-chain, $\Delta$ is a $(-2,-4,-2,-2)$-chain, and $\Theta$ is a $(-2,-3,-2,-2,-2,-2)$-chain. Let $\Gamma_i$, $\Delta_i$, and $\Theta_i$ denote the $i$-th components of $\Gamma$, $\Delta$, and $\Theta$ respectively. An argument with cusp widths of $X_0(m)$ and \Cref{lemma:j-mult-inf} (analogous to the case $(N,r) = (17,1)$ above) implies that $\widetilde{F}_3 \cdot \Gamma_2 = 1$, that $\widetilde{F}_6 \cdot \Gamma_1 = \widetilde{F}_6 \cdot \Delta_2 = 1$, that $\widetilde{F}_{11} \cdot \Theta_1 = 1$, that $\widetilde{F}_{12} \cdot \Gamma_2 = \widetilde{F}_{12} \cdot \Delta_4 = 1$, and that $\widetilde{F}_{23} \cdot \Gamma_1 = 1$.

First note that $F_6^*$, $F_{11}^*$, and $F_{23}^*$ are blown down in forming $\WNro{17}{3}$. In particular, the curves $\Gamma_1^{\oo}$ and $\Theta_1^{\oo}$ on $\WNro{17}{3}$ are $(-1)$-curves. We define $\WNrmin{17}{3}$ to be the surface obtained from $\WNro{17}{3}$ by blowing down $\Gamma_1^{\oo}$ and $\Theta_1^{\oo}$. Let $K_\minn$ be a canonical divisor for $\WNrmin{17}{3}$.

We claim that the curves $\Delta_2^{\minn}$, $\Delta_3^\minn$, $\Delta_4^\minn$, $F_{12}^\minn$, and $\Gamma_2^\minn$ form an elliptic configuration on $\WNrmin{17}{3}$ which meets the $(-2)$-curve $\Delta_1^\minn$. By \Cref{lemma:KW-dot-Fm} we have $K_W \cdot F_{12}^* \leq 0$ and clearly $K_\minn \cdot \Delta_3^\minn \leq 0$ and $K_\minn \cdot \Delta_4^\minn \leq 0$. It therefore suffices to show that $\Gamma_2^\minn$ and $\Delta_2^\minn$ intersect transversally and that $K_\minn \cdot \Gamma_2^\minn \leq 0$ and $K_\minn \cdot \Delta_2^\minn \leq 0$ (if this is the case then the components of our configuration are smooth by an analogous argument to the case $(N,r) = (17,1)$).

Recall that $\widetilde{F}_3$ is blown down in forming $\ZNro{17}{3}$, hence $K_{Z^\circ} \cdot \Gamma_2^\circ \leq 3$ and $\Ramo \cdot \Gamma_2^\circ \geq 1$ (and the same is true for its image under $\tau^\circ$). In particular $K_W \cdot \Gamma_2^* \leq 2$. In forming $\WNro{17}{3}$ we have blown down the image of the component of $E_{3,1}$ which meets $\widetilde{F}_3$. Thus $\KWo \cdot \Gamma_2^{\mathsf{small}} \leq 1$.

Also note that $\KWo \cdot \Delta_2^\oo \leq 1$ since $\Delta_2^*$ meets $F_6^*$. But then $\Gamma_1^\oo \cdot \Gamma_2^\oo = \Gamma_1^\oo \cdot \Delta_2^\oo = 1$, and therefore $\Gamma_2^\minn \cdot \Delta_2^\minn = 1$ and both $K_\minn \cdot \Gamma_2^\minn \leq 0$ and $K_\minn \cdot \Delta_2^\minn \leq 0$ as required.

\case{The case $(N,r) = (18,1)$} \pdfbookmark[3]{The case (N,r)=(18,1)}{18-1}
Let $g = \mymat{-7}{9}{3}{-4} \in \GL_2(\bbZ/18\bbZ)$. We claim that the image of the curve
\begin{equation*}
  \mathscr{C}_{18,1} = (\Fgplus{g})^* + F_{13}^* + F_{25}^* + E_{\infty,9,4}^* + E_{\infty,18,7}^* + E_{\infty,18,13}^*
\end{equation*}
on $\WNro{18}{1}$ contains an elliptic configuration of type $\mathrm{III}^*$ which meets a $(-2)$-curve. 

We first claim that $(\Fgplus{g})^*$ is a $(-2)$-curve. Note that $\Xgplus{g} \cong X(\ns 2) \times_{X(1)} X(\s 3)$ (which has index $24$, $4$ cusps, genus $1$, and LMFDB label \LMFDBLabelMC{6.24.1.a.1}). By construction $\Fgplustil{g}$ is smooth outside possibly those points above $j \in \{0, 1728, \infty\}$ on the diagonal $\DiagDiv \subset \Zone$. For each $j=0,1728,\infty$ let $G_j \subset \GL_2(\bbZ/18\bbZ)$ be as in \Cref{lemma:fix-1728,lemma:fix-0,lemma:fixed-cusps-even}. By direct calculation we see that the number of double cosets $G_j \, h \, G_j$ (where $h$ ranges over the $\GL_2(\bbZ/18\bbZ)$-conjugacy class of $g$) is equal to the number of points on $\Xgplus{g}$ above $j \in X(1)$. By the projection formula and \cite[Proposition~2.5]{KS_MDQS} none may be a multiple point of $\Fgplustil{g}$, which is therefore smooth.

The involution $\tilde{\tau}$ acts on $\Fgplustil{g} \cong X(\ns 2) \times_{X(1)} X(\s 3)$ diagonally on the components of the fibre product, fixing $4$ points. Hence $(\Fgplus{g})^*$ is a rational curve and $K_W \cdot (\Fgplus{g})^* \leq \frac{1}{2} (24/3 - 4 - 4) = 0$ (by the same argument as \Cref{lemma:flat-smooth,lemma:calF-3-dot-K}). In particular $(\Fgplus{g})^*$ is a $(-2)$-curve.

By \Cref{lemma:KW-dot-Fm} we have $K_W \cdot F_{13}^* \leq 0$ and $K_W \cdot F_{25}^* \leq 0$. An argument with cusp widths shows that $\widetilde{F}_{7}$ and $\widetilde{F}_{25}$ meet a component of $E_{\infty,18,7}$ of self-intersection $-3$, that $\widetilde{F}_{13}$ meets the first component of a $(-2,-2)$-chain in $E_{\infty,18,13}$, and that $\Fgplustil{g}$ meets the first component of a $(-2,-2,-2)$-chain in $E_{\infty,9,4}$.

Let $E/\bbC$ be the elliptic curve with CM by the order of discriminant $-11$ and let $\phi = 7(-1 + 3\varphi)\vert_{E[18]}$ where $\varphi = \frac{1 + \sqrt{-11}}{2}$. By \Cref{lemma:CM-pts} the curves $\Fgplustil{g}$ and $\widetilde{F}_{25}$ intersect at $\ffz = (E, E, \phi)$. In more detail, $-1 + 3\varphi$ is a cyclic isogeny of degree $25$ hence $\ffz$ lies on $\widetilde{F}_{25}$. By \Cref{lemma:CM-pts} there exists a choice of basis for $E[18]$ such that $\phi = g$, so $\ffz$ lies on $\Fgplustil{g}$. Since $j(E) \neq 0,1728,\infty$ this intersection is transversal as its image under the morphism $\ZNrtil{18}{1} \to \Zone$ is transversal by \cite[Section~2]{H_KADHMUK}.

We therefore have a chain of $(-2)$-curves of length $9$, consisting of two components of $E_{\infty,18,13}^\oo$, together with $F_{13}^\oo$, $E_{2,1}^\oo\!\!\mymat{3}{-8}{8}{3}$, $F_{25}^\oo$, $(\Fgplus{g})^\oo$ and three components of $E_{\infty,9,4}^\oo$. The claim that $\mathscr{C}_{18,1}^\oo$ contains an elliptic configuration of type $\mathrm{III}^*$ follows by noting that $F_{25}^\oo$, which is the central component of this chain, transversally meets a $(-2)$-curve in $E_{\infty,18,7}^\oo$. 
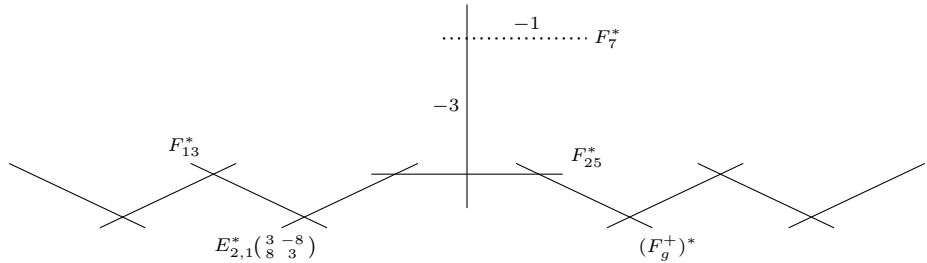
\begin{figure}[H]
  \centering
  \begin{tikzpicture}[x=0.7cm,y=1cm,scale=0.9]
    \draw (-2,0)-- (2,0);
    \draw (0,-0.5)-- (0,2.5);
    \draw[dotted,thick] (-0.5,2)-- (2.5,2);
    \draw (1.025658350974743,0.158113883008419)-- (3.871708245126285,-0.790569415042095);
    \draw (2.9230249470757714,-0.7905694150420967)-- (5.769074841227313,0.15811388300841744);
    \draw (9.563808033429362,0.15811388300841847)-- (6.717758139277819,-0.7905694150420952);
    \draw (7.666441437328333,-0.790569415042097)-- (4.820391543176791,0.1581138830084175);
    \draw (-1.025658350974743,0.15811388300841914)-- (-3.871708245126285,-0.7905694150420945);
    \draw (-2.9230249470757714,-0.7905694150420963)-- (-5.769074841227313,0.15811388300841814);
    \draw (-7.666441437328333,-0.7905694150420961)-- (-4.820391543176791,0.15811388300841808);
    \draw (-9.563808033429362,0.15811388300841964)-- (-6.717758139277819,-0.7905694150420944);
    \begin{scriptsize}
      \draw[color=black,anchor=east] (0,1) node {$-3$};
      \draw[color=black,anchor=south] (1.25,2) node {$-1$};
      \draw[color=black,anchor=south west] (2,0) node {$F_{25}^*$};
      \draw[color=black,anchor=west] (2.5,2) node {$F_7^*$};
      \draw[color=black,anchor=north] (4.2,-0.790569415042095) node {$(\Fgplus{g})^*$};
      \draw[color=black,anchor=north] (-4.2,-0.790569415042095) node {$E_{2,1}^*\!\!\mymat{3}{-8}{8}{3}$};
      \draw[color=black,anchor=south] (-5.9,0.15811388300841814) node {$F_{13}^*$};
    \end{scriptsize}
  \end{tikzpicture}
  \caption{A configuration of curves on $\WNr{18}{1}$ whose image on $\WNro{18}{1}$ contains an elliptic configuration of type $\mathrm{III}^*$ meeting (two) $(-2)$-curves. The lines on the left are components of $E_{\infty,18,13}^*$ and unlabelled lines on the right are components of $E_{\infty,9,4}^*$. The $(-3)$-curve meeting $F_7^*$ and $F_{25}^*$ is a component of $E_{\infty,18,7}^*$. Components without a recorded self-intersection are $(-2)$-curves. Dotted lines are blown down in forming $\WNro{18}{1}$.}
  \label{fig:18-1}
\end{figure}


\case{The case $(N,r) = (20,1)$} \pdfbookmark[3]{The case (N,r)=(20,1)}{20-1}
We claim that the curve
\begin{equation*}
  \mathscr{C}_{20,1} = F_{41}^* + F_{21}^* + E_{2,1}^*\!\!\mymat{5}{-4}{4}{5}
\end{equation*}
contains an elliptic configuration of type $\mathrm{I}_3$ which meets a $(-2)$-curve.

First note that the curves $\widetilde{F}_{21,1}$ and $\widetilde{F}_{41,1}$ meet the divisor $E_{\infty, 20, 1}$ transversally at the same points (the points where it meets $\widetilde{C}_{\infty,i}$). We do not claim that this intersection is transversal.

The curves $F_{41,1}^*$ and $F_{41,11}^*$ both meet the curve $E_{2,1}^*\!\!\mymat{5}{-4}{4}{5}$ (take $\varphi = 5 + 4i$ in \Cref{lemma:CM-pts}) and both meet the point $(E, E, \phi) \in \ZNr{20}{1}$ where $E$ has CM by the order of discriminant $-16$ and $\phi = \mymat{5}{-8}{2}{5}$ (take $\varphi = 5 + 2i$ in \Cref{lemma:CM-pts}). Considering the action of $\LamN{20}$ (see \Cref{sec:atkin-lehn-invol}) on the tangent space to $\WNr{20}{1}$ at the intersections, we see that the intersections are transversal.
\begin{figure}[H]
  \centering
  \begin{tikzpicture}[x=1cm,y=1cm,scale=1.1]
    \draw (-0.5,0)-- (2.5,0);
    \draw (0.4330127018922193,1.25)-- (-2.165063509461096,-0.25);
    \draw (-2.165063509461096,0.25)-- (0.4330127018922194,-1.25);
    \draw (0,-1.5)-- (0,1.5);
    \begin{scriptsize}
      \draw[anchor=west] (2.5,0) node {$F_{21,1}^*$};
      \draw[anchor=north east] (-2.165063509461096,-0.25) node {$E_{2,1}^*\!\!\mymat{5}{-4}{-4}{5}$};
      \draw[anchor=south east] (-2.165063509461096,0.25) node {$F_{41,11}^*$};
      \draw[anchor=south] (0,1.5) node {$F_{41,1}^*$};
      %
    \end{scriptsize}
  \end{tikzpicture}
  \caption{An elliptic configuration of type $\mathrm{I}_3$ on $\WNr{20}{1}$ which meets a $(-2)$-curve. All components are $(-2)$-curves. We do not claim that the intersection of $F_{21,1}^*$ and $F_{41,1}^*$ is transversal. The non-trivial element of $\LamN{20}$ acts on $\WNr{20}{1}$ by the reflection of the triangle fixing the point in the bottom right.}
  \label{fig:20-1}
\end{figure}
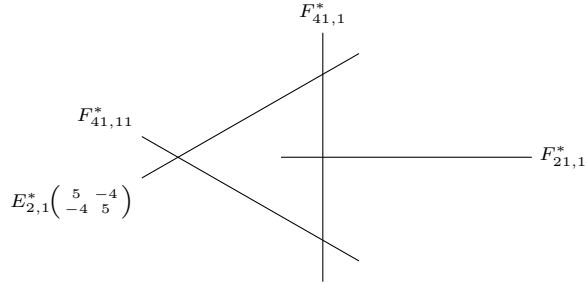


\case{The case $(N,r) = (20,3)$} \pdfbookmark[3]{The case (N,r)=(20,3)}{20-3}
  We claim that the image of the curve
  \begin{equation*}
    \mathscr{C}_{20,3} = F_{27}^* + E_{\infty,20,3}^*
  \end{equation*}
  on $\WNro{20}{3}$ contains an elliptic configuration of type either $\mathrm{I}_1$ or $\mathrm{I}_2$ meeting a $(-2)$-curve.

  Note that $E_{\infty,20,3}$ consists of a pair of $(-3,-7)$-chains, let $\Gamma$ be one such chain. The $(-7)$-component, $\Gamma_2$, of $\Gamma$ meets $\widetilde{F}_{27,1}$ transversally at two points. Note that $\Gamma_2$ also meets $\widetilde{F}_3$, $\widetilde{F}_{23}$, and $\widetilde{F}_{3 \circ (\borel, I)}^+$ (each of which is blown down in forming $\WNro{20}{3}$). In particular $\Gamma_2^\oo$ is a smooth $(-2)$-curve which meets $F_{27,1}^\oo$ at two points. If the points are not distinct then $F_{27,1}^\oo$ is an elliptic configuration of type $\mathrm{I}_1$, otherwise $\Gamma_2^\oo + F_{27,1}^\oo$ is an elliptic configuration of type $\mathrm{I}_2$. 

  Finally, the curve $\widetilde{F}_{27,1}$ meets the second chain $\Delta$ in $E_{\infty,20,3}$ at the $(-3)$-component $\Delta_1$. But $\Delta_1$ also meets $\widetilde{F}_{7}$, so $\Delta_1^\oo$ is a $(-2)$-curve.
\begin{figure}[H]
  \centering
  \begin{tikzpicture}[x=1cm,y=1cm,yscale=0.8]
    \draw[dotted,thick] (1.7353224048995965,-0.39466856765003716)-- (4.103333810799819,0.3946685676500365);
    \draw (2.6840057029501114,-0.3946685676500373)-- (0.3159942970498883,0.3946685676500372);
    \draw [shift={(-2.1,0)}]  plot[domain=-0.9336408172308062:0.9336408172308064,variable=\t]({1*2.9*cos(\t r)+0*2.9*sin(\t r)},{0*2.9*cos(\t r)+1*2.9*sin(\t r)});
    \draw [shift={(2.1,0)}]  plot[domain=2.2079518363589865:4.075233470820599,variable=\t]({1*2.9*cos(\t r)+0*2.9*sin(\t r)},{0*2.9*cos(\t r)+1*2.9*sin(\t r)});
    \draw[dotted,thick] (-0.12213151776324008,-1)-- (-2.12213151776324,-1);
    \draw[dotted,thick] (-0.12213151776323997,1)-- (-2.12213151776324,1);
    \draw[dotted,thick] (-0.12213151776323997,0)-- (-2.12213151776324,0);
    \begin{scriptsize}
      \draw[anchor=west] (4.103333810799819,0.3946685676500365) node {$F_{7,1}^*$};
      \draw[anchor=east] (-0.27640721803866214,2.4932788911710655) node {$F_{27,1}^*$};
      \draw[anchor=east] (-2.122,-1) node {$\left(F_{3 \circ (\borel, I)}^+\right)^*$};
      \draw[anchor=east] (-2.122,1) node {$F_{23,1}^*$};
      \draw[anchor=east] (-2.122,0) node {$E_{3,1}^* \!\! \mymat{1}{2}{-2}{-1}$};
      \draw[color=black] (0.6,-1.6) node {$-2$};
      \draw[color=black] (-0.6,-1.6) node {$-5$};
      \draw[color=black,anchor=south] (-1.4,-1) node {$-1$};
      \draw[color=black,anchor=south] (-1.4,1) node {$-1$};
      \draw[color=black,anchor=south] (-1.4,0) node {$-1$};
      \draw[color=black] (2.85,0.2) node {$-1$};
      \draw[color=black] (1.55,0.2) node {$-3$};
    \end{scriptsize}
  \end{tikzpicture}
  \caption{A configuration of curves on $\WNr{20}{3}$ whose image on $\WNro{20}{3}$ is an elliptic configuration of type $\mathrm{I}_{2}$ (or $\mathrm{I}_1$) which meets a $(-2)$-curve. The two unlabelled lines (the $(-5)$ and $(-3)$-curves) are components of different chains in $E_{\infty,20,3}^*$. Dotted lines are blown down in forming $\WNro{20}{3}$.}
  \label{fig:20-3}
\end{figure}
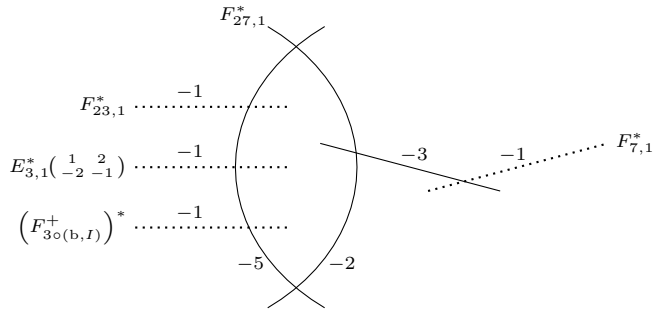


\case{The case $(N,r) = (21,2)$} \pdfbookmark[3]{The case (N,r)=(21,2)}{21-2}
We claim that there exists a blow-down $\WNrtiny{21}{2}$ of $\WNr{21}{2}$ on which the image of the curve
\begin{equation*}
  \mathscr{C}_{21,2} = F_{29}^* + F_{32}^* + E_{2,1}^*\!\mymat{1}{-8}{8}{1} + E_{\infty,21,8}^*
\end{equation*}
contains an elliptic configuration of type $\mathrm{I}_6$ meeting a $(-2)$-curve.

Let $E/\bbC$ be an elliptic curve with $j$-invariant $1728$ and consider the $29$-isogeny $\varphi = -5 + 2i \in \End(E)$. By \Cref{lemma:CM-pts} we have $(E, E, \phi) \in F_{29,1}$ where $\phi = 4 \varphi \vert_{E[21]} = \mymat{-5}{-2}{2}{-5} = \mymat{1}{-8}{8}{1}$. But $8 \mymat{1}{-8}{8}{1} = \mymat{8}{-1}{1}{8}$ and therefore $F_{29,1}$ and $F_{29,8}$ intersect at $(E, E, \phi)$.

By \Cref{lemma:extra-smooth} the curves $\widetilde{F}_{29,\lambda}$ are smooth except possibly at points above $(E, E, \phi) \in F_{29,\lambda}$ where $j(E) = 1728$. But there are two such points, so by the projection formula $\widetilde{F}_{29,\lambda}$ must be smooth on their strict transforms. Therefore $F_{29,\lambda}^*$ is smooth and meets $E_{2,1}^*\!\mymat{1}{-8}{8}{1}$ transversally. 

Note that $E_{\infty,21,8}$ consists of a pair of $(-3,-3,-3)$-chains swapped by the action of $\tilde{\tau}$. The first and third components of each chain meet separate components of $\widetilde{F}_{8}$, $\widetilde{F}_{29}$, and $\widetilde{F}_{32}$. The second components both meet $\Fgplustil{\three{\s}{I}}$. It follows that $E_{\infty,21,8}^\oo$ is a $(-2,-2,-2)$-chain. In particular the curves $F_{29,1}^\oo$, $E_{2,1}^\oo \!\mymat{1}{-8}{8}{1}$, $F_{29,8}^\oo$, and $E_{\infty,21,8}^\oo$ form an elliptic configuration of type $\mathrm{I}_6$.

Finally, by \Cref{lemma:KW-dot-Fm} we have $K_W \cdot F_{29,\lambda}^* \leq 0$ and $K_W \cdot F_{32,\lambda}^* \leq 1$. The components of $E_{\infty,21,11}^*$ which meet $F_{11}^*$ become exceptional on $\WNro{21}{2}$ (and also meet $F_{32}^*$). We define $\WNrtiny{21}{2}$ to be the surface obtained from $\WNro{21}{2}$ by blowing down these curves.
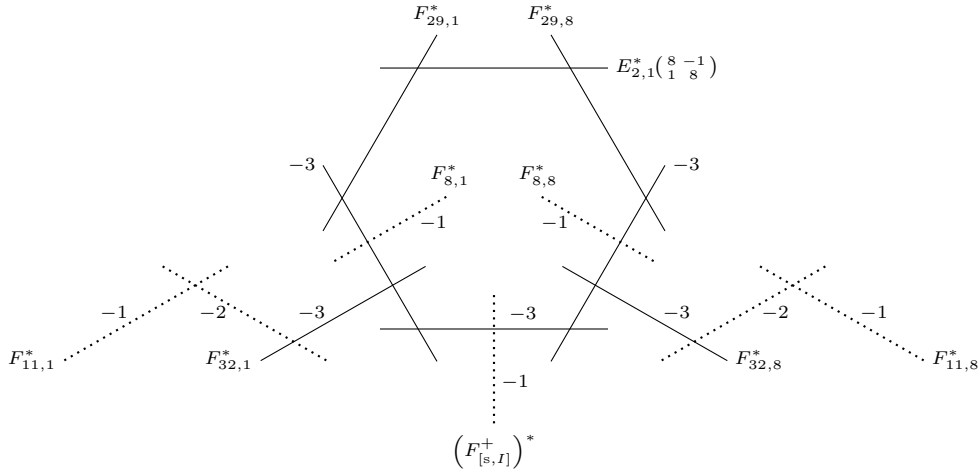
\begin{figure}[H]
  \centering
  \begin{tikzpicture}[x=1cm,y=1cm]
    \draw (-2.25,2.17)-- (-0.75,-0.43);
    \draw (-1.5,0)-- (1.5,0);
    \draw (0.75,-0.43)-- (2.25,2.17);
    \draw (2.25,1.3)-- (0.75,3.9);
    \draw (1.5,3.46)-- (-1.5,3.46);
    \draw (-0.75,3.9)-- (-2.25,1.3);
    \draw[dotted,thick] (-2.1,0.9)-- (-0.58,1.78);
    \draw (-0.9,0.83)-- (-3.07,-0.42);
    \draw[dotted,thick] (-4.36,0.83)-- (-2.2,-0.42);
    \draw[dotted,thick] (-3.5,0.83)-- (-5.66,-0.42);
    \draw[dotted,thick] (3.5,0.83)-- (5.66,-0.42);
    \draw[dotted,thick] (4.36,0.83)-- (2.2,-0.42);
    \draw (0.9,0.83)-- (3.07,-0.42);
    \draw[dotted,thick] (2.1,0.9)-- (0.58,1.78);
    \draw[dotted,thick] (0,0.45)-- (0,-1.3);
    \begin{scriptsize}
      \draw[anchor=west] (1.5,3.46) node {$E_{2,1}^*\!\!\mymat{8}{-1}{1}{8}$};
      \draw[anchor=south] (-0.75,3.9) node {$F_{29,1}^*$};
      \draw[anchor=south] (0.75,3.9) node {$F_{29,8}^*$};
      \draw[anchor=east] (-5.66,-0.42) node {$F_{11,1}^*$};
      \draw[anchor=west] (5.66,-0.42) node {$F_{11,8}^*$};
      \draw[anchor=west] (3.07,-0.42) node {$F_{32,8}^*$};
      \draw[anchor=east] (-3.07,-0.42) node {$F_{32,1}^*$};
      \draw[anchor=south] (-0.58,1.78) node {$F_{8,1}^*$};
      \draw[anchor=south] (0.58,1.78) node {$F_{8,8}^*$};
      \draw[anchor=north] (0,-1.3) node {$\left(\Fgplus{\three{\s}{I}}\right)^*$};
      \draw (-5,0.205) node {$-1$};
      \draw (5,0.205) node {$-1$};
      \draw (-3.7,0.205) node {$-2$};
      \draw (3.7,0.205) node {$-2$};
      \draw (2.4,0.205) node {$-3$};
      \draw (-2.4,0.205) node {$-3$};
      \draw[anchor=west] (0,-0.7) node {$-1$};
      \draw[anchor=south west] (0.1,0) node {$-3$};
      \draw (-0.8,1.4) node {$-1$};
      \draw (0.8,1.4) node {$-1$};
      \draw[anchor=west] (2.25,2.17) node {$-3$};
      \draw[anchor=east] (-2.25,2.17) node {$-3$};
    \end{scriptsize}
  \end{tikzpicture}
\caption{A configuration of curves on $\WNr{21}{2}$ whose image on $\WNrtiny{21}{2}$ contains an elliptic configuration of type $\mathrm{I}_6$ which meets a $(-2)$-curve. The unlabelled lines in the hexagon are components of $E_{\infty,21,8}^*$, and the lines meeting $F_{11,\lambda}^*$ are components of $E_{\infty,21,11}^*$. Curves without a recorded self-intersection are $(-2)$-curves. Dotted lines are blown down in forming $\WNrtiny{21}{2}$. The non-trivial element of $\LamN{21}$ acts via reflection in the vertical axis.}
\label{fig:21-2}
\end{figure}

\begin{proof}[Proof of \Cref{thm:WNr-KD}(ii)]
  That the curves $F_{m,\lambda}^*$ are smooth follows from \Cref{lemma:extra-smooth} similarly to the cases described when $(N,r) = (16,5)$ and $(17,1)$ (except for the case $F_{29,\lambda}^*$ when $(N,r)=(21,2)$ discussed above). We omit the details in the remaining cases (when $N$ is prime we have \cite[Hilfsatz~3]{H_SMDDp2}), which can be checked with the code at \cite{ME_ELECTRONIC_HERE}. In each case we proved the curves $(\Fgplus{g})^*$ appearing in \Crefrange{fig:16-5}{fig:21-2} are smooth.

  It follows that each configuration in \Crefrange{fig:16-5}{fig:21-2} is an elliptic configuration meeting a $(-2)$-curve (we showed that $K \cdot C \leq 0$ for each component $C$ and therefore equality holds, else we obtain a contradiction to \Cref{lemma:rat-crit}). It follows from \Cref{prop:ell-cong-K3} that $\WNr{N}{r}$ is a blown-up elliptic K3 surface in each case.
\end{proof}

\subsection{The properly elliptic cases}
\label{sec:prop-ellipt-cases}
Similarly to \Cref{sec:elliptic-K3} we now exhibit (pseudo-)elliptic configurations on (surfaces birational to) those $\WNr{N}{r}$ appearing in \Cref{thm:WNr-KD}(iii). Part (iii) of \Cref{thm:WNr-KD} then follows from \Cref{prop:pseudo-ell} together with the geometric genera computed in \Cref{thm:geom-genus} (see also \Cref{thm:p_g} and \Cref{tab:invariants-WNr1}). 

\case{The case $(N,r) = (15,7)$} \pdfbookmark[3]{The case (N,r)=(15,7)}{15-7} 
Define $\WNrmin{15}{7} = \WNro{15}{7}$. We claim that the curve
\begin{equation*}
  \mathscr{C}_{15,7} = (\Fgplus{\three{I}{\ns}})^* + (F_{7 \circ \three{\ns}{I}}^+)^*
\end{equation*}
on $\WNr{15}{7}$ is a pseudo-elliptic configuration of type $\mathrm{I}_{2}$.

Note that $\Xgplus{\three{I}{{\ns}}} \cong X(\ns 5)$ (which has index 20 and 4 cusps), and using the argument in \Cref{lemma:flat-smooth,lemma:calF-3-dot-K} it follows that $(\Fgplus{\three{I}{{\ns}}})^*$ is a smooth rational curve with $K_{W} \cdot (\Fgplus{\three{I}{\ns}})^* \leq 0$.

We consider the CM points on $F_{7 \circ \three{\ns}{I}}^+$ of discriminants $-7$ and $-28$. Let $E/\bbC$ be the elliptic curve with CM by the order of discriminant $D = -7$ or $D = -28$ and consider the $7$-isogeny $\varphi = \sqrt{-7} \in \End(E)$. For each discriminant there are $6$ points on $F_{7 \circ \three{\ns}{I}}^+$ above $( j(E), j(E) ) \in \Zone$, namely the points $(E, E, \phi) \in \ZNr{15}{7}$ where $\phi = g \, \varphi\vert_{E[15]}$ and $g$ ranges over the $\GL_2(\bbZ/15\bbZ)$-conjugacy class of $\three{{g_{\ns}}}{g_I}$. By \Cref{lemma:CM-pts} we may assume that $\varphi\vert_{E[15]} = \mymat{-1}{4}{-2}{1}$ if $D = -7$, and that $\varphi\vert_{E[15]} = \mymat{0}{-7}{1}{0}$ if $D = -28$. By direct calculation we see that in each case there exist $4$ choices of $g$ in the conjugacy class of $\three{{g_{\ns}}}{g_I}$ for which $\phi = g \, \varphi\vert_{E[15]}$ is conjugate to $\three{{g_{\ns}}}{g_{\ns}}$, and $2$ choices of $g$ such that $\phi$ is conjugate to $\three{{g_I}}{g_{\ns}}$. Since $g_{\weyl} = \three{{g_{\ns}}}{g_{\ns}}$, this implies that $\Ramtil \cdot \widetilde{F}_{7 \circ \three{\ns}{I}}^+ =  \Fgplustil{\weyl} \cdot \widetilde{F}_{7 \circ \three{\ns}{I}}^+ \geq 8$ and that $F_{7 \circ \three{\ns}{I}}^+$ and $\Fgplus{\three{I}{\ns}}$ meet at $4$ distinct points. In particular $(\Fgplus{\three{I}{\ns}})^*$ and $(F_{7 \circ \three{\ns}{I}}^+)^*$ intersect at $2$ points.

To show that $(\Fgplus{\three{I}{\ns}})^*$ and $(F_{7 \circ \three{\ns}{I}}^+)^*$ form a pseudo-elliptic configuration it remains to check that $(F_{7 \circ \three{\ns}{I}}^+)^*$ is a rational curve for which $K_W \cdot (F_{7 \circ \three{\ns}{I}}^+)^* \leq 0$. By \Cref{lemma:Fmg-bir} the curve $F_{7 \circ \three{\ns}{I}}^+$ is birational to the modular curve $X(\ns 3) \times_{X(1)} X_0(7)$ (LMFDB label \LMFDBLabelMC{21.48.3.a.1}). Using the code accompanying \cite{RSZB_LAIOGFECOQ} we check that $X_{7 \circ \three{\ns}{I}}^+$ has genus $3$ and has cusps with widths $3$, $3$, $21$, and $21$. Moreover, $\widetilde{F}_{7 \circ \three{\ns}{I}}^+$ meets $D_{\infty}$ only in the components $E_{\infty, 5, 2}$ and $E_{\infty, 5, 3}$. Note that $E_{\infty, 5, 2}$ consists of two disjoint $(-3,-2)$-chains whose components have multiplicities $(6, 3)$ in $\tilde{\ffj}^*(\infty)$ and $(3, 9)$ in $(\tilde{\ffj}')^*(\infty)$ by \Cref{lemma:j-mult-inf}. From the projection formula $\widetilde{F}_{7 \circ \three{\ns}{I}}^+$ meets $D_{\infty}$ at the intersection points of $\widetilde{C}_{\infty,1}$ and $E_{\infty, 5, 2}$. Since $\tilde{\tau}(E_{\infty, 5, 2}) = E_{\infty, 5, 3}$ a symmetric argument shows that the intersections with $E_{\infty, 5, 3}$ occur where it meets $\widetilde{C}_{\infty,2}$. Applying \eqref{eq:canonical-Z} we have 
\begin{equation*}
  K_W \cdot (F_{7 \circ \three{\ns}{I}}^+)^* \leq \frac{1}{2} \bigg( K_{\tilde{Z}} \cdot \widetilde{F}_{7 \circ \three{\ns}{I}}^+ - \Ramtil \cdot \widetilde{F}_{7 \circ \three{\ns}{I}}^+ \bigg) \leq \frac{1}{2}(48/3 - 8 - 8) = 0.
\end{equation*}
Since $\tilde{\tau}$ fixes at least $8$ points on $\widetilde{F}_{7 \circ \three{\ns}{I}}^+$ the quotient $(F_{7 \circ \three{\ns}{I}}^+)^*$ is a rational curve (by the Riemann--Hurwitz formula).

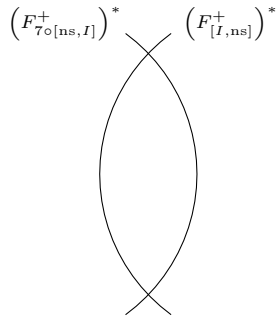
\begin{figure}[H]
  \centering
  \begin{tikzpicture}[x=1cm,y=1cm,yscale=0.8,xscale=0.8]
    \draw [shift={(-2.1,0)}]  plot[domain=-0.9336408172308062:0.9336408172308065,variable=\t]({1*2.9*cos(\t r)+0*2.9*sin(\t r)},{0*2.9*cos(\t r)+1*2.9*sin(\t r)});
    \draw [shift={(2.1,0)}]  plot[domain=2.207951836358987:4.075233470820599,variable=\t]({1*2.9*cos(\t r)+0*2.9*sin(\t r)},{0*2.9*cos(\t r)+1*2.9*sin(\t r)});
    \begin{scriptsize}
      \draw[anchor=east] (-0.3065983510131155,2.4945390669049314) node {$\left( F_{7 \circ \three{{\ns}}{{I}}}^+ \right)^*$};
      \draw[anchor=west] (0.4420704858559453,2.4945390669049314) node {$\left(\Fgplus{\three{I}{\ns}}\right)^*$};
    \end{scriptsize}
  \end{tikzpicture}
  \caption{A pseudo-elliptic configuration on the surface $\WNr{15}{7}$ of type $\mathrm{I}_{2}$. In particular, both curves satisfy $K_W \cdot C \leq 0$. We do not claim that the components are smooth nor that the intersections are transversal.}
  \label{fig:15-7}
\end{figure}


\case{The case $(N,r) = (19,1)$} \pdfbookmark[3]{The case (N,r)=(19,1)}{19-1}
This case was treated by Hermann~\cite[p.~178]{H_SMDDp2}. We repeat their proof that there exists a blow-down $\WNrmin{19}{1}$ of $\WNr{19}{1}$ such that image of the curve
\begin{equation*}
  \mathscr{C}_{19,1} = F_{49}^* + E_{3,1}^*\!\mymat{2}{7}{-7}{-5} + E_{\infty,19,7}^*
\end{equation*}
contains a pseudo-elliptic configuration of type $\mathrm{I}_4$.

Let $E/\bbC$ be an elliptic curve with $j(E) = 0$. Note that $3 + \zeta_3$ and $5 + 8\zeta_3 \in \End(E)$ are cyclic isogenies of degree $7$ and $49$ respectively. In particular taking $ \phi = 7 \mymat{3}{1}{-1}{2} = \mymat{2}{7}{-7}{-5}$ we have $(E,E,\phi) \in F_{7} \subset \ZNr{19}{1}$ by \Cref{lemma:CM-pts}. But $8 \mymat{5}{8}{-8}{-3}$ is equal to $\mymat{2}{7}{-7}{-5}$, and therefore $(E,E,\phi) \in F_{49}$. It follows that both $F_7^*$ and $F_{49}^*$ meet $E_{3,1}^* \!\! \mymat{2}{7}{-7}{-5}$. Since $F_7^*$ is blown-down in forming $\WNro{19}{1}$ we have $K_\oo \cdot E_{3,1}^\oo \!\! \mymat{2}{7}{-7}{-5} \leq 0$.

Finally note that $E_{\infty,19,7}$ is therefore a $(-3,-4,-2)$-chain. The first component meets $\widetilde{F}_7$, the second meets $\widetilde{F}_6$, and the third meets $\widetilde{F}_{11}$ and $\widetilde{F}_{49}$. The image $E_{\infty,19,7}^\oo$ is a $(-2,-3,-1)$-chain and we define $\WNrmin{19}{1}$ to be the surface obtained from $\WNro{19}{1}$ by blowing down the third component. The claim follows since, by \Cref{lemma:KW-dot-Fm} we have $K_W \cdot F_{49}^* \leq 1$.

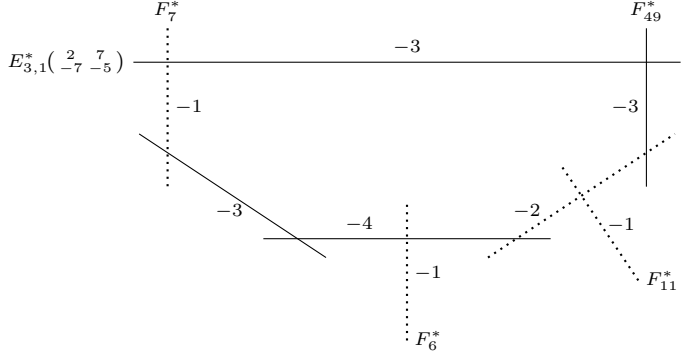
\begin{figure}[H]
  \centering
  \begin{tikzpicture}[x=1cm,y=1cm,scale=0.9]
    \draw (0,3)-- (8,3);
    \draw[dotted,thick] (0.5,3.5)-- (0.5,1.166666666666667);
    \draw (0.08397485283107821,1.9440167647792812)-- (2.81602514716892,0.12264990188738678);
    \draw (1.9,0.4)-- (6.1,0.4);
    \draw[dotted,thick] (5.183974852831078,0.12264990188738523)-- (7.916025147168923,1.9440167647792816);
    \draw (7.5,3.5)-- (7.5,1.1666666666666674);
    \draw[dotted,thick] (4,0.9)-- (4,-1.1);
    \draw[dotted,thick] (6.272649901887384,1.4493584805022572)-- (7.382050294337844,-0.2147421081734331);
    \begin{scriptsize}
      \draw[color=black,anchor=east] (0,3) node {{$E_{3,1}^* \!\! \mymat{2}{7}{-7}{-5}$}};
      \draw[color=black,anchor=west] (8,3) node {\hphantom{$E_{3,1}^* \!\! \mymat{2}{7}{-7}{-5}$}};
      \draw[color=black,anchor=south] (4,3) node {$-3$};
      \draw[color=black,anchor=south] (0.5,3.5) node {$F_{7}^*$};
      \draw[color=black,anchor=west] (0.5,2.3333333333) node {$-1$};
      \draw[color=black] (1.4,0.8) node {$-3$};
      \draw[color=black,anchor=south] (3.3,0.4) node {$-4$};
      \draw[color=black] (5.75,0.8) node {$-2$};
      \draw[color=black,anchor=south] (7.5,3.5) node {$F_{49}^*$};
      \draw[color=black,anchor=east] (7.5,2.3333333333) node {$-3$};
      \draw[color=black,anchor=west] (4,-1.1) node {$F_6^*$};
      \draw[color=black,anchor=west] (4,-0.1) node {$-1$};
      \draw[color=black,anchor=west] (7.382050294337844,-0.2147421081734331) node {$F_{11}^*$};
      \draw[color=black,anchor=west] (6.82,0.6) node {$-1$};
    \end{scriptsize}
  \end{tikzpicture}
  \caption{A configuration of curves on $\WNr{19}{1}$ whose image on $\WNrmin{19}{1}$ is a pseudo-elliptic configuration of type $\mathrm{I}_{4}$. Unlabelled lines are components of $E_{\infty,19,7}^*$. Dotted lines are blown down in forming $\WNrmin{19}{1}$. For each component we display the value $-(2 + K_W \cdot C)$ (which is equal to $C^2$ if $C$ is smooth). We do not claim that the components are smooth nor that the intersections are transversal.}
  \label{fig:19-1}
\end{figure}


\case{The case $(N,r) = (19,2)$} \pdfbookmark[3]{The case (N,r)=(19,2)}{19-2}
This is the final case treated by Hermann~\cite[p.~178]{H_SMDDp2}. We define $\WNrmin{19}{2} = \WNro{19}{2}$ and repeat Hermann's proof that the curve
\begin{align*}
  \mathscr{C}_{19,2} =& \, F_{10}^* + E_{\infty,19,10}^* + F_{29}^* + E_{2,1}^*\!\! \mymat{4}{9}{-9}{4} + F_{13}^* + E_{\infty,19,13}^* \\
  & + F_{14}^* +  E_{\infty,19,14}^* + F_{18}^* + F_{41}^* + E_{2,1}^* \!\! \mymat{6}{-2}{2}{6} 
\end{align*}
on $\WNr{19}{2}$ is a pseudo-elliptic configuration of type $\mathrm{I}_{13}$.

By \Cref{lemma:KW-dot-Fm} we have $K_W \cdot F_m^* \leq 0$ for each curve $F_m^*$ in the support of $\mathscr{C}_{19,2}$. Note that $\widetilde{F}_{10}$ and $\widetilde{F}_{29}$ meet the same component of $E_{\infty,19,10}$, that $\widetilde{F}_{13}$ and $\widetilde{F}_{14}$ are joined by two components of $E_{\infty,19,13}$, and that $\widetilde{F}_{14}$ and $\widetilde{F}_{18}$ are joined by two components of $E_{\infty,19,14}$. Each of these linking components has self-intersection $-2$.

If $E/\bbC$ is an elliptic curve with $j(E) = 1728$ then by \Cref{lemma:CM-pts} considering the isogenies $5(3 + 2i)$ and $2(2 + 5i)$ shows that $\widetilde{F}_{13}$ and $\widetilde{F}_{29}$ meet $E_{2,1}^*\!\! \mymat{4}{9}{-9}{4}$, and considering $2(3 + i)$ and $11(4 + 5i)$ shows that $\widetilde{F}_{10}$ and $\widetilde{F}_{41}$ meet $E_{2,1}^* \!\! \mymat{6}{-2}{2}{6}$. 

It remains to show that that on $\WNr{19}{2}$ the curves $F_{18}^*$ and $F_{41}^*$ meet. Let $E/\bbC$ be the elliptic curve with CM by $\bbZ[\sqrt{-2}]$. Let $\psi_1 = 8(3 + 4\sqrt{-2})$ and $\psi_2 = 6(4 - \sqrt{-2})$. Then $(E, E, \phi_1) \in F_{41}$ and $(E, E, \phi_2) \in F_{18}$ where $\phi_i = (\psi_i) \vert_{E[19]}$ for each $i = 1,2$. But we have $\phi_1 \phi_2^{-1} = (\psi_1 \widehat{\psi_2}) \vert_{E[19]} = 48(4 + 19\sqrt{2}) \vert_{E[19]} = 1$ (here $\widehat{\psi_2}$ denotes the dual of $\psi_2$) and therefore $F_{18}$ and $F_{41}$ intersect at $(E, E, \phi_1)$.
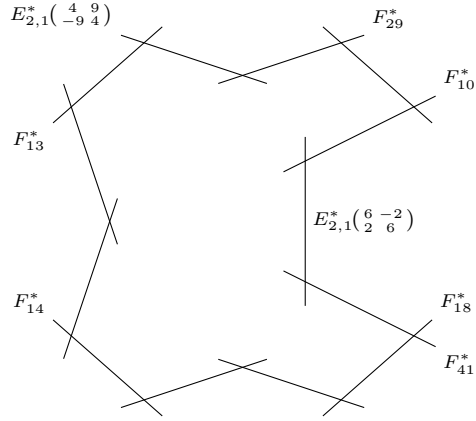
\begin{figure}[H]
  \centering
  \begin{tikzpicture}[x=0.8cm,y=0.8cm,scale=0.8]
    \draw (0,0)-- (3,1);
    \draw (2,1)-- (5,0);
    \draw (0.8484431129141767,-0.17362109361624622)-- (-1.3961656704193426,1.816788766131745);
    \draw (4.151556887085826,-0.173621093616239)-- (6.396165670419347,1.8167887661317534);
    \draw (-0.07338090847990841,3.382420385608106)-- (-1.180178089538841,6.702811928784904);
    \draw (-1.1801780895388412,1.010712140481822)-- (-0.07338090847990841,4.33110368365862);
    \draw (0.8484431129141767,7.887145162882973)-- (-1.3961656704193426,5.896735303134982);
    \draw (0,7.713524069266727)-- (3,6.713524069266727);
    \draw (2,6.713524069266727)-- (5,7.713524069266727);
    \draw (4.151556887085826,7.887145162882966)-- (6.396165670419347,5.896735303134974);
    \draw (3.340991719174924,4.882420385608112)-- (6.468909621089696,6.4528119287848895);
    \draw (3.340991719174917,2.8311036836586183)-- (6.468909621089651,1.2607121404818336);
    \draw (3.7878371337341736,5.6067620346333635)-- (3.787837133734163,2.106762034633363);
    \begin{scriptsize}
      \draw[color=black, anchor=south east] (-1.39,1.81) node {$F_{14}^*$};
      \draw[color=black, anchor=south west] (6.396165670419347,1.8167887661317534) node {$F_{18}^*$};
      \draw[color=black, anchor=north east] (-1.3961656704193426,5.896735303134982) node {$F_{13}^*$};
      \draw[color=black, anchor=south east] (0,7.71) node {$E_{2,1}^*\!\! \mymat{4}{9}{-9}{4}$};
      \draw[color=black, anchor=south west] (5,7.71) node {$F_{29}^*$};
      \draw[color=black, anchor=south west] (6.468909621089696,6.4528119287848895) node {$F_{10}^*$};
      \draw[color=black, anchor=north west] (6.468909621089696,1.26) node {$F_{41}^*$};
      \draw[color=black,anchor=west] (3.78,3.865) node {$E_{2,1}^* \!\! \mymat{6}{-2}{2}{6}$};
    \end{scriptsize}
  \end{tikzpicture}
  \caption{A pseudo-elliptic configuration of type $\mathrm{I}_{13}$ on $\WNr{19}{2}$. All components have $K_W \cdot C \leq 0$. The lines along the bottom are components of $E_{\infty,19,14}^*$, the unlabelled lines on the left are components of $E_{\infty,19,13}^*$, and the unlabelled line on the top right is a component of $E_{\infty,19,10}^*$. We do not claim that the components are smooth nor that the intersections are transversal.}
  \label{fig:19-2}
\end{figure}


\case{The case $(N,r) = (21,5)$} \pdfbookmark[3]{The case (N,r)=(21,5)}{21-5}
We claim that the image of the curve
\begin{equation*}
  \mathscr{C}_{21,5} = F_{17,1}^* + F_{20,1}^* + F_{41,1}^* + E_{2,1}^*\!\!\mymat{8}{-2}{2}{8} + E_{\infty,21,5}^*
\end{equation*}
on $\WNro{21}{5}$ contains a pseudo-elliptic configuration of type $\mathrm{I}_8$.

First note that $\widetilde{F}_{20,1}$ and $\widetilde{F}_{41,1}$ meet $E_{\infty,21,20}$ at the same component. Moreover, since the components of $E_{\infty,21,20}$ are blown-down in forming $\WNro{21}{5}$ from \Cref{lemma:KW-dot-Fm} we have $\KWo \cdot F_{20,1}^\oo \leq 0$ and $\KWo \cdot F_{41,1}^\oo \leq 0$. Similarly $\KWo \cdot F_{17,1}^\oo \leq 0$. Computing the Hirzebruch--Jung continued fraction of $21/5$ one sees that $E_{\infty,21,5}$ is a $(-5,-2,-2,-2,-2)$-chain, and it meets $\widetilde{F}_{20,1}$ (resp. $\widetilde{F}_{17,1}$) at the second (resp. last) component.

The curves $\widetilde{F}_{17,1}$ and $\widetilde{F}_{41,1}$ both meet the curve $E_{2,1}^*\!\!\mymat{8}{-2}{2}{8}$. To see this, consider $2(4 + i)$ and $10(5 + 4i)$ and apply \Cref{lemma:CM-pts} to the elliptic curve $E/\bbC$ with $j(E) = 1728$. The claim follows.
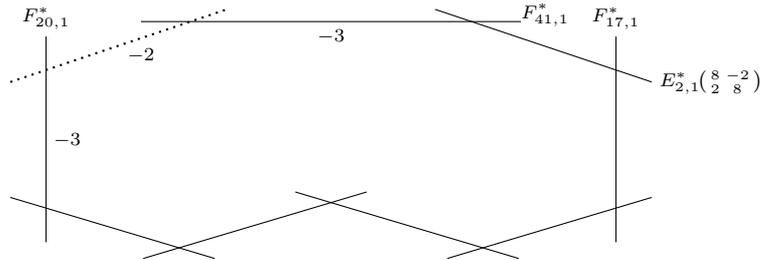
\begin{figure}[H]
  \centering
  \begin{tikzpicture}[x=1.0cm,y=1.0cm,yscale=0.9]
    \draw (2,3.5)-- (7,3.5);
    \draw (0.75,0.25)-- (0.75,3.28);
    \draw (0.28,0.91)-- (2.97,0.01);
    \draw[dotted,thick] (0.28,2.61)-- (3.13,3.68);
    \draw (5.87,3.68)-- (8.72,2.61);
    \draw (8.25,3.28)-- (8.25,0.25);
    \draw (8.72,0.91)-- (6.03,0.01);
    \draw (6.97,0.01)-- (4.03,0.99);
    \draw (4.97,0.99)-- (2.03,0.01);
    \begin{scriptsize}
      \draw[color=black,anchor=west] (6.9,3.6) node {$F_{41,1}^*$};
      \draw[color=black,anchor=south] (0.75,3.28) node {$F_{20,1}^*$};
      \draw[color=black] (2,3) node {$-2$};
      \draw[color=black,anchor=west] (0.75,1.75) node {$-3$};
      \draw[color=black,anchor=north] (4.5,3.5) node {$-3$};
      \draw[color=black,anchor=east] (0.28,2.61) node {\hphantom{$E_{2,1}\!\!\mymat{8}{-2}{2}{8}$}};
      \draw[color=black,anchor=west] (8.72,2.61) node {$E_{2,1}^*\!\!\mymat{8}{-2}{2}{8}$};
      \draw[color=black,anchor=south] (8.25,3.28) node {$F_{17,1}^*$};
    \end{scriptsize}
  \end{tikzpicture}
  \caption{A configuration of curves on $\WNr{21}{5}$ whose image on $\WNro{21}{5}$ is a pseudo-elliptic configuration of type $\mathrm{I}_{8}$. The line in the top left is a component of $E_{\infty,21,20}^*$ and the unlabelled lines at the bottom are components of $E_{\infty,21,5}^*$. We record the value $-(2 + K_W \cdot C)$ (which is equal to $C^2$ if $C$ is smooth) on each component $C$ (when it is not recorded it is equal to $-2$). The dotted component of $E_{\infty,21,20}^*$ is blown down in forming $\WNro{21}{5}$. We do not claim that the components are smooth nor that the intersections are transversal.}
  \label{fig:21-5}
\end{figure}


\case{The case $(N,r) = (22,1)$} \pdfbookmark[3]{The case (N,r)=(22,1)}{22-1}
We claim that the image of the curve
\begin{equation*}
  \mathscr{C}_{22,1} = ( F_{23 \circ (\borel, I)}^+ )^* + F_{27}^*
\end{equation*}
on $\WNro{22}{1}$ is a pseudo-elliptic configuration of type $\mathrm{I}_2$.

First, by \Cref{lemma:KW-dot-Fm} we have $K_W \cdot F_{27}^* \leq 0$. We now prove that $\KWo \cdot (F_{23 \circ (\borel, I)}^+)^\oo \leq 0$, before showing that $( F_{23 \circ (\borel, I)}^+ )^*$ and $F_{27}^*$ meet at two distinct points.

By \Cref{lemma:Fmg-bir} the curve $\widetilde{F}_{23 \circ (\borel,I)}^+$ is birational to $X_0(46)$ and $\tilde{\tau}$ acts as the Atkin--Lehner involution $w_{23}$. We have $p_g(X_0(46)) = 5$ and the involution $w_{23}$ has $h(-4 \cdot 23) + 3h(-23) = 12$ fixed points (see e.g.,~\cite{K_OTNOG0N}). Applying the Riemann--Hurwitz formula, we see that $(F_{23 \circ (\borel, I)}^+)^*$ is a rational curve. Also note that $\widetilde{F}_{23 \circ (\borel, I)}^+$ meets the locus $D_\infty$ only in a neighbourhood of the divisors $E_{\infty, 22, 1}$ and $E_{\infty, 11, 1}$ on $\ZNrtil{22}{1}$. The curve $E_{\infty,22,1}$ is a $(-22)$-curve which by \Cref{lemma:j-mult-inf} has multiplicity $1$ in $\tilde{\ffj}^*(\infty)$ and $(\tilde{\ffj}')^*(\infty)$. Similarly $E_{\infty, 11, 1}$ is a $(-11)$-curve which has multiplicity $2$ in $\tilde{\ffj}^*(\infty)$ and $(\tilde{\ffj}')^*(\infty)$.

The widths of the cusps on $X_0(46)$ are equal to $46$, $23$, $2$, and $1$. The action of $w_{23}$ partitions the cusps into two orbits, the pair with widths $46$ and $2$, and the pair with widths $23$ and $1$. The image of the former pair on $\widetilde{F}_{23 \circ (\borel, I)}^+$ lies on $E_{\infty,11,1}$ and the image of the latter lies on $E_{\infty,22,1}$. It follows that in a neighbourhood of the cusp of width $46$ the curve $\widetilde{F}_{23 \circ (\borel, I)}^+$ meets $E_{\infty,11,1}$ transversally and meets $\widetilde{C}_{\infty,1}$ with multiplicity $2$. Similarly in a neighbourhood of the cusp of width $23$ the curve $\widetilde{F}_{23 \circ (\borel, I)}^+$ meets both $E_{\infty,11,1}$ and $\widetilde{C}_{\infty,1}$ transversally. A symmetrical result is true for the other pair of cusps. Therefore  $D_{\infty} \cdot \widetilde{F}_{23 \circ (\borel, I)}^+ = 10$. Moreover, it follows that $\widetilde{F}_{23 \circ (\borel, I)}^+$ and $\widetilde{F}_{23}$ intersect $E_{\infty, 22, 1}$ at the same point (the point at which it meets $\widetilde{C}_{\infty,1}$). By \eqref{eq:canonical-Z} we have
\begin{equation*}
  K_W \cdot (F_{23 \circ (\borel, I)}^+)^* \leq \frac{1}{2} \big( K_{\tilde{Z}} \cdot \widetilde{F}_{23 \circ (\borel, I)}^+ - \Ramtil \cdot \widetilde{F}_{23 \circ (\borel, I)}^+ \big) \leq \frac{1}{2} \left( 72/3 - 10 - 12 \right) = 1
\end{equation*}
and therefore $\KWo \cdot (F_{23 \circ (\borel, I)}^+)^\oo \leq 0$ (since $F_{23}^*$ is blown-down in forming $\WNro{22}{1}$).

Let $E/\bbQ(i)$ and $E'/\bbQ(i)$ be the geometrically non-isomorphic elliptic curves with CM by $\bbZ[\sqrt{-5}]$ (this is possible since the Hilbert class field of $\bbQ(\sqrt{-5})$ is equal to $\bbQ(i, \sqrt{-5})$). Consider the ideals $\ffa = (23, 8 + \sqrt{-5})$ and $\ffb = (27, 13 + 2\sqrt{-5})$ of  $\bbZ[\sqrt{-5}]$. Let $\bar{\ffa}$ denote the complex conjugate of $\ffa$, and note that $\bar{\ffa} \ffb$ is a principal ideal generated by $4 - 11 \sqrt{-5}$. Let $\psi_1, \psi_2 \colon E \to E'$ denote the isogenies with kernels annihilated by $\ffa$ and $\ffb$ respectively.

Let $\phi_1 = (g_{\borel}, g_I) \, \psi_1 \vert_{E[22]}$ and $\phi_2 = 3\psi_2\vert_{E[22]}$, so that $(E, E', \phi_1) \in F_{23 \circ (\borel,I)}^+$ and $(E, E', \phi_2) \in F_{27}$. After fixing a basis for $E[22]$ we have $\phi_1 \phi_2^{-1} = (g_{\borel}, g_I) (12 - 33\sqrt{-5}) \vert_{E[22]} = \mymat{1}{0}{0}{1}$. In particular $F_{23 \circ (\borel,I)}^+$ and $F_{27}$ intersect at the point $P = (E, E', \phi_1)$. Swapping the roles of $E$ and $E'$ gives a second point of intersection $P' \neq P, \tau(P)$. The claim follows.
\begin{figure}[H]
  \centering
  \begin{tikzpicture}[x=1cm,y=1cm,yscale=0.8,xscale=0.8]
    \draw [shift={(-2.1,0)}]  plot[domain=-0.9336408172308062:0.9336408172308065,variable=\t]({1*2.9*cos(\t r)+0*2.9*sin(\t r)},{0*2.9*cos(\t r)+1*2.9*sin(\t r)});
    \draw [shift={(2.1,0)}]  plot[domain=2.207951836358987:4.075233470820599,variable=\t]({1*2.9*cos(\t r)+0*2.9*sin(\t r)},{0*2.9*cos(\t r)+1*2.9*sin(\t r)});
    \draw[dotted,thick] (0.3,0)-- (2.3,0);
    \begin{scriptsize}
      \draw[anchor=east] (-0.3065983510131155,2.4945390669049314) node {$\left( F_{23 \circ (\borel,I)}^+ \right)^*$};
      \draw[color=black] (0.8809453212619466,-0.9) node {$-3$};
      \draw[anchor=west] (0.4420704858559453,2.4945390669049314) node {$F_{27}^*$};
      \draw[color=black] (-1,-0.9) node {$-2$};
      \draw[anchor=west] (2.3,0) node {$F_{23}^*$};
      \draw[color=black,anchor=south] (1.5,0) node {$-1$};
    \end{scriptsize}
  \end{tikzpicture}
  \caption{A configuration of curves on $\WNr{22}{1}$ whose image on $\WNro{22}{1}$ is a pseudo-elliptic configuration of type $\mathrm{I}_2$. We record the values $-(2 + K_W \cdot C)$ for each component $C$ (which is equal to $C^2$ if $C$ is smooth). The dotted curve $F_{23}^*$ is blown down in forming $\WNro{22}{1}$. We do not claim that the components are smooth, nor that the intersections are transversal.}
  \label{fig:22-1}
\end{figure}
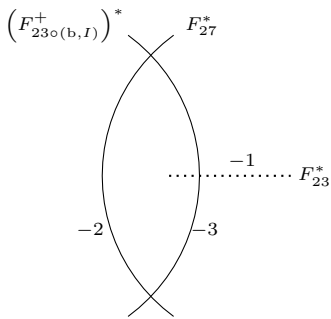

 
\case{The case $(N,r) = (24,11)$} \pdfbookmark[3]{The case (N,r)=(24,11)}{24-11}
We claim that the image of the curve
\begin{equation*}
  \mathscr{C}_{24,11} = (F_{11 \circ (\borel, I)}^+)^* + (F_{11 \circ 11(\borel, I)}^+)^* + E_{\infty,24,11}^*
\end{equation*}
contains a pseudo-elliptic configuration of type $\mathrm{I}_8$.

The curves $(F_{11 \circ (\borel, I)}^+)^*$ and $(F_{11 \circ 11(\borel, I)}^+)^*$ meet the same points of $\ZNr{24}{11}$ which resolve to components of $E_{\infty,24,11}$ and $E_{\infty, 12, 11}$. Indeed, by the same argument as in the case $(22,1)$ we see that $(F_{11 \circ (\borel, I)}^+)^*$ and $(F_{11 \circ 11(\borel, I)}^+)^*$ meet the same component of $E_{\infty, 12, 11}$ (which is blown down in forming $\WNro{24}{11}$) and that
\begin{equation*}
  K_W \cdot (F_{11 \circ (\borel, I)}^+)^* \leq \frac{1}{2} \left( 36/3 - 4 - 6 \right) = 1.
\end{equation*}
In particular $\KWo \cdot (F_{11 \circ (\borel, I)}^+)^\oo \leq 0$, the curves $(F_{11 \circ (\borel, I)}^+)^\oo$ and $(F_{11 \circ 11(\borel, I)}^+)^\oo$ intersect and meet the first and last components of a chain in $E_{\infty,24,11}^\oo$ respectively.

Finally, note that $E_{\infty,24,11}^*$ consists of $(-3,-2,-2,-2,-2,-3)$-chains. The first and last components of each chain meet a component of $\widetilde{F}_{11}$. Since the curves $F_{11,\lambda}^*$ are blown down in forming $\WNro{24}{11}$ the divisor $E_{\infty,24,11}^\oo$ consists of $(-2,-2,-2,-2,-2,-2)$-chains, as required.
\begin{figure}[H]
  \centering
  \begin{tikzpicture}[x=1.0cm,y=1.0cm,scale=0.8]
    \draw (2,0)-- (5,1);
    \draw (4,1)-- (7,0);
    \draw (6,0)-- (9,1);
    \draw (8,1)-- (11,0);
    \draw (10,0)-- (12.72,0.91);
    \draw[dotted,thick] (0.25,3)-- (12.75,3);
    \draw (12.25,3.5)-- (12.25,0.25);
    \draw (0.75,3.5)-- (0.75,0.25);
    \draw (0.28,0.91)-- (3,0);
    \draw[dotted,thick] (1.8,0.93)-- (1.16,-0.97);
    \draw[dotted,thick] (11.34,0.97)-- (11.97,-0.92);
    \begin{scriptsize}
      \draw[color=black,anchor=south] (0.75,3.5) node {$\left(F_{11 \circ (\borel,I)}^+\right)^*$};
      \draw[color=black,anchor=south] (12.25,3.5) node {$\left(F_{11 \circ 11(\borel,I)}^+\right)^*$};
      \draw[color=black,anchor=south] (1.8,0.93) node {$F_{11,1}^*$};
      \draw[color=black,anchor=south] (11.34,0.97) node {$F_{11,11}^*$};
      \draw[color=black] (1,-0.6) node {$-1$};
      \draw[color=black] (12.1,-0.6) node {$-1$};
      \draw[color=black] (0.3,0.65) node {$-3$};
      \draw[color=black] (12.7,0.65) node {$-3$};
      \draw[anchor=south,color=black] (6.5,3) node {$-2$};
      \draw[anchor=west,color=black] (12.25,2) node {$-3$};
      \draw[anchor=east,color=black] (0.75,2) node {$-3$};
    \end{scriptsize}
  \end{tikzpicture}
  \caption{A comfiguration of curves on $\WNr{24}{11}$ whose image on $\WNro{24}{11}$ is a pseudo-elliptic configuration of type $\mathrm{I}_{8}$. The unlabelled line at the top is a component of $E_{\infty,12,11}^*$ and those at the bottom are components of $E_{\infty,24,11}^*$. We record the value $-(2 + K_W \cdot C)$ (which is equal to $C^2$ if $C$ is smooth) on each component $C$ (when it is not recorded it is equal to $-2$). Dotted lines are blown down in forming $\WNro{24}{11}$. We do not claim that the components are smooth nor that the intersections are transversal.}
  \label{fig:24-11}
\end{figure}
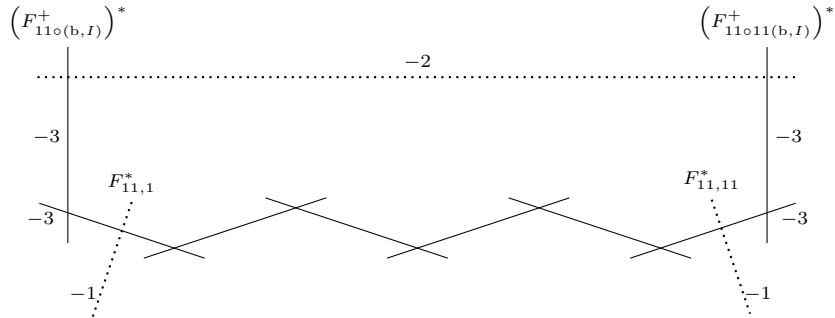


\subsection{A note on minimal models}
\label{sec:rmk-min-models}
Let $K_{{\mathsf{min}}}$ be a canonical divisor for $\WNrmin{N}{r}$ whenever $\WNrmin{N}{r}$ is defined (in the case analysis in \Cref{sec:elliptic-K3,sec:prop-ellipt-cases}). When $\WNr{N}{r}$ is an elliptic surface it is simple to check whether $\WNrmin{N}{r}$ is a minimal model for $\WNr{N}{r}$, since for elliptic fibrations this is equivalent to the condition that $K_{{\textsf{min}}}^2 = 0$ (see \cite[Proposition~IX.3]{B_CAS}). The following corollary is then immediate and extends a result of Hermann\footnote{Hermann's claim is false as stated when $(N,r) = (17,1)$. They claim that the surface which they denote $(Y_{17}^+)^\circ$ (which is birational to $\WNro{17}{1}$ in our notation) is a minimal elliptic surface, though they do not contract the exceptional component of $E_{\infty,17,9}^*$.}~\cite[p.~178]{H_SMDDp2} when $N = 17,19$.

\begin{coro}
  \label{coro:min-models}
  The surface $\WNrmin{N}{r}$ is a minimal model for $\WNr{N}{r}$ when it is defined, i.e., for $(N,r) = (15,7)$, $(17,1)$, $(17,3)$, $(19,1)$, or $(19,2)$.
\end{coro}

\begin{proof}
  The self-intersection $\KWo^2$ is computed in \Cref{lemma:Ko2} (in each case, these integers are recorded in \Cref{tab:invariants-WNr1}). The claim the follows from the construction of $\WNrmin{N}{r}$ (in each case $\WNrmin{N}{r}$ is constructed from $\WNro{N}{r}$ by blowing down $-\KWo^2$ exceptional curves).
\end{proof}

\begin{remark}
  \label{rmk:17-whats-going-on}
  It was observed by Fisher~\cite{F_OPO17CEC} that the (blown up) elliptic K3 surfaces $\WNr{17}{1}$ and $\WNr{17}{2}$ are birational (over $\bbQ$). By \Cref{coro:min-models} our models $\WNrmin{17}{r}$ and $\WNrmin{19}{r}$ are minimal models for $\WNr{17}{r}$. In particular by uniqueness of minimal models the K3 surfaces $\WNrmin{17}{1}$ and $\WNrmin{17}{2}$ are isomorphic (over $\bbC$). Tom Fisher has also shared with us models for $\WNr{19}{1}$ and $\WNr{19}{2}$ which, remarkably, are also birational (and therefore $\WNrmin{19}{1}$ and $\WNrmin{19}{2}$ are isomorphic over $\bbC$). The same is not true in general for the surfaces $\WNr{p}{r}$ (for example the surfaces $\WNr{23}{1}$ and $\WNr{23}{5}$ are not birational since by \Cref{thm:geom-genus} they have different geometric genera).
\end{remark}

It would be interesting to compute minimal models for the surfaces $\WNr{N}{r}$ whenever they are not rational. Indeed, to the best of our knowledge this problem remains open for the surface $\ZNr{N}{r}$, and Hilbert modular surfaces more generally (though several cases are known \cite{H_TUMSZDHMRQK}). We make the following conjecture (cf. \cite{HV_HMSATCOAS} and \cite[Conjecture~VII.4.4]{vdG_HMS}).

\begin{conj}
  \label{conj:min-mods}
  For sufficiently large integers $N$ coprime to $6$ the surfaces $\ZNro{N}{r}$ and $\WNro{N}{r}$ (as defined in \Cref{sec:nonsing-mod,sec:gt-cases}) are minimal.
\end{conj}

\begin{remark}
  It is claimed by Hermann~\cite[Satz~4]{H_SMDDp2} that the techniques of \cite{H_TUMSZDHMRQK} can be used to show that $\ZNro{p}{r}$ is a minimal surface of general type for each prime number $p > 7$ for which $p \equiv 7 \pmod{8}$, though the details are omitted entirely and we have not checked them.
\end{remark}

\appendix
\section{Elements of order \texorpdfstring{$2$ in $\GL_2(\bbZ/N\bbZ)/\{\pm 1\}$}{2 in GL\_2(\unichar{"2124}/N\unichar{"2124})/\{±1\}}}
\label{sec:order-2-app}
The purpose of this appendix is to prove \Cref{lemma:fixer-g-N} which classifies the elements $g \in \GL_2(\bbZ/N\bbZ)$ for which $g^2 = \pm \det(g)$. For the convenience of the reader we provide a \texttt{Magma} script in \cite{ME_ELECTRONIC_HERE} to verify the symbolic claims in the proof of \Cref{lemma:applicable-fixer-g}.

\begin{lemma}
  \label{lemma:applicable-fixer-g}
  Let $p$ be a prime number, $k \geq 1$ be an integer, and let $r$ be coprime to $p$. The conjugacy classes of elements $g \in \GL_2(\bbZ/p^k\bbZ) / \{ \pm 1 \}$ of determinant $r$ for which $g^2 = \pm \det(g)$ are as recorded in \Cref{table:order2-p}.
\end{lemma}

\begingroup
\renewcommand{\arraystretch}{1.3}
\begin{table}[H]
  \centering
  \begin{tabular}{c|c|c|c}
    $p$                  & $k$                       & $r$        & \makecell[c]{Representatives for conjugacy classes of    \\ $g \in \GL_2(\bbZ/p^k\bbZ) / \{\pm 1\}$ with $g^2 = \pm \det(g)$} \\
    \hline
    \multirow{2}{*}{odd} & \multirow{2}{*}{$\geq 1$} & square     & $\lambda g_I$, \, $g_{\weyl}$                              \\
                         &                           & non-square & $g_{\weyl}$                                              \\
    \hdashline
    \multirow{7}{*}{$2$} & $1$                       & $*$        & $\lambda g_I$, \, $\lambda g_{\borel}$, \, $g_{\antidiag}$ \\
    \cdashline{2-4}[0.6pt/2pt]
                         & \multirow{2}{*}{$2$}      & $1$        & $\lambda g_I$, \, $\lambda g_{\borel}$, \, $g_{\antidiag}$ \\
                         &                           & $3$        & $g_{\ns}$, \, $g_{\s}$, \, $g_{\antidiag}$               \\
    \cdashline{2-4}[0.6pt/2pt]
                         & \multirow{4}{*}{$\geq 3$} & $1 \mod 8$ & $\lambda g_I$, \, $\lambda g_{\borel}$, \, $g_{\antidiag}$ \\
                         &                           & $3 \mod 8$ & $g_{\ns}$, \, $\omega g_{\ns}$, \, $g_{\antidiag}$       \\
                         &                           & $5 \mod 8$ & $g_{\antidiag}$                                          \\
                         &                           & $7 \mod 8$ & $g_{\s}$, \, $\omega g_{\s}$, \, $g_{\antidiag}$         \\
  \end{tabular}
  \caption{Conjugacy classes of elements $g \in \GL_2(\bbZ/N\bbZ) / \{\pm 1\}$ for which $g^2 = \pm \det(g)$. Here $\lambda$ ranges over representatives for $\LamN{p^k}$ and $\omega = 2^{k-1} + 1$. The elements $g_\bullet$ are defined in \Cref{def:defining-the-g}.}
  \label{table:order2-p}
\end{table}
\endgroup

\begin{proof}
  When $p$ is odd, this is is \cite[Lemma~3.5]{F_COECAFNSMNGR}, so it suffices to consider the case where $p = 2$. It is simple to check using computer algebra that the claim holds for $k = 1$ and $2$. We therefore assume that $k \geq 3$.

  Let $g = \mymat{a}{b}{c}{d} \in \GL_2(\bbZ/2^k\bbZ)$ have determinant $r$. Note that $g^2 = \pm \det(g)$ if and only if $\mymat{a}{b}{c}{d} = \pm \mymat{d}{-b}{-c}{a}$. In particular either $a = -d$, or $a = d$ and $2b = 2c = 0$.

  If $a = d$ and $2b = 2c = 0$, then $g$ is conjugate to either $\lambda g_I$ or $\lambda g_{\borel}$ for some choice of $\lambda \in \LamN{2^k}$.
  Suppose, therefore, that $a = -d$. If either $b$ or $c$ is invertible we may choose $u,v \in \bbZ/2^k\bbZ$ such that
  \begin{equation*}
    \gamma = \begin{pmatrix} u - bv & a u \\ a v & c u + rv \end{pmatrix}
  \end{equation*}
  is invertible and note that $\gamma^{-1} g \gamma = g_{\antidiag}$.

  Otherwise we have $r = -a^2 - 4b'c' \equiv 3 \pmod{4}$ where $b = 2b'$ and $c = 2c'$. In this case, choose $\alpha \in \bbZ/2^{k}\bbZ$ so that
  \begin{equation*}
    \alpha^2 =
    \begin{cases}
      r/3 & \text{if $r \equiv 3 \pmod{8}$,} \\
      -r  & \text{if $r \equiv 7 \pmod{8}$.} \\
    \end{cases}
  \end{equation*}
  Let $s,t \in \bbZ/2^k\bbZ$ be such that $2 s = a + \alpha$ and $2t = a - \alpha$. There exist $u,v \in \bbZ/2^k\bbZ$ such that the matrix
  \begin{equation*}
    \gamma =
    \begin{cases}
      \mymat{-b'}{0}{t}{\alpha} & \text{if $r \equiv 3 \pmod{8}$,} \\[3mm]
      \mymat{s u - b'v}{-t u - b'v}{c'u + t v}{- c' u + s v} & \text{if $r \equiv 7 \pmod{8}$.}
    \end{cases}   
  \end{equation*}
  is invertible. If $r \equiv 3 \pmod{8}$ then $\gamma^{-1} g \gamma \equiv g_{\ns}(2^k,r) \pmod{2^{k-1}}$ and if $r \equiv 7 \pmod{8}$ then $\gamma^{-1} g \gamma = g_{\s}(2^k,r) \pmod{2^{k-1}}$. Therefore
  \begin{equation*}
    \gamma^{-1} g \gamma =
    \begin{cases}
      \mymat{x}{2^{k-1}y - 2}{2^{k-1}z + 2}{-x} & \text{if $r \equiv 3 \pmod{8}$,} \\[3mm]
      \mymat{x}{2^{k-1}y}{2^{k-1}z}{-x} & \text{if $r \equiv 7 \pmod{8}$.}
    \end{cases}
  \end{equation*}
  Conjugating $\gamma^{-1} g \gamma$ by the matrix
  \begin{equation*}
    \begin{pmatrix}
      1 & 2^{k-y-1} \\
      2^{k-z-1} & 1
    \end{pmatrix}
  \end{equation*}
  gives $\omega g_{\ns}(2^k, r)$, respectively $\omega g_{\s}(2^k, r)$, for some choice of $\omega \in \LamN{2^k}$, as required.
\end{proof}

\begin{proof}[Proof of \Cref{lemma:fixer-g-N}]
  First suppose that $N = M$ is odd, so that $k = 0$. Note that by \Cref{lemma:applicable-fixer-g} after conjugating by an element of $\GL_2(\bbZ/M\bbZ)$ we are free to assume that if $g^2 = \pm \det(g)$, then $g \equiv \lambda g_I$ or $\pm g_{\weyl} \pmod{p^k}$ for each prime power $p^k$ dividing $M$. Note that $(\lambda g_I)^2 = \det(\lambda g_I)$ and $g_{\weyl}^2 = -\det(g_{\weyl})$. Also note that $g_{\weyl}$ is conjugate to $-g_{\weyl}$. In particular, we may assume that either $g = \lambda g_I$ for some $\lambda \in \LamN{M}$, or $g \equiv g_{\weyl} \pmod{p^k}$ for each prime power $p^k$ dividing $M$. But the latter condition uniquely determines an element of $\GL_2(\bbZ/M\bbZ)$.

  Now suppose that $k > 0$. By \Cref{lemma:applicable-fixer-g} after conjugating an element of $\GL_2(\bbZ/N\bbZ)$ we may assume that $g \equiv \lambda g_I$ or $g_{\weyl} \pmod{M}$ and that $g \equiv \lambda g_I$, $\lambda g_{\borel}$, $\pm \omega g_{\s}$, $\pm \omega g_{\ns}$, or $\pm g_{\antidiag} \pmod{2^k}$ where $\lambda \in \LamN{N}$ and $\omega \in \LamN{2^k}$. In $\GL_2(\bbZ/2^k\bbZ)$ we have $(\lambda g_I)^2 = \det(\lambda g_I)$, $(\lambda g_{\borel})^2 = \det(\lambda g_{\borel})$, $(\omega g_{\s})^2 = -\det(\omega g_{\s})$, $(\omega g_{\ns})^2 = -\det(\omega g_{\ns})$, and $g_{\antidiag}^2 = -\det(g_{\antidiag})$ for each $\lambda \in \LamN{2^k}$. Hence if $g \equiv \lambda g_I \pmod{M}$ then either $g \equiv \lambda I$ or $\lambda g_{\borel} \pmod{2^k}$ for some choice of $\lambda \in \LamN{N}$. As above note that if $h \in \{\omega g_{\ns}, \omega g_{\s}, g_{\antidiag}\}$ then $h$ is conjugate to $-h$, so if $g \equiv g_{\weyl} \pmod{M}$ then we may take $g \equiv \omega g_{\s}$, $\omega g_{\ns}$, or $g_{\antidiag} \pmod{2^k}$.
\end{proof}

\section{Tables of numerical invariants}
\label{sec:table-numer-invar}
For the convenience of the reader, in \Cref{tab:invariants-WNr1}, we list the birational invariants $p_g(\ZNrtil{N}{r})$, $\kappa(\ZNrtil{N}{r})$, $p_g(\WNr{N}{r})$, and $\kappa(\WNr{N}{r})$ of the surface $\WNr{N}{r}$ for each $6 \leq N \leq 33$. In addition we record the intersection numbers $K_{\overbar{W}} \cdot \overbar{C}_\infty$ (as defined in \Cref{prop:rational-cases}), together with $K_W^2$ (defined in \Cref{sec:nonsing-mod}), and $\KWo^2$ (defined in \Cref{sec:prop-ellipt-cases}).

The numerical invariants in \Cref{tab:invariants-WNr1} were generated using the \texttt{python} code accompanying this article, which is freely available at \cite{ME_ELECTRONIC_HERE}.

\begingroup
\renewcommand{\arraystretch}{1.2}
\begin{table}[p]
  \centering
  \begin{tabular}{cc|ccccccc}
    $N$ & $r$ & $p_g(\ZNrtil{N}{r})$ & $\kappa(\ZNrtil{N}{r})$ & $p_g(\WNr{N}{r})$ & $K_{\overbar{W}} \cdot \overbar{C}_\infty$ & $K_W^2$ & $\KWo^2$ & $\kappa(\WNr{N}{r})$ \\
    \hline
    \multirow{1}{*}{$6$} & $\boldsymbol{5}$ & $1$ & $0$ & $0$ & $-4$ & $-2$ & $ $ & $-1$\\
    \hdashline
    \multirow{1}{*}{$7$} & $\boldsymbol{3}$ & $1$ & $0$ & $0$ & $-4$ & $-1$ & $ $ & $-1$\\
    \hdashline
    \multirow{3}{*}{$8$} & $3$ & $1$ & $0$ & $0$ & $-4$ & $-3$ & $ $ & $-1$\\
    & $5$ & $1$ & $0$ & $0$ & $-3$ & $0$ & $ $ & $-1$\\
    & $\boldsymbol{7}$ & $2$ & $1$ & $0$ & $-5$ & $-9$ & $ $ & $-1$\\
    \hdashline
    \multirow{2}{*}{$9$} & $1$ & $1$ & $0$ & $0$ & $-3$ & $-2$ & $ $ & $-1$\\
    & $\boldsymbol{2}$ & $2$ & $1$ & $0$ & $-5$ & $-6$ & $ $ & $-1$\\
    \hdashline
    \multirow{2}{*}{$10$} & $\boldsymbol{1}$ & $2$ & $1$ & $0$ & $-5$ & $-7$ & $ $ & $-1$\\
    & $3$ & $2$ & $1$ & $0$ & $-3$ & $-7$ & $ $ & $-1$\\
    \hdashline
    \multirow{2}{*}{$11$} & $1$ & $2$ & $1$ & $0$ & $-3$ & $-6$ & $ $ & $-1$\\
    & $\boldsymbol{2}$ & $3$ & $2$ & $0$ & $-3$ & $-10$ & $ $ & $-1$\\
    \hdashline
    \multirow{4}{*}{$12$} & $1$ & $1$ & $0$ & $0$ & $-3$ & $-3$ & $ $ & $-1$\\
    & $5$ & $4$ & $2$ & $0$ & $-4$ & $-14$ & $ $ & $-1$\\
    & $7$ & $3$ & $2$ & $0$ & $-3$ & $-14$ & $ $ & $-1$\\
    & $\boldsymbol{11}$ & $6$ & $2$ & $0$ & $-6$ & $-28$ & $ $ & $-1$\\
    \hdashline
    \multirow{2}{*}{$13$} & $\boldsymbol{1}$ & $4$ & $2$ & $0$ & $-2$ & $-13$ & $ $ & $-1$\\
    & $2$ & $4$ & $2$ & $0$ & $-1$ & $-12$ & $ $ & $-1$\\
    \hdashline
    \multirow{2}{*}{$14$} & $1$ & $4$ & $2$ & $0$ & $-2$ & $-15$ & $ $ & $-1$\\
    & $\boldsymbol{3}$ & $6$ & $2$ & $0$ & $-3$ & $-23$ & $ $ & $-1$\\
    \hdashline
    \multirow{4}{*}{$15$} & $1$ & $4$ & $2$ & $0$ & $-2$ & $-15$ & $ $ & $-1$\\
    & $2$ & $6$ & $2$ & $0$ & $-2$ & $-20$ & $ $ & $-1$\\
    & $7$ & $8$ & $2$ & $2$ & $2$ & $-6$ & $0$ & $1$\\
    & $\boldsymbol{11}$ & $10$ & $2$ & $0$ & $-4$ & $-35$ & $ $ & $-1$\\
    \hdashline
    \multirow{4}{*}{$16$} & $1$ & $4$ & $2$ & $0$ & $-1$ & $-13$ & $ $ & $-1$\\
    & $3$ & $7$ & $2$ & $0$ & $-1$ & $-26$ & $ $ & $-1$\\
    & $5$ & $8$ & $2$ & $1$ & $1$ & $-14$ & $-4$ & $0$\\
    & $\boldsymbol{7}$ & $11$ & $2$ & $0$ & $-3$ & $-43$ & $ $ & $-1$\\
    \hdashline
    \multirow{2}{*}{$17$} & $\boldsymbol{1}$ & $10$ & $2$ & $1$ & $3$ & $-20$ & $-2$ & $0$\\
    & $3$ & $10$ & $2$ & $1$ & $4$ & $-16$ & $-2$ & $0$\\
    \hdashline
    \multirow{2}{*}{$18$} & $1$ & $8$ & $2$ & $1$ & $1$ & $-18$ & $-6$ & $0$\\
    & $\boldsymbol{5}$ & $13$ & $2$ & $0$ & $-2$ & $-46$ & $ $ & $-1$\\
    \hdashline
    \multirow{2}{*}{$19$} & $1$ & $14$ & $2$ & $2$ & $7$ & $-16$ & $-1$ & $1$\\
    & $\boldsymbol{2}$ & $15$ & $2$ & $2$ & $7$ & $-19$ & $0$ & $1$\\
    \hdashline
    \multirow{4}{*}{$20$} & $1$ & $12$ & $2$ & $1$ & $1$ & $-27$ & $-6$ & $0$\\
    & $3$ & $14$ & $2$ & $1$ & $2$ & $-34$ & $-2$ & $0$\\
    & $\boldsymbol{11}$ & $18$ & $2$ & $0$ & $-1$ & $-62$ & $ $ & $-1$\\
        & $13$ & $16$ & $2$ & $4$ & $6$ & $-3$ & $5$ & $2$\\
    \hdashline
    \multirow{4}{*}{$21$} & $1$ & $15$ & $2$ & $4$ & $8$ & $-3$ & $9$ & $2$\\
        & $2$ & $17$ & $2$ & $1$ & $4$ & $-37$ & $-6$ & $0$\\
        & $\boldsymbol{5}$ & $25$ & $2$ & $2$ & $4$ & $-52$ & $-2$ & $1$\\
        & $10$ & $23$ & $2$ & $6$ & $10$ & $-2$ & $12$ & $2$\\
  \end{tabular}
  \caption{Numerical invariants of the surfaces $\WNr{N}{r}$ (which are defined whenever $\ZNr{N}{r}$ is not rational). The surface $\WNro{N}{r}$ is defined only when $\WNr{N}{r}$ is not rational, so the column recording $\KWo^2$ is left empty if $\WNr{N}{r}$ is rational. For bolded $r$ the surface $\WNr{N}{r}$ is birational to the Humbert surface $\mathcal{H}_{N^2}$ in \Cref{coro:humbert-geom}.}
  \label{tab:invariants-WNr1}
\end{table}

\begin{table}[p]
  \ContinuedFloat
  \captionsetup{list=off,format=cont}
  \centering
  \begin{tabular}{cc|ccccccc}
     $N$ & $r$ & $p_g(\ZNrtil{N}{r})$ & $\kappa(\ZNrtil{N}{r})$ & $p_g(\WNr{N}{r})$ & $K_{\overbar{W}} \cdot \overbar{C}_\infty$ & $K_W^2$ & $\KWo^2$ & $\kappa(\WNr{N}{r})$ \\
    \hline
    \multirow{2}{*}{$22$} & $1$ & $17$ & $2$ & $2$ & $6$ & $-27$ & $-3$ & $1$\\
    & $\boldsymbol{7}$ & $23$ & $2$ & $3$ & $7$ & $-35$ & $1$ & $2$\\
    \hdashline
    \multirow{2}{*}{$23$} & $1$ & $23$ & $2$ & $4$ & $15$ & $-9$ & $5$ & $2$\\
    & $\boldsymbol{5}$ & $32$ & $2$ & $7$ & $17$ & $-3$ & $22$ & $2$\\
    \hdashline
    \multirow{8}{*}{$24$} & $1$ & $13$ & $2$ & $3$ & $6$ & $-12$ & $0$ & $2$\\
    & $5$ & $25$ & $2$ & $3$ & $6$ & $-32$ & $0$ & $2$\\
    & $7$ & $25$ & $2$ & $3$ & $6$ & $-42$ & $0$ & $2$\\
    & $11$ & $29$ & $2$ & $2$ & $4$ & $-64$ & $-4$ & $1$\\
    & $13$ & $21$ & $2$ & $6$ & $10$ & $10$ & $14$ & $2$\\
    & $17$ & $25$ & $2$ & $3$ & $6$ & $-34$ & $0$ & $2$\\
    & $19$ & $25$ & $2$ & $6$ & $10$ & $-8$ & $14$ & $2$\\
    & $\boldsymbol{23}$ & $37$ & $2$ & $0$ & $0$ & $-124$ & $ $ & $-1$\\
    \hdashline
    \multirow{2}{*}{$25$} & $\boldsymbol{1}$ & $36$ & $2$ & $6$ & $16$ & $-20$ & $13$ & $2$\\
    & $2$ & $36$ & $2$ & $10$ & $21$ & $23$ & $36$ & $2$\\
    \hdashline
    \multirow{2}{*}{$26$} & $\boldsymbol{1}$ & $33$ & $2$ & $5$ & $13$ & $-32$ & $9$ & $2$\\
    & $5$ & $33$ & $2$ & $6$ & $15$ & $-12$ & $12$ & $2$\\
    \hdashline
    \multirow{2}{*}{$27$} & $1$ & $42$ & $2$ & $13$ & $24$ & $40$ & $50$ & $2$\\
    & $\boldsymbol{2}$ & $46$ & $2$ & $7$ & $18$ & $-32$ & $17$ & $2$\\
    \hdashline
    \multirow{4}{*}{$28$} & $1$ & $38$ & $2$ & $11$ & $18$ & $25$ & $38$ & $2$\\
    & $\boldsymbol{3}$ & $47$ & $2$ & $5$ & $13$ & $-73$ & $9$ & $2$\\
    & $5$ & $43$ & $2$ & $10$ & $17$ & $10$ & $36$ & $2$\\
    & $11$ & $42$ & $2$ & $7$ & $16$ & $-28$ & $12$ & $2$\\
    \hdashline
    \multirow{2}{*}{$29$} & $\boldsymbol{1}$ & $59$ & $2$ & $15$ & $34$ & $40$ & $71$ & $2$\\
    & $2$ & $59$ & $2$ & $16$ & $35$ & $54$ & $69$ & $2$\\
    \hdashline
    \multirow{4}{*}{$30$} & $1$ & $39$ & $2$ & $8$ & $14$ & $-15$ & $17$ & $2$\\
    & $7$ & $47$ & $2$ & $15$ & $20$ & $49$ & $63$ & $2$\\
    & $\boldsymbol{11}$ & $57$ & $2$ & $4$ & $10$ & $-108$ & $3$ & $2$\\
    & $17$ & $49$ & $2$ & $8$ & $16$ & $-32$ & $18$ & $2$\\
    \hdashline
    \multirow{2}{*}{$31$} & $1$ & $69$ & $2$ & $19$ & $41$ & $73$ & $91$ & $2$\\
    & $\boldsymbol{3}$ & $78$ & $2$ & $22$ & $43$ & $80$ & $111$ & $2$\\
    \hdashline
    \multirow{4}{*}{$32$} & $1$ & $55$ & $2$ & $15$ & $28$ & $45$ & $61$ & $2$\\
    & $3$ & $63$ & $2$ & $14$ & $28$ & $15$ & $57$ & $2$\\
    & $5$ & $67$ & $2$ & $20$ & $32$ & $76$ & $92$ & $2$\\
    & $\boldsymbol{7}$ & $75$ & $2$ & $11$ & $24$ & $-60$ & $42$ & $2$\\
    \hdashline
    \multirow{4}{*}{$33$} & $1$ & $80$ & $2$ & $29$ & $42$ & $131$ & $147$ & $2$\\
    & $\boldsymbol{2}$ & $88$ & $2$ & $17$ & $34$ & $-12$ & $66$ & $2$\\
    & $5$ & $84$ & $2$ & $19$ & $36$ & $37$ & $80$ & $2$\\
    & $7$ & $84$ & $2$ & $28$ & $42$ & $130$ & $154$ & $2$\\
  \end{tabular}
  \caption{Numerical invariants of the surfaces $\WNr{N}{r}$ (which are defined whenever $\ZNr{N}{r}$ is not rational). The surface $\WNro{N}{r}$ is defined only when $\WNr{N}{r}$ is not rational, so the column recording $\KWo^2$ is left empty if $\WNr{N}{r}$ is rational. For bolded $r$ the surface $\WNr{N}{r}$ is birational to the Humbert surface $\mathcal{H}_{N^2}$ in \Cref{coro:humbert-geom}.}
  \label{tab:invariants-WNr2}
\end{table}
\endgroup

\FloatBarrier

\newcommand{\etalchar}[1]{$^{#1}$}
\providecommand{\bysame}{\leavevmode\hbox to3em{\hrulefill}\thinspace}
\providecommand{\MR}{\relax\ifhmode\unskip\space\fi MR }
\providecommand{\MRhref}[2]{%
  \href{http://www.ams.org/mathscinet-getitem?mr=#1}{#2}
}
\providecommand{\bibtitleref}[2]{%
  \hypersetup{urlbordercolor=0.8 1 1}%
  \href{#1}{#2}%
  \hypersetup{urlbordercolor=cyan}%
}
\providecommand{\href}[2]{#2}

\end{document}